\documentclass[12pt]{amsart}
\usepackage{amsmath}
\usepackage{amsthm}
\usepackage{amssymb}
\usepackage{mathtools}
\usepackage{amscd}
\usepackage{hyperref}

\usepackage{stackengine}

\usepackage{color}
\usepackage{pgfplots}
\pgfplotsset{compat=1.15}

\usepackage{tikz-cd}
\usepackage[all,cmtip]{xy}

\newcommand{\gothic}{\mathfrak}
\newcommand{\p}{{\gothic{p}}}
\newcommand{\q}{{\gothic{q}}}

\newcommand{\Q}{{\mathbb {Q}}}
\newcommand{\m}{{\gothic{m}}}
\newcommand{\n}{{\gothic{n}}}
\newcommand{\mf}{\mathfrak}
\newcommand{\hght}{\operatorname{ht}}
\newcommand{\cotimes}[1]{\mathbin{\widehat{\otimes_{#1}}}}

\newcommand{\Ann}{\operatorname{Ann}}
\newcommand{\charac}{\operatorname{char}}

\newcommand{\height}{\operatorname{ht}}

\newcommand{\Spec}{\operatorname{Spec}}

\newcommand{\sh}{\operatorname{sh}}

\newcommand{\red}[1]{{#1}_{\textup{red}}}
\newcommand{\unm}[1]{{#1}^{\textup{un}}}

\newcommand{\length}{\ell}
\newcommand{\sO}{\mathcal O}

\renewcommand{\phi}{\varphi}

\DeclareMathOperator{\ord}{ord}
\DeclareMathOperator{\Proj}{Proj}
\DeclareMathOperator{\HH}{H}
\DeclareMathOperator{\lm}{c_{LM}}
\DeclareMathOperator{\grlm}{c^{gr}_{LM}}
\DeclareMathOperator{\limlm}{c^{\infty}_{LM}}
\DeclareMathOperator{\eh}{e}

\DeclareMathOperator{\Gr}{gr}

\DeclareMathOperator{\vol}{vol}
\DeclareMathOperator{\type}{type}

\DeclareMathOperator{\charc}{char}
\DeclareMathOperator{\Bl}{Bl}
\DeclareMathOperator{\init}{in}
\DeclareMathOperator{\edim}{edim}
\DeclareMathOperator{\rank}{rank}
\DeclareMathOperator{\Hom}{Hom}
\DeclareMathOperator{\num}{num}
\DeclareMathOperator{\lcm}{lcm}
\DeclareMathOperator{\sep}{sep}

\DeclareRobustCommand{\stirlingII}[2]{\genfrac{ \{ }{ \} }{0pt}{}{#1}{#2}}

\DeclareRobustCommand{\stirlingI}[2]{\genfrac{ [ }{ ] }{0pt}{}{#1}{#2}}

\newcommand{\N}{\mathbb N}
\newcommand{\CC}{\mathbb C}

\newcommand{\gr}{\text{gr}}

\newtheorem{theoremx}{Theorem}

\newtheorem{theoremR}{Theorem}

\newtheorem{theorem}{Theorem}
\newtheorem{lemma}[theorem]{Lemma}
\newtheorem{proposition}[theorem]{Proposition}
\newtheorem{corollary}[theorem]{Corollary}
\newtheorem{claim}[theorem]{Claim}
\newtheorem{conjecture}[theorem]{Conjecture}
\newtheorem*{statement*}{Statement}
\newtheorem*{theorem*}{Theorem}
\newtheorem*{lemma*}{Lemma}
\newtheorem*{fact*}{Fact}

\theoremstyle{definition}
\newtheorem{definition}[theorem]{Definition}
\newtheorem*{definition*}{Definition}
\newtheorem{example}[theorem]{Example}
\newtheorem*{example*}{Example}

\theoremstyle{remark}
\newtheorem{remark}[theorem]{Remark}
\newtheorem{setting}[theorem]{Setting}
\newtheorem*{remark*}{Remark}
\newtheorem{question}[theorem]{Question}

\numberwithin{theorem}{subsection}
\numberwithin{equation}{theorem}

\begin{document}

\title{Lech--Mumford constant and stability of local rings}
\author{Linquan Ma and Ilya Smirnov}
\address{Department of Mathematics, Purdue University, West Lafayette, IN 47907, USA}
\email{ma326@purdue.edu}
\address{BCAM – Basque Center for Applied Mathematics, Mazarredo 14, 48009 Bilbao, Basque Country – Spain}
\address{Ikerbasque, Basque Foundation for Science, Plaza Euskadi 5, 48009 Bilbao, Basque Country– Spain}
\email{ismirnov@bcamath.org}
\date{\today}

\begin{abstract}
This work systematically develops a theory of the Lech--Mumford constant, an invariant defined 
as an optimal constant in the classical Lech's inequality and underlined Mumford's notion of local semistability.
We establish a number of properties of semistable singularities, and, in particular, 
show that semistable singularities are log canonical under mild assumptions. We also provide new examples of semistable singularities. 
\end{abstract}

\maketitle

\tableofcontents

\addtocontents{toc}{\protect\setcounter{tocdepth}{1}}

\section*{Introduction}

\subsection*{Overview} 
Let $(R,\m)$ be a Noetherian local ring. In \cite{Lech2}, Lech proved a simple inequality relating two basic invariants of an $\m$-primary ideal $I \subseteq R$ -- the Hilbert--Samuel multiplicity (denoted $\eh(I)$) and the colength -- as follows: 

\begin{theoremR}[Lech's inequality {\cite[Theorem 3]{Lech2}}]
\label{thm Lech}
Let $(R,\m)$ be a Noetherian local ring of dimension $d$, and let $I\subseteq R$ be an $\m$-primary ideal. Then $$\eh(I)\leq d!\eh(R)\length(R/I),$$ where $\eh(R) \coloneqq \eh(\m)$.
\end{theoremR}

Lech also observed that his inequality is never sharp when $d \geq 2$ (see \cite[Page 74, after (4.1)]{Lech2}); that is, the inequality in Theorem~\ref{thm Lech} is always strict in this case. This naturally raises the question of whether the constant $\eh(R)$ in Lech's inequality can be improved. As far as the authors are aware, this problem was first considered by Mumford \cite{Mumford}. Hence, for a Noetherian local ring $(R,\m)$ of dimension $d$, it seems proper to introduce the following invariant: 
\[
\lm(R) \coloneqq \sup_{\sqrt{I}=\m}\left\{\frac{\eh(I)}{d!\length(R/I)}\right\}.
\]
We call $\lm(R)$ the {\it Lech--Mumford constant} of $R$. Note that Mumford introduced the same invariant but called and denoted it differently, see Remark~\ref{rmk: LM notation different from Mumford}.

Mumford's motivation for studying this invariant came from understanding singularities on the compactification of the moduli spaces of smooth projective varieties constructed via Geometric Invariant Theory (GIT). More precisely, he proved the following result.

\begin{theoremR}
[{\cite[Proposition~3.12]{Mumford}, see also \cite[Proposition~1.3]{Shah}}]
\label{thm Mumford GIT}
Let $X$ be a projective scheme over a field $k$ and $L$ be an ample line bundle on $X$. Set $N_n \coloneqq \dim_k \HH^0 (X, L^n)$ and
let $\Phi_n\colon X \to \mathbb P_k^{N_n - 1}$ be the map defined by $L^n$ for sufficiently large $n$.
Suppose $(X,L)$ is asymptotically Chow (or Hilbert) semistable,\footnote{Abusing notation, we will abbreviate both of them as ``asymptotically GIT semistable".} that is, for all $n\gg0$, the Chow (or the Hilbert) point corresponding to $\Phi_n(X)$ is semistable under the natural action of $SL(N_n)$. Then for every closed point $x\in X$, the local ring $\sO_{X,x}$ satisfies $\lm(\sO_{X,x}[[T]])=1$.
\end{theoremR}

In accordance with Theorem~\ref{thm Mumford GIT}, Mumford called a Noetherian local  ring $(R, \m)$ {\it semistable} if $\lm(R[[T]]) = 1$ and initiated the study of such singularities, primarily in dimensions one and two.
He also introduced a notion of {\it stable} singularities, whose definition is more technical but is still coming from and inspired by GIT (see Definition~\ref{def: stability of local rings}). Continuing Mumford's work, 
Shah \cite{Shah, Shah2} attempted to classify semistable singularities in dimension two and obtained necessary conditions for such surface singularities. Ultimately, he obtained a list of candidates for semistable singularities of small multiplicity (but did not prove that all singularities on the list are semistable). 


In this paper, we improve and generalize Mumford's and Shah's classification results by providing necessary conditions for semistable singularities in all dimensions. Our results are based on the Minimal Model Program (MMP), which is natural given that the recent progress in the MMP has brought restrictions on singularities appearing on asymptotically GIT semistable varieties. Inspired by Shah's classification, Odaka has shown in \cite{Odaka} that asymptotically GIT semistable varieties have at worst \emph{semi-log canonical} singularities -- a class of singularities arising from the MMP introduced by 
Koll{\' a}r and Shepherd-Barron to present an MMP approach to the compactification of the moduli space of smooth surfaces \cite{KollarShepherdBarron}. The connection between GIT semistability and MMP singularities is mediated by \emph{K-semistability}: it is well-known that asymptotic GIT semistability implies K-semistability (see \cite[Theorem~3.9]{RossThomas}), and Odaka proved that singularities on K-semistable varieties are semi-log canonical. 

Our first theorem established a direct relation between semistable singularities and semi-log canonical singularities:

\begin{theoremx}[Theorem~\ref{thm: lim-stable implies slc}]
\label{thm Main A}
Let $(R,\m)$ be a local ring essentially of finite type over a field of characteristic zero that satisfies $(\textnormal{S}_2)$ and $(\textnormal{G}_1)$. Suppose $R$ is $\Q$-Gorenstein and semistable. Then $R$ is semi-log canonical.   
\end{theoremx}

We note and highlight that Theorem~\ref{thm Main A} is a local result without assuming global geometry behind, and it naturally fits into the following diagram:
\[
\begin{tikzcd}[column sep = huge]
\text{asymptotic GIT semistability} \ar[r, Rightarrow] \ar[d, Rightarrow, "\text{Mumford}"]  & \text{K-semistability} \ar[d, Rightarrow, "\text{Odaka}"] \\
\text{semistable singularities} \ar[r, Rightarrow, "\text{Theorem~\ref{thm Main A}}"] & \text{semi-log canonical singularities}
\end{tikzcd}
\]
We also point out that K-semistable varieties may not have semistable singularities in general, see
\cite[Section 4.2]{WangXu}. Thus, although there have been tremendous breakthroughs in the study of K-stability \cite{XuKstabilityBook} and we hope to relate them with our study of stability of local rings, we have not been able to see a precise connection or application at the moment.

Our approach to Theorem~\ref{thm Main A} is inspired by Odaka's method, and the key step is similar in spirit: using (semi-)log canonical modification to construct a ``destabilizing" graded family of ideals. A crucial difference, however, is that we must reduce to a situation in which the graded family consists of $\m$-primary ideals. This recruitment naturally led us to introduce a weaker notion of {\it lim-stable} singularities, which localizes well by Theorem~\ref{thm: localization}. (By contrast, it remains unclear whether semistability itself localizes.) A further technical difficulty arises while proving that the graded family constructed via the (semi-)log canonical modification actually violates {lim-stability}. This hinges on a delicate analysis of the Hilbert series of certain rational powers of ideals, see Theorem~\ref{thm: lim-stable main technical result slc}.   

In fact, Theorem~\ref{thm: lim-stable implies slc} is stronger than Theorem~\ref{thm Main A} in many ways: we only need to assume lim-stability, and in dimension two, the result holds in all characteristics $\neq 2$. Moreover, when $R$ is normal, we obtain even stronger results in Theorem~\ref{thm: lim-stable implies log canonical general}. In what follows, we will give a more detailed explanation of our main results and our main computations.


\subsection*{Detailed results} 
The definition of semistability is based on the following important feature of the Lech--Mumford constant: 
$$\lm (R) \geq \lm (R[[T]]) \text{ and strict inequality is possible (and often).}$$ 
Moreover, it is easy to see that $\lm(R)\geq 1$ always holds. Thus, the limit 
$$\lim_{n \to \infty} \lm(R[[T_1, \dots, T_n]])$$
exists and we will denote it by $\limlm(R)$. This prompted us to introduce two new classes of singularities: \emph{Lech-stable} (when $\lm(R) = 1$) and \emph{lim-stable} (when $\limlm(R)=1$). These two classes as well as Mumford's notions of local stability are related as follows:
\[
\begin{tikzcd}
\text{Lech-stable} \ar[r, Rightarrow, "\dim(R)\geq 1"]  &[1.7em] \text{stable} \ar[r, Rightarrow] &
\text{semistable} \ar[r, Rightarrow] & \text{lim-stable}.
\end{tikzcd}
\]
The following is our main result on the connection between stability of local rings and MMP singularities when $R$ is normal:

\begin{theoremx}[Theorem~\ref{thm: lim-stable implies log canonical general}, Theorem~\ref{thm: Lech-stable implies canonical}]
\label{thm Main B}
Let $(R,\m)$ be an excellent normal local domain admitting a dualizing complex. Suppose that either
\begin{enumerate}
    \item $\dim(R)\leq 2$, or
    \item $R$ is essentially of finite type over a field of characteristic zero that is numerically $\Q$-Gorenstein.
\end{enumerate}
Then we have that
\begin{enumerate}
    \item[(a)] if $R$ is lim-stable, then $R$ is $\Q$-Gorenstein and log canonical;
\item[(b)] if $R$ is Lech-stable and $R$ has canonical singularities on the punctured spectrum, then $R$ is $\Q$-Gorenstein and canonical. 
\end{enumerate}
\end{theoremx}

We also obtain similar results by relaxing the normality condition, as already mentioned in Theorem~\ref{thm Main A}. We note that in Theorem~\ref{thm Main B} (b), there is an additional assumption on the singularities of the punctured spectrum of $R$. This assumption cannot be removed, see Remark~\ref{rmk: Lech-stable but not canonical} (which relies on a weak semicontinuity result for the Lech--Mumford constant, Proposition~\ref{prop: new weak semicontinuity}). We also point out that lim-stability (or even Lech-stability) alone does not imply $\Q$-Gorenstein in general (Theorem~\ref{thm: max minors are Lech-stable}).

The proof of Theorem~\ref{thm Main B} (a) relies on a variation of log canonical modification essentially due to Hashizume \cite{HashizumeSingularityArbitraryPair}, see Theorem~\ref{thm: Hashizume}, as well as a lower bound on $\limlm(R)$ in terms of the Hilbert--Poincare series of graded families of $\m$-primary ideals, see Proposition~\ref{prop: exponential bound}. The proof then proceeds by examining the ``second coefficient" in the asymptotic colengths of certain rational powers, see Theorem~\ref{thm: criterion for lim unstable via derivative} and Theorem~\ref{thm: lim-stable main technical result}. The proof of part (b) is similar and easier (we use canonical modifications \cite{BCHM}). The key estimate on $\limlm(R)$, Proposition~\ref{prop: exponential bound}, applied to powers of maximal ideals, already yields some consequences: 

\begin{theoremx}[Corollary~\ref{cor: strict complete intersections}, Corollary~\ref{cor: large multiplicity is lim-unstable}, Remark~\ref{rmk: values of E(d)}]
Let $(R,\m)$ be a Noetherian Cohen--Macaulay local ring of dimension $d$. Then 
\begin{itemize}
    \item if $R$ is a strict complete intersection and lim-stable (resp., Lech-stable), then $\eh(R)\leq \edim(R)$ (resp., $\eh(R)<\edim(R)$);
    \item there exists a constant $C(d)$ depending only on $d$ so that if $R$ is lim-stable, then $\eh(R)< C(d)$. Moreover, one can take $C(1)=3$ and $C(2)=17$.  
\end{itemize}
\end{theoremx}

The definition of Lech--Mumford constant involves a supremum, and this creates a major difficulty in computing (or even bounding) the invariant. We have managed to compute this invariant for some surface singularities, which include pseudo-rational singularities ($\lm(R) = \eh(R)/2$, Corollary~\ref{cor: RDP Lech stable}); standard graded normal domains ($\lm(R) = \eh(R)/2$, Corollary~\ref{cor: Watanabe}); simple elliptic singularities and minimally elliptic singularities of degree one (Theorem~\ref{thm: minimal elliptic irred exc} and Theorem~\ref{thm: minimal elliptic deg one}).

We next summarize our main results on classification of (Lech-, semi-, lim-)stable singularities and computations in examples. 

\begin{theoremx}
\label{thm Main D}
    Let $(R, \m)$ be a Noetherian local ring.
    \begin{enumerate}
    \item If $\dim(R)=0$, then Lech-stability, semistability, and lim-stability of $R$ are all equivalent to $R$ being a field. Moreover, $R$ is never stable. (Proposition~\ref{prop: Artinian case})
    \item 
    If $\dim(R) = 1$, then 
    \begin{enumerate}
        \item $R$ is Lech-stable if and only if $R$ is stable if and only if the unmixed part of $\widehat{R}$, $\unm{\widehat{R}}$, is regular. (Proposition~\ref{prop: dimension one Lech-stable})
        \item $R$ is semistable if and only if $R$ is lim-stable if and only if $\unm{\widehat{R}}$ is either regular or a node. (Theorem~\ref{thm: dimension one semistable}) 
    \end{enumerate}

    \item 
    If $\dim (R) = 2$, then Cohen--Macaulay Lech-stable singularities are essentially rational double points, given explicitly via the ADE classification. (Theorem~\ref{thm: Lech-stable CM surface})
    \item If $\dim(R)=2$, then the following singularities are semistable:
    \begin{enumerate}
        \item cone of the twisted cubic curve (Proposition~\ref{prop: cone of twisted cubic semistable});
        \item certain rational triple points (Proposition~\ref{prop: triple A});
        \item elliptic and rational polygonal cones in Mumford's terminology (Proposition~\ref{prop: Mumford elliptic cones} and Proposition~\ref{prop: Mumford rational cones});
        \item Veronese subrings $k[x,y]^{(n)}$ for $n \leq 6$ (Example~\ref{example: Veronese in two variables}); 
        \item most of semi-log canonical hypersurfaces (Theorem~\ref{thm: slc semistable}).
        \end{enumerate}    
    \item In higher dimensions, simple normal crossings are semistable (Theorem~\ref{thm: snc is semistable}), and generic determinantal rings of maximal minors are Lech-stable (Theorem~\ref{thm: max minors are Lech-stable}).
    \end{enumerate}
\end{theoremx}

Establishing Theorem~\ref{thm Main D} relies on a variety of tools used in computing and bounding the Lech--Mumford constant, including the analysis of its behavior under Gr\"{o}bner deformations (Proposition~\ref{prop: LM multigraded}, Corollary~\ref{cor: grober degeneration} and Corollary~\ref{cor: LM monomial ideal}) and the development of sharper, often optimal, versions of Lech's inequality (Theorem~\ref{thm: main Lech incl/excl} and Theorem~\ref{thm: optimal Lech dim three}).

Lastly, we note that while the main motivation of this paper comes from understanding Mumford's notion of semistability \cite{Mumford}, which has geometric origin, our interest on the algebraic side lies in finding refinements, generalizations, and variations of the classical Lech's inequality (Theorem~\ref{thm Lech}). We have developed a series of works in this direction -- see  \cite{KMQSY,MQS,HMQS,MaSmirnovUniformLech,MQSColengthII}. This paper builds on these works by employing several of their techniques and unifying various aspects.

\subsection*{Acknowledgement} During the (long) preparation of this manuscript, we were inspired by and benefited from discussions and conversations with many colleagues. 

We would like to thank Craig Huneke for encouraging us initiating this project and for many fruitful discussions and suggestions, in particular for his suggestion of using rational powers of ideals. We would also like to thank him and Pham Hung Quy for collaborating with us on Lech's inequality which motivates a large portion of this manuscript. 

We would like to thank Joe Waldron and Ziquan Zhuang for their numerous help on the Minimal Model Program and for their patience in answering our questions on the materials presented in Section~\ref{section: MMP}. Among other things, we learned the proof of Theorem~\ref{thm: Hashizume} from Ziquan, and we learned Remark~\ref{rmk: more lim-stable implies log canonical} from Joe. 

We would like to thank Kei-ichi Watanabe for discussions on two-dimensional singularities related to the techniques in Section~\ref{section: surface}. In particular, Lemma~\ref{lem: Watanabe} is due to Kei-ichi. We would also like to thank Hailong Dao for many communications on low-dimensional singularities. In particular, Lemma~\ref{lem: LM for two dimensional stable} and Corollary~\ref{cor: RDP Lech stable} were generalized from the Gorenstein case in discussion with him.

We would like to thank Yuchen Liu for explaining to us several arguments in Mumford's paper \cite{Mumford} and teaching us the degeneration techniques. We would like to thank Chi Li for inspiring us to consider the second coefficient of the colengths of certain graded family of ideals (Proposition~\ref{prop: consequence asymptotic RR}). 

Finally, we would like to thank Harold Blum, Aldo Conca, Igor Dolgachev, David Eisenbud, Hang Huang, Shihoko Ishii, Mircea Musta\c{t}\u{a}, Yuji Odaka, Suchitra Pande, Quentin Posva, Marilina Rossi, Karl Schwede, Shunsuke Takagi, Kevin Tucker, Bernd Ulrich, Jakub Witaszek, and Chenyang Xu for their expertise and comments on various parts of this paper.    

This material is based upon work supported by the National Science Foundation under Grant No. DMS-1928930 and by the Alfred P. Sloan Foundation under grant G-2021-16778, while the author was in residence at the Simons Laufer Mathematical Sciences Institute (formerly MSRI) in Berkeley, California, during the Spring 2024 semester. 

The first author was supported in part by NSF Grant DMS \#2302430. 

The second author was supported by the State Research Agency of Spain through the Ramon y Cajal fellowship RYC2020-028976-I funded by MCIN/AEI/10.13039/501100011033 and by FSE ``invest in your future'' and grants PID2021-125052NA-I00 and EUR2023-143443 funded by MCIN/AEI/10.13039/501100011033 and the European Union NextGenerationEU/PRTR and by BCAM Severo-Ochoa accreditation CEX2021-001142-S funded by MCIN/AEI/10.13039/501100011033. For a period he was supported through a fellowship from “la Caixa” Foundation
(ID 100010434), and the European Union’s Horizon 2020 research and innovation programme under the Marie Sk{\l}odowska-Curie grant agreement No 847648 (fellowship code ``LCF/BQ/PI21/11830033'').

\subsection*{Notations and Conventions}
Throughout the rest of this article, all rings are commutative, Noetherian, with multiplicative identity $1$. We will use $(R, \m)$ to denote a (Noetherian) local ring with unique maximal ideal $\m$. All schemes are Noetherian and separated.

\newpage

\addtocontents{toc}{\protect\setcounter{tocdepth}{2}}
\section{Preliminaries}

\subsection{Hilbert functions, multiplicities, and integral closure}
Let $G = \oplus_{n = 0}^\infty G_n$ be an $\mathbb{N}$-graded ring such that $G_0$ is an Artinian local ring and $G$ is finitely generated over $G_0$. 
For any finitely generated $\mathbb{Z}$-graded $G$-module $M$ one defines the 
\emph{Hilbert function} $H_M(n) \coloneqq \length_{G_0} (M_n)$ 
and the \emph{Hilbert--Poincare series} 
$$h_M(t) \coloneqq \sum_{n = -\infty}^\infty \length_{G_0} (M_n) t^n \in \mathbb{Z}[[t]].$$ 

Since $G$ is a finitely generated $G_0$-algebra, we may choose homogeneous generators $x_1, \ldots, x_s$ and let $d_1, \ldots, d_s > 0$ be their degrees. 
By the Hilbert--Serre theorem \cite[Theorem~11.1]{AtiyahMacDonald}, $h_M(t)$ is a rational function 
of the form
\[
h_M(t) = \frac{f(t)}{\prod_{i = 1}^s (1 - t^{d_i})},
\]
where $f(t) \in \mathbb{Z}[t, t^{-1}]$ (and if $M$ is generated in non-negative degrees, then we have $f(t)\in \mathbb{Z}[t]$) and the order of the pole of 
$h_M(t)$ at $t = 1$ is $\dim (M) - 1$. It follows that $H_M(n)$ is eventually, for $n \gg 0$, given by a quasi-polynomial of degree $\dim(M) - 1$. 
In the case where $d_1 = \cdots = d_s = 1$, the function $H_M(n)$ is eventually a polynomial.
We will refer to this (quasi-)polynomial as the Hilbert (quasi-)polynomial of $M$.

Let $I$ be an $\m$-primary ideal of a local ring $(R, \m)$. The \emph{associated graded ring} of $I$, 
\[
\gr_I (R) \coloneqq R/I \oplus I/I^2 \oplus I^2/I^3 \oplus \cdots
\]
is an $\N$-graded ring generated in degree one over $R/I$.
If $M$ is a finitely generated $R$-module then 
$\gr_I (M) \coloneqq M/I \oplus IM/I^2M \oplus \cdots$ is a finitely 
generated graded $\gr_I (R)$-module. Hence its {Hilbert function} $H_{\gr_I(M)}(n)$ is given by a polynomial of degree $\dim (M) - 1$ for $n\gg0$, which is called the \emph{Hilbert polynomial} of $M$ with respect to $I$.

We will be mostly working with the {\it Hilbert--Samuel} polynomial $p_I(t)$ of $I$ which computes $n \mapsto \length (R/I^n)$ for $n \gg 0$.
The Hilbert and Hilbert--Samuel polynomials are related by the sum/difference transform: if we write the Hilbert--Samuel polynomial as 
\[
p_I(n) = \eh_0(I) \binom{n + d - 1}{d} - \eh_1(I) \binom{n + d - 2}{d - 1} + \cdots + (-1)^d \eh_d(I),
\]
then 
\[
\Delta p_I(n) \coloneqq 
p_I (n+1) - p_I(n) = \eh_0(I) \binom{n + d - 1}{d-1} - \eh_1(I) \binom{n + d - 2}{d - 2} + \cdots + (-1)^{d-1} \eh_{d-1}(I) 
\]
is the Hilbert polynomial of $I$. 
Since the Hilbert--Samuel polynomial is integer-valued, 
the coefficients $\eh_i(I)$ are integers and are called the \emph{Hilbert coefficients} of $I$.  

\begin{definition}
Let $(R,\m)$ be a local ring of dimension $d$ and let $I\subseteq R$ be an $\m$-primary ideal. The integer $\eh(I) \coloneqq \eh_0(I)$ is called the {\it Hilbert--Samuel multiplicity} of $I$ and can be directly defined as
\[
\eh (I) \coloneqq \lim_{n \to \infty} 
\frac{d! \length (R/I^n)}{n^{d}}.
\]
The multiplicity of a finitely generated $R$-module $M$ with respect to $I$ is defined similarly as 
\[
\eh(I, M) \coloneqq \lim_{n \to \infty} 
\frac{d!\length (M/I^nM)}{n^{d}}.
\]
\end{definition}

We note that, under the definition above, $\eh (I, M) = 0$ when $\dim (M) < d$, and if $\dim(M)=d$, then $\eh(I,M)$ is the normalized leading coefficient (i.e., $(d-1)!$ times the leading coefficient) of the Hilbert polynomial of $M$ with respect to $I$. 

The Hilbert--Samuel multiplicity is additive in short exact sequences of $R$-modules, this property yields 
what is often called the {additivity or associativity formula} for multiplicities \cite[Theorem 11.2.4]{SwansonHuneke} which we will use throughout the article: 
\[
\eh (I, M) = \sum_{\mf p} \length (M_\mf p) \eh (I, R/\p), 
\]
where the sum varies through all primes $\p$ such that 
$\dim (R/\p) = \dim (R)$.

The Hilbert--Samuel multiplicity is intrinsically connected to {integral closure} of ideals. We summarize the main ideas and refer to \cite{SwansonHuneke} for more background information. 

\begin{definition}
    An element $x \in R$ is \emph{integral} over an ideal $I \subseteq R$ if it satisfies a monic polynomial equation
    \[
    x^n + a_1x^{n-1} + \cdots + a_n=0,
    \]
    where $a_k \in I^k$ for $k = 1, \ldots, n$. The set of all elements $x$ integral over $I$ forms an ideal and we call this ideal the \emph{integral closure} of $I$ and denoted it by $\overline{I}$. An ideal $I$ is \emph{integrally closed} if $I = \overline{I}$ and is {\it normal} if $I^m = \overline{I^m}$ for all $m\in\N$. 
\end{definition}

We say that an ideal $Q \subseteq J$ is a {\it reduction} of $J$ if there is an integer $k$ such that 
$QJ^k = J^{k +1}$. It is well-known that $Q$ is a reduction of $J$ if and only if $J \subseteq \overline{Q}$ (see \cite[Corollary 1.2.5]{SwansonHuneke}, note that we always work with Noetherian rings). It follows easily that $\eh (I) = \eh (\overline{I})$ for all $\m$-primary ideals $I\subseteq R$. Conversely, a celebrated theorem of Rees shows that if $(R,\m)$ is formally equidimensional (i.e., the $\m$-adic completion $\widehat{R}$ is equidimensional), then for two $\m$-primary ideals $I\subseteq J$, we have $\eh (I) = \eh (J)$ if and only if $J \subseteq \overline{I}$ \cite[Theorem 11.3.1]{SwansonHuneke}. Thus the study of multiplicities can often be reduced to integrally closed ideals.

In the opposite way, the study of multiplicities can be often reduced to {parameter ideals}, i.e., ideals generated by a system of parameters. It is well-known that for a local ring $(R,\m)$ with an infinite residue field, 
any $\m$-primary ideal $I\subseteq R$ has a reduction generated by a system of parameters, and any such reduction is called a {\it minimal reduction} of $I$. 
In fact, it turns out that any $d=\dim(R)$ \emph{general} elements $x_1, \ldots, x_d \in I$ form a minimal reduction of $I$, see \cite[Theorem 8.6.6]{SwansonHuneke}. Here, we define general elements of $I$ by equipping the vector space $I/\m I$ with the Zariski topology. 

More generally, the theory of superficial elements extends the formula $\eh(I) = \eh(J)$ to shorter sequences, thus providing a suitable tool for inductive proofs, see \cite[Section 8.6]{SwansonHuneke}. We record the following fact which is a consequence of \cite[Proposition 11.1.9]{SwansonHuneke}. If $(R,\m)$ is a local ring of dimension $d$ with an infinite residue field, then for any $\m$-primary ideal $I\subseteq R$ and any general elements $x_1, \ldots, x_k \in I$ with $k < d$, we have
$$\eh (I, R) = \eh (IR/(x_1, \ldots, x_k), R/(x_1, \ldots, x_k)).$$

We next recall some basics on mixed multiplicities that will be used in this paper. Mixed multiplicities originate from the work of Risler and Teissier (see \cite[Section 2]{TeissierMixed}) which uses the work of Bhattacharya \cite{Bhattacharya}. Given two $\m$-primary ideals $I, J$ in a local ring $(R, \m)$ of dimension $d$, the function $(n, m) \mapsto \length (R/I^nJ^m)$
is eventually, i.e., for $n,m \gg 0$, given by a polynomial function of degree $d$.
One then defines the {\it mixed multiplicities} 
$\eh (I^{[k]} \mid J^{[d-k]})$
by writing the degree $d$ bihomogeneous part of that polynomial as 
\[
\frac{1}{d!} \sum_{k = 0}^d \binom{d}{d-k} \eh (I^{[k]} \mid J^{[d-k]}) n^km^{d-k}.
\]
All mixed multiplicities are positive integers, 
and we have $\eh(I) = \eh (I^{[d]} \mid J^{[0]})$ and 
$\eh(J)= \eh (I^{[0]} \mid J^{[d]})$. 
In \cite{TeissierMixed}, the theory of superficial sequences was extended to prove that
\[
\eh (I^{[k]} \mid J^{[d-k]}) = \eh (IR/(x_1, \ldots, x_{d-k}))
\]
for $k\geq 1$, where $x_1, \ldots, x_{d-k} \in J$ are general elements (assuming $R$ has an infinite residue field). We will mostly use mixed multiplicities when $\dim(R) = 2$. In this case, we will clean up the notation by using $\eh (I \mid J)$ instead of $\eh (I^{[1]} \mid J^{[1]})$.

We will often express multiplicities in terms of intersection numbers on certain blowups. For this we will use intersection theory for proper schemes over an Artinian local ring, see \cite{CutkoskyMontanoMultiplicities} or \cite[Chapter VI.2]{KollarBookRationalCurves}. This theory was originated from \cite{SnapperMultiplesDivisors}, and was further developed in  \cite{MumfordLecturesCurvesSurface,KleimanNumericalTheoryAmpleness}. 
We refer to \cite[Theorem A]{CutkoskyMontanoMultiplicities} for the generality stated here.
Let $(R,\m)$ be a local ring of dimension $d$ and let $I\subseteq R$ be an $\m$-primary ideal. Let $Y\to \Spec(R)$ be a projective birational map such that $I\cdot \sO_Y \cong \sO_Y(-E)$ is an invertible sheaf corresponding to the (non-effective) Cartier divisor $-E$. Then we have 
$$\eh(I) = -(-E)^{d}.$$
One way to interpret the intersection product $-(-E)^d$ is to restrict the line bundle $\sO_Y(-E)$ to $E$ and then compute the intersection number $(\sO_Y(-E)|_E)^{d-1}$, here $E$ is projective over an Artinian local ring $R/I$ and the intersection product is defined in \cite[Definition I.2.1]{KleimanNumericalTheoryAmpleness}, see also \cite[Chapter VI.2]{KollarBookRationalCurves} and \cite[Chapter 20]{FultonBookIntersection} for intersection theory in more general setups. Similarly, we can also express mixed multiplicities as intersection numbers on certain blowups \cite[Theorem E]{CutkoskyMontanoMultiplicities}.

\subsection{Rational powers}
We will use the theory of rational powers of ideals. We begin by collecting some definitions.  

\begin{definition}
    Let $I\subseteq R$ be an ideal and let $a, b \in\mathbb{N}$. We define the $(b/a)$-th {\it rational power} of $I$ as  
    $$I^{\frac b a} \coloneqq \{x \in R \mid x^a \in \overline{I^b}\}.$$
\end{definition}

\begin{remark}
It is easy to see that rational powers are integrally closed and that $I^{\frac n 1} = \overline{I^n}$. By the valuation criterion of integral closure \cite[Theorem 6.8.3]{SwansonHuneke}, we have $x \in I^{\frac b a}$ if and only if $\nu (x) \geq (b/a) \nu (I)$ for all valuations $\nu$ supported at the minimal primes of $R$.
It follows easily that rational powers are well-defined, i.e., $I^{\frac b a}= I^{\frac {b'}{a'}}$ whenever $b/a=b'/a'$. 
\end{remark}

Let $I\subseteq R$ be an ideal. There exists finitely many discrete valuations $v_i$ supported at minimal primes of $R$ so that for each $n$, $x\in \overline{I^n}$ if and only if $v_i(x)\geq n$ for all $i$. The minimal collection of such $v_i$'s are called the {\it Rees valuations} of $I$, see \cite[Definition 10.1.1]{SwansonHuneke} and Proposition~\ref{prop: geometric rational powers} for the construction. It follows that $x \in I^{\frac b a}$ if and only if $\nu (x) \geq (b/a) \nu (I)$ for all Rees valuations $\nu$ of $I$. 

\begin{definition}
    We define the \emph{Rees period} of $I\subseteq R$ as 
    \[
    \rho(I) \coloneqq \lcm \{v_i (I) \mid v_i \text{ are the Rees valuations of } I\}.
    \]
\end{definition}

The Rees period is a common denominator for all rational powers, i.e., if $\rho$ is the Rees period of $I$ and $\alpha \in \mathbb{Q}_{\geq 0}$, then $I^\alpha = I^\frac{\lceil \rho \alpha \rceil}{\rho}$, see \cite[Proposition~10.5.5]{SwansonHuneke}. The common denominator allows one to show that the Hilbert--Samuel function of rational powers of $I$
is eventually a quasi-polynomial. This fact is well-known to experts and is closely related to \cite[Proposition~2.13]{BDHM} (and alternatively can be deduced from Corollary~\ref{cor: Riemann-Roch for rational powers} in Section~\ref{section: MMP}). We record it here for completeness. 

\begin{proposition}\label{prop: rational powers are quasi-polynomial}
    Let $(R, \m)$ be an analytically unramified local ring (i.e., the $\m$-adic completion $\widehat{R}$ is reduced) of dimension $d \geq 1$
    and let $I\subseteq R$ be an $\m$-primary ideal. 
    If $\rho$ is the Rees period of $I$ then 
    for any integer $0 \leq r < \rho$
    the function 
    $
    n \mapsto \length (R/I^\frac {n\rho + r}{\rho})
    $
    is given by a polynomial for $n \gg 0$.
    Moreover, 
    \[
        \length (R/I^\frac {n\rho + r}{\rho}) = \frac{\eh(I)}{d!}n^d + O(n^{d-1}). 
    \]
\end{proposition}
\begin{proof}
It was shown in \cite[Proposition~2.13]{BDHM} (note that its proof is still valid when $R$ is not necessarily a domain) that there exists $a \in \mathbb{Z}$ such that $I^{\frac{w + \rho n}{\rho}} = \overline{I^n} I^\frac{w}{\rho}$ for all integers $w \geq a$ and $n \geq 1$.
Since $R$ is analytically unramified, there exists $b \in \mathbb{Z}$ so that $\overline{I^{n+ b}} = I^n\overline{I^b}$ for all $n \geq 1$  \cite[Corollary~9.2.1]{SwansonHuneke}. 
Hence, after enlarging $a$ if needed, we may assume that $I^{\frac{w + \rho n}{\rho}} = I^n I^\frac{w}{\rho}$.
Thus we have an exact sequence 
\[
        0 \to I^\frac{w}{\rho}/I^nI^\frac{w}{\rho}        \to R/I^\frac{w + n \rho}{\rho} \to R/I^\frac{w}{\rho} \to 0
\]
and the conclusion follows after noting that the function $n \mapsto \length (I^\frac{w}{\rho}/I^nI^\frac{w}{\rho})$ is eventually a polynomial for $n\gg0$ with the prescribed leading coefficient. 
\end{proof}

We will need the following geometric description of rational powers of ideals, which is also well-known to experts. 

\begin{proposition}\label{prop: geometric rational powers}
    Let $R$ be a domain and $I\subseteq R$ be an ideal. Suppose that 
    $f \colon Z\to \Spec(R)$ is a projective birational map such that $I\sO_Z=\sO_Z(-E)$ for a Cartier divisor $E$ on $Z$ and the generic point of each irreducible component of $\red{E}$ is regular. Then the Rees valuations of $I$ are the divisorial valuations attached to the irreducible components of $\red{E}$
    and we may compute  
    \[
        I^{\frac b a} = \Gamma (Z, \sO_Z (- \lceil  \frac{b}{a}E\rceil)) \cap R.
    \]

    Furthermore, if $R$ is excellent then the result holds for the natural map
    \[
        Y\coloneqq\Proj (R \oplus \overline{I}T \oplus \overline{I^2}T^2 \oplus \overline{I^3}T^3 \oplus \cdots)
        \to \Spec (R).
    \]  
\end{proposition}
\begin{proof}
    Let $X \to \Spec (R)$ be the blow-up of $I$ and $\overline{X}$ be its normalization. 
    The construction given in \cite[Proposition~10.2.5, Exercises~10.5,6]{SwansonHuneke} identifies
    the Rees valuations as the valuations arising from the minimal primes of $I\sO_{\overline{X}}$. Since localization and normalization commute, we could equivalently normalize the local rings of the minimal primes of $I\sO_X$. 
    Now, by construction $Z \to \Spec (R)$ factors through $X$, so the DVRs at the minimal primes of $I\sO_Z$ (i.e., the generic points of the irreducible components of $\red{E}$) factor through the normalizations of the minimal primes of $I\sO_X$. Therefore, those are precisely the Rees valuations. The formula for rational powers now follows directly from their description in terms of Rees valuations. 

    We now verify that $Y \to \Spec (R)$ satisfies the necessary assumptions. In the extended Rees algebra $S = R[T^{-1}, \overline{I}T, \overline{I^2}T^2, \ldots]$
    the principal ideal $(T^{-1})$ is normal
    by \cite[Proposition~5.2.1 and Proposition~5.3.1]{SwansonHuneke}. 
    Hence, for every associated prime $Q$ of $T^{-1}S$ the ring $S_Q$ is a DVR \cite[Lemma~4.2]{Rees}. Since $R$ is excellent, 
    we know that $S$ is Noetherian (e.g., by \cite[Corollary 9.2.1]{SwansonHuneke}). Since $S_Q$ is a DVR, it follows that the homogeneous localization $S_{(Q)}$ (which is a $\mathbb{Z}$-graded ring) is regular by \cite[Proposition 1.2.5]{GotoWatanabeGradedII}. Thus $[S_{(Q)}]_0$ is a DVR: it is normal since it is a direct summand of $S_{(Q)}$ and it is one-dimensional. It is easy to see that $[S_{(Q)}]_0$ are exactly the local rings of the minimal primes of $I\sO_Y$, and thus they are regular. 
\end{proof}

We end this subsection by providing a formula for the colength of integral closures of powers of $(I,T^m)$ in the ring $R[[T]]$ using rational powers of $I$. 

\begin{lemma}\label{lemma: first rational powers expansion}
Let $I\subseteq R$ be an ideal. Then for all positive integers $a, n$ we have 
\[
        (I, T^a)^{\frac n a} = \sum_{k = 0}^{n} I^{\frac k a} T^{n - k}.
    \]
\end{lemma}
\begin{proof}   
    It is clear that the right hand side is contained in $(I, T^a)^{\frac n a}$. Namely, 
    \[
     (I^{\frac k a} T^{n - k})^a \subseteq \overline{I^k} (T^{a})^{n-k} \subseteq \overline{(I, T^a)^n}. 
    \]
    For the opposite inclusion, we note that $T^n \in (I, T^a)^{\frac n a}$ and so it suffices to characterize when $x T^{n-s} \in (I, T^a)^{\frac n a}$. 
    Given that $x^a T^{a(n-s)} \in \overline{(I, T^a)^{n}}$, it satisfies an equation of integral dependence 
    \[
        (x^aT^{a(n-s)})^{m} + f_1 (x^aT^{a(n-s)})^{m-1} + \cdots + f_k (x^aT^{a(n-s)})^{m-k} + \cdots + f_m =0
    \]
    where $f_k \in (I, T^a)^{nk}$.
    We now group the terms of the coefficients of $T^{a(n-s)m}$ to obtain
    \[
        x^{am} + y_1 x^{a(m-1)} + \cdots + y_m=0, 
    \]
    where $y_k T^{ak(n-s)} \in (I, T^a)^{nk}$ and 
    so $y_k \in I^{ks}$. This shows that $x^a \in \overline{I^s}$, thus $x \in I^{\frac s a}$
    and the conclusion follows. 
\end{proof}

We next express the colength of rational powers of $(I, T^m)$ in $R[[T]]$ when $m$ is a multiple of the Rees period of $I$. 

\begin{proposition}
\label{prop: rational powers of I+T^m}
Let $(R, \m)$ be a local ring and $I\subseteq R$ be an $\m$-primary ideal with the Rees period $\rho$. Let $m$ be a positive integer divisible by $\rho$. 
For positive integers $n$ and $0 \leq r < m$ we write 
$s = \lceil \frac{r\rho}{m} \rceil$ and $t = s \frac{m}{\rho} - r$.
Then we have
    \[
        \length (R[[T]]/(I, T^m)^{\frac {nm + r} m} )
        = \sum_{h = 0}^{s-1} \sum_{k = 0}^n \frac m\rho \length (R/I^\frac{k\rho + h}{\rho}) 
        + \sum_{h = s}^{\rho-1} \sum_{k = 0}^{n-1} \frac m\rho \length (R/I^\frac{k\rho + h}{\rho}) 
        + \left(\frac m \rho - t \right) \length (R/I^{\frac{n\rho+s}{\rho}}).
    \]    
\end{proposition}
\begin{proof}
Using Lemma~\ref{lemma: first rational powers expansion}, we first rewrite 
\begin{equation}\label{eq: for rational powers sum}
    (I, T^m)^{\frac {nm + r} m} =
\sum_{k = 0}^{nm + r} I^{\frac k m }T^{nm + r - k}
= \sum_{k = 0}^{nm + r} I^{\frac{\lceil k\rho/m \rceil }{\rho} } T^{nm + r - k}.
\end{equation}
It is left to count which of the appearing fractions in the exponent of $I$ are distinct.
Note that we have
\[
\left \lceil \frac {\rho (nm + r)} m \right \rceil = \rho n + s.
\]
We see that:
\begin{enumerate}
    \item For each $0 \leq h \leq s$, all of $\frac{h}{\rho}, \ldots, \frac{n\rho + h}{\rho}$
appear among the exponents $\frac{\lceil k\rho/m\rceil }{\rho}$, and 
\begin{itemize}
    \item if $0\leq h<s$, then each $I^{\frac{i\rho + h}{\rho}}$ above appears $\frac m \rho$ times in the sum;
    \item if $h=s$, then each $I^{\frac{i\rho + h}{\rho}}$ above appears $\frac m \rho$ times in the sum when $0\leq i\leq n-1$, while $I^{\frac{n\rho + h}{\rho}}$ appears $\frac m \rho - s\cdot \frac m \rho +r = \frac m \rho - t$ times in the sum.
\end{itemize}
    \item For $s< h \leq \rho - 1$, only $\frac{h}{\rho}, \ldots, \frac{(n-1)\rho + h}{\rho}$
appear among the exponents $\frac{\lceil k\rho/m\rceil }{\rho}$, and in this case each $I^{\frac{i\rho + h}{\rho}}$ appears $\frac m \rho$ times for all $h$ and all $0\leq i\leq n-1$. 
\end{enumerate}
Putting these together we obtain the formula in the proposition.
\end{proof}

\begin{corollary}
\label{cor: integral closures of powers of I+T^m}
    Let $(R, \m)$ be a local ring and $I\subseteq R$ be an $\m$-primary ideal with Rees period $\rho$. Let $m$ be a positive integer divisible by $\rho$ and $n$ be an arbitrary positive integer. 
    Then we have
    \[
        \length (R[[T]]/\overline{(I, T^m)^{n}})
        = \frac m\rho \sum_{h = 1}^{\rho - 1} \sum_{k = 1}^{n-1} \length (R/I^\frac {k\rho + h}\rho)
        + \frac m\rho \sum_{k = 0}^{n} \length (R/\overline{I^k}) = \frac{m}{\rho}\sum_{j=0}^{\rho n}\length(R/I^{\frac{j}{\rho}}).
    \]
\end{corollary}
\begin{proof}
This follows from Proposition~\ref{prop: rational powers of I+T^m}: in this case we have $r=s=t=0$, the first term in Proposition~\ref{prop: rational powers of I+T^m} vanish and we obtain 
$$\sum_{h = 0}^{\rho-1} \sum_{k = 0}^{n-1} \frac m\rho \length (R/I^\frac{k\rho + h}{\rho}) 
        + \frac m \rho  \length (R/\overline{I^{n}})
$$
which is easily checked to agree with both expressions in the corollary.
\end{proof}

\subsection{Stable ideals and pseudo-rational rings}
In this subsection, we briefly recall the concept of stable ideals and its connections with rings of minimal multiplicity and pseudo-rational rings.

\begin{definition}
Let $(R, \m)$ be a local ring and $I\subseteq R$ be an $\m$-primary ideal. We say that $I$ is \emph{stable} if there is a parameter ideal $J \subseteq I$ such that $I^2 = JI$.
\end{definition}

\begin{remark}
We note that if $I$ is stable, then the parameter ideal $J$ in the definition is necessarily a minimal reduction of $I$ and clearly $I^2\subseteq J$. Furthermore, when $(R,\m)$ is a Cohen--Macaulay local ring with an infinite residue field and $I$ is integrally closed, then $I^2\subseteq J$ for some minimal reduction $J$ implies $I$ is stable. For if $I^2\subseteq J$, then we have
$$I^2\subseteq \overline{I^2}\cap J=\overline{J^2}\cap J = J\overline{J}=JI \subseteq I^2$$
where the second equality above follows from \cite[Theorem 1]{ItohIntegralClosure}. Thus we have equalities throughout and $I$ is stable. 
\end{remark}

As far as we know, the notion of stable ideals first appeared in the work of Lipman for one-dimensional Cohen--Macaulay local rings, see \cite[Lemma 1.11]{LipmanArf}. The definition in higher dimensions appeared in \cite{Kubota}, the definition there is equivalent to ours when the residue field is infinite due to the following result that was independently observed by Huneke and Ooishi. 

\begin{theorem}[{\cite[Theorem~2.1]{Huneke} or \cite[Theorem~3.3]{Ooishi}}]
\label{thm: Northcott}
Let $(R, \mf m)$ be a Cohen--Macaulay local ring of dimension $d$ with an infinite residue field.
Then the following conditions are equivalent:
\begin{enumerate}
    \item $I$ is stable, 
    \item $I^2=JI$ for any minimal reduction $J$ of $I$,
    \item $\length (R/I) = \eh_0(I) - \eh_1(I)$.
\end{enumerate}
If these conditions hold, then $\eh_2(I), \ldots, \eh_d(I) = 0$
and $$\length (R/I^n) = \eh_0(I) \binom{n + d-1}{d} - \eh_1(I) \binom{n + d-2}{d-1}$$ for all $n \geq 1$.
\end{theorem}

The notion of minimal multiplicity is defined through the Abhyankar's inequality. 

\begin{theorem}[Abhyankar's inequality, \cite{Abhyankar}]
    Let $(R, \m)$ be a Cohen--Macaulay local ring of dimension $d$. 
    Then $\eh(R) \geq \edim(R) - d + 1$. Moreover, if the residue field of $R$ is infinite, then equality holds if and only if $\m$ is stable. 
\end{theorem}

A Cohen--Macaulay local ring $(R,\m)$ has {\it minimal multiplicity} if Abhyankar's inequality is an equality, that is, $\eh(R) = \edim(R) - \dim(R) + 1$. This property is stable on extending residue field (since multiplicity and embedding dimension are preserved under such extensions). Hence by Theorem~\ref{thm: Northcott}, if $(R,\m)$ is a $d$-dimensional Cohen--Macaulay local ring of minimal multiplicity, then 
the Hilbert--Samuel polynomial of $\m$ is 
$$p_\m(n)= \eh(R) \binom{n+d-1}{d} + (1 - \eh(R)) \binom{n+d-2}{d-1}.$$

We next recall that a local ring $(R,\m)$ is {\it pseudo-rational} in the sense of Lipman--Teissier \cite{LipmanTeissier}, if it is Cohen--Macaulay, normal, analytically unramified (i.e., $\widehat{R}$ is reduced) and for every proper birational map $\pi \colon Y\to X\coloneqq\Spec(R)$, the canonical map $H_\m^d(R)\to H^d(R\Gamma_\m(R\pi_*\sO_Y))$ is injective where $d=\dim(R)$. When $X$ admits a dualizing complex $\omega_X^\bullet$ (e.g., when $R$ is a homomorphic image of a Gorenstein local ring), then $\omega_Y^\bullet\coloneqq\pi^!\omega_X^\bullet$ is a dualizing complex of $Y$, and in this case, the last condition in the definition is equivalent to that $\pi_*\omega_Y=\omega_X$.

It is well-known that if $(R,\m)$ is essentially of finite type over a field of characteristic zero or if $\dim(R)\leq 2$, then pseudo-rational singularities are exactly rational singularities: that is, $Rf_*\sO_Y=\sO_X$ for some (equivalently, any) resolution of singularities $f$: $Y\to X=\Spec(R)$. 
This is essentially a consequence of the Grauert--Riemenschneider vanishing theorem \cite{GRVanishing,LipmanResolution}. The following theorem characterizes pseudo-rational surface singularities in terms of integral closure of ideals.

\begin{theorem}[{\cite[Proposition 5.5]{LipmanTeissier}}]
\label{thm: Lipman--Teissier}
Let $(R, \mf m)$ be a normal analytically unramified local domain of dimension two.
Then $R$ is pseudo-rational if and only if $I\overline{I} = \overline{I^2}$ for every $\m$-primary ideal $I$. In particular, if $R$ is pseudo-rational with an infinite residue field, then every $\m$-primary integrally closed ideal is stable.
\end{theorem}

\newpage
\section{Definition and basic properties}

\subsection{Lech--Mumford constant} 
We start by introducing the key concept of this paper. 
\begin{definition}
Let $(R, \mf m)$ be a local ring of dimension $d$. We call the quantity
\[
\lm(R) \coloneqq \sup_{\sqrt{I}=\m}\left\{\frac{\eh(I)}{d!\length(R/I)}\right\}
\]
the {\it Lech--Mumford constant} of $R$. 
\end{definition}

\begin{remark}
\label{rmk: LM notation different from Mumford}
In \cite{Mumford}, $\lm(R)$ is called the $0$th {\it flat multiplicity} of $R$ and is denoted by $\eh_0(R)$. Furthermore, $\lm(R[[T_1,\dots,T_n]])$ is called the $n$th flat multiplicity of $R$ and is denoted by $\eh_n(R)$. We decided not to follow Mumford's notation because $\eh_i$ is widely used in commutative algebra for the $i$th Hilbert coefficient.
\end{remark}

We collect some basic properties of $\lm(R)$. 

\begin{proposition}
\label{prop: basic results on cLM}
Let $(R,\m)$ be a local ring. Then we have the following: 
\begin{enumerate}
  \item the supremum in the definition of $\lm(R)$ can be restricted to integrally closed ideals;\label{part restrict to integrally closed}
  \item $\lm(R)=\lm(\widehat{R})$; \label{part restrict to complete}
  \item $\eh(R)\geq \lm(R)\geq \eh(R)/\eh(\red{R})$, in particular $\lm(R)\geq 1$; \label{part multiplicity bounds}
  \item $\lm(R)=\eh(R)$ if $\dim(R)\leq 1$; \label{part LM dim 1}
  \item if $S$ is a quotient of $R$ such that $\dim(S) = \dim(R)$, then $\lm(S) \leq \lm(R)$; \label{part LM of big image}
  \item if $\unm{R}$ is the unmixed part of $R$,\footnote{Recall that $\unm{R}=R/j(R)$ where $j(R)$ is the largest $R$-submodule of $R$ such that $\dim(j(R))<\dim(R)$, see \cite{HochsterHunekeIndecomposable}.} then $\lm(R)=\lm(\unm{R})$. \label{part restrict to unmixed}
\end{enumerate}
\end{proposition}
\begin{proof}
First of all, (\ref{part restrict to integrally closed}) is obvious: replacing $I$ by its integral closure will keep the multiplicity and will not increase the colength. Second, there is a one-to-one correspondence between (integrally closed) $\m$-primary ideals of $R$ and $\widehat{R}$ \cite[Lemma 9.1.1]{SwansonHuneke}, and multiplicity and colength are preserved by passing to completion, (\ref{part restrict to complete}) follows. 

Next, in (\ref{part multiplicity bounds}) we have $\lm(R)\leq \eh(R)$ by Theorem~\ref{thm Lech}. Now let $d\coloneqq\dim(R)$ and consider the sequence of ideals $I_n = \mf m^n + \sqrt{0}$. Note that $I_n$ is integral over $\mf m^n$. It follows that
\[
\frac{\eh(I_n)}{d! \length (R/I_n)}= \frac{\eh(\m^n)}{d! \length (\red{R}/\mf m^n\red{R})}
= \frac{\eh(R)}{d! \length (\red{R}/\mf m^n\red{R})/n^{d}} \to \frac{\eh(R)}{\eh(\red{R})} \text{ as $n\to\infty$}. 
\]
This proves that $\lm(R)\geq \eh(R)/\eh(\red{R})\geq 1$. 

To see (\ref{part LM dim 1}), note that if $\dim(R)\leq 1$, then $\eh(\m)/d!\length(R/\m)=\eh(R)$ and thus $\lm(R)\geq \eh(R)$, this together with (2) shows that $\lm(R)=\eh(R)$. 

To see (\ref{part LM of big image}), let $I$ be an $\mf m$-primary ideal of $S$. We may consider its preimage $I'$ in $R$. Then $\length (R/I') = \length (S/I)$ and, since $S$ is a quotient of $R$ with $\dim(S)=\dim(R)$, we have $\eh (I') \geq \eh(I)$. Thus $\lm(S) \leq \lm(R)$.

To see (\ref{part restrict to unmixed}), note that $\unm{R}$ is a quotient of $R$ with $\dim(\ker(R\to \unm{R}))<\dim(R)$. Thus for any $\m$-primary ideal $I$ of $R$, we have $\eh(I)=\eh(I\unm{R})$ and $\length(R/I)\geq\length(\unm{R}/I\unm{R})$. It follows that $\lm(R)\leq \lm(\unm{R})$. Since $\lm(R)\geq \lm(\unm{R})$ by (4), we have $\lm(R)=\lm(\unm{R})$.
\end{proof}

We can completely characterize when $\lm(R)=\eh(R)$. This is essentially a reformulation of the main result of \cite{MaSmirnovUniformLech} (see also \cite[Theorem A (2)]{HMQS}).

\begin{theorem}
\label{thm: uniform Lech}
Let $(R,\m)$ be a local ring. Then $\lm(R)=\eh(R)$ if and only if one of the following holds.
\begin{enumerate}
  \item[(a)] $\dim(R)\leq 1$.
  \item[(b)] $\dim(R)\geq 2$ and $\eh(\red{\widehat{R}})=1$. 
\end{enumerate}
\end{theorem}
\begin{proof}
First of all, whenever (a) or (b) holds, we have $\lm(R)=\eh(R)$ by (\ref{part LM dim 1}) and (\ref{part multiplicity bounds}) of Proposition~\ref{prop: basic results on cLM}. Conversely, suppose $\dim(R)\geq 2$ and $\eh(\red{\widehat{R}})>1$, by \cite[Theorem 3]{MaSmirnovUniformLech}, there exists $\epsilon>0$ such that $\eh(I)\leq d!(\eh(R)-\epsilon)\length(R/I)$ for all $\m$-primary ideals $I\subseteq R$. This implies $\lm(R)\leq \eh(R)-\epsilon<\eh(R)$. 
\end{proof}

It will be useful to be able to restrict in the definition of $\lm (R)$ to special classes of ideals. The discussion and results below extend some ideas in \cite[Proposition~3.4]{Shah}. Recall that a graded family of ideals in a ring $R$ is a sequence of ideals $\{J_n\}_n$ such that $J_{n+1} \subseteq J_n$ and $J_n J_m \subseteq J_{n+m}$. For convenience, we shall also set $J_0 = R$. Examples of such families are powers of an ideal, or integral closures of powers of an ideal, or valuation ideals. For each graded family $\{J_n\}_n$ we can form
\begin{itemize}
  \item the Rees algebra $R[J_\bullet T^\bullet] \coloneqq \oplus_{n=0}^\infty J_nT^n$,
  \item the extended Rees algebra $R[J_\bullet T^\bullet, T^{-1}]$,
  \item the associated graded ring $\gr_{J_\bullet}(R)\coloneqq R[J_\bullet T^\bullet, T^{-1}]/(T^{-1})= \oplus_{n = 0}^\infty J_n/J_{n+1}$.
\end{itemize}
The graded family $\{J_n\}_n$ is called Noetherian if $R[J_\bullet T^\bullet]$ is a Noetherian ring. Note that this implies $R[J_\bullet T^\bullet, T^{-1}]$ and $\gr_{J_\bullet}(R)$ are both Noetherian. Moreover, for a Noetherian graded family, by \cite[Remark 2.4.3]{RatliffFiltration} there exists $m\in\mathbb{N}$ such that 
$$J_{n+m}=J_nJ_m \text{ for all } n\geq m.$$
This implies (when $R$ is local) that $\cap_nJ_n\subseteq \cap_n J_m^n\subseteq \cap_n\m^n=0$. Furthermore, we have that $R[J_\bullet T^\bullet, T^{-1}]$ is an integral extension of the usual extended Rees algebra $R[J_m^nT^n, T^{-1}]$, and thus by \cite[Theorem 2.2.5 and Theorem 5.1.4]{SwansonHuneke}, we have
$$\dim(R[J_\bullet T^\bullet, T^{-1}])=\dim(R[J_m^nT^n, T^{-1}])=\dim(R)+1, \text{ and } \dim(\gr_{J_\bullet}R)=\dim(R).$$

For the ease of the presentation of the next few results, we introduce a graded version of Lech--Mumford constant for graded rings. We will show that, in our cases of interest, it agrees with the usual Lech--Mumford constant of the localization at the unique homogeneous maximal ideal so it will not cause any ambiguity later. Let $R$ be an $\mathbb{N}^n$-graded algebra over a local ring $(R_0,\m_0)$. We define
$$\grlm(R)\coloneqq\sup\left\{\frac{\eh(I)}{d!\length(R/I)}\right\}$$
where the supremum runs over all $\mathbb{N}^n$-graded ideals of finite colength, i.e., those $\mathbb{N}^n$-graded ideals $I$ such that $\sqrt{I}=\m\coloneqq\m_0+R_{>0}$ ($\m$ is the unique $\mathbb{N}^n$-graded  maximal ideal).

\begin{proposition}
\label{prop: LM associated graded}
Let $(R, \mf m)$ be a local ring and $\{J_n\}_n$ be a Noetherian graded family.
Then $\lm(R) \leq \grlm(\Gr_{J_\bullet}(R))$. In particular, $\lm(R)\leq \grlm(\Gr_J(R))$ for any ideal $J\subseteq R$.
\end{proposition}
\begin{proof}
Let $I$ be an $\mf m$-primary ideal of $R$ and let $\init(I)$ be its form ideal in $\Gr_{J_\bullet}(R)$, i.e.,
\[
\init(I) = \bigoplus_{k \geq 0} \frac{I \cap J_k + J_{k + 1}}{J_{k+1}}.
\]
Since the graded family is Noetherian, we have $\cap_k J_k = 0$. It follows that
\[
\length(\gr_{J_\bullet}(R)/\init(I)) =
\sum_{k = 0}^\infty \length\left(\frac{J_k}{(I \cap J_k) + J_{k + 1}}\right)
= \sum_{k = 0}^\infty \length\left(\frac{I + J_k}{I + J_{k + 1}}\right) = \length(R/I).
\]
Moreover, it is easy to see that $\init(I_1)\init(I_2) \subseteq \init(I_1I_2)$ for any two ideals $I_1, I_2$ of $R$, thus
\[ \length(\gr_{J_\bullet}(R)/\init(I)^n) \geq \length(\gr_{J_\bullet}(R)/\init(I^n)) = \length(R/I^n).\]
Since $\dim (R) = \dim (\gr_{J_\bullet}(R))$, the above discussion shows that $\eh(I) \leq \eh(\init(I))$. Since $\init(I)$ is clearly a homogeneous ideal, we have $\lm(R)\leq\grlm(\gr_{J_\bullet}(R))$ as wanted.
\end{proof}

\begin{proposition}
\label{prop: LM multigraded}
If $R$ is an $\mathbb N^n$-graded ring over a local ring $(R_0,\m_0)$
and $\m=\m_0+R_{>0}$ is its unique homogeneous maximal ideal, then $\lm(R_\m) = \grlm(R)$ and we may restrict to multi-homogeneous $\m$-primary ideals in the supremum computing $\lm(R_\m)$. 
\end{proposition}
\begin{proof}
We consider the degree lexicographic order filtration on $\mathbb N^n$: for two multi-degrees, $(a_1,...,a_n) \prec (b_1,...,b_n)$ if either 
\begin{itemize}
    \item $\sum^n_{i = 1} a_i < \sum^n_{i = 1} b_i$, or 
    \item $\sum^n_{i = 1} a_i=\sum^n_{i = 1} b_i$, and there is $j$ such that  $a_j < b_j$ and $a_k=b_k$ for all $k<j$. 
\end{itemize}

Given an element $f \in R$, we define its initial form $\init_\prec (f)$ to be the smallest graded piece with respect to $\prec$. For any ideal $J\subseteq R$, we define $\init_\prec(J)$ to be the ideal generated by all $\{\init_\prec(f) \mid f \in J\}$. Furthermore, for any vector $w \in \mathbb N^n$, we define the weight filtration as follows: if $g$ is any multi-homogeneous element then define $w(g) = w\cdot \deg(g)$. For an arbitrary element $f \in R$, we define its initial form $\init_w(f)$ to be the smallest graded piece with respect to $w$, and define $\init_w(J)$ similarly.

Now we fix an $\m$-primary ideal $I\subseteq R$ and consider $\init_\prec(I)$. By Noetherianness, $\init_\prec(I)$ is generated by $\init_\prec (f_i)$ for finitely many $f_1,\ldots, f_r \in I$. It follows that we can find a weight vector $w \in \mathbb N^n$ so that 
$\init_\prec(f_i) = \init_w(f_i)$ for all $f_i$. Thus $\init_\prec(I) = \init_w(I)$ (since the two ideals have the same colengths).

Next, we consider a graded family $J_\bullet \coloneqq \{J_n\}_n$, where $J_n$ is the ideal in $R$ generated by all multi-homogeneous elements $g$ such that $w(g) \geq n$, and set $\gr_w (R)\coloneqq\gr_{J_\bullet}(R)$ to be the associated graded ring with respect to $J_\bullet$.
Note that $\gr_w (R) \cong R$ since the grading induced by the weight is coarser than the original $\mathbb N^n$-grading on $R$. The form ideal of $I$ in $\gr_w(R)$ is $\init_w(I) = \init_\prec (I)$ by construction. By the argument in Proposition~\ref{prop: LM associated graded}, we know that 
$$\frac{\eh(I)}{\length (R/I)} \leq \frac{\eh(\init_w(I))}{\length (R/\init_w(I))} = \frac{\eh(\init_\prec(I))}{\length (R/\init_\prec(I))}.$$

Finally, by varying $I$ it follows that $\lm(R_\m) \leq \grlm (R)$. But the other direction $\grlm(R)\leq \lm(R_\m)$ is clear since each multi-homogeneous ideal of finite colength corresponds to an $\m$-primary ideal of $R_\m$ (and $R_\m$ has possibly more $\m$-primary ideals). Thus $\lm(R_\m) = \grlm (R)$ and the latter can be computed by multi-homogeneous ideals.
\end{proof}

As a special case of Proposition~\ref{prop: LM multigraded}, we recover the following fundamental result of Mumford.

\begin{corollary}[{\cite[Lemma~3.6]{Mumford}}]
\label{cor: compute LM by homogeneous}
Let $(R, \mf m)$ be a local ring.
Then $\lm(R[[T]])$ can be computed by looking only at homogeneous
ideals, i.e., ideals of the form $I = \oplus_{k \geq 0} I_k T^k$ where $I_k \subseteq R$ and $I_k = R$ for $k \gg 0$. 
\end{corollary}
\begin{proof}
By Proposition~\ref{prop: LM multigraded}, we have
$$\grlm(R[T]) =\lm(R[T]_{\m+(T)}) =\lm(R[[T]]).$$ 
where the last equality follows from Proposition~\ref{prop: basic results on cLM}(\ref{part restrict to complete}) since $R[T]_{\m+(T)}$ and $R[[T]]$ agree up to completion. The conclusion follows from the definition of $\grlm$.
\end{proof}

\begin{lemma}
\label{lem: associated graded with respect to x}
Let $(R, \mf m)$ be a local ring and $x$ be a parameter element.
Then we have $\grlm (\Gr_{(x)}(R)) \leq \grlm ((R/xR)[T])$.
\end{lemma}
\begin{proof}
We have an isomorphism $(x^n)/(x^{n+1}) \cong R/((x) + (0:x^n))$. Thus
we have a natural surjective graded homomorphism $(R/xR)[T] \to \Gr_{(x)}(R)$. Moreover, it is clear that $\dim((R/xR)[T])=\dim(\Gr_{(x)}(R))=\dim(R)$. Thus $\grlm (\Gr_{(x)}(R)) \leq \grlm ((R/xR)[T])$ by the graded version of Proposition~\ref{prop: basic results on cLM}(\ref{part LM of big image}). 
\end{proof}

We next generalize \cite[Proposition~3.5]{Mumford} and show that Lech--Mumford constant does not decrease under specialization.

\begin{proposition}
\label{prop: Mumford prop 3.5}
Let $(R, \mf m)$ be a local ring and $x$ be a parameter element.
Then we have  $\lm(R) \leq \lm(R/xR)$.
Moreover, if the equality holds and $\dim(R) > 1$,
then the supremum in the definition of $\lm(R)$ is not attained.
\end{proposition}
\begin{proof}
By Proposition~\ref{prop: LM associated graded} and Lemma~\ref{lem: associated graded with respect to x}, we have that
$$\lm(R) \leq \grlm(\Gr_{(x)}(R)) \leq \grlm((R/xR)[T]).$$
Thus it suffices to show that $\grlm(S[T]) \leq \lm(S)$
for a local ring $(S, \mf m)$.
This was proved by Mumford in \cite[Proposition~3.5]{Mumford} and we include a proof here for completeness.

Fix a homogeneous ideal $I = \oplus_{k \geq 0} I_kT^k$ in $S[T]$ of finite colength. We want to bound $\eh(I)$.  
Suppose that $I_{N + 1} = S$ and note that $I^n$ contains the ideal
\begin{align*}
&\left(I_0^n \oplus I_0^{n-1}I_1T \oplus \cdots \oplus I_0I_1^{n-1}T^{n-1} \right)
\oplus \left(I_1^{n}T^{n} \oplus \cdots \oplus
I_1I_2^{n-1}T^{2n-1}\right) \oplus \cdots \oplus\\
&\left(I_{N-1}^{n}T^{n(N - 1)} \oplus \cdots \oplus
I_{N-1}^{1}I_N^{n-1}T^{nN - 1}\right)
\oplus\left(I_{N}^{n}T^{Nn} \oplus \cdots \oplus
I_{N}^{1}T^{Nn + n - 1}
\right) \oplus T^{n(N+1)} \oplus \cdots.
\end{align*}
Thus using that $I_k \subseteq I_{k+1}$ we obtain that
\begin{align*}
I^n \supseteq  \bigoplus_{i = 0}^{n- 1} I_0^nT^{i} \oplus
\bigoplus_{i = 0}^{n- 1} I_1^nT^{n + i} \oplus \cdots \oplus
\bigoplus_{i = 0}^{n- 1} I_{N-1}^nT^{(N-1)n + i} \oplus
\bigoplus_{i = 0}^{n - 1} I_N^{n-i}T^{Nn + i} \oplus
T^{(N + 1)n} \oplus \cdots.
\end{align*}
Let $d=\dim(S)$. The inclusion of ideals above implies that
\begin{align*}
\length(S[T]/I^n) &\leq \sum_{i = 0}^{N - 1} n\length(S/I_i^n) + \sum_{j = 1}^{n} \length(S/I_{N}^j)
\leq \frac{n^{d+1}}{d!}\sum_{i = 0}^{N - 1} \eh(I_i) + \frac{n^{d+1}}{(d+1)!} \eh(I_N) + O(n^{d}).
\end{align*}
It follows that
\[
\frac{\eh(I)}{(d+1)!\length(S[T]/I)} \leq
\frac{(d+1)\sum\limits_{i = 0}^{N - 1} \eh(I_i) + \eh(I_N)}{(d + 1)!\sum\limits_{i = 0}^{N} \length (S/I_i)}
\leq \frac{\sum\limits_{i = 0}^{N} \eh(I_i)}{d!\sum\limits_{i = 0}^{N}\length(S/I_i)} \leq \max_i \left\{\frac{\eh(I_i)}{d!\length(S/I_i)}\right\}.
\]
Therefore $\grlm(S[T]) \leq \lm(S)$. Finally, for the last claim, we simply note that the middle
inequality in the above formula is strict unless $d=0$. This completes the proof. 
\end{proof}

\begin{corollary}
\label{cor: Mumford 3.5}
Let $(R, \mf m)$ be a local ring of dimension $d$.
Then we have 
$$\lm(R[[T]]) \leq \lm (R) \leq (d + 1)\lm(R[[T]]).$$
\end{corollary}
\begin{proof}
The first inequality follows from Proposition~\ref{prop: Mumford prop 3.5}. For the second inequality, let $I$ be an $\mf m$-primary ideal in $R$ and
consider the ideal $J = IR[[T]] + (T)$. It is easy to see that $\length (R[[T]]/J) = \length (R/I)$ and $\eh(J)=\eh(I)$. 
Thus
\[
\frac{\eh(I)}{d!\length (R/I)} =
\frac{(d+1)\eh(J)}{(d+1)!\length (R[[T]]/J)} \leq (d+1)\lm(R[[T]]).
\]
This completes the proof.
\end{proof}

\begin{remark}
It is not true that $\lm (R/xR) \leq d\lm (R)$ for a linear form $x \in \m-\m^2$, where $d=\dim(R)$. We will see in Theorem~\ref{thm: Lech-stable CM surface} in Section~\ref{section: surface} that the $E_8$ singularity
$R = \mathbb C [[x,y,z]]/(x^2 + y^3 + z^5)$ satisfies $\lm (R) = 1$, but it follows from Proposition~\ref{prop: basic results on cLM}(\ref{part LM dim 1}) that $\lm(R/xR) = \eh(R/xR) = 3$. On the other hand, we do not know whether $\lm(R/xR)\leq d\lm(R)$ for a general linear form $x\in \m-\m^2$. This holds when $d = 2$, because in this case 
$$\lm (R/xR)/2 = \eh(R/xR)/2= \eh(R)/d! \leq \lm (R)$$ 
where the first equality follows from Proposition~\ref{prop: basic results on cLM}(\ref{part LM dim 1}) and the second equality follows from the fact that $x$ is a general linear form.
\end{remark}

We conclude this subsection by recalling another result of Mumford, which extends the main step in Proposition~\ref{prop: Mumford prop 3.5} by bounding the multiplicity of an ideal in the power series using mixed multiplicities of its components. This result will be used several times in Section~\ref{section: example} and Section~\ref{section: Lech's inequality revisited}.

\begin{proposition}[{\cite[Proposition~4.3]{Mumford}}]\label{prop: Mumford 4.3}
    Let $(R, \m)$ be a local ring and 
    \[
        I = I_0 + I_1 T + \cdots + I_{N-1}T^{N-1} + T^N
    \]
    be an $(\m, T)$-primary ideal of $R[[T]]$. Set $I_N = R$.
    Then for all sequences of integers $0 = r_0 < r_1 < \cdots < r_l = N$
    we have
    \[
        \eh(I) \leq \sum^{l - 1}_{k = 0} (r_{k+1} - r_k) \sum_{i = 0}^d \eh (I_k^{[i]} \mid I_{k + 1}^{[d-i]}). 
    \]
\end{proposition}
\begin{proof}
    One can verify directly that 
    \begin{align*}
        I^m & \supseteq I^m_{r_0} (1 + \cdots + T^{r_1 - 1})
        + I^{m-1}_{r_0}I_{r_1} (T^{r_1} + \cdots + T^{2r_1 - 1})
        + \cdots +  I_{r_0}I^{m-1}_{r_1} (T^{(m-1)r_1} + \cdots + T^{mr_1 - 1})\\
        & + I^m_{r_1} (T^{mr_1} + \cdots + T^{(m-1)r_1 + r_2 - 1})
        + \cdots + I_{r_1}I^{m-1}_{r_2} (T^{r_1 + (m-1)r_2} + \cdots + T^{mr_2 - 1})
        \\
        &\quad\vdots \\
        & + I_{r_{l- 1}}^m (T^{mr_{l-1}} + \cdots + T^{(m-1)r_{l-1} + r_l - 1}) + \cdots + T^{mr_l}.
    \end{align*}
    It follows that
    \[
    \length (R[[T]]/I^m)\leq \sum_{k = 0}^{l-1} (r_{k+1}-r_k) \sum_{i = 0}^m \length (R/I_{k}^{m-i}I_{k+1}^i).
    \]
    By the definition of mixed multiplicities, we have
    \begin{align*}
        \sum_{i = 0}^m \length (R/I_{k}^{m-i}I_{k+1}^i)
        &= \frac{1}{d!} \sum_{i = 0}^m \left( \sum_{j = 0}^d \binom{d}{d-j} \eh (I_k^{[j]} \mid I_{k+1}^{[d-j]}) i^j(m-i)^{d-j} + O(m^{d-1})\right)
        \\
        &= \sum_{j = 0}^d \eh (I_k^{[j]} \mid I_{k+1}^{[d-j]}) \left( \sum_{i = 0}^m \frac{i^j}{j!}\frac{(m-i)^{d-j}}{(d-j)!} \right) + O(m^{d}).
    \end{align*}
    Since $\eh (I) = \lim_{m \to \infty} \frac{(d+1)!}{m^{d+1}}\length (R[[T]]/I^m)$, it suffices to show that 
    \[
    \sum_{i = 0}^m \frac{i^j}{j!}\frac{(m-i)^{d-j}}{(d-j)!}
    = \frac{m^{d+1}}{(d+1)!} + O(m^d). 
    \]
    This formula follows by setting $N = m-d, a = j, b = d - j$ and expanding the binomial coefficients in the identity 
    \[
    \sum_{i = 0}^N \binom{a + i}{a}\binom{b + N -i}{b} = \binom{a+b+N+1}{a+b+1}
    \]
    obtained by comparing the coefficients at $x^{N}$
    of the formal power series identity
    \[(1+x)^{-a-1}(1+x)^{-b-1}=(1 + x)^{-a-b-2}.\qedhere \]
\end{proof}

\subsection{Stability of local rings} As an immediate consequence of Proposition~\ref{prop: basic results on cLM} and Proposition~\ref{prop: Mumford prop 3.5}, for any local ring $(R,\m)$, we have the following inequalities
\begin{equation*}
\label{eqn: chain of inequalities}
\eh(R)\geq \lm(R)\geq \lm(R[[T_1]]) \geq \lm(R[[T_1, T_2]]) \geq  \cdots \geq 1. \tag{$\dagger$}
\end{equation*}
\normalsize
We thus define 
$$\limlm(R)\coloneqq\displaystyle\lim_{n\to\infty}\lm(R[[T_1,\dots,T_n]]).$$

\begin{definition}
\label{def: stability of local rings}
We say a local ring $(R,\m)$ is
\begin{itemize}
  \item {\it Lech-stable} if $\lm(R)=1$,
  \item {\it semistable} if $\lm(R[[T]])=1$,
  \item {\it lim-stable} if $\limlm(R)=1$,
  \item {\it stable} if $R$ is semistable and the supremum in the definition of $\lm(R[[T]])$ is not attained.
\end{itemize}
\end{definition}

\begin{remark}
We have the following implications of these stability conditions, which are straightforward from Proposition~\ref{prop: Mumford prop 3.5} (see (\ref{eqn: chain of inequalities})): 
\[
\xymatrix{
& \text{stable} \ar@{=>}[d] & \\
\text{Lech-stable} \ar@{=>}[ru]^{\dim(R)\geq 1} \ar@{=>}[r] & \text{semistable} \ar@{=>}[r] & \text{lim-stable}.
}
\]

We will see in Proposition~\ref{prop: Artinian case}, Proposition~\ref{prop: dimension one Lech-stable} and Theorem~\ref{thm: dimension one semistable} a complete characterization of these stability notions for local rings of dimension $0$ and $1$. We also have a good understanding of them in dimension $2$, see Theorem~\ref{thm: Lech-stable CM surface}, Theorem~\ref{thm: lim-stable implies log canonical general} and Theorem~\ref{thm: lim-stable implies slc}; and for many class of rings of combinatorial nature, see Theorem~\ref{thm: snc is semistable}, Theorem~\ref{thm: max minors are Lech-stable},   and Example~\ref{example: Veronese in two variables}. From these results one easily obtains examples of semistable singularities that are not stable, and examples of stable singularities that are not Lech-stable. On the other hand, quite embarrassingly, we do not know/have a single example of a lim-stable local ring that is not semistable: in fact, even in some simple examples, computing the entire chain (\ref{eqn: chain of inequalities}) is challenging (for us), see Example~\ref{example: double drop} for a partial attempt. 
\end{remark}

\begin{remark}
\label{rmk: mult bound for Lech-stable and semistable}
Let $(R,\m)$ be a local ring of dimension $d$. By considering $\m$ and $\m^2$ in the definition of $\lm(R)$, we see that if $R$ is Lech-stable, then 
$$\eh(R)\leq \min \left\{d!, \frac{d!(\edim(R)+1)}{2^d}\right \}.$$
Similarly, if $R$ is semistable, then 
\[
\eh(R)\leq \min \left \{(d+1)!, \frac{(d+1)!(\edim(R)+2)}{2^{d+1}} \right \}.
\]
\end{remark}

\begin{proposition}
\label{prop: limit-stable reduced}
Let $(R,\m)$ be a local ring that is lim-stable. Then $\unm{\widehat{R}}$ is reduced. In particular, Cohen--Macaulay semistable singularities are reduced.
\end{proposition}
\begin{proof}
By Proposition~\ref{prop: basic results on cLM} (\ref{part restrict to complete}) and (\ref{part restrict to integrally closed}), we may replace $R$ by $\unm{\widehat{R}}$ to assume that $R$ is complete and unmixed. Since $R$ is lim-stable, unwinding the definition we have that for any $\epsilon>0$, there exists $n$ such that 
$$1+\epsilon \geq \lm(R[[T_1,\dots,T_n]]) \geq \frac{\eh(R[[T_1,\dots,T_n]])}{\eh(\red{R[[T_1,\dots,T_n]]})}=\frac{\eh(R)}{\eh(\red{R})},$$
where the second inequality follows from Proposition~\ref{prop: basic results on cLM}(\ref{part multiplicity bounds}). Letting $\epsilon\to 0$ we obtain that $\eh(R)=\eh(\red{R})$. Since $R$ is unmixed, if $R$ is not reduced then the nilradical $\sqrt{0}\neq 0$ have full dimension and thus $\eh(R)=\eh(\red{R})+\eh(\m, \sqrt{0})>\eh(R)$ which is a contradiction. Thus $R$ is reduced as wanted. The last conclusion follows from the fact that Cohen--Macaulay rings are unmixed and semistability implies lim-stability. 
\end{proof}

We now give characterizations of various notions of stability of local rings in small dimensions. We start with the case that $\dim(R)=0$. 

\begin{proposition}
\label{prop: Artinian case}
Let $(R, \mf m)$ be an Artinian local ring. The following are equivalent:
\begin{enumerate}
\item $R$ is regular (i.e., $R$ is a field),
\item $R$ is Lech-stable,
\item $R$ is semistable,
\item $R$ is lim-stable.
\end{enumerate}
Furthermore, $R$ is never stable.
\end{proposition}
\begin{proof}
Obviously we have $(1)\Rightarrow(2)\Rightarrow(3)\Rightarrow(4)$ (and these hold in any dimension). To see $(4)\Rightarrow(1)$, note that since $R$ is Artinian, Proposition~\ref{prop: limit-stable reduced} shows that if $R$ is lim-stable then $R$ is reduced and thus regular. For the last part, note that since $\dim(R[[T]])=1$, we have $\lm(R[[T]])=\eh(R[[T]])$ (see Proposition~\ref{prop: basic results on cLM}) and it is easy to see that the supremum is achieved at the maximal ideal of $R[[T]]$. Thus $R$ is not stable. 
\end{proof}

We end this subsection by providing a characterization of Lech-stable and stable singularities of dimension one. The semistable and lim-stable singularities in dimension one will be treated later, after computing the necessary surface examples.

\begin{proposition}
\label{prop: dimension one Lech-stable}
Let $(R, \mf m)$ be a local ring of dimension one. Then the following conditions are equivalent:
\begin{enumerate}
  \item $R$ is Lech-stable,
  \item $R$ is stable,
  \item $\unm{\widehat{R}}$ is a regular local ring.
\end{enumerate}
\end{proposition}
\begin{proof}
We first note that by Proposition~\ref{prop: basic results on cLM} (\ref{part restrict to complete}) and (\ref{part restrict to unmixed}), we have $\lm(R)=\lm(\unm{\widehat{R}})$ and $\lm(R[[T_1,\dots,T_r]])=\lm(\unm{\widehat{R}}[[T_1,\dots, T_r]])$ for all $r>0$. Hence to prove the proposition, we may replace $R$ by $\unm{\widehat{R}}$ to assume that $R$ is a Cohen--Macaulay complete local ring of dimension one.

We already know that $(1)\Rightarrow(2)$, and clearly if $R$ is regular, then it is Lech-stable and thus $(3)\Rightarrow(1)$. It remains to prove $(2)\Rightarrow (3)$. Suppose $R$ is stable, then $R$ is semistable and thus by Remark~\ref{rmk: mult bound for Lech-stable and semistable}, we know that $\eh(R)\leq 2$. Since $R$ is a one-dimensional Cohen--Macaulay local ring, $\edim(R)\leq \eh(R)+1-1 \leq 2$ by Abhyankar's inequality.
If $\edim(R)=2$, then we have equalities throughout. In particular, $\eh(R)=2$ and it is easy to see that $\lm(R[[T]])=1$ is attained at the maximal ideal of $R[[T]]$, in contradiction with $R$ being stable. Therefore we must have $\edim(R)=1$, i.e., $R$ is regular. 
\end{proof}

\subsection{Localization and faithfully flat extension} In this subsection, we study the behavior of the Lech--Mumford constant under certain ring homomorphisms. We start with a key lemma. 

\begin{lemma}
\label{lem: localization key}
Let $(R,\m)$ be a local ring and $\mf p\in \Spec(R)$ such that $\hght\mf p+\dim(R/\mf p)=\dim(R)$. Let $x_1,\dots,x_r$ be part of a system of parameters of $R$ whose images in $R/\mf p$ form a system of parameters of $R/\mf p$. Then we have 
$$\lm(R_\mf p)\leq \lm(R/(x_1,\dots,x_r)).$$
\end{lemma}
\begin{proof}
We use induction on $r$ and we first assume that $r=1$. Let $d\coloneqq \dim(R)$ and thus $\dim(R_\mf p)=d-1$. Let $J$ be a $\mf p R_\mf p$-primary ideal of $R_\mf p$ and let $I$ be the contraction of $J$ to $R$. Note that $I$ is a $\mf p$-primary ideal by construction and thus $x\coloneqq x_1$ is a nonzerodivisor on $R/I$. It follows that 
$$\length(R/I+(x)) = \eh(x, R/I)=\eh(x, R/\mf p) \length(R_\mf p/IR_\mf p)=\eh(x, R/\mf p)\length(R_\mf p/J),$$
where the second equality follows from the associativity formula for multiplicity. Now let $I_n$ be the contraction of $J^n$ to $R$, we clearly have $I^n\subseteq I_n$, thus we have 
$$\length(R/I^n+(x))\geq \length(R/I_n+(x))=\eh(x, R/I_n)=\eh(x, R/\mf p) \length(R_\mf p/I_nR_\mf p)=\eh(x, R/\mf p)\length(R_\mf p/J^n),$$
where the first equality follows from the fact that $I_n$ is $\mf p$-primary (and thus $x$ is a nonzerodivisor on $R/I_n$) and the second equality follows again from the associativity formula for multiplicity. Letting $n\to\infty$ and taking limit, we obtain that 
$\eh(I, R/(x)) \geq \eh(x, R/\mf p)\eh(J, R_\mf p).$
Therefore we have 
$$\frac{\eh(J, R_\mf p)}{(d-1)!\length(R_\mf p/J)} = \frac{\eh(x, R/\mf p)\eh(J, R_\mf p)}{(d-1)!\eh(x, R/\mf p)\length(R_\mf p/J)}\leq \frac{\eh(I, R/(x))}{(d-1)!\length(R/I+(x))}\leq\lm(R/(x)).$$
It follows that $\lm(R_\mf p)\leq \lm(R/(x))$. This completes the proof when $r=1$.

Now suppose $r\geq 2$. We can choose a minimal prime $\mf q$ of $\mf p+(x_1)$ such that $\hght\mf q + \dim(R/\mf q)=\dim(R)$. It follows from the choice of $\mf q$ that the images of $x_2,\dots,x_r$ form a system of parameters of $R/\mf q$, and that 
$\hght(\mf p R_\mf q)+1=\dim(R_\mf q)$ and $\hght(\mf q/(x_1))+\dim(R/\mf q)=\dim(R/x_1R)$. Therefore we have 
$$\lm(R_\mf p)\leq\lm(R_\mf q/x_1R_\mf q)=\lm((R/x_1R)_\mf q)\leq \lm(R/(x_1,x_2,\dots,x_r))$$
where the first inequality follows from the $r=1$ case applied to $R_\mf q$, and the second inequality follows by the inductive hypothesis applied to $R/x_1R$. This completes the proof.
\end{proof}

We now prove the effect of localization on Lech--Mumford constants.

\begin{theorem}
\label{thm: localization}
Let $(R, \mf m)$ be a local ring
and $\mf p\in\Spec(R)$ such that $\hght \mf p + \dim (R/\mf p) = \dim(R)$.
Then we have 
\begin{enumerate}
  \item $\lm (R) \geq \lm (R_\mf p[[T_1, \ldots, T_r]])$ where $r = \dim (R/\mf p)$;
  \item $\limlm(R)\geq \limlm(R_\mf p)$. 
\end{enumerate}
In particular, if $R$ is equidimensional, catenary, and lim-stable, then so is $R_{\mf p}$ for all $\mf p\in\Spec(R)$.
\end{theorem}
\begin{proof}
Let $S\coloneqq R[T_1,\dots,T_r]_{(\m+(T_1,\dots,T_r))}$ and let $P\coloneqq \mf p + (T_1,\dots,T_r)$ be a prime ideal of $S$. We can choose $x_1,\dots,x_r$ part of a system of parameters of $R$ such that their images in $R/\mf p$ form a system of parameters of $R/\mf p$. It follows that that $\hght(P)+\dim(S/P)=\dim(S)$ and that $x_1-T_1,\dots,x_r-T_r$ is part of a system of parameters of $S$ such that their images in $S/P$ form a system of parameters of $S/P$. Therefore we have
$$\lm(R_\mf p[[T_1,\dots,T_r]])=\lm(S_P)\leq \lm(S/(x_1-T_1,\dots, x_r-T_r))=\lm(R)$$
where the first equality follows from Proposition~\ref{prop: basic results on cLM}(\ref{part restrict to complete}) as $S_P$ and $R_\mf p[[T_1,\dots,T_r]]$ agree upon completion at their maximal ideals, and the inequality above follows from Lemma~\ref{lem: localization key}. This completes the proof of the part $(1)$.

We next proceed with part $(2)$. Note that by the definition of $\limlm(R)$, for any $\epsilon>0$, there exists $n$ such that $\limlm(R)\geq \lm(R[[T_1,\dots,T_n]])-\epsilon$. It follows from part $(1)$ that
\begin{align*}
  \limlm(R) &\geq \lm(R[[T_1,\dots,T_n]])-\epsilon \\
   & \geq \lm(R_\mf p[[T_1,\dots, T_n, T'_1,\dots, T'_r]]) -\epsilon \\
   & \geq \limlm(R_\mf p)-\epsilon.
\end{align*}
for any $\epsilon>0$. Thus, we conclude that $\limlm(R)\geq\limlm(R_\mf p)$, as required.
\end{proof}

\begin{remark}
\label{rmk: lm does not localize well}
One cannot expect that $\lm(R)\geq \lm(R_\mf p)$ in general: for example let $R = k[[x,y,z,t]]/(xyz)$ and $\mf p = (x,y,z)$, then we will see from Theorem~\ref{thm: snc is semistable} and Proposition~\ref{prop: LM of xyz} from Section~\ref{section: example} that  
$\lm(R) = 1$ while $\lm(R_\mf p) = 3/2$. 
\end{remark}

\begin{remark}
We expect the following generalization of Theorem~\ref{thm: localization}. Suppose $(R,\m)$ is a local ring such that $\widehat{R}$ is equidimensional (more generally, one can work on an excellent and biequidimensional scheme). Then we suspect that the function
$\Phi\colon \Spec(R) \to \mathbb{R}$ defined by
$\Phi(\mf p) = \lm(R_\mf p[[T_1,\dots,T_{\dim(R/\mathfrak{p})}]])$, and the function $\Phi'\colon \Spec(R) \to \mathbb{R}$ defined by $\Phi'(\mf p) = \limlm(R_\mf p)$ are upper semicontinuous. 
\end{remark}

\begin{corollary}\label{c loc bound}
Let $(R, \mf m)$ be a local ring
and let $\mf p\in\Spec(R)$ be such that $\hght \mf p + \dim (R/\mf p) = \dim (R)$.
Then $\dim (R)!\lm (R) \geq (\hght \mf p)!\lm(R_\mf p)$.
\end{corollary}
\begin{proof}
Let $d=\dim(R)$, $r=\dim(R/\mf p)$ and $h=\hght \mf p$. By Theorem~\ref{thm: localization} we have 
\begin{align*}
  d!\lm(R) & \geq d!\lm(R_\mf p[[T_1,\dots,T_r]]) \\
   & =h! (d(d-1)\cdots (h+1))\lm(R_\mf p[[T_1,\dots,T_r]]) \\
   & \geq h!\lm(R_\mf p).
\end{align*}
where the last line above is due to Corollary~\ref{cor: Mumford 3.5}.
\end{proof}

We next study the behavior of Lech--Mumford constant under faithfully flat extension. 

\begin{proposition}
\label{prop: faithfully flat extension}
Let $(R, \mf m) \to (S, \mf n)$ be a flat local extension of local rings.
Then we have $\lm (S) \geq \lm (R [[T_1, \ldots, T_r]])$ where $r = \dim (S) - \dim (R)$. Furthermore, we have $\limlm(S)\geq\limlm(R)$.
\end{proposition}
\begin{proof}
Since $S$ is faithfully flat over $R$, $\dim (S) = \dim (R) + \dim (S/\mf mS)$.
Let $\mf p$ be a minimal prime of $\mf m S$ such that $\dim (S/\mf p) = \dim (S/\mf mS)=r$. By Theorem~\ref{thm: localization}, we have 
$\lm(S)\geq \lm(S_\mf p[[T_1,\dots, T_r]])$. By replacing $R\to S$ by $R[[T_1,\dots, T_r]]\to S_\mf p[[T_1,\dots, T_r]]$, we may assume that $\dim(R)=\dim(S)$.
In this case, all we need to show is that $\lm(R)\leq \lm(S)$. But for any $\m$-primary ideal $I$ of $R$, we have
\[
\length_R(R/I) \length_S (S/\mf m S) = \length_S (S/IS), \text{ and } \eh(I)\length_S(S/\m S) = \eh(IS).
\]
It follows that
\[
\lm (R) = \sup_{\sqrt{I} = \mf m} \left\{\frac{\eh(I)}{d! \length_R (R/I)}\right\}
= \sup_{\sqrt{I} = \mf m} \left\{\frac{\eh(IS)}{d! \length_S (S/IS)}\right\} \leq \lm (S).
\]
The proof of $\limlm(S)\geq \limlm(R)$ is entirely similar and we omit the details.
\end{proof}

Putting Theorem~\ref{thm: localization} and Proposition~\ref{prop: faithfully flat extension} together allows us to prove a graded version of the localization inequality. 

\begin{corollary}
\label{cor: graded localization}
Let $R$ be a $\mathbb{N}$-graded ring generated in degree one over a 
local ring $(R_0,\m_0)$ and let $\m=\m_0+R_{>0}$.
Let $\mf p$ be a homogeneous prime ideal in $R$ 
such that $\hght \mf p + \dim (R/\mf p) = \dim (R)$.
Then we have $\grlm(R)=\lm (R_{\m}) \geq \lm (R_{(\mf p)}[[T_1, \ldots, T_r]])$ where $r = \dim (R/\mf p)$ and $R_{(\mf p)}$ is the homogeneous localization.
\end{corollary}
\begin{proof}
By a similar reduction as in the proof of Theorem~\ref{thm: localization}, we may assume that $r=\dim(R/\mf p)=1$. Recall that
$R_{(\mf p)}$ consists of equivalence classes of elements 
of the form $a/b$, where $a,s$ are homogeneous, $\deg a = \deg b$, and $b \notin \mf p$.
Thus, if $z \notin \mf p$ with $\deg z = 1$, then $R_{(\mf p)}$
is the localization of $A \coloneqq R[z^{-1}]_0$ at $\mf q = A \cap \mf p$.
Note that $R_z \cong A[z, z^{-1}]$ and $\mf q R_z=\mf p R_z$.
By Theorem~\ref{thm: localization}, $\lm (R_{\m})\geq \lm(R_{\mf p }[[T]])$.
Finally, since $R_{\mf p}\cong  A[z, z^{-1}]_{\mf q [z, z^{-1}]}\cong A_\mf q (z)$, by Proposition~\ref{prop: faithfully flat extension}, 
$$\lm(R_{\mf p }[[T]]) = \lm(A_\mf q (z)[[T]]) \geq \lm (A_\mf q [[T]]) = \lm (R_{(\mf p)}[[T]]).$$
This completes the proof.
\end{proof}

We also have the following corollary as a consequence of the results discussed so far. This corollary will allow us to reduce to the case of infinite residue field when studying the Lech--Mumford constant. 

\begin{corollary}
\label{cor: purely transcendental field extension}
Let $(R,\m)$ be a local ring and let $R(t)\coloneqq R[t]_{\m R[t]}$. Then $\lm(R) = \lm(R(t))$. 
\end{corollary}
\begin{proof}
Let $S \coloneqq R[t]_{\m + (t)}$ and let $P \coloneqq \m S$. Then clearly $\height(P)+\dim(S/P)=\dim(S)$ and $t$ is a parameter in $S$ whose image is a system of parameters on $S/P$. Thus by Lemma~\ref{lem: localization key}, $\lm(R(t))=\lm(S_P)\leq \lm(S/(t))=\lm(R)$. On the other hand, $\lm(R)\leq \lm(R(t))$ by Proposition~\ref{prop: faithfully flat extension}.
\end{proof}

Given Corollary~\ref{cor: purely transcendental field extension}, one might wonder whether the Lech--Mumford constant is preserved under arbitrary field extensions. Here we record an example showing that this is not always the case. 

\begin{example} \label{Ex: field matters}
Let $R=\overline{\mathbb{F}}_2[x,y,t]_\m/(x^2+ty^2)$ where $\m=(x,y,z)$. By Proposition~\ref{prop: BS for non normal limit} in Section~\ref{section: surface}, we have $\lm(R)=\lm(\widehat{R})=1$. It follows from Theorem~\ref{thm: localization} that $$\lm\left(\frac{\overline{\mathbb{F}}_2(t)[[x,y,z]]}{(x^2+ty^2)}\right)=\lm\left(\frac{\overline{\mathbb{F}}_2(t)[x,y]_{(x,y)}}{(x^2+ty^2)}[[z]]\right)=\lm(R_{(x,y)}[[z]])=1.$$ However, $\lm(\overline{\mathbb{F}_2(t)}[[x,y,z]]/(x^2+ty^2))=2$ by Proposition~\ref{prop: basic results on cLM}(\ref{part multiplicity bounds}).
\end{example}

On the other hand, we do not know whether the Lech--Mumford constant is preserved under separable field extensions.

\begin{question}
\label{question: field extension}
Let $(R,\m, k) \to (R', \m', k')$ be a flat local extension of complete local rings such that $\m'=\m R'$ and $k'$ is separable over $k$. Then is it true that $\lm(R)=\lm(R')$?
\end{question}

We want to point out that 
if we restrict to ideals of fixed colength, i.e., to the invariants 
\[
L_N (R) \coloneqq  \max_{\length(R/I)=N} \left\{ \frac{\eh (I)}{\dim(R)! N}\right\}, 
\]
then it is not true that $L_N(R)$ is invariant under separable field extensions (i.e., the analogous question has a negative answer). 

\begin{proposition}
    Consider a local ring of the form $R = \mathbb{F}_5[[x,y,z]]/(x^4 + y^4 + z^4 + g(x, y, z))$
    or $\mathbb{Q}[[x,y,z]]/(3x^3 + 4x^3 + 5x^3 + g(x,y,z))$, where $g$ consists of higher order terms. 
    Then 
    \[
    \frac {\eh (R)}4 = L_2 (R) < L_2(R \cotimes{k} \overline{k}),
    \]
    where $k$ is the corresponding coefficient field. 
\end{proposition}
\begin{proof}
    We observe that if $\length (R/I) = 2$ 
    then $\eh (I) > \eh (R)$ if and only if $I$
    is integrally closed. Therefore, 
    $L_2 (R) = \eh(R)/4$ if and only if there is 
    no integrally closed ideal $I$ of colength $2$.
    
    By the proof of \cite[Proposition~4]{MQSColengthII} (or taking $J = \m^n$ for $n \gg 0$ in the statement),
    integrally closed ideals of colength $2$ in a local ring $(S, \n, \ell)$
    are parametrized by $\ell$-rational points 
    in $\Proj \Gr_\n (S)$. In our cases, $\Proj \Gr_\m (R)$ does not have any $k$-rational point.
    For the curve $\Proj (\mathbb{F}_5[x,y,z]/(x^4 + y^4 + z^4))$ 
    this follows from the fact that $\{0, 1\}$ are the only forth powers in $\mathbb{F}_5$
    and for Selmer's curve $\Proj (\Q[x,y,z]/(3x^3 + 4y^3 + 5z^3))$, an elementary exposition can be found in \cite{ConradSelmer}.
    
    On the other hand, $\Proj (\Gr_\m (R \cotimes{k} \overline{k}))$ always has $\overline{k}$-rational points as $\overline{k}$ is algebraically closed. 
    Thus 
    \[
    L_2(R \cotimes{k} \overline{k}) \geq 
    \frac{\eh (R) + 1}{4}. \qedhere
    \]    
\end{proof}

We next record the following observation regarding the behavior of the Lech--Mumford constant under blowing up.

\begin{proposition}
\label{prop: blow up cLM}
Let $(R, \mf m)$ be a local domain and 
let $\pi \colon X \to \Spec (R)$ be the blowup of an ideal $I$.
For any closed point $x \in X$ we have $\lm(\sO_{X, x})  \leq \grlm(\Gr_{I}(R))$.
\end{proposition}
\begin{proof}
Since $R$ is local, every closed point $x\in X$ lies in the exceptional fiber $E \coloneqq \pi^{-1}(V(I)) = \Proj (\Gr_{I}(R))$. Let $\pi$ be a local generator of $I\sO_{X,x}$.  
Since $\pi$ is a nonzerodivisor on $\sO_{X,x}$, we have 
$\Gr_{(\pi)} (\sO_{X, x}) \cong (\sO_{X, x}/\pi\sO_{X, x})[T]=\sO_{E, x}[T]$. By Proposition~\ref{prop: LM associated graded}, we have that 
$$\lm(\sO_{X, x}) \leq \grlm(\Gr_{(\pi)} (\sO_{X, x}))=\grlm(\sO_{E,x}[T])=\lm(\sO_{E,x}[[T]]).$$
Since $\sO_{E, x}=(\Gr_{I}(R))_{(P)}$ for a homogeneous prime $P$ such that $\dim(\gr_I(R)/P)=1$, 
by Corollary~\ref{cor: graded localization}, we have
\[
\grlm(\Gr_{I}(R))\geq \lm(\sO_{E,x}[[T]]) \geq \lm(\sO_{X, x}).
\qedhere \]
\end{proof}

\begin{remark}
Tracing the proof of Proposition~\ref{prop: blow up cLM} and Corollary~\ref{cor: graded localization} more carefully, it applies to any numerical measure of singularities $f$ such that:
\begin{enumerate}
\item $f(R_{\mf p}) \leq f(R_{\mf q})$ whenever $\mf p \subseteq \mf q$ in $\Spec(R)$,
\item $f(R) \leq f(\Gr_{(\pi} (R))$ for a regular element $\pi\in R$,
\item $f(R) \leq f(R(t))$ where $R(t)=R[t]_{\m R[t]}$. 
\end{enumerate}
In particular, it applies to the Hilbert--Samuel and Hilbert--Kunz multiplicities. In the case of  Hilbert--Samuel multiplicity, our proof resembles  the argument of Orbanz in \cite{Orbanz} who generalized a classical result of Dade \cite{Dade}: the multiplicity does not increase when blowing up a regular equimultiple prime ideal
(i.e., $R/\mf p$ is regular and $\eh (R_\mf p) = \eh(R_\mf q)$ for any $\mf p \subseteq \mf q$), since these assumptions imply that $\eh(\Gr_{\mf p}(R)) \leq \eh(R_\mf p)$.
\end{remark}

We also point out that Proposition~\ref{prop: blow up cLM} might fail without passing to the associated graded ring. 

\begin{example}
Let $R = \CC[x,y,z]/(x^2 + y^3 + z^6)$. By Theorem~\ref{thm: minimal elliptic deg one} and Example~\ref{example: Laufer} in Section~\ref{section: surface},
$\lm(R) = 8/7$. After blowing up the closed point, the chart corresponding to $z \neq 0$
is defined  by the equation $a^2 + b^3z + z^4$. At the origin it has a simple elliptic singularity of degree $2$ (see Example~\ref{example: Laufer}), so the Lech--Mumford constant of the origin is $4/3$ by Theorem~\ref{thm: minimal elliptic irred exc} in Section~\ref{section: surface}. 
\end{example}

Finally, we study the behavior of the Lech--Mumford constant under finite extensions.
\begin{proposition}
\label{prop: finite extension}
Let $(R, \mf m, k)$ be a local domain and let $(S, \mf n, \ell)$ be a finite extension of $R$. Then $\rank_R(S) \lm (R) \geq \lm(S)$. In particular, if $(S,\mf n, \ell)$ is a finite birational extension of $R$, then $\lm(R)\geq \lm(S)$.\footnote{This appeared without proof in \cite[Proposition~3.10]{Mumford}, which asserts that $\lm(R)\geq\lm(S)$ when $S$ is birational and integral over $R$, we have not been able to verify this fact without mild assumptions.}
\end{proposition}
\begin{proof}
Let $J$ be an $\mf n$-primary ideal in $S$ and denote $I \coloneqq J \cap R$. 
Then
\[
\length_R (R/I) \leq \length_R (S/J)
= \length_S (S/J) [\ell: k].
\]
On the other hand, since $R$ is a domain, we have 
$$\eh(I, R) \rank_R(S)= \eh_R(I, S)=\eh(IS, S)[\ell: k] \geq \eh(J, S)[\ell: k].$$
By combining these inequalities we obtain that 
\[
\frac{\eh(J)}{d!\length_S(S/J)} 
= \frac{\eh (J) [\ell: k]}{d!\length_S (S/J) [\ell: k]}
\leq \frac {\rank_R(S)\eh(I)}{d! \length_R (R/I)}
\leq \rank_R(S)\lm(R).
\]
Therefore, $\lm(S) \leq \rank_R(S)\lm(R)$ and the proposition follows.
\end{proof}

\subsection{Weak semicontinuity}
We now turn our attention to the behavior of Lech--Mumford constants in the fibers of a finite type morphism. We will work in the following setting and notation.

\begin{setting}\label{families setting}
Let $T \to R$ be a finite type flat homomorphism of rings and suppose there is an ideal $I\subseteq R$
such that the composition $T \to R \to R/I$ is an isomorphism (geometrically, $\Spec(R)\to \Spec(T)$ is a flat family together with a section $\sigma\colon \Spec(T)\to \Spec(R)$).  

For each $\p \in \Spec (T)$ we will denote 
$R(\p) \coloneqq R \otimes_{T} T_\p/\p T_\p$. 
Note that $R(\p)$ is a finitely generated algebra over the field $k(\p) \coloneqq T_\p/\p R_\p$
and $IR(\p)$ is a maximal ideal of $R(\p)$. 
We denote by $\overline{R(\mf p)}$ the local ring obtained by the localization of 
$R(\mf p) \otimes_{k(\mf p)} \overline{k(\mf p)}$
at $IR(\mf p)$. 
\end{setting}

We will study the function
$\p \mapsto \lm (\overline{R(\mf p)})$
culminating in a weak semicontinuity result of Proposition~\ref{prop: new weak semicontinuity}.
First we record the following lemma on equidimensionality of fibers, which should be well-known. 

\begin{lemma}\label{lemma: constant dimension}
In Setting~\ref{families setting} suppose that $T$ is a domain and $R(0)$ is equidimensional\footnote{Here $R(0)$ is a finite type $K$-algebra where $K$ is the fraction field of $T$, thus equidimensional means that each irreducible component of $R(0)$, as an affine domain over $K$, has the same dimension.} (e.g., $R$ is a domain). 
Then the function $\Spec(T) \ni \p \mapsto \height IR(\p)$ is constant.
\end{lemma}
\begin{proof}
    The proof follows the outline of 
    \cite[Corollaire 14.2.2]{EGAIV3} (see also \cite{ConradFiberdim}), which shows that if $X \to Y$ is an open dominant locally finite type morphism of irreducible Noetherian schemes, then all fibers $X_y$ are of pure dimension and $\dim X_y$ is constant for all $y\in Y$. 

    We will show that for each $0\neq \p \in \Spec (T)$, $\hght(IR(\p))=\hght(IR(0))$. We first reduce to the case where $T=(V,\pi,\kappa)$ is a DVR and $\p=(\pi)$. By the Krull--Akizuki theorem, there exists a DVR $(V,\pi,\kappa)$ over $T$ in the fraction field of $T$ whose center on $T$ is $\p$.  The generic fiber of $V \to R_V \coloneqq R \otimes_{T} V$ is the same as $R(0)$ while the special fiber is $R_V /\pi R_V  \cong R(\p) \otimes_{k(\p)} \kappa$. Since $k(\p)\to \kappa$ is a field extension, it is clear that $\height I R(\p) = \height I (R_V /\pi R_V)$. Thus we may replace $T$ by $V$ and $R$ by $R_V$ to assume that $T=V$. 
    
    Now we assume that $(V,\pi,\kappa) \to R$ is a finite type flat map. Observe that $\pi \notin P$ for any minimal prime $P$ of $R$ since $\pi$ is a regular element of $V$ and the map is flat. Since the generic fiber is obtained by inverting $\pi$, we have a one-to-one correspondence between minimal primes of $R$ and $R(0)$. Moreover, we have that $R/P$ is torsion-free and thus flat over $V$. Hence, by \cite[Corollaire 14.2.2]{EGAIV3}, for every minimal prime $P$ of $R$, $R/(P,\pi)$ is equidimensional (as a finite type $\kappa$-algebra) and $$\dim (R/(P, \pi)) = \dim (R/P)(0) = \dim (R(0)) = \hght(IR(0))$$
where the second equality is by equidimensionality of $R(0)$ and the third equality follows as $IR(0)$ is a maximal ideal of $R(0)$.  
Since $I(R/\pi R)$ is a maximal ideal of $R/\pi R$, $\n \coloneqq (I, \pi)$ is a maximal ideal of $R$ and we have 
$$\hght(I(R/\pi R))= \dim(R_{\n})-1 = \dim (R_{\n}/PR_{\n}) -1 =\dim(R_{\n} /(P,\pi)R_{\n})=\dim(R/(P,\pi)) $$
for some minimal prime $P$ of $R$, where the last equality above follows from the equidimensionality of $R/(P, \pi)$. It follows that $\hght(I(R/\pi R))= \hght(IR(0))$ as wanted.
\end{proof}

\begin{remark}
Consider the map $f \colon T \coloneqq k[[t]] \to R \coloneqq k[[x,y,z]]/(xy, xz)$ that sends $t \mapsto x+y$. Since $x + y$ is a regular element of $R$, $f$ is flat. Setting $I = (y, z)$, we have $R/I\cong T$ which produce a section of $f$. On the other hand, $\height (I(R/tR)) = 1 > 0 = \height (I (R[1/t]))$. This shows that the equidimensional assumption on $R(0)$ is necessary in Lemma~\ref{lemma: constant dimension}. 
\end{remark}

We next record some consequences of Lemma~\ref{lem: localization key},

\begin{corollary}
    \label{cor: inequality DVR base}
    Let $(D, \pi)$ be a DVR and $D \to R$ be a finite type flat homomorphism of rings with a section given by $P\subseteq R$. 
    Suppose that $\dim (R_P) = \height P(R/\pi R)$
    (for example, this holds if  $R[1/\pi]$ is equidimensional by Lemma~\ref{lemma: constant dimension}).
    Then $\lm (R_{P}) \leq \lm ((R/\pi R)_{(P, \pi)})$.
\end{corollary}
\begin{proof}
Note that $P$ is a prime ideal of $R$ and $(P, \pi)$ is a maximal ideal of $R$. The assumption implies that 
$$\hght P = \dim(R_{(P,\pi)})-1 = \dim(R_{(P,\pi)})- \dim(R_{(P,\pi)}/PR_{(P,\pi)}).$$ Now the conclusion follows immediately from Lemma~\ref{lem: localization key}.
\end{proof}

\begin{corollary}
\label{cor: grober degeneration technical form}
    Let $R=S/I$ where $S= k[x_1,\dots,x_n]$ is a polynomial ring over a field $k$ and let $\m=(x_1,\dots,x_n)$. Let $J \subseteq \m S[t]$ be an ideal such that 
    \begin{enumerate}
        \item $S[t, t^{-1}]/JS[t, t^{-1}] \cong R[t, t^{-1}]$, and 
        \item $t$ is a nonzerodivisor on $S[t]/J$.
    \end{enumerate}
    Then 
    $
    \lm (R_\m) \leq \lm ((S[t]/(t, J))_\m).
    $
\end{corollary}
\begin{proof}
    First, we want to show that 
    \begin{equation}
    \label{eqn: chain of equalities on dimensions}
    \dim (R_\m) = \dim ((S[t]/(t, J))_\m) = \dim (S[t]/J)_{(\m, t)} - 1 = \dim\left((S[t]/J)_{(\m, t)}[t^{-1}]\right).
    \end{equation} 
    Note that the second equality holds since $t$ is a nonzerodivisor on $S[t]/J$, and the last equality holds by the following claim. 
    
    \begin{claim}
    Let $(T, \n)$ be a catenary local ring and $x \in \n$
    be a parameter. Then $\dim (T_x) = \dim(T) - 1$.
    \end{claim}
    \begin{proof}
        Clearly, $\dim (T_x) \leq \dim (T) - 1$, so it suffices to show the opposite inequality. Let $\p$ be a minimal prime of $T$ such that $\dim (T/\p) = \dim (T)$. Note that $x \notin \p$ because it is a parameter, so it suffices to prove the claim in the case when $T$ is a domain. 

        We may now extend $x$ to a system of parameters $x_1, \ldots, x_{d-1}, x_d = x$ and let $\q$ be a minimal prime of $(x_1, \ldots, x_{d-1})$. Since $T$ is a catenary local domain, we have $\dim (T_\q) = \dim(T)-1$. Since $x \notin \q$ by the construction, the claim follows. 
    \end{proof}

Therefore, to prove (\ref{eqn: chain of equalities on dimensions}) it remains to show that $\dim (R_\m) = \dim\left( (S[t]/J)_{(\m, t)}[t^{-1}]\right)$. We may present $(S[t]/J)_{(\m, t)}[1/t]$ as a localization of $R_\m[t, t^{-1}]$ at a multiplicative closed subset $W$ that avoids $\m$. Thus it suffices to show that $W$ intersects all maximal ideals 
    of $R_\m[t, t^{-1}]$. But those ideals can be lifted to ideals of the form $(\m, f(t))$ for some polynomial $f(t) \in S[t]/J$ 
    and it is clear that they are trivialized in $(S[t]/J)_{(\m, t)}[1/t]$.

    We now apply Corollary~\ref{cor: inequality DVR base} 
    to the map $k[t]_{(t)} \to (S[t]/J)\otimes_{k[t]}k[t]_{(t)}$ (with $P=\m$)
    and obtain by Corollary~\ref{cor: purely transcendental field extension} that 
    \[
    \lm(R_\m) =  \lm (R_\m(t)) \leq \lm ((S[t]/(t, J))_\m). \qedhere
    \]
\end{proof} 

We include two major applications of Corollary~\ref{cor: grober degeneration technical form}. We will use these results to compute examples in later sections. 

\begin{corollary}
\label{cor: grober degeneration}
Let $R=S/I$ where $S= k[x_1,\dots,x_n]$ is a polynomial ring over a field $k$ and let $\m=(x_1,\dots,x_n)$. Suppose we give weights to the variables such that $\deg(x_i)=d_i\in\mathbb{N}$ and let $\init(I)$ be the initial ideal of $I$ with respect to the weights. Then we have 
$$\lm(R_\m)\leq \lm((S/\init(I))_\m).$$
\end{corollary}
\begin{proof}
Consider the $\mathbb{Z}$-graded ring $S[t]$ by assigning $\deg(t)=-1$. Let $J\coloneqq\hom(I)\subseteq S[t]$ be the homogenization of $I$ in $S[t]$. The conclusion follows from Corollary~\ref{cor: grober degeneration technical form} by noting that $S[t, t^{-1}]/JS[t, t^{-1}] \cong R[t, t^{-1}]$ and $(S[t]/(t, J))_\m=(S/\init(I))_\m$. 
\end{proof}

\begin{corollary}
\label{cor: LM monomial ideal}
Let $R=S/I$ where $S=k[x_1,\dots,x_n]$ is a standard graded polynomial ring over a field $k$ with $\m=(x_1,\dots,x_n)$, and $I\subseteq R$ is a homogeneous ideal. Fix a monomial order $\prec$ on $S$ and let $\init_\prec(I)$ be the initial ideal with respect to $\prec$. Then we have  
$\lm(R_\m)\leq \lm((S/\init_{\prec}(I))_\m)$.
\end{corollary}
\begin{proof}
We follow the standard construction as in \cite[15.16 and 15.17]{EisenbudBook}. We choose an appropriate weight function $w$ so that $\init_w(I)=\init_\prec(I)$. Let $J\coloneqq\hom(I)\subseteq S[t]$ be the $w$-homogenization of $I$ in $S[t]$. The desired inequality follows from Corollary~\ref{cor: grober degeneration technical form} by noting that $S[t, t^{-1}]/JS[t, t^{-1}] \cong R[t, t^{-1}]$ and $(S[t]/(t, J))_\m=(S/\init_{\prec}(I))_\m$. 
\end{proof}

To prove the weak semicontinuity result on Lech--Mumford constant in a flat family, we need a lemma that reduces a general field extension to an algebraic field extension.

\begin{lemma}
\label{lem: technical field bound}
    Let $k$ be a field and $R = k[x_1, \ldots, x_m]/I$ be a finitely generated $k$-algebra with $\m = (x_1, \ldots, x_m)$.
    Then for any field extension $k \subseteq \ell$ we have
    \[
    \lm ((R \otimes_k \ell)_\m) \leq \lm ((R \otimes_k \overline{k})_\m).
    \]
\end{lemma}
\begin{proof}
First note that any $\m$-primary ideal $J$ of $\ell [x_1, \ldots, x_m]/I$  has finitely many generators, which are polynomials in $x_1,\dots, x_m$ with finitely many coefficients in $\ell$. Hence, we can find a finitely generated field extension $k \subset \ell'$ inside $\ell$ and an $\m$-primary ideal $J'\subseteq R' \coloneqq R\otimes_k\ell'$ such that $J = J' (R\otimes_k\ell)$. Since the multiplicities and the colengths of $J$ and $J'$ are equal, we may assume that $\ell$ is a finitely generated field extension of $k$. We may further assume that $\ell$ is separably generated over $k$: by \cite[\href{https://stacks.math.columbia.edu/tag/04KM}{Tag 04KM}]{stacks-project}, there exists a commutative diagram:
\[
\xymatrix{
k \ar[r]  \ar[d] & \ell \ar[d] \\
k' \ar[r] & \ell' 
}
\]
where $k \subseteq k'$ and $\ell \subseteq \ell'$ are finite purely inseparable and $k'\to \ell'$ is separably generated. We may then replace $k$ by $k'$, $R$ by $R\otimes_kk'$ and $\ell$ by $\ell'$: this will not change the right hand side and will only possibly increase the left hand side of the desired inequality. 

We now write $\ell = k(t_1, \ldots, t_n)[x]/(f(x))$
where $t_1, \ldots, t_n$ are a transcendental basis of $\ell$ over $k$ and $f(x)$ is a monic separable polynomial. 
    By collecting denominators of $f(x)$, we may assume that $f(x) \in k[t_1, \ldots, t_n, g^{-1}][x]$ for some $g\in k[t_1,\dots,t_n]$. Let $A \coloneqq k[t_1, \ldots, t_n]$. By the assumption, $k[t_1, \ldots, t_n, g^{-1}][x]/(f(x))$ is generically \'etale over $A$, so we may find $0 \neq h \in A$ such that $A_h [x]/(f(x))$ is \'etale over $A_h$ (we may assume that $g$ divides $h$ so that $f(x)$ is well-defined in $A_h[x]$). 
    Thus for any maximal ideal
    $h \notin \n \subseteq A$ the quotient ring 
    $(A_h/\n)[x]/(f(x)) \cong (A/\n)[x]/(f(x))$ is a finite \'etale extension of $A/\n$, 
    so it is a product of finite separable field extensions of $A/\n$. 
    
    By lifting a maximal ideal corresponding to one component, say $L$, of $(A_h/\n)[x]/(f(x))$, we find a maximal ideal $\mathfrak{N}$ of $R \otimes_k A_h[x]/(f(x))$ such that 
    the extension of $(\m, \n)$ to $S \coloneqq (R \otimes_k A_h[x]/({f}(x)))_{\mathfrak N}$ is the maximal ideal $\mf N$, $\m S$ is a prime ideal ($S/\m S$ is the localization of $A_h[x]/(f(x))$ at one of its maximal ideal, so it is a regular local ring), and $S/{\mf N} = L$ is an algebraic field extension of $k$.

    Let $g_1, \ldots, g_n \in A$ be elements whose images are minimal generators of $\n A_\n$. By Lemma~\ref{lem: localization key} applied to $S$ we have that 
    $
    \lm (S_{\m S}) \leq  \lm (S/(g_1, \ldots, g_n)).
    $
    Now, observe that 
    \[
    S_{\m S} \cong (R (t_1, \ldots, t_n)[x]/(f(x)))_\m
    \cong (R \otimes_k \ell)_\m
    \]
    while 
    \begin{align*}
    S/(g_1, \ldots, g_n) 
    &\cong \big(R \otimes_k A_h[x]/(f(x), g_1, \ldots, g_n)\big)_{\mf N}
    \cong 
     \big(R \otimes_k A_\n[x]/(f(x), g_1, \ldots, g_n) \big)_{\mf N}
    \\ &\cong 
    \big(R \otimes_k (A_\n/\n A_\n)[x]/(f(x)) \big)_{\mf N}
    \cong 
    (R \otimes_k L)_\m. 
    \end{align*}
Thus we have $\lm((R \otimes_k \ell)_\m)\leq \lm((R \otimes_k L)_\m)\leq \lm((R \otimes_k \overline{k})_\m)$ where the last inequality follows from the fact that $L$ is algebraic over $k$.  
\end{proof}

We can now prove the main result of this subsection.

\begin{proposition}\label{prop: new weak semicontinuity}
Let $T \to R$ be a finite type flat map of rings so that $T$ is a domain and the generic fiber $R(0)$ is equidimensional. Let $I\subseteq R$ be an ideal 
such that the composition $T \to R \to R/I$ is an isomorphism. Let $\mf p \subseteq \mf q$ be two prime ideals of $T$. Then we have 
$$\lm(\overline{R(\mf p)})\leq \lm(\overline{R(\mf q)}).$$
\end{proposition}
\begin{proof}
   By base change to $(T/\p)_\q$ and using induction on its dimension (via Lemma~\ref{lemma: constant dimension}), we may assume that $T$ is a one-dimensional local domain with maximal ideal $\q$ and fraction field $F$.
   Since every $I\overline{R(0)}$-primary ideal of $\overline{R(0)}$ is extended from $(R(0) \otimes_F L)_I$ for some finite field extension $L$ of $F$, it suffices to show that for any finite field extension $F \subseteq L$ we have 
   $\lm ((R(0) \otimes_F L)_I) \leq \lm(\overline{R(\mf q)})$. By localizing the integral closure of $T$ in $L$ 
   at a maximal ideal, we may find a DVR $(V, \pi)$ such that 
   the finite type flat map $V \to R' \cong R \otimes_T V$ with a section $IR'$ has generic fiber $R(0) \otimes_F L$.
   Therefore, by Corollary~\ref{cor: inequality DVR base}
    \[
    \lm ((R(0) \otimes_F L)_I) \leq \lm ((R'/\pi R')_{(\pi, I)}).
    \]
   Since $R'/\pi R' \cong R(\q) \otimes_{k(\q)} V/\pi V$, it follows from Lemma~\ref{lem: technical field bound} that 
   \[\lm ((R'/\pi R')_{(\pi, I)})= \lm \big((R(\q) \otimes_{k(\q)} V/\pi V)_{I}\big)\leq \lm \big((R(\q) \otimes_{k(\q)} \overline{k(\q)})_{I}\big)= \lm(\overline{R(\mf q)}). \qedhere \] 
   \end{proof}

\begin{remark}
We expect that the Lech--Mumford constant function in a flat family should be upper semicontinuous. More precisely, we suspect that, with notations as in Proposition~\ref{prop: new weak semicontinuity}, fix $\mathfrak{q}\in\Spec(T)$, then for any $\epsilon>0$, there exists $f\notin \mathfrak{q}$ so that 
$$\lm(\overline{R(\mathfrak{p})}) < \lm(\overline{R(\mathfrak{q})})+\epsilon$$
for all $\mathfrak{p}\in\Spec(T)$ that do not contain $f$.
\end{remark}

\newpage
\section{Examples: surface singularities}
\label{section: surface}

In this section, we study Lech-Mumford constant of certain two-dimensional singularities. 

\subsection{Lech-stable singularities} We aim to determine which surface singularities are Lech-stable. We begin with a key lemma.
\begin{lemma}
\label{lem: LM for two dimensional stable}
Let $(R, \mf m)$ be a Cohen--Macaulay local ring of dimension two with an infinite residue field. Suppose $I\subseteq R$ is an $\m$-primary ideal such that $I^2\subseteq J$ for some minimal reduction $J$ of $I$ (e.g., $I$ is stable). Then 
$$\eh(I)\leq (\type(R)+1)\length(R/I).$$
Furthermore, if every $\mf m$-primary integrally closed ideal is stable, then either $R$ is regular or $\lm(R) = \eh(R)/2$. 
\end{lemma}
\begin{proof}
All our assumptions are preserved when passing from $R$ to $\widehat{R}$, so we may replace $R$ by $\widehat{R}$ to assume that $R$ has a canonical module $\omega_R$. Let $J$ be a minimal reduction of $I$ such that $I^2\subseteq J$. Now in $S \coloneqq R/J$ we have
\[
    \length (\omega_S) = \length (\omega_S/I\omega_S) + \length (I\omega_S) \leq 
 \length (\omega_S/I\omega_S) + \length (\Hom_S (S/I, \omega_S))
 \leq (\type(S) + 1) \length (S/IS).
\]
Since $S$ is Artinian, 
$\length(\omega_S) = \length(S)=\length(R/J)=\eh(J)=\eh(I)$
and putting these together yields that
\[\eh(I)\leq (\type(S) + 1) \length (S/IS)=(\type(R) + 1) \length (R/I).
\]

We now prove the second part. The proof of the first part and Proposition~\ref{prop: basic results on cLM}(\ref{part restrict to integrally closed}) show that $\lm(R) \leq (\type (R) + 1)/2$. Let $\mf q$ be a minimal reduction of $\m$ such that $\m^2\subseteq \mf q$. Since we may assume that $R$ is not regular, $\m \neq \mf q$ and thus we have $\m/\mf q=(\mf q : \m)/\mf q$. It follows that
\[
\eh(R)=\eh(\mf q)=\length(R/\mf q)=\length(\m/\mf q)+1=\length((\mf q : \m)/\mf q)+1=\type(R)+1.
\]
This shows that $\lm(R) = (\type(R) + 1)/2 = \eh(R)/2$.
\end{proof}

We can now compute the Lech--Mumford constant for two-dimensional pseudo-rational singularities. 

\begin{corollary}
\label{cor: RDP Lech stable}
Let $(R, \mf m)$ be a pseudo-rational local ring of dimension two. Then either $R$ is regular or $\lm(R)=\eh(R)/2$. In particular, if $R$ is Gorenstein and pseudo-rational of dimension two (i.e., $R$ is a rational double point), then $R$ is Lech-stable. 
\end{corollary}
\begin{proof}
Suppose $R$ is pseudo-rational, then so is $R(t)=R[t]_{\m R[t]}$. Thus by Corollary~\ref{cor: purely transcendental field extension} we may assume that $R$ has an infinite residue field. Now the first statement follows from Theorem~\ref{thm: Lipman--Teissier} and Lemma~\ref{lem: LM for two dimensional stable}. The second statement follows from the first statement and the fact that two-dimensional Gorenstein pseudo-rational singularities of are hypersurfaces of multiplicity two, see \cite[Theorem 4.17]{Reid} or \cite[Lemma 4.1]{CRMPST}.
\end{proof}

Our goal in this subsection is to classify Lech-stable Cohen--Macaulay singularities in dimension two. We start with the normal case. In this case, the converse to Corollary~\ref{cor: RDP Lech stable} was originally obtained by Goto--Iai--Watanabe \cite[Proposition 7.5]{GotoIaiWatanabe}. We include a short proof here for completeness (and to emphasize that the result holds in arbitrary characteristic). We will use Kato's Riemann--Roch formula which connects colength and multiplicity using intersection theory. We provide a proof since we cannot find a good reference in the generality stated below.

\begin{theorem}[{\cite{Kato} or \cite[Theorem~2.8]{OWY}}]
\label{thm: Kato RR}
Let $(R, \mf m)$ be a two-dimensional excellent normal local ring admitting a dualizing complex.
Let $f \colon Y \to \Spec(R)$ be a resolution of singularities. Suppose $E$ is an exceptional effective divisor on $Y$ with $I \coloneqq \Gamma (Y, \sO_Y (-E)) \subseteq R$. Then
\[
\length(R/I) + h^1(Y, \sO_Y(-E))  = -\frac{E^2 + K_Y\cdot E}{2} + p_g(R),
\]
where $p_g(R) \coloneqq h^1(Y, \sO_Y)$.
\end{theorem}
\begin{proof}
We follow the approach of \cite[Lemma 23.1]{LipmanRationalSingularity}.
By taking the cohomology exact sequence associated with the short exact sequence
\[
0 \to \sO_Y(-E) \to \sO_Y \to \sO_E \to 0
\]
and noting that $\Gamma (Y, \sO_Y) = R$, we obtain that 
\[
\length (R/I) + h^1(Y, \sO_Y(-E)) = \chi(E) + p_g(R). 
\]
By adjunction and the Riemann--Roch for curves 
we obtain that 
$$K_Y\cdot E + E^2 = \deg(\omega_E) = \chi(E, \omega_E) - \chi(E, \sO_E) = -2\chi(E),$$ 
The conclusion follows by combining the two equations above.
\end{proof}

\begin{theorem}[Goto--Iai--Watanabe]
Let $(R,\m)$ be a two-dimensional complete normal local ring of dimension two. Then the following are equivalent: 
\begin{enumerate}
    \item $R$ is Lech-stable.
    \item $R$ is Gorenstein and pseudo-rational.
\end{enumerate}
\end{theorem}
\begin{proof}
First note that $(2)\Rightarrow(1)$ is contained in Corollary~\ref{cor: RDP Lech stable} and we only need to prove $(1)\Rightarrow(2)$. 

So we assume that $R$ is Lech-stable, we know that $\eh(R)\leq 2$ by Remark~\ref{rmk: mult bound for Lech-stable and semistable}. It follows from Abhyankar's inequality that $\edim(R)\leq 3$, and thus $R$ is either regular or a hypersurface. In particular, $R$ is Gorenstein. Let $\pi\colon Y\to X\coloneqq\Spec(R)$ be the minimal resolution, we know that $K_{Y/X} \coloneqq K_Y-\pi^*K_X\leq 0$. If $K_{Y/X}=0$, then $R^1\pi_*\sO_Y\cong R^1\pi_*\omega_Y=0$ and thus $R$ is pseudo-rational. Thus we may assume that $K_{Y/X}<0$. Let $I\subseteq R$ be an $\m$-primary ideal such that $I\sO_Y=\sO_Y(-E)$ is $\pi$-ample (for example, any $I$ such that $Y\cong \Bl_I(R)$). By Theorem~\ref{thm: Kato RR} and noting that $h^1(Y,\sO_Y(-nE))=0$ for $n\gg0$ as $-E$ is $\pi$-ample, we have
$$\length(R/I^n) = \frac{-E^2}{2} n^2 + \frac{K_Y\cdot (-E)}{2} n + p_g(R)$$
for all $n\gg0$. Since $K_{Y/X}<0$ and $-E$ is $\pi$-ample, we know that $K_Y\cdot (-E)=K_{Y/X}\cdot (-E)<0$ and thus for $n\gg0$, we have  
$$2\length(R/I^n) < -E^2n^2 = -(-nE)^2 = \eh(I^n).$$
In particular, this shows that $\lm(R)>1$ contradicting that $R$ is Lech-stable.
\end{proof}

We next state our classification of Lech-stable Cohen--Macaulay surface singularities. 

\begin{theorem}
\label{thm: Lech-stable CM surface}
Let $(R, \mf m, k)$ be a Lech-stable complete Cohen--Macaulay local ring of dimension two such that $k$ is separably closed.  
Suppose that $\charc(k) = 0$ or $\charc (k) > 5$. 
Then $R$ is Lech-stable if and only if $R$ is either regular or a possibly non-normal ADE singularity in the following sense: there exists a complete regular local ring $S$ of dimension three with the maximal ideal $\mf n = (x,y,z)$ such that $\widehat{R} \cong S/fS$ with $f$ one of the following:
\[
\begin{cases}
\,\ A_n\colon f = x^2 + y^2 + z^{n + 1}, \quad n \geq 1, \\ 
\,\ D_n\colon f = x^2 + z(y^2 + z^{n - 2}),  \quad n\geq 4, \\
\,\ A_\infty\colon f = x^2 + y^2, \\
\,\ D_\infty\colon f = x^2 + zy^2, \\
\,\ E_6\colon f = x^2 + y^3 + z^4, \\
\,\ E_7\colon f = x^2 + y^3 + yz^3,\\
\,\ E_8\colon f = x^2 + y^3 + z^5.    
\end{cases}
\]
Moreover, these equations define Lech-stable singularities (even if $k$ is not separably closed).
\end{theorem}

\begin{remark}
\begin{enumerate}
\item We expect that even when the residue characteristic is $2, 3, 5$, Lech-stable Cohen--Macaulay singularities are still the rational double points and their non-normal limits as classified by Artin \cite{Artin}.
\item We also expect that, without assuming $k$ is separably closed in Theorem~\ref{thm: Lech-stable CM surface}, the completion of the strict Henselization $\widehat{R^{\sh}}$ should be isomorphic to $S/fS$ as in the theorem. But our theorem does not imply this because it is not clear to us that $\widehat{R^{\sh}}$ is Lech-stable, see Question~\ref{question: field extension}. 
\item The Cohen--Macaulay assumption in Theorem~\ref{thm: Lech-stable CM surface} is necessary for having a viable classification since, by Proposition~\ref{prop: finite extension}, any finite birational extension of a Lech-stable singularity is still Lech-stable.
\end{enumerate}
\end{remark}

To prove Theorem~\ref{thm: Lech-stable CM surface}, we will first extend Corollary~\ref{cor: RDP Lech stable} from rational double points to their non-normal limits $A_\infty$ and $ D_\infty$. We shall do this by showing that a version of the Brian{\c c}on--Skoda theorem holds in these rings. 
The first ingredient is the characterization of stable ideals in Theorem~\ref{thm: Northcott} and 
the second is the $\mf m$-adic stability of the associated graded rings due to Srinivas--Trivedi (see \cite{SrinivasTrivedi}, and \cite{MaPhamSmirnovFilterRegularSequence} for more general results).

\begin{theorem}[{\cite[Theorem~3]{SrinivasTrivedi} or \cite[Main Theorem]{MaPhamSmirnovFilterRegularSequence}}]
\label{thm: small perturbation}
Let $(R, \mf m)$ be a local ring,
$I = (f_1, \ldots, f_k)$ be an ideal generated by a regular sequence,
and $J$ be an ideal such that $I + J$ is $\mf m$-primary. Then there exists $N$ such that
for any ideal $I' = (f_1 + g_1, \ldots, f_k + g_k)$ with $g_i \in \mf m^N$ we have
\[
\Gr_J( R/I) \cong \Gr_J (R/I').
\]
In particular, the Hilbert--Samuel functions of $J(R/I)$ and $J(R/I')$ are the same.
\end{theorem}

We can now prove the promised Brian{\c c}on--Skoda theorem for $A_\infty$ and $D_\infty$ singularities.

\begin{proposition}
\label{prop: BS for non normal limit}
Let $(S, \m)$ be a complete regular local ring of dimension three with an infinite residue field and let $x,y,z$ be a regular system of parameters of $S$ (i.e., $\m = (x,y,z)$). Let $R = S/(f)$ where $f = xy$ or $f=x^2 + y^2z$.
Then for any integrally closed $\mf m$-primary ideal $I$ of $R$, we have $I^2 = JI$ for some minimal reduction $J$ of $I$. As a consequence, $R$ is Lech-stable. 
\end{proposition}
\begin{proof}
By Theorem~\ref{thm: small perturbation}, we know that there exists $N\gg0$ such that
\[
\gr_I(R) \cong \gr_I (S/(f + z^{N+1})) \text{ when $f=xy$,}
\]
\[
\gr_I(R) \cong \gr_I (S/(f + yz^N)) \text{ when $f=x^2+y^2z$.}
\]
Since $S/(xy + z^N)$ and $S/(x^2+y^2z+yz^N)$ are rational double points in all characteristics,\footnote{These are $A_N$ and $D^0_{2N}$ in Artin's list \cite{Artin}, and they are rational singularities in all characteristics.} we know that $\eh_2 (IS) = 0$ by Theorem~\ref{thm: Lipman--Teissier} and Theorem~\ref{thm: Northcott}.
It follows that $\eh_2(I) = 0$ and thus $I$ is stable by Theorem~\ref{thm: Northcott} again. The last assertion then follows from Lemma~\ref{lem: LM for two dimensional stable}.
\end{proof}

\begin{remark}
In the notation of Proposition~\ref{prop: BS for non normal limit},  set $R'=S/(xy + z^N)$ when $f=xy$ and $R'\coloneqq S/(x^2+y^2z+yz^N)$ when $f=x^2+y^2z$. Another way to see that $R$ is Lech-stable is to observe that 
$$\frac{\eh(I)}{2\length(R/I)}= \frac{\eh(IR')}{2\length(R'/IR')}\leq \lm(R')=1$$
where the first equality follows from Theorem~\ref{thm: small perturbation} and the last equality follows from Theorem~\ref{thm: Lipman--Teissier} as $R'$ has rational double point as in the proof of Proposition~\ref{prop: BS for non normal limit}. Since this is true for all $\m$-primary ideals $I\subseteq R$, we have $\lm(R)=1$ and thus $R$ is Lech-stable. This avoids the use of Theorem~\ref{thm: Northcott}, but does not shed light on the Brian{\c c}on--Skoda property of $A_\infty$ and $D_\infty$.
\end{remark}

Based on Theorem~\ref{thm: small perturbation} and Proposition~\ref{prop: BS for non normal limit}, it seems natural to ask whether the Lech--Mumford constant is stable under small perturbations. We do not know this even for hypersurfaces so we record it here as a question. 

\begin{question}
Let $(S,\m)$ be a complete regular local ring and $0\neq f\in S$. Then does there exist $N$ such that for all $g\in \m^N$, we have 
$\lm(S/f)=\lm(S/(f+g))$?
\end{question}

We next record a useful lemma, essentially due to Mumford, which will be needed in the proof of Theorem~\ref{thm: Lech-stable CM surface}.

\begin{lemma}
\label{lem: Mumford useful lemma}
Let $(S, \mf m)$ be a three-dimensional regular local ring and let $x,y,z$ be a regular system of parameters. Suppose $R = S/(x^2 + g)$ such that $g \in (y, z^2)^3$. Then the ideal $I =(y, z^2)^2 + x\mf m$ is integral over $(y^2, z^4)$.
\end{lemma}
\begin{proof}
It is enough to check that $I$ is integral over $(y, z^2)^2$. To see this, simply note that 
\[(xy)^2 = y^2g \in (y, z^2)^4, \text{ and } (xz)^4 = g^2z^4 \in (y, z^2)^{8}. \qedhere \]
\end{proof}

Now we are ready to prove Theorem~\ref{thm: Lech-stable CM surface}.

\begin{proof}[Proof of Theorem~\ref{thm: Lech-stable CM surface}]
We begin by proving the last claim: the singularities in the list are Lech-stable. It is well-known that the equations defining $A_n, D_n, E_6, E_7, E_8$ are rational double points when the residue characteristic is zero or larger than five (see \cite{CRMPST}). Thus by Corollary~\ref{cor: RDP Lech stable} they are Lech-stable. We have also shown in Proposition~\ref{prop: BS for non normal limit} that $A_\infty$ and $D_\infty$ are Lech-stable. 

It remains to show that if $R$ is Lech-stable, then $\widehat{R}$ is isomorphic to one of the singularities in the list. We may assume that $R$ is complete. Since $R$ is Lech-stable, we know that $\eh(R)\leq 2$ by Remark~\ref{rmk: mult bound for Lech-stable and semistable}. It follows from Abhyankar's inequality that $\edim(R)\leq 3$, and thus $R$ is either regular or a hypersurface of multiplicity two, i.e., $R=S/fS$ where $S$ is a complete regular local ring of dimension three. We will analyze the possible shape of $f$, following the strategy in \cite[Section 3]{CRMPST}. 

First of all, by \cite[Lemma~3.1]{CRMPST} we may write 
$f = x^2 + g$, $f = x^2 + y^2 + g$, or $f = x^2 + y^2 + z^2 + g$ with $g \in \mf n^3$. Next, following \cite[Proposition~3.2]{CRMPST}, the case that $f = x^2 + y^2 + z^2 + g$ can be transformed to $f = x^2 + y^2 + z^2$ after a change of generators. 
Similarly, following \cite[Proposition~3.3]{CRMPST}, the case $f = x^2 + y^2 + g$ can be transformed to $f = x^2 + y^2 + z^n$ for $n\geq 3$ or $f = x^2 + y^2$.

Now in the remaining case that $f=x^2+g$ where $g\in\n^3$, the argument in \cite[Lemma~3.5]{CRMPST} shows that after a change of generators we may assume that 
$g \in (y,z)^3$.  Furthermore, we must have $\ord_{\m}(g) \leq 3$, since otherwise $x$ is integral over $\mf m^2$, so $\length(R/\overline{\m^2})\leq 3$, and thus 
$$\lm(R)\geq \frac{\eh(\overline{\m^2})}{2\length(R/\overline{\m^2})}\geq \frac{8}{6} >1$$
contradicting that $R$ is Lech-stable. Therefore we may assume that 
$$f = x^2 + ay^3 + b y^2 z + c y z^2 + d z^3 + h$$ 
where $a,b,c, d$ are either units or zero (but not all of them are zero) and $h \in (y,z)^4$. Then by the same argument as in \cite[Lemma~3.5]{CRMPST}, we may reduce to the case that $f$ is one of the following three forms: 
$f = x^2 + y^2 z + z^3 + h$, $f = x^2 + yz^2 + h$, or $f = x^2 + z^3 + h$ with $h \in (y,z)^4$. 
In the first case, \cite[Proposition~3.6]{CRMPST} brings us to a $D_4$ singularity.
In the second case, the proof of \cite[Proposition~3.7]{CRMPST} shows that $f$ is either a $D_n$ singularity for some $n\geq 5$ or a $D_\infty$ singularity. 

Finally, we arrive at $f = x^2 + y^3 + h$ with $h \in (y,z)^4$. In this case, following the argument of \cite[Proposition 3.9]{CRMPST}, after a change of generators we may write $f=x^2+y^3+z^4$, $f=x^2+y^3+yz^3$, or $f=x^2+y^3+z^5$ unless there exist $\eta_1$ and $\eta_2$ such that 
$$f=x^2+y^3+\eta_1yz^4+\eta_2z^6,$$ 
see the last paragraph in the proof of \cite[Proposition 3.9]{CRMPST}. However, this last case cannot happen since in this case we have $f-x^2\in(y, z^2)^3$ and hence by Lemma~\ref{lem: Mumford useful lemma},  $\length(R/\overline{(y^2, z^4)})\leq \length(R/((y,z^2)^2+x\m))\leq 7$ and thus
$$\lm(R)\geq \frac{\eh(\overline{(y^2, z^4)})}{2\length(R/\overline{(y^2, z^4)})}\geq \frac{16}{14} >1$$
contradicting that $R$ is Lech-stable. This completes the proof.
\end{proof}

\begin{remark}
Via Theorem~\ref{thm: Lech-stable CM surface} and Proposition~\ref{prop: Mumford prop 3.5}, we obtain examples of Lech-stable singularities in higher dimensions that are not regular. For example, the ADE singularities are Lech-stable in all dimensions. One particular family of such examples are compound Du Val (cDV) singularities
\[
\widehat {R} \cong {\mathbb C}[[x,y,z,t]]/(f(x,y,z) + tg(x,y,z,t)),
\]
where $f$ is one of the ADE equations in Theorem~\ref{thm: Lech-stable CM surface} and $g$ is arbitrary. It was shown by Reid in \cite[1.1]{ReidC3f} that isolated cDV singularities are precisely Gorenstein terminal singularities in dimension three.
\end{remark}


\subsection{Minimally elliptic singularities}
Using the rich theory of surface singularities, we are able to compute the Lech--Mumford constant of singularities 
closest to rational singularities. We refer the reader to \cite{ArtinIsolatedRationalSurface, Laufer, Reid, IshiiBook, OWY} for additional background information.

\begin{definition}
Let $(R, \mf m)$ be a two-dimensional excellent normal local ring containing an algebraically closed field $k\cong R/\m$.
Let $X \to \Spec(R)$ be the minimal resolution with reduced exceptional divisor $E\coloneqq E_0+\cdots +E_r$ where each $E_i$ is irreducible. The {\it fundamental cycle} $Z_f$ is the unique nonzero minimal effective divisor on $X$ supported in $E$ such that $Z_fE_i\leq 0$ for all $i$ (the existence of $Z_f$ follows from the Hodge index theorem). 

We say that $R$ has a {\it minimally elliptic singularity} if $R$ is Gorenstein with geometric genus $p_g(R) \coloneqq \length(H^1(X, \sO_X)) = 1$. The number $e\coloneqq-Z_f^2$ is called the degree of $R$. 

If, in addition, the reduced exceptional divisor $E=E_0$ is a non-singular elliptic curve, then 
we say that $R$ is a {\it simple elliptic singularity}. In this case $Z_f=E_0$ and thus the degree of $R$ is $-E_0^2$.
\end{definition}

Laufer (\cite[Theorem~3.13]{Laufer}) showed that minimally elliptic singularities of degree at most three are hypersurfaces and 
listed the dual graphs of their minimal resolutions. We list some examples from \cite[Section 5, Tables 1-3]{Laufer}. 

\begin{example}
\label{example: Laufer}
In $\mathbb{C}[[x,y,z]]$, we have:
\begin{enumerate}
\item $z^2 + y^3 + x^6$ is a simple elliptic singularity of degree $1$; 
\item $z^2 + xy^3 + x^4$ is a simple elliptic singularity of degree $2$;
\item $z^3 + y^3 + x^3$ is a simple elliptic singularity of degree $3$;
\item $z^2 + y^3 + x^7$ is a minimally elliptic singularity of degree $1$, its exceptional set is a rational curve with a cusp singularity;
\item $z^2 + y^3 + x^8$ is a minimally elliptic singularity of degree $1$, its exceptional set consists of three non-singular rational
curves meeting transversely at the same point;
\item $z^2 + x(y+x)(y^2+x^3)$ is a minimally elliptic singularity of degree $2$, its exceptional set is a rational curve with a node singularity;
\item $xz^2 + y^3 + x^4$ is a minimally elliptic singularity of degree $3$, its exceptional set is a rational curve with a cusp singularity.
\end{enumerate}
\end{example}

Our techniques allow us to compute the Lech--Mumford constant for minimally elliptic singularities whose exceptional set is irreducible or whose degree is $1$. In particular, the list in Example~\ref{example: Laufer} is covered by our theorems.

\begin{theorem}
\label{thm: minimal elliptic irred exc}
Let $(R, \mf m)$ be a minimally elliptic singularity of degree $e$. 
Let $X \to \Spec(R)$ be the minimal resolution, and suppose that $E_0$ is the only irreducible exceptional divisor (e.g., $R$ is simple elliptic). Then, the following hold:
\[
\begin{cases}
\text{ If $e \geq 3$, then $\lm (R) = e/2$ and $\lm(R)$ is computed by $\m$.} \\
\text{ If $e = 2$, then $\lm(R) = 4/3$ and $\lm(R)$ is computed by $\overline{\mf m^2}$. } \\
\text{ If $e = 1$, then $\lm(R) = 8/7$.}
\end{cases}
\]
\end{theorem}
\begin{proof}
Let $I\subseteq R$ be an integrally closed $\mf m$-primary ideal. By choosing a resolution $Y \to \Spec(R)$ that dominates the blowup of $I$, we have that $I\sO_Y=\sO_Y(-Z)$ where $Z$ is an anti-nef effective exceptional divisor on $Y$ and $\Gamma(Y, \sO_Y(-Z))=I$ (as $I$ is integrally closed). 

By \cite[Theorem 3.1]{OWY}, we know that $h^1 (Y, \sO_Y(-Z))\leq p_g(R)=1$. If $h^1 (Y, \sO_Y(-Z)) = 1$, then $I$ is a $p_g$-ideal in the sense of \cite[Definition 3.2]{OWY} and thus by \cite[Corollary 3.6]{OWY}, $I$ is stable. Therefore by Lemma~\ref{lem: LM for two dimensional stable}, $\eh(I) \leq (\type(R)+1)\length(R/I)=2\length(R/I)$. In other words, in this case we have 
$$\frac{\eh(I)}{2\length(R/I)}\leq 1.$$

Now we suppose $h^1 (Y, \sO_Y(-Z)) = 0$. By Theorem~\ref{thm: Kato RR}, we have
\[
\length(R/I) = \frac{-Z^2 - K_YZ}{2} + 1.
\]
Since $X$ is the minimal resolution, we have a map $f \colon Y \to X$. Let $W \coloneqq f_* Z$. Since $E_0$ is the only irreducible exceptional divisor on $X$, we have that $W = nE_0$ for some $n>0$.
By \cite[Theorem 3.4]{Laufer} or \cite[Theorem 4.21]{Reid}, $-E_0=-Z_f$ is numerically equivalent to the canonical divisor $K_X$. 

We first consider the case $e\geq 2$. By \cite[Theorem 3.13]{Laufer} or \cite[Theorem 4.23]{Reid} we know that $\m\sO_X=\sO_X(-E_0)$. Thus by Theorem~\ref{thm: Kato RR}, we have
\[
1 = \length(R/\m) = \frac{-E_0^2 - K_XE_0}{2} + p_g(R) - h^1 (X, \sO_X(-E_0)) = 1 - h^1 (X, \sO_X(-E_0)).
\]
Therefore $h^1 (X, \sO_X(-E_0)) = 0$, and thus by \cite[Proposition~2.5]{OWY}, $h^1 (X, \sO_X(-W))=h^1 (X, \sO_X(-nE_0))=0$. Let $I_{W}\coloneqq \Gamma(X, \sO_X(-W))=\Gamma(X, \sO_X(-nE_0))=\overline{\m^n}$. By Theorem~\ref{thm: Kato RR}, we have 
\[
\length(R/I_W) = \frac{-W^2 - K_XW}{2} + 1.
\]
We next claim that 
\begin{equation}
\label{eqn: comparison to cycles on minimal resolution}
\frac{\eh(I)}{2\length(R/I)}\leq \frac{\eh(I_W)}{2\length(R/I_W)}.   
\end{equation}
Note that since $W=f_*Z$, we have $I\sO_X=\mathfrak{I}\cdot \sO_X(-W)$ for some ideal sheaf $\mathfrak{I}\subseteq \sO_X$ and 
$$\sO_Y(-Z)=I\sO_Y=\mathfrak{I}\cdot \sO_Y(-f^*W).$$
It follows that $Z=f^*W+V$ where $V\geq 0$ is $f$-anti-nef (since $\sO_Y(-V)=\mathfrak{I}\sO_Y$ is $f$-relatively globally generated). Write $K_{Y/X}=K_Y-f^*K_X$, then by Theorem~\ref{thm: Kato RR} 
\[
\frac{\eh(I)}{2\length(R/I)} = \frac{Z^2}{Z^2+K_YZ-2} = \frac{W^2 + V^2}{W^2 + V^2 + K_{Y/X}V + K_XW -2}
\]
Hence, to prove (\ref{eqn: comparison to cycles on minimal resolution}), it is enough to show that 
\[
\frac{V^2}{V^2 + K_{Y/X}V} \leq \frac{W^2}{W^2 + K_XW - 2} =\frac{\eh(I_W)}{2\length(R/I_W)}.
\]
This amounts to proving that 
$$(V^2)(K_X W - 2) \leq (W^2)(K_{Y/X}V).$$ 
To see this, note that $K_{Y/X}$ is $f$-exceptional and effective since $X$, $Y$ are non-singular. Since $V$ is $f$-anti-nef, it follows that $K_{Y/X}V\leq 0$ and thus the right-hand side is non-negative. On the other hand, we have $K_X W = n (-E_0^2) =ne \geq 2$ and thus the left-hand side is non-positive. This completes the proof of (\ref{eqn: comparison to cycles on minimal resolution}). It remains to optimize 
\[\frac{\eh(I_W)}{2\length(R/I_W)} = \frac{W^2}{W^2 + K_XW - 2} =\frac{(nE_0)^2}{(nE_0)^2 + n (-E_0^2) - 2} =\frac{n^2e}{n^2e - ne + 2}\]
for various $n>0$. Equivalently, we need to minimize $2/(n^2e) - 1/n$ for $n\geq 1$. 
By analyzing the derivative, 
we see that the function on positive real numbers $x \mapsto 2/(x^2e) - 1/x$  has its only local minimum at $x = 4/e$. It follows easily that we have 
\[
\begin{cases}
\text { for $e \geq 3$, the minimum is computed at $n=1$, so $\lm(R)=e/2$ is computed by $\m$;}\\
\text { for $e=2$, the minimum is computed at $n=2$, so $\lm(R)=4/3$ is computed by $\overline{\mf m^2}$.}
\end{cases}
\]
This completes the proof when $e\geq 2$.

Now we suppose $e=1$. By \cite[Lemma 4.23]{Reid}, the line bundle $\mathcal{L} = \sO_{X} (-nE_0)|_{E_0}$ is base point free when $n \geq 2$ and thus  
$\sO_{X} (-nE_0)$ is globally generated on $X$. This implies that $I_{nE_0}\coloneqq \Gamma(X, \sO_X(-nE_0))$ is represented on $X$ when $n\geq 2$
(see \cite[2.3, 2.6]{OWY}). Furthermore, by \cite[4.25, Corollary and Remark]{Reid}, we know $R$ is a hypersurface and we can choose a generating set $x,y,z$ of $\m$ such that $I_{2E_0} = (x^2, y, z)$ and $I_{3E_0}=(x^3, xy, y^2, z)$. In particular, we have $\length(R/I_{2E_0})=2$ and $\length(R/I_{3E_0})=4$. By Theorem~\ref{thm: Kato RR}, we have 
$$2 = \length(R/I_{2E_0}) = \frac{-(2E_0)^2 - 2K_XE_0}{2} + p_g(R) - h^1 (X, \sO_X(-2E_0)) = 1 +1 - h^1 (X, \sO_X(-2E_0));$$
$$4 = \length(R/I_{3E_0}) = \frac{-(3E_0)^2 - 3K_XE_0}{2} + p_g(R) - h^1 (X, \sO_X(-3E_0)) = 3 +1 - h^1 (X, \sO_X(-3E_0)).$$
Hence $h^1 (X, \sO_X(-2E_0)) =h^1 (X, \sO_X(-3E_0)) = 0$ and thus by \cite[Proposition 2.5]{OWY} (applied to the globally generated line bundle $\sO_X(-2E_0)$), we have $h^1 (X, \sO_X(-nE_0)) = 0$ for all $n \geq 2$.
 
Now if $W = f_*Z  =n E_0$ for some $n\geq 2$. Then we claim that 
$$\frac{\eh(I)}{2\length(R/I)}\leq \frac{\eh(I_W)}{2\length(R/I_W)}.$$
To see this, note that by the same notation as in the $e\geq 2$ case, it reduces to showing that $V^2 (K_X W - 2) \leq (W^2)(K_{Y/X}V)$. But as $n\geq 2$, $K_X W \geq -2 E_0^2 = 2$, the left-hand side above is non-positive while the right-hand side is non-negative (since $V$ is $f$-anti-nef and $K_{Y/X}$ is $f$-exceptional and effective). Thus we have 
\[\frac{\eh(I)}{2\length(R/I)}\leq \frac{\eh(I_W)}{2\length(R/I_W)} = \frac{W^2}{W^2 + K_XW - 2} =\frac{(nE_0)^2}{(nE_0)^2 + n (-E_0^2) - 2} =\frac{n^2}{n^2 - n + 2} \leq 8/7,\]
and it is easy to check that the maximum is achieved when $n=4$.

Finally, if $W=f_*Z=E_0$, then there exists $g\in I \setminus I_{2E_0}$. By Lech's inequality (Theorem~\ref{thm Lech}),
\[\frac{\eh(I)}{2 \length(R/I)} \leq \frac{\eh(I, R/(g))}{2 \length(R/I)} \leq {\eh(R/(g))}/2. 
\] 
By \cite[4.25, Theorem, Corollary and Remark]{Reid}, we can choose a generating set $x,y,z$ of $\m$ such that if we assign weights $1, 2, 3$ to $x,y,z$ respectively, the associated graded ring $G$ of the Noetherian graded family $\{I_{nE_0}\}_{n}$ coincides with the associated graded ring of this weight filtration. Moreover, $G\cong k[x,y,z]/(f)$ where $f=z^2+y^3+ h$ and $h$ is homogeneous of degree $6$ whose terms involve $x$. Since $I_{2E_0} = (x^2, y, z)$ and $g\notin I_{2E_0}$, the image of $g$ in $G$ agrees with $x$ up to units. It follows that $\eh(R/(g))= \eh(G) =2$ and thus $\eh(I)/2\length(R/I)\leq 1$ in this case. This completes the proof when $e=1$.
\end{proof}

It is natural to ask whether the Lech--Mumford constant of a normal surface singularity $R$ can always be computed from the minimal resolution. If this is the case, then the computation of $\lm (R)$ becomes a finite-dimensional optimization problem on the lattice of cycles on the minimal resolution.  Kei-ichi Watanabe pointed to us the following lemma that gives a partial affirmative answer to this question.

\begin{lemma}[Watanabe]
\label{lem: Watanabe}
Let $(R, \mf m)$ be a two-dimensional excellent normal local ring containing an algebraically closed field $k\cong R/\m$ and let $X\to \Spec(R)$ be the minimal resolution. Suppose $I\subseteq R$ is an integrally closed $\m$-primary ideal and $Y\to \Spec(R)$ is a resolution such that $I\sO_Y=\sO_Y(-Z)$ where $Z$ is an anti-nef effective divisor on $Y$. Let $f\colon Y\to X$ be the canonical map and let $W \coloneqq f_*Z$. Suppose $\sO_X(-W)$ is globally generated. Then we have
$$\frac{\eh(I)}{2\length(R/I)}\leq  \max \left \{1,  \frac{\eh(I_W)}{2\length(R/I_W)} \right \}$$
where $I_W \coloneqq \Gamma(X, \sO_X(-W))$.

In particular, if every anti-nef cycle on $X$ is globally generated, then $\lm(R)$ can be computed from the minimal resolution. 
\end{lemma}
\begin{proof}
First of all, since $W=f_*Z$, we have $I\sO_X=\mathfrak{I}\cdot \sO_X(-W)$ for some ideal sheaf $\mathfrak{I}\subseteq \sO_X$ so that $\sO_X/\mathfrak{I}$ is supported at points, and we have
$$\sO_Y(-Z)=I\sO_Y=\mathfrak{I}\cdot \sO_Y(-f^*W).$$
Thus $Z=f^*W+V$ where $V\geq 0$ is $f$-exceptional and $\sO_Y(-V)=\mathfrak{I}\sO_Y$ is $f$-relatively globally generated. It follows that 
\begin{equation}
\label{eqn: Lech inequality in Watanabe proof}
   -V^2= \eh(\mathfrak{I}, \sO_X) \leq 2h^0(V, \sO_V) 
\end{equation}
where the inequality follows from Lech's inequality (Theorem~\ref{thm Lech}) in the regular case.\footnote{Here the notation $\eh(\mathfrak{J},\sO_X)$ should be interpreted as $\sum_x\eh(\mathfrak{J}_x, \sO_{X,x})$ where $x$ runs over all (finitely many) points on $X$ that are the support of $f_*\sO_V$, and we are using Lech's inequality for each $\eh(\mathfrak{J}_x, \sO_{X,x})$ and noting that each $\sO_{X,x}$ is a two-dimensional regular local ring.}

Since $f_*\sO_V$ is supported at points, from the short exact sequence
    \[
    0 \to \sO_Y (-Z) \to \sO_Y(-f^*W) \to \sO_V(-f^*W)\cong \sO_V \to 0,
    \] 
we obtain that $R^1f_*\sO_Y(-f^*W)\twoheadrightarrow R^1f_*\sO_V\cong H^1(V,\sO_V)$. But we have 
$$R^1f_*\sO_Y(-f^*W)= (R^1f_*\sO_Y) \otimes \sO_X(-W) =0$$ 
since $X$, $Y$ are regular. Thus $H^1(V,\sO_V)=0$. Now after taking global sections, the short exact sequence above induces an exact sequence
    \[
    0 \to I \to I_W \to H^0(V, \sO_V) \to H^1(Y, \sO_Y (-Z)) \xrightarrow{\alpha} H^1(Y, \sO_Y(-f^*W))\to 0.
    \]
In particular, we have that $h^1(Y, \sO_Y(-Z))\geq h^1(Y, \sO_Y(-f^*W))$. On the other hand, since $\sO_Y(-V)$ is $f$-relatively globally generated, each $f$-exceptional prime divisor is not in the base locus of $\sO_Y(-V)$, and since $V$ is $f$-exceptional, (the strict transform of) each exceptional prime divisor on $X$ also cannot be in the base locus of $\sO_Y(-V)$. Thus $\sO_Y(-V)$ has no fixed component and thus $h^1(Y, \sO_Y(-Z))\leq h^1(Y, \sO_Y(-f^*W))$ by \cite[Proposition 2.5]{OWY}. It follows that the surjection $\alpha$ must be an isomorphism and we have 
\[
    0 \to I \to I_W \to H^0(V, \sO_V) \to 0.
    \]
Thus we have
$$\length(R/I)= \length(R/I_W) + h^0(V,\sO_V).$$
Since $\eh(I)=-W^2-V^2$ and $\eh(I_W)=-W^2$, we have 
$$\frac{\eh(I)}{2\length(R/I)}= \frac{\eh(I_W)-V^2}{2\length(R/I_W)+2h^0(V,\sO_V)}\leq \frac{\eh(I_W)-V^2}{2\length(R/I_W)-V^2}\leq  \max \left \{1,  \frac{\eh(I_W)}{2\length(R/I_W)} \right \}$$
where the first inequality follows from (\ref{eqn: Lech inequality in Watanabe proof}).
\end{proof}

One can use Lemma~\ref{lem: Watanabe} to give an alternative proof of Theorem~\ref{thm: minimal elliptic irred exc} when $e\geq 2$. The point is that, for minimally elliptic singularity of degree $e\geq 2$ whose exceptional divisor $E_0$ on the minimal resolution is irreducible, we know that $\sO_X(-nE_0)$ is globally generated for all $n\geq 0$ and thus Lemma~\ref{lem: Watanabe} tells us immediately that $\lm(R)$ can be computed from $\eh(I_{W})/2\length(R/I_W)$ where $W=nE_0$ for some $n$, and we come down to optimize $n^2e/(n^2e-ne+2)$ when $n$ varies, as in the proof of Theorem~\ref{thm: minimal elliptic irred exc}. In fact, we can use Lemma~\ref{lem: Watanabe} to compute $\lm(R)$ for all two-dimensional standard graded normal domains.

\begin{corollary}
\label{cor: Watanabe}
Let $(R,\m, k)$ be (the localization of) a two-dimensional standard graded normal domain over an algebraically closed field $k$. If $R$ is not regular, then $\lm(R)=\eh(R)/2$ and $\lm(R)$ is computed by $\m$.
\end{corollary}
\begin{proof}
The hypothesis implies that $X \coloneqq \Bl_\m(R)\to \Spec(R)$ is the minimal resolution, and moreover, the exceptional divisor $E$ on $X$ is irreducible (we have $E\cong \Proj(R)$). It is also easy to see that $\sO_X(-nE)=\m^n\sO_X$ is globally generated  and that $\m^n=\overline{\m^n}$ for all $n\geq 0$. By Lemma~\ref{lem: Watanabe}, $\lm(R)$ is computed from cycles on $X$, i.e., we only need to find the supremum of $\eh(\m^n)/2\length(R/\m^n)$ when $n$ varies. 

By taking $\m$, we see that $\lm(R) \geq \eh(R)/2$. Thus
it suffices to show that $\length (\m^n/\m^{n+1}) \geq 2n + 1$ for all $n \geq 0$, since this would imply that for all $n \geq 1$
\[
\frac{\length (R/\m^n)}{\eh(\m^n)} \geq \frac{1}{n^2\eh(R)} \left( \sum_{k = 0}^{n-1} (2k + 1) \right ) = \frac{1}{\eh(R)}
\]
and thus $\eh(\m^n)/2\length(R/\m^n) \leq \eh(R)/2$ as wanted. 
By contradiction, suppose $\length (\m^{n_0}/\m^{n_0+1}) \leq 2n_0$ for some $n_0$. Note that $\length (\m/\m^2) \geq 3$ because $R$ is singular, so $n_0 \geq 2$. Now by Macaulay's theorem  on Hilbert function (see \cite[Theorem 4.2.10]{BrunsHerzogBook}, which originates from \cite{Macaulay1927}), for any $d\in\mathbb{N}$, if we write
\[
\length (\m^{d}/\m^{d+1}) = \dim_k[R]_{d}= \binom{k_{d}}{d} + \binom{k_{d - 1}}{d-1} + \cdots + \binom{k_j}{j},
\]
where $d \geq j \geq 1$ and 
$k_{d} > k_{d - 1} > \cdots > k_j \geq j$,
then
\[
\length (\m^{d + 1}/\m^{d+2})=\dim_k[R]_{d+1} \leq 
\binom{k_{d} + 1}{d + 1} + \binom{k_{d - 1} + 1}{d} + \cdots + \binom{k_j + 1}{j + 1}.
\]
Now if $n_0=2$, then we can write 
$$2n_0=4=\binom{n_0+1}{n_0}+\binom{n_0-1}{n_0-1}$$ and it follows by induction that $$\length (\m^{n}/\m^{n + 1}) \leq \binom{n+1}{n}+\binom{n-1}{n-1}=n+2$$ and thus $\eh(R)=1$ contradicting that $R$ is singular. If $n_0\geq 3$, then we have
\[
2n_0 = \binom{n_0+1}{n_0} + \binom{n_0-1}{n_0-2} 
\]
and again it follows by induction that 
$$\length (\m^{n}/\m^{n + 1}) \leq\binom{n+1}{n} + \binom{n-1}{n-2}=2n.$$
Therefore, $\eh(R) = 2$ and thus $R$ is a hypersurface (of degree $2$) by Abhyankar's inequality. But then a direct computation shows that the Hilbert function of $R$ must be 
$\length (\m^n/\m^{n+1}) = 2n + 1$, a contradiction. 
\end{proof}

\begin{corollary}
Let $(R,\m,k)$ be a two-dimensional local ring so that $\Gr_\m(R)$ is geometrically normal. Then $\lm(R)=\eh(R)/2$ and $\lm(R)$ is computed by $\m$. 
\end{corollary}
\begin{proof}
First of all, we know $\lm(R)\geq \eh(R)/2$ by considering the maximal ideal. By Proposition~\ref{prop: LM associated graded} and Proposition~\ref{prop: LM multigraded}, we know that $\lm(R)\leq \lm(\Gr_\m(R)) \leq \lm(\Gr_\m(R) \otimes_k\overline{k})$, the latter is a two-dimensional standard graded normal domain by hypothesis and thus $\lm(\Gr_\m(R) \otimes_k\overline{k})= \eh(R)/2$ by Corollary~\ref{cor: Watanabe} (note that multiplicity does not change upon passing to the associated graded ring and base change the field).
\end{proof}

We next use a different argument to handle all minimally elliptic singularities of degree one: the key point is that we can compute the Lech--Mumford constant for the ``non-normal limit" for this family. We treat this case algebraically.

\begin{lemma}
\label{lem: LM cusp sing}
Let $(R,\m) = K[[x, y,z]]/(z^2 + y^3)$ where $K$ is a field. Then $\lm(R) = 8/7$, and $8/7$ is achieved only by ideals $\overline{(y^2, x^{4n})}$ where $n \geq 1$ is an integer.
\end{lemma}
\begin{proof}
First of all, it is straightforward to see that $\eh(\overline{(y^2, x^{4n})})=16n$ and $\length(R/\overline{(y^2, x^{4n})})=7n$ (for example, using integral dependence as in Lemma~\ref{lem: Mumford useful lemma}) and thus we achieved $8/7$. Next we note that $R \cong K[[x, t^2, t^3]]$. By Proposition~\ref{prop: LM multigraded} $\lm(R)$ can be computed from monomial ideals $I \subset R$, thus it suffices to show that 
\[\frac{\eh(I)}{2\length(R/I)} \leq 8/7\]
for all integrally closed $\m$-primary monomial ideals $I \subset R$. 

Let $S \coloneqq K[[x, t]]$ be the integral closure of $R$. We will transfer the computation to $S$. Set $J \coloneqq \overline{IS}$, then $J$ is an integrally closed $\m$-primary monomial ideal of $S$ and $I=J\cap R$ by \cite[Proposition 1.6.1]{SwansonHuneke}. Since the difference between $R$ and $S$ is the vector space spanned by monomials $\{tx^n\}_{n \geq 0}$, we have $\eh(I) = \eh(J)$ and 
\[\length(S/J)=\length(R/I)+\length(S/(J:t)+(t)).\]
Our goal is to show that
\begin{equation}
\label{eqn: 8/7 goal}
\eh(J)\leq \frac{16}{7} \big(\length(S/J)- \length(S/(J:t)+(t))\big).
\end{equation}

At this point, we analyze both sides of (\ref{eqn: 8/7 goal}) using data of the Newton polytope associated to $J$. More specifically, recall that $\eh(J)$ is twice the volume of the complement of the Newton polytope, $\length(S/J)$ is the number of integer points in the complement of the Newton polytope, and $\length(S/(J:t)+(t))$ is the number of integer points in the complement of the Newton polytope whose $t$-coordinate is $1$. 

Suppose $\{(a_i, b_i)\}_{i=0}^{i=k}$ are all integer points on the boundary of the Newton polytope associated to $J$ so that $a_0=0$, $b_0=\length(S/J+(t))$ and $a_k=\length(S/J+(x))$, $b_k=0$. By dividing the complement of the Newton polytope by the lines $t=a_i$ we may write it as a union of $k$ trapezoids $T_i$, $0\leq i\leq k-1$, where $T_i$ is defined by the end points $(a_i, 0), (a_i, b_i), (a_{i + 1}, b_{i + 1}), (a_{i+1}, 0)$. 

\vspace{2mm}
  \begin{tikzpicture}
    \centering
    \begin{axis}[
    ylabel={x},
    xmin=0, xmax=27,
    ymin=0, ymax=6,
    grid=both,
    xticklabels={\,,$a_1$,\,,\,,\, ,$a_2$,\,,\,,$a_3$,\,,\,,\,,$t$},
    xtick =     {2 , 4   , 6, 8, 10, 12  ,14,16, 18, 20, 22, 24, 26},
    yticklabels={\, ,$b_3$,$b_2$,\, , \, ,$b_1$, \, ,\,, \, ,\,, $b_0$},
    ytick =     {0.5, 1   , 1.5 , 2 , 2.5, 3   , 3.5, 4, 4.5, 5, 5.5},
    ]
    \addplot[
    color=blue,
    mark=ball,
    thick, 
    ]
    coordinates {
    (0,5.5)(4,3)(12, 1.5)(18, 1)(45, 0)
    };
    \draw[color = black, very thick] (0,0) -- (0,5.5);
    \draw[color = black, very thick] (4,0) -- (4,3);
    \draw[color = black, very thick] (12,0) -- (12,1.5);
    \draw[color = black, very thick] (18,0) -- (18,1);
    \draw[color = black, very thick, mark = ball] (3, 6);

    \node[label={90:{$T_0$}}, inner sep=2pt] at (axis cs: 1.5,2) {};
    \node[label={90:{$T_1$}}, inner sep=2pt] at (axis cs: 7.8,0.9) {};
    \node[label={90:{$T_2$}}, inner sep=2pt] at (axis cs: 15,0.3) {}; 
    \node[label={90:{$T_3$}}, inner sep=2pt] at (axis cs: 21,0.0) {}; 
    
    \end{axis}
    \end{tikzpicture}

Since $\eh(J)$ is twice the volume of the complement of the Newton polytope, we have
\[
\eh(J) = 2 \sum_{i = 0}^{k-1} \vol (T_i) = \sum_{i = 0}^{k - 1} (b_{i + 1} + b_i)(a_{i+1} - a_i).
\]
Since $\length(S/J)$ is the total number of integer points inside the union of all the $T_i$'s, and by our assumption, $k$ is the number of integer points on the boundary of the Newton polytope, by Pick's theorem we have
$$\frac{\eh(J)}{2}= \big(\length(S/J)-b_0-a_k+1\big) + \frac{b_0+a_k+k}{2} -1.$$
It follows that 
$$\length(S/J) = \frac 12 \big(\eh(J) + b_0+ a_k -k\big).$$
Moreover, it is easy to see that the number of integer points on the intersection of the line $t=1$ and the complement of the Newton polytope is 
$$\length(S/(J:t)+(t)) = \lfloor b_0 - \frac{b_0 - b_1}{a_1 - a_0} \rfloor.$$
Thus we may rewrite (\ref{eqn: 8/7 goal}) as
\begin{equation}\label{eqn: unrefined goal}
    7 \eh (J) \leq 8\big(\eh(J) + b_0 + a_k - k\big) - 16 b_0 + 16 \lceil \frac{b_0 - b_1}{a_1 - a_0} \rceil,
\end{equation}
and since $a_k\geq k$, it is enough to show that
\begin{equation}
\label{eqn: refined goal in 8/7}
8b_0 \leq \eh(J) +  16 \cdot \frac{b_0 - b_1}{a_1 - a_0} = \sum_{i=0}^{k -1} (b_{i + 1} + b_i)(a_{i + 1} - a_i)
+  16  \cdot \frac{b_0 - b_1}{a_1 - a_0} .
\end{equation}

Next we note that, since $J=\overline{IS}$ with $I\subseteq R$, none of the vertices defining the Newton polytope of $J$ has the $t$-coordinate $1$. Thus if $a_1=1$, then $(a_2, b_2)$ must be on the same edge, so that 
\[\frac{b_0 - b_1}{a_1 - a_0} = \frac{b_0 - b_2}{a_2 - a_0}.\]
Therefore, we can omit the point $(a_1, b_1)$ when $a_1=1$\footnote{We caution the readers that, omitting the point $(a_1, b_1)$ will decrease $k$, however, this will not affect (\ref{eqn: refined goal in 8/7}) by the fact that $(b_0 - b_1)(a_2-a_0)=(b_0-b_2)(a_1-a_0)$.} to assume that $a_1\geq 2$. Now we set $x_i = a_{i + 1} - a_i$, note that each $x_i$ is a positive integer and $x_0\geq 2$ (as $a_1\geq 2$). Since the Newton polytope is convex, for all $i$ we have 
\begin{equation}
\label{eqn: convexity}
\frac{b_i - b_{i+1}}{x_i} = \frac{b_i - b_{i+1}}{a_{i + 1} - a_i} \geq \frac{b_{i + 1} - b_{i+2}}{a_{i+2}- a_{i+1}} =
\frac{b_{i + 1} - b_{i+2}}{x_{i+1}}.
\end{equation}
Thus our goal becomes proving
\[
8b_0 \leq b_0x_0 + \sum_{i = 1}^{k-1} b_i (x_{i-1}+x_i)
+   16 \cdot\frac{b_0 - b_1}{x_0}
\]
under the condition (\ref{eqn: convexity}) and that $x_0\geq 2$. We rewrite the above as
\begin{equation*}
0 \leq 
b_0 \left(x_0 + \frac{16}{x_0}-8\right) +  b_1\left(x_0 + x_1 - \frac{16}{x_0} \right) + \sum_{i = 2}^{k-1} b_i (x_{i-1}+x_i).
\end{equation*}
Using the quadratic formula, we further rewrite the above as
\begin{equation}
\label{eqn: refined linear optimization in 8/7}
0 \leq \frac{b_0}{x_0} (x_0 - 4)^2 + 
 b_1\left(x_0 + x_1 - \frac{16}{x_0} \right) +
\sum_{i = 2}^{k-1} b_i (x_{i-1}+x_i)
\end{equation}

Now if $x_0\geq 4$, then the coefficient $x_1 + x_0- 16/x_0$ at $b_1$ is non-negative and thus (\ref{eqn: refined linear optimization in 8/7}) follows. Observe that none of these ideals will attain $8/7$. Namely, for the equality to hold in  (\ref{eqn: refined linear optimization in 8/7}) we must take $x_0 = 4$
and $b_1 = 0$, otherwise $x_1 > 0$. 
But then $\eh(J) = 4b_0$ and (\ref{eqn: refined goal in 8/7})
has strict inequality.

We still have two cases remaining: $x_0 = 2$ or $x_0=3$. Note that by (\ref{eqn: convexity}), 
\[
\frac{b_0}{x_0} \geq \frac{b_1}{x_0} + \frac{b_1}{x_1} - \frac{b_2}{x_1}.
\]
By plugging this into (\ref{eqn: refined linear optimization in 8/7}) and rearranging the terms, it is enough to show that
\begin{align}
\label{eqn: refined further linear optimization in 8/7}
0 \leq & \,\ b_1 \left(x_0 + x_1 - \frac{16}{x_0} + \frac{(x_0 - 4)^2}{x_0} + \frac{(x_0 - 4)^2}{x_1} \right) \\
& + b_2 \left(x_1 + x_2 - \frac{(x_0 - 4)^2}{x_1}\right) 
+ 
\sum_{i = 3}^{k-1} b_i (x_{i-1}+x_i). \notag
\end{align}

Suppose that $x_0 = 3$. Then the coefficient at $b_1$ is 
$x_1 + 1/x_1 -2 \geq 0$ and at $b_2$ is $x_1 + x_2 - 1/x_1 \geq 0$. 
Hence (\ref{eqn: refined further linear optimization in 8/7}) holds 
and becomes equality if and only if $x_1 = 1, x_2 = 0$.
In this case, (\ref{eqn: convexity}) shows that $b_0 > 4b_1$. Note that the equality is not possible, otherwise the integer point $(2b_1, 2)$ 
is on the line through $(b_0, 0)$ and $(b_1, 3)$
contradicting the definition of $b_1$.
Thus (\ref{eqn: refined goal in 8/7}) is a strict inequality:
\[
\eh(J) + 16 \frac{b_0- b_1}{a_1 - a_0}
= 4b_0 + 3b_1 + 16 \frac{b_0 - b_1}{3} 
> 8b_0.
\]

Last assume that $x_0 = 2$. 
The first possibility for the equality to hold in (\ref{eqn: refined further linear optimization in 8/7}) is to have $b_1 = 0$.
In this case, $\eh(J) = 2b_0$ and (\ref{eqn: refined goal in 8/7}) is
always strict. 
Now, $b_1 > 0$ and 
the coefficient of $b_1$ in (\ref{eqn: refined further linear optimization in 8/7}) is $x_1 + 4/x_1 -4 \geq 0$.
If $b_2 = 0$, then the equality in (\ref{eqn: refined further linear optimization in 8/7}) may only happen when $x_1 = 2$. 
We have $\eh (J) = 2 b_0 + 4b_1$
and $b_0 \geq 2b_1$ by (\ref{eqn: convexity}). 
First, suppose that $b_0 - b_1$ is odd. 
Then 
\[
8(\eh(J) + b_0 + a_k-k) - 16b_0 + 
8 \left \lceil 
\frac{b_0 - b_1}{2}
\right \rceil
= 16b_0 + 24b_1 + 8 + 8(a_k - k)
\geq 14b_0 + 28b_1 + 8, 
\]
and we have a strict inequality in (\ref{eqn: unrefined goal}).
Second, if $b_0 - b_1$ is even, then, in fact, $b_0 > 2b_1$,  
otherwise $(b_0, 0), (b_1, 2), (0, 4)$ are on the same line and 
we would have an uncounted integer point $(b_1/2, 3)$. 
Hence, in this case (\ref{eqn: refined goal in 8/7}) is a strict inequality.

Thus, we arrived to the case where $x_0 = 2$ and $b_2 > 0$.
Now, the coefficient of $b_2$ in (\ref{eqn: refined further linear optimization in 8/7}) is 
$x_1 + x_2 - 4/x_1$.
Since $x_1$ is a positive integer, this coefficient is always non-negative unless $x_0=2$ and $x_1=1$. 
If the coefficients are non-negative, we must make them $0$, so 
$x_1 = 2$ and $x_2 = 0$, but this would contradict $b_2 > 0$.

Hence, we now proceed to analyze the last case where $x_0 = 2$ and $x_1 = 1$. Summing up, the coefficient at 
$b_1$ in (\ref{eqn: refined further linear optimization in 8/7}) 
is $1$ and the coefficient of $b_2$ is 
$x_2-3$. By (\ref{eqn: convexity}) again, we have 
$$b_1 =\frac{b_1}{x_1} \geq \frac{b_2}{x_1} + \frac{b_2}{x_2}  -\frac{b_3}{x_2} =b_2 + \frac{b_2}{x_2}  -\frac{b_3}{x_2} .$$ 
Plugging these into (\ref{eqn: refined further linear optimization in 8/7}), it is enough to show that 
\[
b_2+ \frac{b_2}{x_2} -\frac{b_3}{x_2} + b_2(x_2-3) +  \sum_{i = 3}^{k-1} b_i (x_{i-1} + x_i)
= b_2(x_2 + \frac{1}{x_2} -2) + b_3(x_2 -\frac{1}{x_2} + x_3) + \sum_{i = 4}^{k-1} b_i (x_{i-1} + x_i)
\]
is nonnegative. This is always true since the coefficients of all $b_i$'s above are non-negative. 
It remains to note that the equality in (\ref{eqn: refined further linear optimization in 8/7}) 
now forces that $x_2 = 1$ and $b_3 = 0$. 
We compute $\eh(J) = 2b_0 + 3b_1 + 2b_2$ and observe that 
if any of the convexity inequalities (\ref{eqn: convexity}) is strict, then
\[
\eh(J) + 8(b_0 - b_1) = 
8b_0 + (2b_0 - 5b_1+ 2b_2) > 8b_0,
\]
showing that (\ref{eqn: refined goal in 8/7}) is a strict inequality.
Hence, we may assume that all those points lay on the same line, 
giving that $b_0 = 4b_2$ and $b_1 = 2b_2$. 
\end{proof}

\begin{theorem}
\label{thm: minimal elliptic deg one}
Let $(R, \mf m)$ be a minimally elliptic singularity of degree $1$. Then we have $\lm(R) = 8/7$.
\end{theorem}
\begin{proof}
Let $X\to \Spec(R)$ be the minimal resolution with the fundamental cycle  $Z_f$, and let $J_n\coloneqq \Gamma(X, \sO_{X}(-nZ_f))$. By \cite[4.25, Theorem, Corollary, and Remark]{Reid}, $\{J_n\}_n$ is a graded family and we have $\Gr_{J_\bullet}(R)\cong k[x,y,z]/(z^2+y^3+xf(x,y,z))$ where $z^2+y^3+xf(x,y,z)$ is homogeneous if we assign $x$, $y$, $z$ weights $1$, $2$, $3$ respectively. It follows that $\Gr_{(x)}\Gr_{J_\bullet}(R)\cong k[x,y,z]/(z^2+y^3)$. 
Thus applying Proposition~\ref{prop: LM associated graded} twice (and using Proposition~\ref{prop: LM multigraded}), we have
$$\lm(R)\leq \lm(k[[x,y,z]]/(z^2+y^3)) =8/7$$
where the last equality follows from Lemma~\ref{lem: LM cusp sing}. 

It remains to find an $\m$-primary ideal of $R$ such that $8/7$ is achieved. Note that by \cite[Lemma 4.23]{Reid}, we know that $\sO_X(-nZ_f)$ is globally generated for all $n\geq 2$ and that by \cite[Corollary 4.25]{Reid}, $J_2= \Gamma(X, \sO_X(-2Z_f))$ satisfies $\length(R/J_2)=2$. By Theorem~\ref{thm: Kato RR}, we have 
\[
2 + h^1(X, \sO_X(-2Z_f))  = -\frac{4Z_f^2 + 2K_X Z_f}{2} + 1 = -\frac{-4+2}{2} +1=2,
\]
where we used that $K_X$ is numerically equivalent to $-Z_f$ and $-Z_f^2=1$. It follows that $h^1(X, \sO_X(-2Z_f))=0$ and hence $h^1(X, \sO_X(-4Z_f))=0$ by \cite[Proposition 2.5]{OWY}. Finally, we apply Theorem~\ref{thm: Kato RR} to $J_4=\Gamma(X, \sO_X(-4Z_f))$, we obtain that 
\[
\length(R/J_4) = -\frac{16Z_f^2 + 4K_X Z_f}{2} + 1 = -\frac{-16+4}{2} +1=7.
\]
Since $\eh(J_4)=-(4Z_f)^2=16$, we have 
$$\frac{\eh(J_4)}{2\length(R/J_4)}=16/14=8/7.$$
This completes the proof of the theorem.
\end{proof}

\subsection{Asymptotic Lech's inequality} For most surface examples studied in the previous two subsections, the Lech--Mumford constant is attained (i.e., the supremum in the definition is actually a maximum). We expect the following more general conjecture in this direction.

\begin{conjecture}
\label{conj: asymptotic Lech consequence}
Let $(R,\m)$ be a local ring. Suppose $\widehat{R}$ has an isolated singularity, i.e.,  $\widehat{R}_P$ is regular for all $P\in\Spec(\widehat{R})-\{\m\}$. If $\lm(R)>1$, then the supremum in the definition of $\lm(R)$ is attained. In particular, $\lm(R)\in \Q$. 
\end{conjecture}

\begin{remark}
\begin{enumerate}
    \item Without the assumption that $\lm(R)>1$, the first conclusion of Conjecture~\ref{conj: asymptotic Lech consequence} is false: consider $R=\mathbb{C}[[x,y,z,w]]/(x^2+y^2+z^2+w^2)$, then $R$ has an isolated singularity, but as $\lm(R/w)=1$ by Theorem~\ref{thm: Lech-stable CM surface}, Proposition~\ref{prop: Mumford prop 3.5} implies that $\lm(R)=1$ and the supremum in the definition of $\lm(R)$ is not attained. 

    \item Without the isolated singularity assumption, the first conclusion of Conjecture~\ref{conj: asymptotic Lech consequence} is also false. Let $A$ be any Artinian local ring that is not a field and let $R \coloneqq A[[T_1,\dots,T_d]]$. By Theorem~\ref{thm: uniform Lech} or Proposition~\ref{prop: basic results on cLM}(\ref{part multiplicity bounds}), $\lm(R)=\eh(R)>1$ for every $d$. But the supremum in the definition of $\lm(R)$ is not attained by Proposition~\ref{prop: Mumford prop 3.5} for $d\geq 2$.
\end{enumerate}
\end{remark}

On the other hand, we do not know any example that $\lm(R)>1$ and $\widehat{R}$ is reduced, but the supremum in the definition of $\lm(R)$ is not attained. Nor do we have an example where $\lm(R)$ is irrational. However, beyond the isolated singularity case, we have limited evidence. In fact, if we examine the descending chain (\ref{eqn: chain of inequalities}): 
$$\lm(R)\geq \lm(R[[T_1]]) \geq \lm(R[[T_1, T_2]]) \geq \cdots \geq \lm(R[[T_1,\dots, T_n]])\geq \cdots,$$
we observe that as long as this chain is not strictly decreasing, 
by Proposition~\ref{prop: Mumford prop 3.5}
there must exist some $n$ such that the supremum in the definition of
$\lm(R[[T_1,\dots, T_n]])$ is not attained. Thus, one way to produce or to seek examples where $\lm(R[[T_1,\dots, T_n]])>1$ and is not attained is to start with a non-lim-stable singularity $R$ and compute the chain (\ref{eqn: chain of inequalities}) above. However, while there are plenty of non-lim-stable singularities (for example, see Theorem~\ref{thm: dimension one semistable} in Section~\ref{section: Estimates on higher LM} or Theorem~\ref{thm: lim-stable implies log canonical general} in Section~\ref{section: MMP}), computing these higher Lech--Mumford constants seems extremely difficult. We will see in Example~\ref{example: double drop} that even producing an example $(R,\m)$ with $\lm(R)>\lm(R[[T_1]]) >\lm(R[[T_1, T_2]])$ requires nontrivial effort. 

We also do not have an example of a local ring $(R,\m)$ with $\lm(R[[T_1, \ldots, T_n]]) = 1$ for some $n \geq 2$ but $R$ is not semistable. In fact, as we mentioned earlier, we do not even know an example of a lim-stable but not semistable singularity.

Despite the difficulty discussed above, we point out the following result providing strong evidence towards Conjecture~\ref{conj: asymptotic Lech consequence}. 

\begin{proposition}
\label{prop: LM attained iso sing}
Conjecture~\ref{conj: asymptotic Lech consequence} holds if $R$ contains a field of characteristic $p>0$ and $R/\m$ is perfect.
\end{proposition}
\begin{proof}
Since $\lm(R)>1$, we can fix an $\epsilon>0$ such that $\lm(R)-\epsilon >1$. By \cite[Corollary 4.4]{HMQS}, we know that there exists $N>0$ such that 
$$\frac{\eh(I)}{d!\length(R/I)} \leq \lm(R)-\epsilon$$
for all $\m$-primary ideals $I\subseteq R$ such that $\length(R/I)>N$. On the other hand, if $\length(R/I)\leq N$, then by Lech's inequality (Theorem~\ref{thm Lech}), $\eh(I)\leq d!\eh(R)\length(R/I)$. It follows that $\eh(I)/(d!\length(R/I))$ can only take values in finitely many rational numbers when $\length(R/I)\leq N$ and thus $\lm(R)$ is attained. 
\end{proof}

Proposition~\ref{prop: LM attained iso sing} is essentially a consequence of \cite[Corollary 4.4]{HMQS}. For non-isolated singularities, one cannot expect the conclusion of \cite[Corollary 4.4]{HMQS} (or the statement of \cite[Conjecture 1.2 (a)]{HMQS}) to hold, as already observed in \cite[Example 4.8]{HMQS}. In fact, Lemma~\ref{lem: LM cusp sing} also exhibits an example that the Lech--Mumford constant is achieved by sufficiently deep ideals. 

Note that, when $R$ is singular, $R[[T]]$ never has an isolated singularity. In this case we do not know whether $\lm(R[[T]])$ can always be approximated (or even achieved) by sufficiently deep ideals. 

\begin{question}
\label{question: asymptotic Lech R[[T]]}
Let $(R,\m)$ be a local ring of dimension $d$. Is it always true that 
    \[
    \lm(R[[T]]) = \lim_{N \to \infty} \sup_{\substack{\sqrt{I}=(\m, T)\\
    I\subseteq (\m, T)^N}} \left\{\frac{\eh(I)}{(d+1)!\length (R[[T]]/I)} \right\}?
    \]
\end{question}

\begin{remark} 
\begin{enumerate}
    \item Question~\ref{question: asymptotic Lech R[[T]]} has an affirmative answer whenever $\lm(R[[T]])$ is computed by $(J, T)$ for some $\m$-primary ideal $J\subseteq R$. By \cite[Lemma 5.10]{HMQS}, it is enough to show that 
    $$\lim_{N \to \infty} \sup_{\substack{\sqrt{I}=(\m, T)\\
    \length(R[[T]]/I)>N}} \left\{\frac{\eh(I)}{(d+1)!\length (R[[T]]/I)} \right\} = \frac{\eh(J)}{(d+1)!\length(R/J)}.
    $$
But this is straightforward by considering $I= (J, T^{N+1})$ for each $N$.
    \item On the other hand, it rarely happens that $\lm(R[[T]])$ can be computed by $(J,T)$ (this fails in simple examples, see Lemma~\ref{lem: LM cusp sing}). In fact, if $\lm(R[[T]])$ is computed by $(J, T)$, then set $\rho$ to be the Rees period of $J$ and consider the ideal $\overline{(J,T^\rho)}$. By Corollary~\ref{cor: integral closures of powers of I+T^m}, we find that 
    $$\phantom{mmm}\frac{\rho \eh(J)}{(d+1)!\sum_{j=1}^\rho\length(R/J^{\frac{j}{\rho}})}=\frac{\eh(\overline{(J,T^\rho)})}{(d+1)!\length(R[[T]]/\overline{(J,T^\rho)})} \leq \lm(R[[T]]) = \frac{\eh(J)}{(d+1!)\length(R/J)}.$$
    It follows that $\sum_{j=1}^\rho\length(R/J^{\frac{j}{\rho}})
    \geq \rho\cdot \length(R/J)$ and thus we must have $J=J^{\frac{j}{\rho}}$ for each $1\leq j\leq \rho$. But since $J^{\frac{1}{\rho}}=\m$, we see that $J=\m$. Thus $\lm(R[[T]])$ is computed by the maximal ideal $(\m, T)$ and $\lm(R[[T]])=\eh(R)/(d+1)!$. 
\end{enumerate}
\end{remark}

We end this subsection by proving that Conjecture~\ref{conj: asymptotic Lech consequence} holds for Gorenstein normal rings of dimension two, in arbitrary characteristic. This includes most of the examples studied in the previous two subsections.

We will need the following lemma, which is well-known to experts (see \cite[Corollary~4.2]{MSTWWAdjoint} and \cite[Corollary~3.16]{EinIshiiMustata} for more general results in all dimensions). We include a short and characteristic-free proof for completeness.

\begin{lemma}
\label{lem: uniform BS dim 2}
Let $(R,\m)$ be a two-dimensional complete normal local domain. Then there exists $N>0$ such that for any $\m$-primary ideal $I\subseteq R$, we have $\m^N I^2 \subseteq J$ for any reduction $J$ of $I$.
\end{lemma}
\begin{proof}
We may replace $R$ by the completion of $R(t)$ to assume $R$ has an infinite residue field. It is enough to prove the result when $J=(x,y)$ is a minimal reduction of $I$. Now let $\pi \colon Y\to X\coloneqq\Spec(R)$ be a resolution of singularities such that $I\sO_Y=(x,y)\sO_Y=\sO_Y(-E)$. We have an exact Koszul complex $0\to \sO_Y\to \sO_Y(-E)^{\oplus 2}\xrightarrow{\cdot(x,y)} \sO_Y(-2E)\to 0.$ Pushforward to $X$ and note that $\pi_*\sO_Y(-nE)=\overline{I^n}$, we obtain that
$\overline{I}^{\oplus 2}\xrightarrow{\cdot (x,y)}\overline{I^2}\to R^1\pi_*\sO_Y$. Since $R$ is normal and $\dim(R)=2$, $R^1\pi_*\sO_Y$ is independent of the resolution and is annihilated by $\m^N$ for some fixed $N>0$. It follows that $\m^N I^2 \subseteq\m^N\overline{I^2}\subseteq (x,y)=J$ as wanted.
\end{proof}

\begin{proposition}
Let $(R, \mf m)$ be a two-dimensional complete Gorenstein normal domain. Then we have
\[
\lim_{N \to \infty} \sup_{\substack{\sqrt{I}=\m \\ \length(R/I)> N}} \left \{\frac{\eh(I)}{2\length(R/I)} \right \} = 1.
\]
As a consequence, if $\lm(R)>1$, then the supremum in the definition of $\lm(R)$ is attained. In particular, $\lm(R)\in\Q$.
\end{proposition}
\begin{proof}
We replace $R$ by the completion of $R(t)$ to assume that $R$ has an infinite residue field. We fix $K > 0$ given by Lemma~\ref{lem: uniform BS dim 2}. Now let $I$ be an $\m$-primary ideal and let $J=(x,y)$ be a minimal reduction of $I$. In $S \coloneq R/J$  we have
by Lemma~\ref{lem: uniform BS dim 2} that 
\[
\eh(I) = \length(R/J)= \length (S) = \length (S/IS) + \length (IS)
\leq \length (S/IS) + \length (\Ann_S\m^KI).
\]
Since $S$ if Gorenstein, we get a further bound  
\[
\eh(I) \leq \length (R/I) + \length (\Hom_S (S/\m^KIS, S)) = \length (R/I) + \length (S/\m^K IS) 
\leq 2\length (R/I) +  \length (I/\m^K I).
\]
Filtering now gives us the inequality
\[
\length(I/\m^K I) = \sum_{n = 0}^{K-1} \length \left (\frac{\m^{n}I}{\m^{n+1} I} \right )
\leq \sum_{n = 0}^{K-1} \mu(\m^n)\mu(I)
= (1+ \mu(\m)+ \cdots + \mu(\m^{K-1}))\mu(I).
\]
By writing $R$ as a quotient of a complete regular local ring $A$
one obtains from \cite[Corollary~3.5]{HMQS} that $\mu(I)/\length(R/I)$
tends to $0$ when $\length(R/I)\to\infty$.
It follows that
$$
\lim_{N \to \infty} \sup_{\substack{\sqrt{I}=\m \\ \length(R/I)> N}} \left \{ \frac{\length(I/\m^K I)}{\length(R/I)}\right \} = 0.
$$
Therefore we have
$$
\lim_{N \to \infty}\sup_{\substack{\sqrt{I}=\m \\ \length(R/I)> N}} \left \{\frac{\eh(I)}{2\length(R/I)}  \right \} \leq \lim_{N \to \infty} \sup_{\substack{\sqrt{I}=\m \\ \length(R/I)> N}} \left \{\frac{\length (R/I) + \length (R/I)}{2\length(R/I)}  \right \} =1.
$$
The rest of the assertions follow from the same argument as in Proposition~\ref{prop: LM attained iso sing}.
\end{proof}

\newpage
\section{Estimates on higher Lech--Mumford constants}
\label{section: Estimates on higher LM}

\subsection{Lower bound on $\limlm(R)$ via graded families of ideals} 
To obtain estimates for $\limlm(R)$, we use the following elementary construction, which transforms a graded family of ideals of $R$ into a controllable graded family of ideals in the ring of power series. This construction is applicable in arbitrary dimensions, and the results presented in this subsection will be also utilized in Section~\ref{section: MMP}.

\begin{lemma}
\label{lem: introduce power family}
Let $(R, \mf m)$ be a local ring and $\{I_n\}_n$ be a graded family of $\m$-primary ideals where we set $I_0=R$. 
Fix non-negative integers $k, N$ and consider the following $\m+(t_1,\dots,t_r)$- primary ideal in $R[[t_1, \ldots, t_r]]$
\[
J_{N, k} \coloneqq \sum_{n = 0}^{N} I_{N - n} (t_1, \ldots, t_r)^{kn}.
\]
Then we have
\begin{align*}
    \length (R[[t_1,\dots,t_r]]/J_{N, k}) &= \sum_{j = 0}^{N - 1} \length (I_{j}/I_{j+1}) \binom{k(N - j) + r - 1}{r}\\
    &= 
    \sum_{n = 0}^{N-1}  \length(R/I_{N-n}) \left(\binom{k(n+1) + r - 1}{r} - \binom{kn + r - 1}{r} \right).
\end{align*}
\end{lemma}
\begin{proof}
    The second formula easily follows after noting that the colength of a graded ideal (with the monomial grading on $t_i$) can be computed in each component, and that the number of monomials in $r$ variables of degree at most $d$ is $\binom{d+r}{r}$.
The first formula follows by reordering the terms in the following way 
\begin{align*}
\length (R[[t_1,\dots,t_r]]/J_{N, k}) &= \sum_{n = 0}^{N-1}  \length(R/I_{N-n}) \left(\binom{k(n+1) + r - 1}{r} - \binom{kn + r - 1}{r} \right)\\ &= \sum_{n = 0}^{N-1} \left(\sum_{j=0}^{N - n - 1} \length (I_{j}/I_{j+1}) \right) \left(\binom{k(n+1) + r - 1}{r} - \binom{kn + r - 1}{r} \right) \\
&= \sum_{j = 0}^{N-1} \length (I_{j}/I_{j+1}) \sum_{n=0}^{N - j - 1} \left(\binom{k(n+1) + r - 1}{r} - \binom{kn + r - 1}{r} \right)  \\
 &=  \sum_{j = 0}^{N - 1} \length (I_{j}/I_{j+1}) \binom{k(N - j) + r - 1}{r}.  
\end{align*}
\end{proof}

We next observe that the multiplicity of the ideal $J_{N,k}$ constructed in Lemma~\ref{lem: introduce power family} can be controlled under rather mild assumptions on the graded family $\{I_n\}_n$.

\begin{lemma}
\label{lem: formula power multiplicity}
Let $(R, \mf m)$ be a local ring of dimension $d \geq 1$ and let $\{I_n\}_n$ be a graded family of $\m$-primary ideals. If there exists $A\in \mathbb{R}$ such that
$\length(R/I_n) = \frac{A}{d!} n^d + O(n^{d-1})$, 
then for the ideal $J_{N, k} \subseteq R[[t_1, \ldots, t_r]]$
constructed in Lemma~\ref{lem: introduce power family}, 
we have 
$$\eh (J_{N,k}) \geq A k^r N^{r + d}. $$
\end{lemma}
\begin{proof}
Since $\{I_n\}_n$ is a graded family of ideals, it is easy to check that for each fixed $k$,
$\{J_{N, k}\}_N$ form a graded family of $\m+(t_1,\dots,t_r)$-primary ideals in $R[[t_1,\dots,t_r]]$. In particular, we have $J_{N, k}^m \subseteq J_{Nm, k}$ for every $m$. It follows  
from Lemma~\ref{lem: introduce power family}
that
\begin{align*}
\length(R[[t_1,\dots,t_r]]/J_{N,k}^m) 
& \geq \sum_{n = 0}^{Nm-1} \length(R/I_{Nm-n}) \left( \binom{k(n+1)+r-1}{r}-\binom{kn+r-1}{r}\right).
\end{align*}
Now we estimate that 
\begin{align*}
\binom{k(n+1)+r-1}{r}-\binom{kn+r-1}{r}  = & \,\ \frac{1}{r!}(kn)^r+ \frac{2k+r-1}{2(r-1)!}(kn)^{r-1} \\ 
& - \frac{1}{r!}(kn)^r - \frac{r-1}{2(r-1)!}(kn)^{r-1} + O(n^{r-2})\\
= & \,\ \frac{k^r}{(r-1)!} n^{r-1} + O(n^{r-2}).
\end{align*}
Putting these together, we obtain that
\begin{align*}
\length(R[[t_1,\dots,t_r]]/J_{N,k}^m)  & \geq \sum_{n = 0}^{Nm-1}  \left(\frac{A}{d!}(Nm - n)^d+ O((Nm-n)^{d-1})\right) \left(\frac{k^r}{(r-1)!} n^{r-1}+ O(n^{r-2})\right) \\
& = \frac{Ak^r}{d!(r-1)!}\sum_{n = 0}^{Nm-1}  \left((Nm - n)^d+ O((Nm-n)^{d-1})\right) \left(n^{r-1}+ O(n^{r-2})\right) \\
& = \frac{Ak^r(Nm)^{d +r}}{(d + r)!} + O(m^{d +r -1})
\end{align*}
where the last equality follows, for example, by \cite[Lemma 4.5]{Mumford} (or see the combinatorial identity at the end of the proof of Proposition~\ref{prop: Mumford 4.3}). It is now clear that we have 
$\eh (J_{N,k}) \geq A k^r N^{r + d}$. 
\end{proof}

The following is our key proposition, providing an estimate on $\limlm(R)$.

\begin{proposition}
\label{prop: exponential bound}
Let $(R, \mf m)$ be a local ring of dimension $d \geq 1$.
Suppose that $\{I_n\}_n$ is a graded family of $\m$-primary ideals and we set $I_0=R$. Suppose there exists $A\in \mathbb{R}$ such that
\[
\length (R/I_n) = \frac{A}{d!}n^{d} + O(n^{d-1}).  
\]
Then for every $x\in\mathbb{R}_{>0}$, we have
\[
\limlm(R) \geq \frac{A}{x^d\sum_{j =0}^{\infty} \length (I_j/I_{j+1})e^{-jx}}.
\]
\end{proposition}
\begin{proof}
Let $S \coloneqq R[[t_1, \ldots, t_r]]$ and consider the $\m+(t_1,\dots,t_r)$-primary ideal 
\[
J_{N, k} =  \sum_{n = 0}^{N} I_{N-n} (t_1, \ldots, t_r)^{kn} 
\]
constructed in Lemma~\ref{lem: introduce power family}. By Lemma~\ref{lem: introduce power family} and Lemma~\ref{lem: formula power multiplicity}, we may bound
\begin{align*}
\lm(S) &\geq \lim_{k \to \infty} 
\frac{\eh(J_{N, k})}{(r+d)! \length(S/J_{N,k})} \\
& \geq 
\lim_{k \to \infty} 
\frac{A k^rN^{r+d}/(r+d)!}{\sum_{j = 0}^{N-1} \length (I_j/I_{j+1})\binom{k(N-j) + r -1}{r} }
\\ &\geq \frac{A N^{r + d}/(r + d)! }
{\sum_{j = 0}^{N-1} \length (I_j/I_{j+1})(N-j)^r/r!}
\end{align*}
for any $N$. Setting $N = r/x$ yields
\begin{align*}
\lm(S) &\geq
\frac{r!}{(r+d)!}\frac{Ar^{r+d}}
{x^{r+d}\sum_{j = 0}^{r/x-1} \length (I_j/I_{j+1}) \left (\frac{r}{x}-j \right)^r }\\
& = \frac{r^d r!}{(r+d)!}\frac{A}
{x^d\sum_{j = 0}^{r/x-1} \length (I_j/I_{j+1}) \left (1-\frac{jx}{r} \right)^r }
\end{align*}
At this point, we note that $\left (1 - \frac{y}{r} \right )^r \leq e^y$ for all $0\leq y \leq r$. Thus we have
\[
\lm(S) \geq 
\frac{r^d r!}{(r+d)!}\frac{A}
{x^d\sum_{j = 0}^{r/x-1} \length (I_j/I_{j+1}) e^{-jx}  }
\geq \frac{r^d r!}{(r+d)!}\frac{A}
{x^d\sum_{j = 0}^{\infty} \length (I_j/I_{j+1}) e^{-jx} }
\]
and the conclusion follows by letting $r \to \infty$.
\end{proof}

\subsection{Semistability in dimension one}
Now we prove our main theorem in dimension one. Our result is a refinement and generalization of Mumford's characterization of (semi)stability in dimension one (see \cite[Proposition 3.14]{Mumford}). 

Following \cite[1.41]{KollarSingularitiesMMP}, we say that a local ring $R$ of dimension one is a {\it node} if $R\cong S/(f)$ where $(S,\m)$ is a regular local ring of dimension two, $f\in \m^2$, and the image of $f$ in $\m^2/\m^3$ is not a square.

\begin{theorem}
\label{thm: dimension one semistable}
Let $(R, \mf m, k)$ be a local ring of dimension one. Then the following conditions are equivalent:
\begin{enumerate}
    \item[(1)] $R$ is semistable,
    \item[(2)] $R$ is lim-stable,
    \item[(3)] $\unm{\widehat{R}}$ is either a regular local ring or a node,
    \item[(4)] $\unm{\widehat{R}}$ is Gorenstein and seminormal.
\end{enumerate}
\end{theorem}
\begin{proof}
We first note that by Proposition~\ref{prop: basic results on cLM} (\ref{part restrict to complete}) and (\ref{part restrict to unmixed}), we know that $\lm(R)=\lm(\unm{\widehat{R}})$ and $\lm(R[[T_1,\dots,T_r]])=\lm(\unm{\widehat{R}}[[T_1,\dots, T_r]])$ for all $r>0$. Hence, to prove the theorem, we can replace $R$ by $\unm{\widehat{R}}$, thereby assuming that $R$ is a Cohen--Macaulay complete local ring of dimension one for the remainder of the argument.

We already know $(1)\Rightarrow(2)$ from (\ref{eqn: chain of inequalities}). Moreover, $(3)\Leftrightarrow(4)$ follows from \cite[Theorem 8.1 (ii)$\Leftrightarrow$(viii)]{GrecoTraversoSeminormal}. 

We next show that $(3)\Rightarrow(1)$. Suppose $R$ is a node. By \cite[Theorem 8.1 (x)]{GrecoTraversoSeminormal}, we know that (the localization of) ${\Gr_\m(R)}$ is still a node and it is enough to show that the latter is semistable by Proposition~\ref{prop: LM associated graded} and Proposition~\ref{prop: LM multigraded}. Thus we may assume that $R$ contains a field. Now by \cite[Theorem 8.1 (xii)]{GrecoTraversoSeminormal} or \cite[1.41]{KollarSingularitiesMMP}, when $\charac(k)\neq 2$, $R$ is an ordinary double point, i.e., $R$ is a simple normal crossing after an \'etale extension. In this case, the completion of the strict henselization of $R$ is isomorphic to $k^{\sep}[[x,y]]/(xy)$. By \cite[Proposition 3.14 (2)]{Mumford} (see also Proposition~\ref{prop: BS for non normal limit} for an alternative treatment), we have $\lm(k^{\sep}[[x,y,z]]/(xy)) = 1$. It follows that $R$ is semistable since $R\to k^{\sep}[[x,y]]/(xy)$ is faithfully flat (for example, see Proposition~\ref{prop: faithfully flat extension}). Now if $\charac(k)=2$, then by \cite[Proposition 8.4]{GrecoTraversoSeminormal} and the discussion above, we may assume that 
$$R\cong \frac{k[x,y]_{(x,y)}}{x^2+\lambda y^2}$$
where $\lambda\in k$ does not have a square root in $k$. To show such $R$ is semistable, we may also assume $k$ is infinite by Corollary~\ref{cor: purely transcendental field extension}. We claim the following: 
\begin{claim}
\label{clm: claim in main dim one}
With notations as above, for all integers $n\geq 1$ the ring
$$R_n \coloneqq \frac{k[x,y,z]}{x^2+\lambda y^2 + z^n}$$
is pseudo-rational.
\end{claim}
\begin{proof}[Proof of Claim]
We can assign $x,y$ degree $n$ and $z$ degree $2$ to make $R_n$ an $\mathbb{N}$-graded ring over $k$. With this grading, $a(R_n)=-2$. Thus to show $R_n$ is pseudo-rational, it is enough to show that $R_n$ is normal.\footnote{This is well-known in characteristic zero, for example see \cite[Theorem 5.1]{KSSW09} (which originates from \cite{WatanabeRationalSingularity}). But the characteristic zero assumption in those references is only used to guarantee resolution of singularities. In our context, we can take $f\colon X\to \Spec(R_n)$ where $X=\Proj(R \oplus R_{\geq 1}t \oplus R_{\geq 2}t^2 \oplus \cdots)$. It is easy to check that, if $R_n$ is normal, then $X$ has pseudo-rational singularities (see \cite[Proof of Lemma 5.4]{KSSW09}). Thus to see $R_n$ is pseudo-rational, it is enough to observe that $R^1f_*\sO_X\cong [H_\m^2(R_n)]_{\geq 0}=0$. } Now clearly, $R$ satisfies $(\textnormal{S}_2)$ as $R$ is a hypersurface. Since $R$ is $\mathbb{N}$-graded and $k$ is infinite, the defining ideal of the nonsingular locus of $R$ is homogeneous and thus to show $R_n$ is $(\textnormal{R}_1)$, it is enough to show that $(R_n)_\m$ satisfies Serre's condition $(\textnormal{R}_1)$ where $\m=(x,y,z)$. 

We use induction on $n$. When $n=1$, $R_1\cong k[x,y]$ is regular and the conclusion is obvious. Consider the blowup $\Bl_{\m}(R_n)$. It has three affine charts isomorphic to 
$$\Spec\left(\frac{k[x,y',z']}{1+\lambda y'^2 + x^{n-2}z'^n}\right), \hspace{1em} \Spec\left(\frac{k[x',y,z']}{x'^2+\lambda + y^{n-2}z'^n}\right), \hspace{1em} \Spec\left(\frac{k[x',y',z]}{x'^2+\lambda y'^2 + z^{n-2}}\right).$$
We are aiming to show $(R_n)_\m$ is $(\textnormal{R}_1)$. It suffices to prove that the three charts are normal at points that map to $\m$. Now the third chart is normal by our inductive hypothesis (when $n=2$, it is easy to see that this chart is regular). In the first chart, $\m=(x)$ and thus it is enough to observe that after modulo $(x)$, the ring becomes $k[y', z']/(1+\lambda y'^2)$ which is regular as $\lambda$ does not have a square root in $k$. Similarly the second chart is regular after modulo $\m=(y)$ and we are done.
\end{proof}

Now we complete the proof of $(3)\Rightarrow(1)$. For any $(x,y,t)$-primary ideal $I\subseteq R[[t]]$, we invoke Theorem~\ref{thm: small perturbation} to see that $\Gr_I(R[[t]])\cong \Gr_I(R_n)$ for some $n\gg0$ (here we still use $I$ to denote the image of the ideal in $R_n$). In particular, $\length(R[[t]]/I)=\length(R_n/I)$ and $\eh(I, R[[t]])=\eh(I, R_n)$. Since $R_n$ is Gorenstein and pseudo-rational by Claim~\ref{clm: claim in main dim one}, Corollary~\ref{cor: RDP Lech stable} implies that $\lm(R_n)=1$. It follows that 
$$\frac{\eh(I, R[[t]])}{2\length(R[[t]]/I)} = \frac{\eh(I, R_n)}{2\length(R_n/I)}\leq \lm(R_n)=1.$$
Thus $\lm(R[[t]])=1$ and $R$ is semistable.

It remains to prove $(2)\Rightarrow(3)$. Suppose $R$ is lim-stable, we first claim that $\eh(R)\leq 2$. By Proposition~\ref{prop: exponential bound} applied to the graded family $\{I_n=\m^n\}_n$, we find that 
\[
\limlm(R) \geq \frac{\eh(R)}{x\sum_{j \geq 0} \length (\mf m^j/\mf m^{j+1})e^{-jx}}.
\]
Setting $x=1$ above and noting that $\length(\m^j/\m^{j+1})\leq \eh(R)$ for all $j\geq 1$, since $R$ is Cohen--Macaulay, we have the bound 
\[
\limlm(R) \geq \frac{\eh(R)}{\sum_{j \geq 1} \length (\mf m^j/\mf m^{j+1})e^{-j}+1} \geq \frac{\eh(R)}{\eh(R)/(e-1)+1} = \frac{1}{1/(e - 1) + 1/\eh(R)}.
\]
If $\eh(R)\geq 3$, then $1/(e - 1) + 1/\eh(R) < $, contradicting that $R$ is lim-stable. Thus, $\eh(R)\leq 2$. Since $R$ is a one-dimensional Cohen--Macaulay ring, by Abhyankar's inequality we must have that $\edim(R)\leq \eh(R)+1-1 \leq 2$.

If $\edim(R)=1$, then $R$ is regular. So we suppose $\edim(R)=2$ and $\eh(R)=2$. Then $R$ is a hypersurface, i.e., we have $R\cong S/(f)$ where $(S, \mf m)$ is a complete regular local ring of dimension two
and $f\in \m^2-\m^3$. We need to show that the image of $f$ in $\m^2/\m^3$ is not a square. Suppose on the contrary that $f=z^2+g$ where $z$ is part of a regular system of parameters of $S$ and $g\in \m^3$. Let $z,y$ be a full regular system of parameters of $S$. We want to construct a graded family of ideals and apply Proposition~\ref{prop: exponential bound}. There are two cases to consider:

\begin{enumerate}
    \item[(a)] $R$ is a domain. We let $V \coloneqq R^{\text{N}}$ be the normalization of $R$. Then $(V,\varpi, k')$ is a DVR, and we set $I_n \coloneqq (\varpi^n)\cap R$. 
    \item[(b)] $R$ is not a domain. Then we can write $f=(z+g_1)(z+g_2)$ where $g_1,g_2\in \m^2$. We have $R^{\text{N}}\cong S/(z+g_1) \times S/(z+g_2)$ and we let $I_n\coloneqq(y^n)R^{\text{N}}\cap R$. 
\end{enumerate}

In case (a), note that $2=\eh(R)=\eh_R(\m, V)=[k':k]\eh(\m V, V),$
thus we have two scenarios: either $\m V=(\varpi)$ and $[k': k]=2$, or $\m V=(\varpi^2)$ and $V/(\varpi)\cong k$. In the first scenario, $z\in (\varpi^2)$ and $yV=(\varpi)$. It follows that 
$\length(R/I_1)=\length(I_1/I_2)=1$ and that $\length(I_n/I_{n+1})\leq 2$  for all $n$.
By Proposition~\ref{prop: exponential bound}, we then have a bound
$$\limlm(R) \geq \frac{2}{x(1+e^{-x}+2\sum_{j \geq 2}e^{-jx})} =\frac{2(1-e^{-x})}{x+xe^{-2x}}$$
contradicting that $R$ is lim-stable since 
\[
\sup_{x\in\mathbb{R}_{>0}} \left \{ \frac{2(1-e^{-x})}{x+xe^{-2x}}  \right \} > 1.1559.
\]
In the second scenario, we have $I_1=I_2=\m$. It follows that we have
$$\length(R/I_1)=1, \length(I_1/I_2)=0, \text{ and } \length(I_n/I_{n+1})\leq 1 \text{ for all $n$.} $$
By Proposition~\ref{prop: exponential bound}, we then have
$$\limlm(R) \geq \frac{1}{x(1+\sum_{j \geq 2}e^{-jx})} =\frac{1-e^{-x}}{x-xe^{-x}+xe^{-2x}}. $$
This bound contradicts again that $R$ is lim-stable since 
\[
\sup_{x\in\mathbb{R}_{>0}} \left \{ \frac{1-e^{-x}}{x-xe^{-x}+xe^{-2x}}  \right \} > 1.0696.
\]

Finally, in case (b), note that in both $S/(z+g_1)$ and $S/(z+g_2)$, we have that $z\in (y^2)$ since $g_1,g_2 \in \m^2$. Thus $I_1=\m$ and $I_2=(y^2,z)$. It follows that 
$$\length(R/I_1)=\length(I_1/I_2)=1, \text{ and } \length(I_n/I_{n+1})\leq 2 \text{ for all $n$.} $$
Thus by exactly the same argument as in the first scenario of case (a), we have $\limlm(R)>1.1559$ contradicting that $R$ is lim-stable. This completes the proof of $(2)\Rightarrow(3)$.
\end{proof}

\begin{corollary}
\label{cor: lim-stable implies demi-normal}
If $(R,\m)$ is equidimensional, $(\textnormal{S}_2)$, and lim-stable, then $R$ is demi-normal.
\end{corollary}
\begin{proof}
By definition (\cite[Definition~5.1]{KollarSingularitiesMMP}), {\it demi-normality} requires Serre's condition $(\textnormal{S}_2)$ and that codimension one points are either regular or nodes. These properties follow from Theorem~\ref{thm: dimension one semistable} and Theorem~\ref{thm: localization} (note that equidimensional and $(\textnormal{S}_2)$ implies catenary).
\end{proof}

\subsection{Strict complete intersections and rings of large multiplicity}
In this subsection, we collect some consequences of Proposition~\ref{prop: exponential bound} in higher dimensions.  First of all, we show that for strict complete intersections -- i.e., local rings $(R,\m)$ such that $\Gr_\m (R)$ is a complete intersection (this includes all hypersurfaces and all standard graded complete intersections) -- various stability conditions force the defining equations to have small orders. This result serves as one of our motivations for studying the relationship between stability of local rings and singularities in the Minimal Model Program (MMP), and will be significantly generalized in Section~\ref{section: MMP}.

\begin{corollary}
\label{cor: strict complete intersections}
Let $(R,\m, k)$ be a strict complete intersection and write 
$$\Gr_\m(R) = k[x_1, \ldots, x_n]/(f_1, \ldots, f_c)$$ 
with $\deg f_i = D_i$. Then the following hold:
\begin{enumerate}
    \item if $R$ is lim-stable, then $\sum_{i=1}^c D_i \leq n$; 
    \item if $R$ is Lech-stable, then $\sum_{i=1}^c D_i <n$.
\end{enumerate}
\end{corollary}
\begin{proof}
We first prove (1). We note that $$\sum_{j=0}^\infty\length(\m^j/\m^{j+1})t^j=h_{\Gr_\m(R)}(t)=\frac{\prod_{i=1}^c(1-t^{D_i})}{(1-t)^n},$$ where the latter is the Hilbert series of the standard graded complete intersection $\Gr_\m(R)$. Thus by applying Proposition~\ref{prop: exponential bound} to $\{\m^n\}_n$, we obtain that
\[
\limlm(R) \geq L(x)\coloneqq \frac{\prod_{i=1}^c D_i}{x^{n-c}\sum_{j = 0}^{\infty} h_{\Gr_\m(R)}(e^{-x})}
=\frac{\prod_{i=1}^c D_i (1 - e^{-x})^{n}}{x^{n-c}\prod_{i = 1}^c (1 - e^{-D_ix})}
\]
for all $x\in\mathbb{R}_{>0}$. Using that
$$\frac{1 - e^{-x}}{x} = 1 - \frac{x}{2} + O(x^2),$$ 
we further estimate that
\begin{align*}
L(x) &= \frac{(1 - \frac{x}{2} + O(x^{2}))^{n}}{\prod_{i = 1}^c (1 - \frac{D_i}{2}x + O(x^{2}))} \\
& = \left(1 - \frac{n}{2}x + O(x^{2}) \right) \left(1 + \frac{\sum_{i=1}^c D_i}{2}x + O(x^{2}) \right)
\\&= 1 + \frac{\sum_{i=1}^c D_i - n }{2}x + O(x^{2}).
\end{align*}
Thus if $\sum_{i=1}^c D_i > n$, then $\limlm(R)\geq L(x_0)>1$ for any $x_0$ sufficiently small. It follows that if $R$ is lim-stable, then $\sum_{i=1}^c D_i\leq n$ as wanted. 

We next prove (2). This is easier and we just need to rewrite $h_{\Gr_\m(R)}(t)$ via Laurent series in $(1-t)$:
$$\frac{\prod_{i=1}^c(1-t^{D_i})}{(1-t)^n}=\frac{\prod_{i=1}^c D_i}{(1-t)^{n-c}}- \frac{\frac{1}{2}(\prod_{i=1}^c D_i)(\sum_{i=1}^c D_i -c)}{(1-t)^{n-c+1}}+ \cdots$$
Note that this series is finite since $1$ is the only pole of $h_{\Gr_\m(R)}(t)$. From this we obtain that 
\begin{align*}
  \length(R/\m^j) &= (\prod_{i=1}^c D_i)\binom{n-c+j-1}{n-c} - \frac{1}{2}(\prod_{i=1}^c D_i)(\sum_{i=1}^c D_i -c) \binom{n-c+j-2}{n-c-1}+ O(j^{n-c-2})\\
  &= \frac{(\prod_{i=1}^c D_i)}{(n-c)!}j^{n-c} + (\prod_{i=1}^c D_i)\frac{n-1-\sum_{i=1}^c D_i}{2(n-c-1)!}j^{n-c-1} + O(j^{n-c-2})
\end{align*}
Since $\eh(\m^j)=(\prod_{i=1}^c D_i)j^{n-c}$, our estimate for $\length(R/\m^j)$ shows that whenever $\sum_{i=1}^c D_i\geq n$, we must have $\eh(\m^j)>(n-c)!\length(R/\m^j)$ for all $j\gg0$. Thus if $R$ is Lech-stable, then $\sum_{i=1}^c D_i <n$ as wanted.
\end{proof}

\begin{remark}
\label{rmk: lim LM of x^3+y^3}
Given an explicit strict complete intersection, we may get a lower bound for $\limlm(R)$ by optimizing $L(s)$ in Corollary~\ref{cor: strict complete intersections}. For example, one easily obtains that $\limlm(k[[x,y]]/(x^3 + y^3)) \geq 1.26$.
\end{remark}

We next prove that Cohen--Macaulay local rings of sufficiently large multiplicity cannot be lim-stable.

\begin{corollary}\label{cor: large multiplicity is lim-unstable}
    For any positive integer $d$ there is a constant $C = C(d)$ with the following property: if $(R, \m)$ is a Cohen--Macaulay local ring of dimension $d$
    such that $\lm(R) \geq C$ then $\limlm(R) > 1$. In particular, if $\eh(R) \geq d! C$, then $\limlm(R) > 1$.
\end{corollary}
\begin{proof}
    By Theorem~\ref{thm: dimension one semistable} and Theorem~\ref{thm: uniform Lech}, we may take $C(1) = 3$. Thus we may assume that $d \geq 2$. Since $R$ is Cohen--Macaulay, by \cite[Corollary~2.6]{RVV}
    the Hilbert series of $\Gr_I (R)$ is bounded above by 
    $(\length (R/I) + (\eh(I) - \length (R/I)t)(1 - t)^{-d}$.
    By Proposition~\ref{prop: exponential bound} applied to $\{I^n\}_n$ 
    we thus estimate that 
    \[
        \limlm(R) \geq \frac{(1 - e^{-x})^d}{x^d \left( \frac{\length (R/I)}{\eh(I)} (1 - e^{-x}) + e^{-x}\right) }. 
    \]
    It follows that 
    \[
        \limlm(R) \geq \frac{(1 - e^{-x})^d}{x^d \left (\frac {1 - e^{-x}}{d!\lm(R)}  + e^{-x}\right )}. 
    \]
    Since $e^x (1- e^{-x})^d x^{-d}$ grows without a bound, it remains to fix any $x_0 > 0$ such that $e^{x_0} (1 - e^{-x_0})^d x_0^{-d} > 1$ and note that 
    \[
        \lim_{B \to \infty} \frac{(1 - e^{-x_0})^d}{x_0^d ((1 - e^{-x_0})/B + e^{-x_0})}
        = \frac{(1 - e^{-x_0})^d}{x_0^{d}e^{-x_0}} > 1.
    \]
The last statement follows from the fact that if $\eh(R)\geq d!C$, then $\lm(R)\geq C$. 
\end{proof}

\begin{remark}
\label{rmk: values of E(d)}
While we know that $C(1) = 3$ by Theorem~\ref{thm: dimension one semistable}, in higher dimensions we can only estimate that
\[
C(d) \leq \inf \left \{\frac{B}{d!} \mathrel{\Big|} \sup_{x > 0}\left\{ \frac{(1 - e^{-x})^d}{x^d ((1 - e^{-x})/B - e^{-x})} \right\} >1 \right\}.
\]
The values on the right-hand side grow quite fast, we have $C(2) \leq 8.375$, $C(3) \leq 77.664$,
$C(4) \leq 754.48$, etc. As a consequence, we have that a degree $n$ Veronese subring of $k[x_1, \ldots, x_d]$ is not lim-stable if $n \geq 17$ and $d = 2$, $n \geq 22$ and $d = 3$, and $n \geq 27$ and $d = 4$, etc.
\end{remark}

\begin{corollary}\label{cor: lim unstable dimension 2}
Let $(R, \mf m)$ be a two-dimensional Cohen--Macaulay local ring.
If $R$ has minimal multiplicity and $\eh(R) \geq 17$ or 
if $R$ does not have minimal multiplicity and $\eh(R) \geq 13$, 
then $\limlm(R) > 1$.
\end{corollary}
\begin{proof}
The first statement follows from Corollary~\ref{cor: large multiplicity is lim-unstable} (see Remark~\ref{rmk: values of E(d)}). If $R$ does not have minimal multiplicity, then \cite[Theorem~3.3]{RVV} implies that
    \[
        \sum_{n \geq 0} \length (\m^n/\m^{n+1}) t^n \leq \frac{1 + (\eh (R) - 2)t + t^2}{(1 - t)^2}.
    \]
Thus we obtain from Proposition~\ref{prop: exponential bound} that 
\[
\limlm(R) \geq 
\frac{\eh(R) (1 - e^{-x})^2}{x^2 (1 + (\eh (R) - 2)e^{-x} + e^{-2x})}.
\]
Setting $x = 1$, it is easy to check that the right-hand side above is greater than one when $\eh (R) \geq 13$.
\end{proof}

Finally, we prove the following result that extends \cite[Corollary~7.6 and Theorem~2]{Shah} from semistable 
to lim-stable singularities. When $R$ is excellent and normal, this result will be revisited in Section~\ref{section: MMP}. 

\begin{proposition}\label{prop: Shah associated graded is CM and reduced}
 Let $(R, \m)$ be a two-dimensional Cohen--Macaulay local ring. 
    Suppose $R$ is lim-stable and $\eh(R) = \edim(R)$. Then $\Gr_\m (R)$ is Cohen--Macaulay and reduced. 
\end{proposition}
\begin{proof}
    If $\Gr_\m(R)$ is not Cohen--Macaulay, then by \cite[Theorem (i) and (ii)]{Sally3} (or \cite[Theorem~3]{Shah}), the Hilbert polynomial of $\Gr_\m (R)$ is at most $n \eh(R) - 1$. Combining this with \cite[Theorem~3.3]{RVV}, we obtain that there exists $N \geq 0$ such that $\length (\m^n/\m^{n+1}) \leq n \eh(R) - 1$ for all $n \geq N$ and $\length (\m^n/\m^{n+1}) \leq n\eh(R)$ for $1 \leq n < N$.
    Therefore by Proposition~\ref{prop: exponential bound} we have the bound
    \begin{align*}
    \limlm(R) \geq f(x) \coloneqq & \,\ \frac{\eh(R)}{x^2(1+ \sum_{j=1}^{N-1} j\eh(R)e^{-jx} + \sum_{j=N}^\infty (j\eh(R)-1)e^{-jx})}\\
    = & \,\  \frac{\eh(R) (1 - e^{-x})^2/x^2}{1 + (\eh (R) - 2)e^{-x} + e^{-2x} - e^{-Nx} + e^{-(N+1)x}}.
    \end{align*}
    Since $\lim_{x \to 0} f(x) = 1$ and, by a straightforward computation, $\lim_{x \to 0} f'(x)=\frac 1 {\eh(R)}>0$, we obtain that $f(x) > 1$ for all positive sufficiently small $x$. Therefore $\limlm(R)>1$ which contradicts our assumption that $R$ is lim-stable. We have proved that $\gr_{\m}(R)$ is Cohen--Macaulay. 
    
    Next, by \cite[Theorem 3 (ii)]{Shah} or by combining \cite[Theorem~3.3, Lemma~3.4]{SallySuper} with \cite[Theorem~2.11]{Huneke},
    the Hilbert function of $R$ is $\length(\m^n/\m^{n+1})=n\eh(R)$ for all $n \geq 1$. Now 
    by \cite[Corollary~3.6]{PWYStrongReesProperty}, $\m^n$ has the strong Rees property for all $n$, i.e., $\length (\m^n/\m^{n+1}) > \length (I/\m I)$ for any ideal $\m^n \subsetneq I$. Thus for any $n$, we have that
\[
\length (\overline{\m^n}/\overline{\m^{n+1}}) \leq 
\length (\overline{\m^n}/\m \overline{\m^{n}})\leq 
\begin{cases}
\length (\m^n/\m^{n+1}) = n\eh(R),  &\text{ if $\overline{\m^n}=\m^n$.} \\
\length (\m^n/\m^{n+1})-1  = n\eh(R)-1, &\text{ if $\overline{\m^n}\neq \m^n$.}
\end{cases}
\] 
    Without loss of generality, we may assume that $R$ is complete.
    By Proposition~\ref{prop: limit-stable reduced}, $R$ is reduced, hence $\lim_{n \to \infty} \length (R/\overline{\m^n})/n^{2} = \eh(R)/2$.
    Now if $\m^N \neq \overline{\m^{N}}$ for some $N$, then by applying Proposition~\ref{prop: exponential bound} to the graded family $\{\overline{\m^n}\}_n$ and using a similar argument as above, we obtain that 
    \begin{equation*}
        \limlm(R) \geq g(x) \coloneqq \frac{\eh(R) (1 - e^{-x})^2/x^2}{1 + (\eh (R) - 2)e^{-x} + e^{-2x} - e^{-Nx}(1 - e^{-x})^2}.  
    \end{equation*}
    In a neighborhood of $0$ we may expand the numerator as 
    $g_1(x) = \eh(R)(1 - x + 7x^2/12) + o(x^2)$
    and the denominator as $g_2(x) = \eh(R)(1 - x + x^2/2) + o(x^2)$.
    Thus we see that $\lim_{x \to 0} g(x) = 1$ and $g'(x) > 0$ for all $x>0$ sufficiently small. Hence $\limlm(R) > 1$, which contradicts $R$ being lim-stable. Thus $\m^n$ is integrally closed for all $n$, i.e., $\m$ is normal. 

    Suppose the Rees period of $\m$ is $\rho > 1$.
    Fix any $0 < c < \rho$ and let $I_n = \m^{\frac{n\rho - c}{\rho}}$. If $I_n \supsetneq \m^n$, then we have
    \[
    \length (I_n/I_{n+1}) \leq 
    \length (I_n/\m I_{n}) \leq n\eh(R) - 1 
    \]
    by the strong Rees property of $\m^n$. 
    Thus, due to Proposition~\ref{prop: rational powers are quasi-polynomial}, we may apply Proposition~\ref{prop: exponential bound} to the graded family $\{I_n\}_n$ and repeat the previous argument to arrive at a contradiction. Hence we have $\m^{\frac{n\rho - c}{\rho}}= \m^n$ for all $0<c<\rho$ and all $n$. But this contradicts the definition of Rees period. Therefore we must have $\rho=1$ and thus $\m^\alpha= \m^{\lfloor\alpha \rfloor}$ for all $\alpha\in\Q_{\geq 0}$. It follows that $\Gr_\m(R)$ is reduced by \cite[Theorem 10.5.6]{SwansonHuneke}.
\end{proof}

\newpage
\section{Connections with singularities in MMP}
\label{section: MMP}

In this section we will establish a connection between stability of local rings and singularities arising in the Minimal Model Program (MMP). Our results can be seen as a generalization of Mumford's heuristic conjecture that semistable surface singularities are a restricted class of rational and strongly elliptic singularities and their non-normal limits (see \cite[Page 77]{Mumford}).   

\subsection{Singularities in the minimal model program} We start by recalling some basic concepts in birational geometry and MMP singularities, following \cite{KollarMori} and \cite{KollarSingularitiesMMP}. 

Let $(R,\m)$ be an excellent local ring admitting a dualizing complex. A \emph{pair} $(X, \Delta)$ (over $\Spec(R)$) is 
a normal integral scheme $X$ of essentially of finite type over $R$ and an effective $\mathbb{Q}$-divisor $\Delta\geq 0$ on $X$. Note that $X$ admits a dualizing complex $\omega_X^\bullet=\pi^!\omega^\bullet_R$ where $\pi \colon X\to \Spec(R)$, in particular $X$ admits a canonical sheaf $\omega_X=H^{-\dim X}\omega_X^\bullet$. The canonical divisor $K_X$ is a divisor on $X$ so that $\sO_X(K_X)\cong \omega_X$, this is well-defined up to linear equivalence.

Let $(X,\Delta)$ be a pair. Let $E$ be a divisor over $X$, that is, there exists a projective birational morphism $f \colon Y\to X$ with $Y$ normal such that $E$ is a prime divisor on $Y$. The divisor $E$ is called \emph{exceptional} over $X$ if the center of $E$ on $X$ has codimension at least two. Suppose $K_X+\Delta$ is $\Q$-Cartier. We define the discrepancy of $E$, $a(E; X, \Delta)$, as the coefficient of $E$ in $K_Y -f^*(K_X+\Delta)$, where we always assume that the canonical divisors $K_X$ and $K_Y$ are chosen so that $f_*K_Y=K_X$.
When $X=\Spec(R)$, we write $a(E; R,\Delta)$ for $a(E; X,\Delta)$.

\begin{definition}
Let $(X,\Delta)$ be a pair such that the coefficients of $\Delta$ are $\leq 1$ and $K_X+\Delta$ is $\Q$-Cartier. We say that $(X,\Delta)$ is
\begin{enumerate}
\item {\it terminal} if $a(E; X,\Delta)> 0$ for all exceptional divisors $E$ over $X$;
\item {\it canonical} if $a(E; X,\Delta)\geq 0$ for all exceptional divisors $E$ over $X$;
\item {\it log terminal} if $a(E; X,\Delta)>-1$ for all divisors $E$ over $X$;
\item {\it log canonical} if $a(E; X,\Delta)\geq -1$ for all divisors $E$ over $X$.
\end{enumerate}
When $X=\Spec(R)$, we will also say $(R,\Delta)$ is terminal (resp., canonical, log terminal, log canonical) if $(X,\Delta)$ is so. We will suppress $\Delta$ from the notation if $\Delta = 0$.
\end{definition}

\begin{remark}
It is well-known that if $X$ is normal, $\Q$-Gorenstein and $X$ admits a log resolution $Y\to X$, then $X$ is terminal (resp., canonical) if and only if $a(E; X)>0$ (resp. $a(E; X) \geq 0$) for all exceptional prime divisors $E$ on $Y$. Similarly, if
$(X,\Delta)$ is a pair such that $K_X+\Delta$ is $\Q$-Cartier and $(X,\Delta)$ admits a log resolution $Y\to (X,\Delta)$, then $(X,\Delta)$ is log terminal (resp., log canonical) if and only if $a(E; X, \Delta)>-1$ (resp., $a(E; X, \Delta) \geq -1$) for all prime divisors $E$ on $Y$. 
\end{remark}

Log canonical modifications will play a critical role in our proofs.

\begin{definition}
Let $(X,\Delta)$ be a pair such that all coefficients of $\Delta$ are $\leq 1$. We say that a projective birational morphism $f\colon Y\to (X,\Delta)$ is a log canonical modification of $(X,\Delta)$ if $Y$ is normal and we have
\begin{itemize}
    \item $(Y,f_*^{-1}\Delta + \red{E})$ is log canonical,
    \item $K_Y+f_*^{-1}\Delta + \red{E}$ is $f$-ample,
\end{itemize}
where $\red{E}$ denotes the reduced exceptional divisor, i.e., the sum of all $f$-exceptional prime divisors with coefficient $1$.
\end{definition}

It is well-known that the existence of log canonical modifications follows from the full log minimal model program (including the abundance conjecture), and when a log canonical modification of $(X,\Delta)$ exists, then it is unique, see \cite{OdakaXu,KollarSingularitiesMMP} for further discussions. We will use the following celebrated result of Odaka--Xu. 

\begin{theorem}[{\cite[Theorem 1.1]{OdakaXu} and \cite[Theorem 1.32]{KollarSingularitiesMMP}}]
\label{thm: Odaka-Xu}
Let $(X,\Delta)$ be a pair such that $X$ is essentially of finite type over a field of characteristic zero and all coefficients of $\Delta$ are $\leq 1$. If $K_X+\Delta$ is $\Q$-Cartier, then $(X,\Delta)$ admits a log canonical modification. 
\end{theorem}

Our full results in characteristic zero require some generalizations of log canonical singularities beyond the $\Q$-Gorenstein setting. Though we will only require the numerically $\Q$-Gorenstein case (see below) in our applications, we discuss this theory in the more general setting of \cite{deFernexHacon}, see also \cite{BdFFU,HashizumeSingularityArbitraryPair}. Let $f \colon Y\to X$ be a projective birational morphism of normal integral schemes. The two types of $m$th limiting relative canonical divisors are defined as
$$K^-_{m,Y/X}\coloneqq K_Y - \frac{1}{m} f^{\natural}(mK_X) \text{ and } K^+_{m,Y/X}\coloneqq K_Y + \frac{1}{m} f^{\natural}(-mK_X),$$
where $f^{\natural}(D)=\text{div}(\sO_X(-D)\cdot \sO_Y)$ for a divisor $D$ on $X$. It is not difficult to see that 
$$K^-_{m,Y/X} \leq K^-_{mn,Y/X} \leq K^+_{mn,Y/X}\leq K^+_{m,Y/X}$$
for all $m, n$, with equalities throughout if $mK_X$ is Cartier. After taking limsup (resp., liminf) of $K^-_{m,Y/X}$ (resp., $K^+_{m,Y/X}$), we obtain $\mathbb{R}$-divisors $K^-_{Y/X}$ and $K^+_{Y/X}$. 

For any prime divisor $E$ on $Y$, we write $a_m(E;X)$ for the coefficient of $E$ in $K^-_{m, Y/X}$ and $a(E;X)$ for the coefficient of $E$ in $K^-_{Y/X}$. Note that $a(E;X)=\limsup a_m(E;X)$ for every divisor $E$ over $X$. 

\begin{definition}
    We say that
\begin{itemize}
    \item $X$ is of {\it log canonical type} if there exists $m$ such that $a_m(E;X)\geq -1$ for all divisors $E$ over $X$, see \cite[Definition 7.1]{deFernexHacon}.
    \item $X$ is {\it pseudo-log canonical} if $a(E;X)\geq -1$ for all divisors $E$ over $X$, see \cite[Definition 4.1, Proposition 4.7 and Theorem 4.8]{HashizumeSingularityArbitraryPair}. 
\end{itemize}
\end{definition}

It is clear that log canonical type implies pseudo-log canonical and that they both agree with log canonical when $K_X$ is $\Q$-Cartier. On the other hand, there are examples of pseudo-log canonical singularities that are not of log canonical type \cite[Example 4.11]{HashizumeSingularityArbitraryPair}.

When $X$ is essentially of finite type over a field of characteristic zero, for every divisor $E$ over $X$ and every $m\geq 2$, there exists an effective $\Q$-divisor $\Delta\geq 0$ on $X$ such that $m\Delta$ is integral, $K_X+\Delta$ is $\Q$-Cartier, and $a_m(E;X)=a(E; X,\Delta)$, as shown in \cite[Theorem 5.4]{deFernexHacon}. It follows that, in this case, we have 
$$a(E;X)=\sup\{a(E;X,\Delta) \mid \Delta\geq 0 \text{ and } K_X+\Delta \text{ is $\Q$-Cartier}\}.$$
Hence, $X$ is of log canonical type if and only if there exists a pair $(X,\Delta)$ such that $K_X+\Delta$ is $\Q$-Cartier and $(X,\Delta)$ is log canonical.

A $\Q$-divisor $D$ on $X$ is called {\it numerically $\Q$-Cartier} if, locally on $X$, there is a regular alteration $\pi \colon Z\to X$ and a $\pi$-numerically trivial $\Q$-divisor $G$ on $Z$ such that $\pi_*G=D$. We call $X$ {\it numerically $\Q$-Gorenstein} if $K_X$ is numerically $\Q$-Cartier.  Note that, if $X$ is numerically $\Q$-Gorenstein and $f\colon Y\to X$ is any projective birational morphism with $Y$ normal, then we have a well-defined $\Q$-divisor $f_{\num}^*K_X$ such that $K_Y-f_{\num}^*K_X$ is $f$-exceptional and $(f_{\num}^*K_X)\cdot F^{d-1}=0$ for all Cartier divisors $F$ on $Y$ that is $f$-exceptional. Explicitly, one may set $f_{\num}^*K_X\coloneqq g_*K$, where $g\colon Z\to Y$ is a sufficiently large regular alteration\footnote{Here we assume the existence of regular alterations, which holds when $X$ is essentially of finite type over a field or over a complete DVR, or when $\dim X\leq 3$.} and $K$ is the numerically trivial $\Q$-divisor on $Z$ with $(f\circ g)_*K=K_X$. The property $(f_{\num}^*K_X)\cdot F^{d-1}=0$ follows by the projection formula. 

Since we can take $G=\pi^*D$, any $\Q$-Cartier divisor is numerically $\Q$-Cartier, in particular, $\Q$-Gorenstein singularities are numerically $\Q$-Gorenstein. More generally, we have 
\begin{itemize}
    \item If $\dim X=2$, then any $\Q$-divisor is numerically $\Q$-Cartier (in particular, $X$ is numerically $\Q$-Gorenstein), see \cite[Example 5.7]{BdFFU} and \cite[Example 3.4]{TakagiFinitisticTestIdeal}. 
    \item If $X$ is essentially of finite type over a field of characteristic zero, then $X$ is numerically $\Q$-Gorenstein if and only if $K^-_{Y/X}=K^+_{Y/X}$ for every resolution of singularities $Y\to X$, see \cite[Proposition 5.9]{BdFFU}. In particular, when $X$ is numerically $\Q$-Gorenstein, $K^-_{Y/X}=K^+_{Y/X}$ is a $\Q$-divisor.
\end{itemize}

We record the following result on log canonicity for numerically $\Q$-Gorenstein varieties.
\begin{theorem}[{\cite[Proposition 3.5]{ZhangIsolatedSingularitiesNonInvertibleEndo}}, see also {\cite[Corollary 5.2 and Lemma 5.4]{HashizumeSingularityArbitraryPair}}]
\label{thm: num Q-Gor lc is Q-Gor lc}
Let $X$ be a normal integral scheme essentially of finite type over a field of characteristic zero. Suppose $X$ is numerically $\Q$-Gorenstein and pseudo-log canonical. Then $X$ is $\Q$-Gorenstein and log canonical. 
\end{theorem}

We end this subsection by proving the following slight generalization of \cite[Theorem 1.1]{HashizumeSingularityArbitraryPair} for numerically $\Q$-Gorenstein rings that we will need later. We refer to \cite{KollarMori}, \cite{BCHM}, and \cite{KollarSingularitiesMMP} for basic language of MMP.

\begin{theorem}
\label{thm: Hashizume}
Let $(R,\m)$ be a numerically $\Q$-Gorenstein normal local domain essentially of finite type over a field of characteristic zero. Then there exists a projective birational morphism $h$: $W\to X\coloneqq\Spec(R)$ such that 
\begin{enumerate}
    \item $a(E_i,X)<-1$ for any $h$-exceptional prime divisor $E_i$;
    \item each $h$-exceptional prime divisor $E_i$ on $W$ is $\Q$-Cartier;
    \item $K_W+\red{E}$ is $\Q$-Cartier and the pair $(W,\red{E})$ is log canonical;
    \item there exists an $h$-ample exceptional Cartier divisor on $W$.
\end{enumerate}
\end{theorem}
\begin{proof}
We follow the approach from \cite[Proof of Theorem 4.14]{HashizumeSingularityArbitraryPair}. Let $f \colon Y\to X$ be a log resolution. We divide the exceptional prime divisors based on their discrepancies with respect to 
$X$. Specifically, let:
\begin{enumerate}
    \item $F_j$ be the exceptional divisors on $Y$ such that $a(F_j;X)\geq -1$,
    \item $G_i$ be the exceptional divisors on $Y$ such that $a(G_i;X)<-1$.
\end{enumerate}
Set $F=\sum F_j$ and consider the pair $(Y, F+\sum t_iG_i)$ where each $t_i\in \Q$ and $0\leq t_i <1$. 

We now run a $(K_Y+F+\sum t_iG_i)$-MMP over $X$ with scaling of an ample divisor. After finitely many steps, we obtain a $\Q$-factorial dlt model $f' \colon (Y', F'+\sum t_iG_i')\to X$ where $K_{Y'}+F'+\sum t_iG_i'$ is the limit of movable divisors over $X$. The divisors $F', G_i'$ are the birational transformations of $F$ and $G_i$ on $Y'$. Moreover, $a(E'; X)\leq -1$ for any $f'$-exceptional prime divisor $E'$ on $Y'$ (see \cite[Step 2 in Proof of Theorem 4.14]{HashizumeSingularityArbitraryPair} or \cite[Theorem 3.1]{ZhangIsolatedSingularitiesNonInvertibleEndo}). 

We next claim that $(Y',  F'+\sum t_iG_i')$ admits a good minimal model over $X$. To establish this, it is enough to check that this pair satisfies the conditions of \cite[Theorem 3.5]{HashizumeSingularityArbitraryPair}. Since $X$ is numerically $\Q$-Gorenstein, we know by construction that $K_{Y'}-f'^*_{\num}K_X = -F' - \sum a_iG_i'$ where each $a_i>1$. Thus $$K_{Y'}+F'+\sum a_iG_i' = (K_{Y'} +F' + \sum t_iG_i') + \sum (a_i-t_i)G_i'$$ is $f'$-numerically trivial. In particular, the divisor $-(K_{Y'} +F' + \sum t_iG_i')$ is $f'$-numerically equivalent to the effective divisor $\sum (a_i-t_i)G_i'$ since $a_i>1$ and $t_i<1$. Since no stratum of $\sum G_i'$ could be an log canonical center of the pair $(Y',  F'+\sum t_iG_i')$, it follows that the pullback of $-(K_{Y'} +F' + \sum t_iG_i')$ to the normalization of any log canonical center of $(Y',  F'+\sum t_iG_i')$ is pseudo-effective. Thus $(Y',  F'+\sum t_iG_i')$ admits a good minimal model $(Y'', F''+\sum t_iG_i'')$ over $X$ by \cite[Theorem 3.5]{HashizumeSingularityArbitraryPair}.

At this point, we let $Y''\to W_{\underline{t}}$ be the contraction over $X$ induced by the semi-ample divisor $K_{Y''}+F''+\sum t_iG_i''$. It is not hard to show that this contracts all exceptional prime divisors $E''$ with $a(E'';X)=-1$ (see \cite[Step 4 in Proof of Theorem 4.14]{HashizumeSingularityArbitraryPair}). Therefore, $(G_i)_{W_{\underline{t}}}$, the birational transformation of $G_i''$, are all the exceptional prime divisors on $W_{\underline{t}}$. By \cite[Theorem 5.1]{MengZhuangMMPlocallystablefamily}, there are only finitely many ample models $W_{\underline{t}}$ as we vary $t_i$'s inside a rational polytope. Since each $K_{W_{\underline{t}}} + \sum t_i (G_i)_{W_{\underline{t}}}$ is $\Q$-Cartier, when we vary the $t_i$'s in a fixed $P^\circ$ (the interior so that the ample model $W_{\underline{t}}$ is constant for all $\underline{t}\in P^\circ$), we see that each $(G_i)_{W_{\underline{t}}}$ is $\Q$-Cartier. We also note that the pair $(W_{\underline{t}}, \sum t_i (G_i)_{W_{\underline{t}}})$ is log canonical. Thus, for $1-\epsilon <t_i$ for all $i$, the pair $(W_{\underline{t}}, (1-\epsilon) \sum (G_i)_{W_{\underline{t}}})$ is also log canonical, note that  $K_{W_{\underline{t}}}+ (1-\epsilon)\sum (G_i)_{W_{\underline{t}}}$ is $\Q$-Cartier (because $K_{W_{\underline{t}}} + \sum t_i (G_i)_{W_{\underline{t}}}$ and each $(G_i)_{W_{\underline{t}}}$ is $\Q$-Cartier). Letting each $t_i\to 1$, we have that $\text{lct}(W_{\underline{t}}, \sum (G_i)_{W_{\underline{t}}})\geq 1-\epsilon$,  and approaching to $1$ as $\min \{t_i\} \to 1$. By the ACC for log canonical thresholds \cite[Theorem 1.1]{HMXACC}, we can find a $\underline{t}$ so that $\text{lct}(W_{\underline{t}}, \sum (G_i)_{W_{\underline{t}}})=1$, i.e., the pair $(W_{\underline{t}}, \sum (G_i)_{W_{\underline{t}}})$ is log canonical. 

Finally, set $W\coloneqq W_{\underline{t}}$ chosen as above with $h \colon W\to X$, and let $E_i \coloneqq (G_i)_{W_{\underline{t}}}$ (these are all the $h$-exceptional prime divisors on $W$). We have already seen that $a(E_i;X)<-1$, that each $E_i$ is $\Q$-Cartier, and that the pair $(W, \red{E}=\sum E_i)$ is log canonical. Moreover, we know that $K_W + \sum t_iE_i$ is $h$-ample. It follows that $K_W - h^*_{\num}K_X + \sum t_iE_i$ is $h$-exceptional, in particular $\Q$-Cartier and thus $h$-ample. Hence there exists an $h$-ample exceptional Cartier divisor on $W$. This completes the proof of the theorem.
\end{proof}

\begin{remark} In \cite[Theorem 1.1]{HashizumeSingularityArbitraryPair}, it was shown that for any pair $(X,\Delta)$, where $X$ is essentially of finite type over a field of characteristic zero and the coefficients of $\Delta$ are $\leq 1$, there exists a projective birational morphism $h \colon W\to X$ such that: $a(E_i;X,\Delta)<-1$\footnote{Here, $a(E_i;X,\Delta)\coloneqq\sup\{a(E_i; X, \Delta+\Delta') \mid \Delta'\geq 0 \text{ and } K_X+\Delta+\Delta' \text{ is $\Q$-Cartier}\}$.} for any $h$-exceptional prime divisor $E_i$, $\red{E}=\sum E_i$ is $\Q$-Cartier, and $(W,h_*^{-1}\Delta+\red{E})$ is log canonical. The same result holds for $\mathbb{R}$-coefficients.

Theorem~\ref{thm: Hashizume} slightly improves this result by ensuring that each $E_i$ is $\Q$-Cartier, assuming $X$ is numerically $\Q$-Gorenstein. In fact, the proof above can be generalized (adapting \cite[Step 3 and Step 4 in Proof of Theorem 4.14]{HashizumeSingularityArbitraryPair}) to obtain a full generalization of \cite[Theorem 1.1]{HashizumeSingularityArbitraryPair}, also for $\mathbb{R}$-coefficients. Specifically, for any pair $(X,\Delta)$, where $X$ is essentially of finite type over a field of characteristic zero and the coefficients of $\Delta$ are $\leq 1$, there exists a projective birational morphism $h \colon W\to X$ such that: $a(E_i;X,\Delta)<-1$ for any $h$-exceptional prime divisor $E_i$, each $E_i$ is $\Q$-Cartier, and $(W,h_*^{-1}\Delta+\red{E})$ is log canonical.  We leave the details to the interested reader.
\end{remark}
\begin{remark}
 The morphism $W\to X$ constructed in Theorem~\ref{thm: Hashizume} (and also in \cite[Theorem 1.1]{HashizumeSingularityArbitraryPair}) is not necessarily a log canonical modification of $X$. In \cite[Theorem 3.6]{ZhangIsolatedSingularitiesNonInvertibleEndo}, the existence of log canonical modification $Y\to X$ was established under the assumption that $X$ is a normal integral numerically $\Q$-Gorenstein scheme essentially of finite type over a field of characteristic zero (this generalizes Theorem~\ref{thm: Odaka-Xu}). However, it is unclear whether $Y$ admits an ample exceptional Cartier divisor in this setup.
\end{remark}

\subsection{Asymptotic Riemann-Roch formula and its consequence} In this subsection, we recall a well-known version of the asymptotic Riemann--Roch theorem. We provide a proof, as we have been unable to find a reference that covers the generality we require. We will follow notation of \cite[Chapter VI.2]{KollarBookRationalCurves}, and more generally \cite{FultonBookIntersection}.

\begin{theorem}
\label{thm: asymptotic RR}
Let $X$ be a projective scheme over an Artinian local ring with $\dim X=d$ and let $L$ be a line bundle on $X$.  
\begin{enumerate}
    \item If $X$ is equidimensional and $(\textnormal{S}_1)$, then we have 
    \[
\chi (X, L^{\otimes n}) = \frac{L^d}{d!} n^d - \frac{([\omega_X]-[\sO_X])\cdot L^{d-1}}{2(d-1)!}n^{d-1} + O (n^{d-2})
\]
where $[\omega_X]-[\sO_X]$ is the class in $K_{d-1}(X)$ following \cite[VI, Definition 2.1]{KollarBookRationalCurves}.
\item If $X$ is $(\textnormal{S}_2)$ and $(\textnormal{G}_1)$ and $D$ is a Weil divisor that is principal in codimension one, then we have
\[
\chi (X, \sO_X(D)\otimes L^{\otimes n}) = \frac{L^d}{d!} n^d - \frac{(K_X-2D)\cdot L^{d-1}}{2(d-1)!}n^{d-1} + O (n^{d-2}).
\]
\end{enumerate}
\end{theorem}
\begin{proof}
We will prove both (1) and (2) simultaneously. Let $D$ be a Weil divisor on $X$ that is principal in codimension one (in case (1), we simply take $D=0$). By \cite[VI, Theorem 2.13]{KollarBookRationalCurves}, for any coherent sheaf $G$ on $X$ (and thus for any $G\in D^b_{coh}(X)$), $\chi (X, G\otimes L^{\otimes n})$ is a polynomial in $n$ of degree at most the dimension of the support of $G$. Thus we can write 
\begin{equation}
\label{eqn: eqn in proof of asym RR}
\chi (X, \sO_X(D)\otimes L^{\otimes n}) = A n^d +Bn^{d-1} + O (n^{d-2}). 
\end{equation}
It is well-known (see \cite[VI, Corollary 2.14]{KollarBookRationalCurves}) that $A=L^d/d!$. We will compute $B$.

Let $\omega_X^\bullet$ be the normalized dualizing complex of $X$. Using Grothendieck duality, we have 
\begin{align*}
\chi(X, R\underline{\Hom}(\sO_X(D), \omega_X^\bullet\otimes L^{\otimes n})) & = \chi(X,\sO_X(D)\otimes L^{\otimes -n}).
\end{align*}
We now truncate $\omega_X^\bullet$. Consider now the exact triangle
$$R\underline{\Hom}(\sO_X(D),\omega_X[d]\otimes L^{\otimes n}) \to R\underline{\Hom}(\sO_X(D),\omega_X^\bullet\otimes L^{\otimes n})\xrightarrow{} R\underline{\Hom}(\sO_X(D),\tau^{>-d}\omega_X^\bullet\otimes L^{\otimes n})\xrightarrow{+1} .$$
Since $X$ is $(\textnormal{S}_1)$, we know that $\tau^{>-d}\omega_X^\bullet$ is supported in codimension at least two. As a result,  
$$\chi(X, R\underline{\Hom}(\sO_X(D),\tau^{>-d}\omega_X^\bullet\otimes L^{\otimes n}))= \chi(X, R\underline{\Hom}(\sO_X(D),\tau^{>-d}\omega_X^\bullet)\otimes L^{\otimes n}) = O(n^{d-2}).$$
Therefore, from (\ref{eqn: eqn in proof of asym RR}) and the duality
we compute that 
\begin{equation}\label{eqn: second eqn in proof of asym RR}
\begin{split}
\chi (X, R\underline{\Hom}(\sO_X(D),\omega_X[d]\otimes L^{\otimes n}))
&= \chi(X, R\underline{\Hom}(\sO_X(D), \omega_X^\bullet\otimes L^{\otimes n})) + O(n^{d-2})  
\\ &= (-1)^d  \left( An^d - Bn^{d-1} + O(n^{d-2}) \right). 
\end{split}
\end{equation}

We now analyze the cases. In case (1), we have $\sO_X(D)=\sO_X$, and thus by (\ref{eqn: eqn in proof of asym RR}) and (\ref{eqn: second eqn in proof of asym RR}) we compute that 
\begin{align*}
\chi (X, \sO_X\otimes L^{\otimes n}) - \chi(X,\omega_X\otimes L^{\otimes n})
&= \chi (X, \sO_X\otimes L^{\otimes n}) - (-1)^d \chi(X,\omega_X[d] \otimes L^{\otimes n})
\\ &= 2Bn^{d-1}+ O(n^{d-2}).
\end{align*}
It thus follows from \cite[VI, Theorem 2.13]{KollarBookRationalCurves} that 
$$B=- \frac{([\omega_X]-[\sO_X])\cdot L^{d-1}}{2(d-1)!}.$$

Finally, in case (2), since $D$ is principal in codimension one, $R\underline{\Hom}(\sO_X(D),\omega_X[d]\otimes L^{\otimes n})$ agrees with $\omega_X[d]\otimes \sO_X(-D)\otimes L^{\otimes n}$ in codimension one. It follows from (\ref{eqn: second eqn in proof of asym RR}) that 
\begin{align*}
\chi(X,\sO_X(K_X-D)\otimes L^{\otimes n}) & =  (-1)^d \chi(X,\omega_X[d]\otimes \sO_X(-D)\otimes L^{\otimes n}) \\
& = An^d - Bn^{d-1} + O(n^{d-2}).
\end{align*}
After combining this formuma with (\ref{eqn: eqn in proof of asym RR}), we immediately express 
\[
\chi (X, \sO_X(D)\otimes L^{\otimes n}) - \chi(X,\sO_X(K_X-D)\otimes L^{\otimes n}) = 2Bn^{d-1}+ O(n^{d-2}).
\]
On the other hand, we also have the formula
$$\chi (X, \sO_X(D)\otimes L^{\otimes n}) - \chi(X,\sO_X(K_X-D)\otimes L^{\otimes n}) = \frac{(2D-K_X)\cdot L^{d-1}}{(d-1)!} n^{d-1} + O(n^{d-2}), $$
because for any effective divisor $E$ one has that 
$$\chi(X,L^n)-\chi(X, \sO_X(-E)\otimes L^{\otimes n}) = \chi(E,L|_E^{\otimes n}) = \frac{E\cdot L^{d-1}}{(d-1)!} n^{d-1} + O(n^{d-2}).$$
By comparing the two computations of
$\chi (X, \sO_X(D)\otimes L^{\otimes n}) - \chi(X,\sO_X(K_X-D)\otimes L^{\otimes n})$ we immediately get the desired formula
$$B=- \frac{(K_X-2D)\cdot L^{d-1}}{2(d-1)!}.$$ 
This completes the proof.
\end{proof}

We will need the following consequence of Theorem~\ref{thm: asymptotic RR} on the asymptotic behavior of colengths of certain graded family of ideals. Note that we do not assume that $R$ is normal.

\begin{proposition}
\label{prop: consequence asymptotic RR}
Let $(R, \mf m)$ be an excellent local domain of dimension $d\geq 2$ admitting a dualizing complex, and let $f \colon Y \to X\coloneqq \Spec(R)$ be a projective birational map such that $Y$ is $(\textnormal{S}_2)$ and $(\textnormal{G}_1)$, and $f_*\sO_Y=\sO_X$. Suppose $f$ is an isomorphism on the punctured spectrum, and there exists an effective Cartier divisor $E$ such that $-E$ is $f$-ample and $f$-exceptional. Let $F$ be an $f$-exceptional Weil divisor on $Y$ and set $I_n \coloneqq  \Gamma(Y, \sO_Y (-nE-F))\subseteq R$.
Then we have
\[
\length (R/I_{n}) = -\frac{(-E)^d}{d!} n^d + \frac{(K_Y+2F) \cdot (-E)^{d-1}}{2(d-1)!}n^{d-1} + O(n^{d-2}).
\]
\end{proposition}
\begin{proof}
First note that, since $-E$ is $f$-ample and $f$-exceptional, $Y$ is the blowup of $X$ along some $\m$-primary ideal $I$ (for example, one can take $I=\Gamma(Y,\sO_Y(-mE))$ for $m$ sufficiently divisible). Taking cohomology of the short exact sequence 
$$0\to \sO_Y(-nE-F) \to \sO_Y \to  \sO_{nE+F} \to 0$$
we obtain that 
\[
\length(R/I_n) = \chi(\sO_{nE+F}) + \sum_{i=1}^{d-1} (-1)^i\big(h^i(Y, \sO_Y(-nE-F)) - h^i(Y, \sO_Y)\big).
\]
Since $-E$ is $f$-ample, we have $H^i(Y, \sO_Y(-nE-F))\cong R^if_*\sO_Y(-nE-F)=0$ for all $i\geq 1$ and $n\gg 0$. Thus, we have asymptotic comparison $\length(R/I_n) = \chi(\sO_{nE+F}) + O(n^{d-2}).$

We now proceed to calculate $\chi(\sO_{nE+F})$. 
From the short exact sequence
$$0\to \sO_F(-nE) \to \sO_{nE+F} \to  \sO_{nE} \to 0,$$
we get that 
$\chi(\sO_{nE+F}) = \chi(\sO_{nE}) + \chi(F, \sO_F(-nE)).$
By \cite[VI, Corollary 2.14]{KollarBookRationalCurves} 
we compute $\chi(F, \sO_F(-nE)) = \frac{(-E|_F)^{d-1}}{(d-1)!}n^{d-1} + O(n^{d-2})$. Therefore it suffices to show that 
\begin{equation}\label{eqn: eqn 5.2.2 to do}
\chi(\sO_{nE})= -\frac{(-E)^d}{d!} n^d + \frac{K_Y \cdot (-E)^{d-1}}{2(d-1)!}n^{d-1} + O(n^{d-2}).    
\end{equation}

Now note that for every $k\geq 1$ there is a short exact sequence 
$$0 \to \sO_E(-kE) \to \sO_{(k+1)E} \to \sO_{kE}\to 0.$$
Taking Euler characteristic gives
$\chi(\sO_{(k+1)E}) = \chi(\sO_{kE}) + \chi(E, \sO_E(-kE))$. 
By induction on $k$ formula (\ref{eqn: eqn 5.2.2 to do}) will 
follow once we prove the formula 
\[
\chi(E, \sO_E(-kE)) = -\frac{(-E)^d}{(d-1)!} k^{d-1} + \frac{K_Y \cdot (-E)^{d-1} - (-E)^d}{2(d-2)!}k^{d-2} + O(k^{d-3}).
\]

On the other hand, since $Y$ satisfies $(\textnormal{S}_2)$, $E$ is equidimensional and $(\textnormal{S}_1)$, allowing us to apply the asymptotic Riemann--Roch formula of Theorem~\ref{thm: asymptotic RR}. This yields 
\[
\chi(E, \sO_E(-kE)) = \left(\frac{(-E|_E)^{d-1}}{(d-1)!} k^{d-1} - \frac{([\omega_E]-[\sO_E]) \cdot (-E|_E)^{d-2}}{2(d-2)!}k^{d-2} + O(k^{d-3})\right).
\]
We now analyze the term $([\omega_E]-[\sO_E]) \cdot (-E|_E)^{d-2}$.
To do so, consider the exact sequence
$$0\to \omega_Y \to \omega_Y(E) \to \omega_E \to h^{-(d-1)}(\omega_Y^\bullet). $$
Since $Y$ satisfies Serre's $(\textnormal{S}_2)$ condition, $h^{-(d-1)}(\omega_Y^\bullet)$ has codimension at least three in $Y$. Therefore, its restriction to $E$  has codimension at least two in $E$. It follows that 
$$([\omega_Y(E)] - [\sO_Y])|_E = [\omega_E] - [\sO_E] \text{ in $K_{d-2}(E)/K_{d-3}(E)$}.$$
Therefore, essentially by the definition of intersection theory, we have 
\begin{align*}
([\omega_E] - [\sO_E]) \cdot (-E|_E)^{d-2} & = -([\omega_Y(E)] - [\sO_Y])\cdot (-E)^{d-1} \\
& = -(K_Y+E) \cdot (-E)^{d-1} \\
& = -K_Y \cdot (-E)^{d-1} + (-E)^d.
\end{align*}
Substituting back, we obtain (\ref{eqn: eqn 5.2.2 to do}) as desired. The proof is now complete.
\end{proof}

\begin{remark}
Note that the assumptions of Proposition~\ref{prop: consequence asymptotic RR} are satisfied if $R$ is normal and $Y$ is the normalized blowup of $R$ along an $\m$-primary ideal (this will be our main case of interest in the next subsection). We also note that, when $R$ is essentially of finite type over a field and $F$ is principal in codimension one, then one can also prove Proposition~\ref{prop: consequence asymptotic RR} by compactifying $X$ and applying Theorem~\ref{thm: asymptotic RR} (2). However, we do not know how to use this approach in a more general situation (e.g., the mixed characteristic case). 
\end{remark}

The following immediate consequence of Proposition~\ref{prop: consequence asymptotic RR} relates some MMP singularities with Hilbert coefficients of $\m$-primary ideals. 

\begin{corollary}
Let $(R, \mf m)$ be an excellent $\Q$-Gorenstein normal local domain of dimension $d\geq 2$ that admits a dualizing complex. 
Let $I$ be an arbitrary $\mf m$-primary ideal and let $\overline{\eh}_1(I)$ be the first normal Hilbert coefficient.
The following hold:
\begin{enumerate}
    \item If $R$ is canonical, then $(d-1)\eh(I) \geq  2\overline{\eh}_1(I),$
    \item if $R$ is terminal, then $(d-1)\eh(I) >  2\overline{\eh}_1(I)$.
\end{enumerate}
\end{corollary}
\begin{proof}
Let $f\colon Y\to X\coloneqq\Spec(R)$ be the normalized blowup of $I$. Then $I\sO_Y=\sO_Y(-E)$ is an $f$-ample and $f$-exceptional Cartier divisor and we have $\Gamma(Y, \sO_Y(-nE))=\overline{I^n}$. By the definition of normal Hilbert coefficient, we have 
\begin{align*}
\length(R/\overline{I^n}) & =\eh(I)\binom{n+d-1}{d} - \overline{\eh}_1(I)\binom{n+d-2}{d-1} + O(n^{d-2}) \\
& = \frac{\eh(I)}{d!}n^d + \frac{(d-1)\eh(I)-2\overline{\eh}_1(I)}{2(d-1)!}n^{d-1} + O(n^{d-2}).
\end{align*}
Therefore, Proposition~\ref{prop: consequence asymptotic RR} provides the formula:
\[
(d-1)\eh(I)-2\overline{\eh}_1(I) = K_Y\cdot (-E)^{d-1}= (K_Y-f^*K_X)\cdot (-E)^{d-1}.
\]
Now when $R$ is canonical (resp., terminal), we have $a(F; X) \geq 0$ (resp., $a(F; X) > 0$) for all exceptional prime divisors $F$ on $Y$, so $K_Y-f^*K_X \geq 0$ (resp., $K_Y-f^*K_X > 0$). 
Since $-E$ is $f$-ample, it follows immediately that 
$(K_Y-f^*K_X)\cdot (-E)^{d-1} \geq 0$ (resp., $> 0$). 
\end{proof}

The second application of Proposition~\ref{prop: consequence asymptotic RR} improves Proposition~\ref{prop: rational powers are quasi-polynomial} by providing a formula for the second term of the Hilbert--Samuel quasi-polynomial for rational powers of ideals.

\begin{corollary}\label{cor: Riemann-Roch for rational powers}
Let $(R,\m)$ be an excellent local domain of dimension $d\geq 2$ admitting a dualizing complex. Let $f$: $Y=\Bl_I(R)\to X\coloneqq\Spec(R)$ be the blowup of an $\m$-primary ideal $I\subseteq R$. Suppose $X$ and $Y$ are $(\textnormal{S}_2)$ and $(\textnormal{G}_1)$, and that the generic point of each exceptional prime divisor on $Y$ is regular. Let $I\sO_Y=\sO_Y(-E)$ and let $\rho$ be the Rees period of $I$. Then we have
\begin{align*}
\length (R/I^{\frac{k\rho + h}{\rho}}) 
&= -\frac{(-E)^d}{d!}k^d + \left( \frac{K_{Y} \cdot (-E)^{d-1} }{2(d-1)!} + \frac{\lceil \frac h \rho E \rceil \cdot (-E)^{d-1}}{(d-1)!} \right)k^{d-1} + O(k^{d-2}).
\end{align*}
\end{corollary}
\begin{proof}
First of all, we have $0\to \sO_X\to f_*\sO_Y \to C \to 0$. Since $I$ is $\m$-primary, $C$ has finite length. But as $R$ and $S$ are $(\textnormal{S}_2)$, $C=0$ and thus $f_*\sO_Y=\sO_X$. Let $E_1, \ldots, E_n$ be all the exceptional prime divisors of $f$ and 
write $E = \sum^n_{i = 1} c_i E_i$. Note that $\rho = \lcm \{c_1, \ldots, c_n\}$,
and thus if we let $F = \lceil \frac{h}{\rho} E \rceil$, then by Proposition~\ref{prop: geometric rational powers}
\[I^{\frac{k\rho + h}{\rho}} = \Gamma (Y, \sO_Y(-kE - F))\cap R = \Gamma (Y, \sO_Y(-kE - F)).\]
Thus we may apply Proposition~\ref{prop: consequence asymptotic RR}
to obtain that, for all $k\gg 0$, we have
\[
\length (R/I^{\frac{k\rho + h}{\rho}}) = -\frac{(-E)^d}{d!} k^d + \frac{(K_{Y}+2F) \cdot (-E)^{d-1}}{2(d-1)!}k^{d-1} + O(k^{d-2})
\]  
which is easily seen to be equivalent to what we want to show.
\end{proof}

\subsection{Lim-stable and log canonical singularities}
We now prove our main result: lim-stable (and in particular, semistable) singularities are log canonical in the $\Q$-Gorenstein case, see Corollary~\ref{cor: lim-stable implies log canonical Q-Gor}. We also present generalizations beyond the $\Q$-Gorenstein setting, which are discussed in Theorem~\ref{thm: lim-stable implies log canonical general}. We begin with the following lemma which is well-known to experts (for example, see \cite[Proof of Theorem 1.2, Step 2]{Odaka}). 

\begin{lemma}
\label{lem: Odaka discrepancy lemma}
Let $X$ be a  numerically $\Q$-Gorenstein normal integral scheme. Let $\pi\colon Y=\Bl_{\mathcal{I}}X\to X$ be the blowup along an ideal sheaf $\mathcal{I}\subseteq \sO_X$ locally of height at least two, and suppose $Y$ is normal with exceptional prime divisors $\{E_i\}_{i=1}^n$. Set $\mathcal{I}\sO_Y=\sO_Y(-\sum_{i=1}^nc_iE_i)$. For any positive integer $m$ such that $b_i=m/c_i$ is an integer for all $i$, let $\pi' \colon Y'\to X'\coloneqq X\times \mathbb{A}^1$ be the normalized blowup of $X'$ along the ideal sheaf $\mathcal{I}+(t^m)\subseteq \sO_X[t]$. Then 
there is a bijection between exceptional prime divisors on $Y'$ and on $Y$ such that for corresponding divisors $E'_i$ on $Y'$ and $E_i$ on $Y$ we have $a(E_i'; X')=b_i(a(E_i;X)+1)$. 

In particular, if $a(E_i;X)<-1$ for all $i$, then $a(E_i'; X') < 0$ for all $m$ sufficiently divisible.
\end{lemma}
\begin{proof}
The computation of the discrepancy $a(E_i;X)$ and $a(E_i';X')$ is local along each $E_i$. Thus after localizing at the generic point of $E_i$, we may assume that $Y=\Spec(V)$ is a DVR with uniformizer $s$, $E_i=\text{div}(s)$, and $\mathcal{I}\sO_Y=(s^{c_i})$. The map $\pi'$ then factors through $g\colon Y'\to Y\times \mathbb{A}^1\cong \Spec(V[t])$, which is the normalized blowup along the ideal $(s^{c_i}, t^m)$. Since $b_i=m/c_i$ is an integer, it is easy to see that this normalized blowup is isomorphic to the blowup along the ideal $(s, t^{b_i})$. This blowup only creates one exceptional prime divisor (corresponding to $E_i'$), and by a direct computation, we have $K_{Y'/Y\times\mathbb{A}^1}=b_iE_i'$ and $g^*E_i=b_iE_i' + \widetilde{E}_i$ where $\widetilde{E}_i$ is the strict transform of $E_i$. It follows that the coefficient of $E_i'$ in 
\begin{align*}
K_{Y'}-\pi'^*_{\num}K_{X'} & =  K_{Y'/Y\times\mathbb{A}^1} + g^*\big(K_{Y\times\mathbb{A}^1} - (\pi\times\text{id})^*_{\num}K_{X\times\mathbb{A}^1}\big) \\
& = K_{Y'/Y\times\mathbb{A}^1} + g^*\big(a(E_i;X)E_i\big)
\end{align*}
is equal to $b_i + b_ia(E_i;X)$, that is, $a(E_i'; X')=b_i(a(E_i;X)+1)$ as claimed.
\end{proof}

We next record the following consequence of Proposition~\ref{prop: exponential bound}. 

\begin{theorem}
\label{thm: criterion for lim unstable via derivative}
Let $(R, \mf m)$ be a local ring of dimension $d \geq 1$.
Suppose $\{I_n\}_{n}$ is a Noetherian graded family of $\m$-primary ideals such that $I_1 = \m$. Set $I_0 = R$. 
Then the Hilbert series of $\Gr_{I_\bullet}(R)$ can be written as
\[
h(t) \coloneqq \sum_{n \geq 0} \length (I_n/I_{n+1})t^n = \frac{f(t)}{(1-t)^{d}},
\]
where $f(t)$ is a rational function with $f(1)\neq 0$ and all poles of $f(t)$ have order at most $d$. Furthermore, if $f'(1) > \frac{d}{2} f(1)$, then we have $\limlm(R) > 1$. 
\end{theorem}
\begin{proof}
Our assumptions imply that $\Gr_{I_\bullet}(R)$ is a finitely generated $\mathbb{N}$-graded ring over $k$. Hence there exist $z_1,\dots,z_d$ with $\deg(z_i)=a_i$ such that $\Gr_{I_\bullet}(R)$ is finite over $k[z_1,\dots,z_d]$. It follows that we can write
\[
h(t)= \frac{p(t)}{(1-t^{a_1})(1-t^{a_2})\cdots (1-t^{a_d})}
\]
where $p(t)\in \mathbb{Z}[t]$ and the first claim follows. 

We now prove the second part. From the expression of $h(t)$, it is easy to check that
$$\length(R/I_n)=\frac{f(1)}{d!}n^d + O(n^{d-1}).$$ 
Thus, applying Proposition~\ref{prop: exponential bound} to $\{I_n\}_n$ yields, for any $x \in \mathbb{R}_{> 0}$, the bound 
\begin{equation}\label{eq: expo bound}
        \limlm (R) \geq \frac{f(1)(1 - e^{-x})^d}{x^d f(e^{-x})}.
    \end{equation}
    Now, if we let $x \to 0$ and use that 
    $f(e^{-x}) = f(1) - f'(1) x + O(x^2)$, 
    then (\ref{eq: expo bound}) becomes 
    \[
        \limlm(R) \geq f(1) \cdot  \frac{1 - \frac{d}{2}x + O(x^2)}{f(1) - f'(1) x  + O(x^2)} 
        =  1+ \left(\frac{f'(1)}{f(1)} -\frac{d}{2}\right)x + O(x^2),
    \] 
    and the conclusion follows. 
\end{proof}

As a corollary of Theorem~\ref{thm: criterion for lim unstable via derivative}, we generalize \cite[Theorem~3]{Shah} 
and partially recover Proposition~\ref{prop: Shah associated graded is CM and reduced}.

\begin{corollary}\label{cor: lim stable embdim}
Let $(R,\m)$ be a Cohen--Macaulay ring of dimension two. If $R$ is lim-stable, then either $R$ has minimal multiplicity or $R$ has almost minimal multiplicity and $\Gr_\m (R)$ is Cohen-Macaulay.
\end{corollary}
\begin{proof}
    By Theorem~\ref{thm: criterion for lim unstable via derivative} applied to the graded family $\{\m^n\}_n$, 
    $R$ is not lim-stable whenever $\eh_1 (R) = f'(1) > f(1) = \eh(R)$. 
    By Northcott's inequality, $\eh_1 (R) \geq \eh(R) - 1$ and it was shown by \cite[Theorem 2.1]{Huneke} that the equality holds if and only if $R$ has minimal multiplicity. 
    Essentially by \cite[Theorem~3.3]{SallySuper} (the statement gives the Hilbert function, but the base case concerns only $\eh_1$, see also \cite[Exercise~4.4]{Rossi}) $\eh_1(R) = \eh(R)$ if and only if $R$ has almost minimal multiplicity and the associated graded ring $\gr_\m(R)$ is Cohen-Macaulay. 
\end{proof}

We are now ready to state and prove the main technical result.

\begin{theorem}
\label{thm: lim-stable main technical result}
Let $(R,\m)$ be an excellent numerically $\Q$-Gorenstein normal local domain admitting a dualizing complex. Suppose there exists a projective birational map $\pi$: $Y\to X\coloneqq\Spec(R)$ such that
\begin{itemize}
    \item there exists a $\pi$-ample $\pi$-exceptional Cartier divisor, and
    \item $a(F;X)<-1$ for all exceptional prime divisors $F$ on $Y$.
\end{itemize}
If $R$ is lim-stable, then $R$ is $\Q$-Gorenstein and log canonical. 
\end{theorem}
\begin{proof}
Let $-E$ be a $\pi$-ample and $\pi$-exceptional Cartier divisor on $Y$. By the negativity lemma \cite[Lemma 3.39]{KollarMori} we know that $E$ is effective. If a curve $C$ is contracted under $\pi$, then $C\cdot (-E) >0$. This implies that $C$ lies in the support of $E$, and consequently, the exceptional set is of pure codimension one. As a result $\pi$ is an isomorphism over the locus where $X$ is $\Q$-Gorenstein and log canonical. This is because none of the exceptional prime divisors can have centers on this locus, given the assumption on the discrepancy.

By Theorem~\ref{thm: num Q-Gor lc is Q-Gor lc}, it suffices to show that $R$ is pseudo-log canonical. If this is not the case, we may choose a prime ideal $Q\in \Spec(R)$ of minimal height such that $R_Q$ is not pseudo-log canonical. By Theorem~\ref{thm: localization}, we can replace $R$ by $R_Q$ and apply Theorem~\ref{thm: num Q-Gor lc is Q-Gor lc} again, which allows us to assume that $R$ is $\Q$-Gorenstein and log canonical on the punctured spectrum. After this replacement, both assumptions are preserved and $\pi$ is an isomorphism over the punctured spectrum of $R$. It follows that $\pi$ is the blowup of some $\m$-primary ideal $I\subseteq R$. 

We now consider the graded family of ideals $I_n \coloneqq \{I^{\frac{n}{\rho}}\}_n$, 
where $\rho$ is the Rees period of $I$. This is a Noetherian graded family of $\m$-primary ideals with $I_1=\m$ by the definition of rational powers. By Theorem~\ref{thm: criterion for lim unstable via derivative}, we thus have
$$h(t) \coloneqq \sum_{j\geq 0} \length(I_j/I_{j+1})t^j = \frac{f(t)}{(1-t)^d}$$
where $f(t)$ is a rational function with $f(1)\neq 0$ and all poles of $f(t)$ have order at most $d\coloneqq\dim(R)$. 
Thus $\length(R/I_n)= ({f(1)}/{d!})n^d + O(n^{d-1})$ and 
it follows from 
Proposition~\ref{prop: rational powers are quasi-polynomial} that
$f(1) = \eh(I)/\rho^d$. 

We next show that $f'(1)>\frac{d}{2}f(1)$. We will deduce this by comparing the multiplicity and the colength of $R[[T]]$-ideals  $\overline{(I+T^m)^n}$, where $m$ is a sufficiently divisible and $n$ a sufficiently large positive integer. Note that by Corollary~\ref{cor: integral closures of powers of I+T^m}, we know that 
\begin{align*}
  \frac{\eh(\overline{(I+T^m)^n})}{(d+1)!\length(R[[T]]/\overline{(I+T^m)^n})} =  \frac{\eh(I) mn^{d+1}}{(d+1)!\frac{m}{\rho}\sum_{j=1}^{n\rho}\length(R/I_{j})} 
  = \frac{(n\rho)^{d+1} f(1)}{{(d+1)!}\sum_{j=1}^{n\rho}\length(R/I_{j})}.
\end{align*}

At this point, we observe that $\length(R/I_j)$ is the coefficient of $t^{j-1}$ in $h(t)/(1-t)$, and that $\sum_{j=1}^n\length(R/I_j)$ is the coefficient of $t^{n-1}$ in $h(t)/(1-t)^2=f(t)/(1-t)^{d+2}$. We can write 
$$\frac{f(t)}{(1-t)^{d+2}}=\frac{f(1)}{(1-t)^{d+2}} -\frac{f'(1)}{(1-t)^{d+1}} + g(t)$$
where all poles of $g(t)$ have order at most $d$. We expand the first two terms above and obtain that the coefficient of $t^{n-1}$ in $f(t)/(1-t)^{d+2}$ is
$$f(1)\binom{d+n}{d+1} - f'(1)\binom{d+n-1}{d} = \frac{f(1)}{(d+1)!}n^{d+1} + \frac{\frac{d}{2}f(1)-f'(1)}{d!} n^d+ O(n^{d-1}).$$
It follows that for $n\gg0$, 
\begin{align}
\label{equation 1 in lim-stable main result}
    \frac{\eh(\overline{(I+T^m)^n})}{(d+1)!\length(R[[T]]/\overline{(I+T^m)^n})} & = 
    \frac{(n\rho)^{d+1} f(1)}{(n\rho)^{d+1}f(1)+ (d+1)\cdot (\frac{d}{2}f(1)-f'(1))\cdot (n\rho)^d + O(n^{d-1})} \notag \\
    & =  \frac{1}{1+ (d+1)(\frac{d}{2} -\frac{f'(1)}{f(1)})\cdot \frac{1}{n\rho} + O(\frac{1}{n^2})}.
\end{align}

On the other hand, we know that $a(F;X)<-1$ for all exceptional prime divisors $F$ on $Y$ by assumption. Let $\pi'\colon Y'\to X'\coloneqq \Spec(R[[T]])$ be the normalized blowup of $I+T^m$
(recall that $m$ is sufficiently divisible). By Lemma~\ref{lem: Odaka discrepancy lemma}, $a(F', X')<0$ for all exceptional prime divisors $F'$ on $Y'$. Moreover, it is clear that $\pi'$ is an isomorphism on the punctured spectrum of $R[[T]]$ and $(I+T^m)\sO_{Y'}=\sO_{Y'}(-E')$. In particular $-E'$ is a $\pi'$-ample $\pi'$-exceptional Cartier divisor on $Y'$. We now apply Proposition~\ref{prop: consequence asymptotic RR} to $Y'\to X'$ and note that $\Gamma(Y',\sO_{Y'}(-nE'))=\overline{(I+T^m)^n}$. Thus we obtain that  
\begin{align*}
\length(R[[T]]/\overline{(I+T^m)^n}) & =-\frac{(-E')^{d+1}}{(d+1)!}n^{d+1}+ \frac{K_{Y'}\cdot (-E')^{d}}{2d!}n^{d}+O(n^{d-1}) \\
& =-\frac{(-E')^{d+1}}{(d+1)!}n^{d+1}+ \frac{(K_{Y'}-\pi'^*_{\num}K_{X'})\cdot (-E')^{d}}{2d!}n^{d}+O(n^{d-1}).
\end{align*}
Since $\{\overline{(I+T^m)^n}\}_n$ is a Noetherian graded family, it follows easily from the expression above that $\eh(\overline{(I+T^m)^n})=-(-E')^{d+1}n^{d+1}$. 
Combining the two expressions yields the formula:
\begin{align}
\label{equation 2 in lim-stable main result}
\frac{\eh(\overline{(I+T^m)^n})}{(d+1)!\length(R[[T]]/\overline{(I+T^m)^n})} & = \frac{-(-E')^{d+1}n^{d+1}}{-(-E'^{d+1})n^{d+1}+ \frac{(d+1)}{2}(K_{Y'}-\pi'^*_{\num}K_{X'})(-E')^dn^d+ O(n^{d-1})} \notag\\   
& = \frac{1}{1+ \frac{d+1}{-(-E')^{d+1}}\cdot \frac{(K_{Y'}-\pi'^*_{\num}K_{X'})(-E')^d}{2}\cdot \frac{1}{n} + O(\frac{1}{n^2})}.
\end{align}
By comparing the coefficients of $1/n$ in the denominators of (\ref{equation 1 in lim-stable main result}) and (\ref{equation 2 in lim-stable main result}), we see that
$$\frac{d}{2} -\frac{f'(1)}{f(1)} = \frac{\rho}{-(-E')^{d+1}}\cdot \frac{(K_{Y'}-\pi'^*_{\num}K_{X'})(-E')^d}{2}.$$
Note that by definition, 
\[
(K_{Y'}-\pi'^*_{\num}K_{X'})\cdot (-E')^{d}
= \sum_{F'} a(F'; X')F' \cdot (-E')^{d}
= \sum_{F'} a(F'; X')  (-E'|_{F'})^{d}.
\]
Since $-E'|_{F'}$ is ample on $F'$ and $a(F'; X') < 0$, 
we have that $(K_{Y'}-\pi'^*_{\num}K_{X'})\cdot (-E')^{d} < 0.$
It follows that $f'(1)>\frac{d}{2}f(1)$ and tbus $R$ is not lim-stable by Theorem~\ref{thm: criterion for lim unstable via derivative}, a contradiction. 
\end{proof}

The following is an immediate consequence of Theorem~\ref{thm: lim-stable main technical result} and the existence of log canonical modification in the $\Q$-Gorenstein setting (Theorem~\ref{thm: Odaka-Xu}).

\begin{corollary}
\label{cor: lim-stable implies log canonical Q-Gor}
Let $(R,\m)$ be a $\Q$-Gorenstein normal local ring essentially of finite type over a field of characteristic zero. If $R$ is lim-stable (e.g., $R$ is semistable), then $R$ is log canonical.
\end{corollary}
\begin{proof}
This follows from Theorem~\ref{thm: lim-stable main technical result} and Theorem~\ref{thm: Odaka-Xu}: if $f$: $Y\to X$ is the log canonical modification, then $K_Y+\red{E}-f^*K_X$ is $f$-ample $f$-exceptional and $a(F;X)<-1$ for all exceptional prime divisors $F$ on $Y$ by the negativity lemma \cite[Lemma 3.39]{KollarMori} so that Theorem~\ref{thm: lim-stable main technical result} can be applied. 
\end{proof}

Beyond the $\Q$-Gorenstein setting, we have the following extension of Corollary~\ref{cor: lim-stable implies log canonical Q-Gor}. This is our main result on lim-stable singularities.

\begin{theorem}
\label{thm: lim-stable implies log canonical general}
Let $(R,\m)$ be an excellent normal local domain admitting a dualizing complex. Suppose one of the following holds: 
\begin{enumerate}
    \item $\dim(R)\leq 2$, or
    \item $R$ is essentially of finite type over a field of characteristic zero that is numerically $\Q$-Gorenstein.
\end{enumerate}
If $R$ is lim-stable, then $R$ is $\Q$-Gorenstein and log canonical. 
\end{theorem}
\begin{proof}
We just need to verify that the hypotheses of Theorem~\ref{thm: lim-stable main technical result} are satisfied. In case $(2)$, this  follows immediately from Theorem~\ref{thm: Hashizume}.  

In case (1), we let $f\colon Y\to X\coloneqq\Spec(R)$ be the log canonical modification (which exists by the full log MMP in dimension two, see \cite{TanakaMMPexcellentsurface}). Note that both $X$ and $Y$ are numerically $\Q$-Gorenstein by \cite[Example 3.4]{TakagiFinitisticTestIdeal}. We need to show the existence of an ample exceptional Cartier divisor on $Y$. Let $g\colon Z\to Y$ be a log resolution that is an isomorphism over $Y\backslash\red{E}$ (the latter is isomorphic to $X\backslash\{\m\}$ which is regular). We will show that $Y$ is numerically log terminal in the sense that $a(F;Y)>-1$ for all $g$-exceptional prime divisors $F$ on $Z$. 

To prove the claim, we start by writing 
$$K_Z+g_*^{-1}\red{E}-g^*(K_Y+\red{E})=\sum a_iF_i,$$
where $a_i\geq -1$ for all $i$. This is possible since $(Y,\red{E})$ is log canonical. We also write 
$K_Z-g_{\num}^*K_Y=\sum b_iF_i.$
It follows that $\sum(a_i-b_i)F_i$ is $g$-numerically equivalent to $g_*^{-1}\red{E}$. In particular, $\sum(a_i-b_i)F_i$ is $g$-nef. By the negativity lemma \cite[Lemma 3.39]{KollarMori}, we have $a_i-b_i\leq 0$ for all $i$. Now suppose that $a_i=b_i$ for some $i$ and let $y \coloneqq g(F_i)$. Note that $y\in \red{E}$ since $Z\to Y$ is an isomorphism over $Y\backslash\red{E}$. We have 
$$0\leq F_i \cdot \sum_j(a_j-b_j)F_j= \sum_{j\neq i}(a_j-b_j)(F_i\cdot F_j)\leq 0,$$
so we must have equalities throughout. Hence $a_j=b_j$ for all $j$ such that $F_i\cap F_j\neq \emptyset$. Since the exceptional set $f^{-1}(y)$ is connected, repeating this procedure we obtain that $a_i=b_i$ for all $i$ such that $F_i\in f^{-1}(y)$. But then 
$$0 = f^{-1}(y) \cdot \sum(a_i-b_i)F_i = f^{-1}(y) \cdot g_*^{-1}\red{E} $$
which is a contradiction to $y\in \red{E}$. Therefore we must have $b_i>a_i\geq -1$ for all $i$, i.e., $a(F;Y)>-1$ for all $g$-exceptional prime divisors $F$. This implies $\lceil K_Z-g_{\num}^*K_Y\rceil$ is effective and thus $g_*\sO_Z(\lceil K_Z-g_{\num}^*K_Y\rceil)=\sO_Y$. Since $g_{\num}^*K_Y$ is $g$-numerically trivial by definition, by relative vanishing in dimension two (for example, see \cite[Theorem 3.4]{TanakaMMPexcellentsurface}), we have that $R^1g_*\sO_Z(\lceil K_Z-g_{\num}^*K_Y\rceil)=0$. Thus the composition map 
$$\sO_Y\to Rg_*\sO_Z\to Rg_*\sO_Z(\lceil K_Z-g_{\num}^*K_Y\rceil)\cong \sO_Y$$
is the identity. In particular, $\sO_Y\to Rg_*\sO_Z$ splits in $D(Y)$ and by duality $Y$ has pseudo-rational singularities. This in turn implies that $Y$ is $\Q$-factorial by \cite[Proposition 17.1]{LipmanRationalSingularity}. It follows that $K_Y-f^*_{\num}K_X+\red{E}$ is an $f$-ample $f$-exceptional $\Q$-Cartier divisor, and that $a(G;X)<-1$ for all exceptional prime divisors $G$ on $Y$ by the negativity lemma again. Thus all hypotheses of Theorem~\ref{thm: lim-stable main technical result} are satisfied and the result follows.
\end{proof}

\begin{remark}
\label{rmk: more lim-stable implies log canonical}
There are some other situations in which Theorem~\ref{thm: lim-stable main technical result} can be applied in low dimensions. Specifically, if $(R,\m)$ is essentially of finite type over a perfect field of characteristic $p>5$, and $R$ is normal and $\Q$-Gorenstein of dimension three, then we claim that log canonical modification of $X\coloneqq\Spec(R)$ exists. Hence  Corollary~\ref{cor: lim-stable implies log canonical Q-Gor} still holds. 

To see the claim, note that by \cite[Corollary 9.21]{BMPSTWW1}, which builds on \cite{HX}, \cite{BirkarExistenceFlipsThreefolds}, we can find a $\mathbb{Q}$-factorial dlt modification $f \colon Z \to X$, i.e.,
$(Z, \red{E})$ is dlt and $K_{Z}+\red{E}$ is $f$-nef. By \cite[Theorem~1.4 (2)]{HNT} (see also \cite[Theorem~1.1]{Waldron}),
$K_{Z}+\red{E}$ is $f$-semi-ample.
Thus the contraction over $X$ induced by $K_{Z}+\red{E}$ is the log canonical modification $Y$ of $X$. In fact, provided that the analog of \cite[Theorem 1.1]{Waldron} holds, this argument applies more generally to any three-dimensional excellent local ring $(R,\m)$ with $\charac(R/\m)>5$ such that $R$ is normal, $\Q$-Gorenstein, and admits a dualizing complex. It is our understanding that the proof in \cite{Waldron} goes through by replacing the use of positive characteristic MMP by \cite{BMPSTWW1} and using the gluing argument of \cite[Theorem 3.1]{PosvaAbundanceSLCsurface}, which works in mixed characteristic with minimal changes.  
\end{remark}

It is natural to expect that lim-stable singularities in characteristic zero are pseudo-log canonical (without assuming numerical $\Q$-Gorensteinness). In fact, we expect the following stronger statement. 

\begin{conjecture}
Let $(R,\m)$ be a normal local domain essentially of finite type over a field of characteristic zero. Suppose $R$ is lim-stable (e.g., $R$ is semistable), then $R$ is of log canonical type (in particular, $R$ is pseudo-log canonical).
\end{conjecture}

We end this subsection by collecting some properties of lim-stable surface singularities, which strengthen some of the results obtained by Mumford and Shah, see \cite[Proposition 3.20]{Mumford} and \cite[Theorem~2]{Shah}. The proof proceeds by combining our results in dimension two and the following consequence of the classification of log canonical surface singularities, see \cite[Remark 4-6-29]{Matsuki} in characteristic zero and \cite[Example 3.27 and 3.39]{KollarSingularitiesMMP} in arbitrary characteristic.

\begin{theorem}
\label{thm: lc but not rational surface}
Let $(R,\m)$ be a two-dimensional excellent normal local ring with a dualizing complex, with $R/\m$ algebraically closed. If $R$ is log canonical and not a rational singularity, then $R$ is either a simple elliptic or a cusp singularity, in particular $R$ is Gorenstein.
\end{theorem}

\begin{proposition}\label{prop: on lim-stable surfaces}
Let $(R, \mf m)$ be a two-dimensional lim-stable Cohen--Macaulay local ring.
Then $\eh(R)\leq \edim(R)$ (i.e.,  either $R$ has minimal multiplicity or almost minimal multiplicity)  and we have the following.
\begin{enumerate}
    \item If $R$ has minimal multiplicity, then $\Gr_\mf m(R)$ is Cohen--Macaulay and its defining equations as a quotient of a polynomial ring, are quadratic.
    \item If $\eh(R) = \edim(R)$, then $\Gr_\mf m(R)$ is Cohen--Macaulay and reduced and its defining equations as a quotient of a polynomial ring are quadratic or cubic.
\end{enumerate}
Furthermore, if $R$ is excellent normal with a dualizing complex, and $R/\m$ is algebraically closed, then \begin{enumerate}
\item[(3)] $R$ is $\Q$-Gorenstein and log canonical;
\item[(4)] $R$ has minimal multiplicity if and only if $R$ has rational singularity; 
\item[(5)] If $\eh(R) = \edim(R)$, then $\Gr_\mf m(R)$ (and thus $R$) is Gorenstein and either $\Gr_\mf m(R)$ is a hypersurface of degree $3$, or its defining equations as a quotient of a polynomial ring are quadratic.
\end{enumerate}
\end{proposition}
\begin{proof}
The bound on the embedding dimension was obtained in Corollary~\ref{cor: lim stable embdim}. Note that $\eh(R)\geq \edim(R)-1$ by Abhyankar's inequality. If $\eh(R)=\edim(R)-1$ then $R$ has minimal multiplicity and it is well-known that $\Gr_\m(R)$ is Cohen--Macaulay with quadratic defining equations, see \cite[Theorem~1]{Sally1}. If $\eh(R)=\edim(R)$, then the conclusion follows from Proposition~\ref{prop: Shah associated graded is CM and reduced}
and \cite[Theorem~3]{Shah}. 

Next, (3) follows from Theorem~\ref{thm: lim-stable implies log canonical general} and (4) follows from Theorem~\ref{thm: lc but not rational surface} (note that rational surface singularities have minimal multiplicity). For (5), note that if $R$ has almost minimal multiplicity, then it does not have rational singularity and thus by Theorem~\ref{thm: lc but not rational surface} $R$ is Gorenstein. Now we apply \cite[3.1, 3.3]{Sally3} to obtain that, if $\eh(R)\geq 4$, then $\Gr_\mf m(R)$ is Gorenstein with quadratic defining equations; otherwise $\eh(R)=\edim(R)=3$ and thus $\Gr_\m(R)$ is a hypersurface of degree $3$. 
\end{proof}

\begin{remark}
We believe it should be possible to prove directly that if $(R,\m)$ is a semistable (or merely lim-stable) Cohen-Macaulay surface singularity such that $\eh(R) = \edim(R)$, then $\gr_\m (R)$ is Gorenstein. Proposition~\ref{prop: on lim-stable surfaces} shows that this is true when $R$ is normal, but we do not see how to prove this without using the classification result (Theorem~\ref{thm: lc but not rational surface}). 

From \cite[Theorem~3 (ii)]{Shah} we know that the Hilbert series of $\gr_\m(R)$ is given by
$$h_{\gr_\m(R)}(t) = \frac{1 + (\eh(R) - 1)t + t^2}{(1-t)^2}.$$ Moreover, by \cite[Corollary~4.4.6(a)]{BrunsHerzogBook} we know that $h_{\gr_\m(R)}(t)=h_{\omega_{\gr_\m(R)}}(t)$, meaning that the Hilbert series $\gr_\m(R)$ and the graded canonical module $\omega_{\gr_\m(R)}$ coincide. Thus it remains to show that there is a graded injection $\gr_\m(R) \to \omega_{\gr_\m(R)}$. In particular, it is immediate that $\gr_\m(R)$ is Gorenstein if it is a domain (see \cite[Corollary~4.4.6 (c)]{BrunsHerzogBook}). However, in general, $\gr_\m (R)$ is only reduced, which is not sufficient to apply \cite[Corollary~4.4.6 (c)]{BrunsHerzogBook}. 
\end{remark}

\subsection{Lech-stable and canonical singularities}
We now present a connection between Lech-stable and canonical singularities. 

\begin{theorem}
\label{thm: Lech-stable implies canonical}
Let $(R,\m)$ be an excellent normal local domain admitting a dualizing complex. Suppose one of the following holds: 
\begin{enumerate}
    \item $\dim(R)\leq 2$, or 
    \item $R$ is essentially of finite type over a field of characteristic zero, is numerically $\Q$-Gorenstein, and has canonical singularities on the punctured spectrum (e.g., $R$ has an isolated singularity).
\end{enumerate}
If $R$ is Lech-stable, then $R$ is $\Q$-Gorenstein and canonical. 
\end{theorem}
\begin{proof}
First of all, in both cases, we know that $R$ is lim-stable and thus $R$ is $\Q$-Gorenstein by Theorem~\ref{thm: lim-stable implies log canonical general}. Let $f \colon Y \to X\coloneqq \Spec(R)$ be a canonical modification (see \cite[Theorem~1.31]{KollarSingularitiesMMP}, which is based on \cite{BCHM}). This means $Y$ has canonical singularities and $K_Y$ is $f$-ample. It follows from $\Q$-Gorensteiness of $R$ that $K_Y-f^*K_X$ is an $f$-ample and $f$-exceptional $\Q$-Cartier divisor. Now since $f$ is an isomorphism on the punctured spectrum by the uniqueness of canonical modification, we know that $f$ is the blowup of some $\m$-primary ideal $I\subseteq R$ and we set $I\sO_Y=\sO_Y(-E)$. 

The rest of the argument follows the proof of Theorem~\ref{thm: lim-stable main technical result}. By Proposition~\ref{prop: consequence asymptotic RR} applied to $f$, and noting that $\Gamma(Y, \sO_Y(-nE))=\overline{I^n}$, we may express that 
\[
\length(R/\overline{I^n})= \frac{-(-E)^d}{d!}n^d + \frac{(K_Y-f^*K_X)\cdot (-E)^{d-1}}{2(d-1)!}n^{d-1} + O(n^{d-2}), 
\]
where $d\coloneqq\dim(R)$. Next, we write $K_Y-f^*K_X=\sum_F a(F;X)F$ where the sum runs over all exceptional prime divisors $F$ on $Y$. Since $K_Y-f^*K_X$ is $f$-amply and by the negativity lemma, we know that $a(F;X)<0$ for all such $F$. It follows that 
$$(K_Y-f^*K_X)\cdot (-E)^{d-1}=\sum_F a(F;X)F\cdot (-E)^{d-1} = \sum_F a(F;X)(-E|_F)^{d-1}<0$$ 
where the last inequality follows since $-E|_F$ is ample and $a(F;X)<0$. Therefore, for sufficiently large $n$, we have $d!\length(R/\overline{I^n})<\eh(\overline{I^n})$,  which implies that $\lm(R)>1$. This contradicts our assumption that $R$ is Lech-stable.
\end{proof}

\begin{remark}
\label{rmk: Lech-stable but not canonical}
Unlike Theorem~\ref{thm: lim-stable main technical result} and Theorem~\ref{thm: lim-stable implies log canonical general}, in Theorem~\ref{thm: Lech-stable implies canonical} one cannot drop the assumption that $R$ has canonical singularities on the punctured spectrum, even when $R$ is Gorenstein. As an example, consider the flat family
$$\mathbb{C}[t] \to \frac{\mathbb{C}[t, x,y,z,w]}{xyz+ t(x^3+y^3+z^3)}.$$
We know by Proposition~\ref{prop: new weak semicontinuity} that 
$$\lm\bigg(\frac{\overline{\mathbb{C}(t)}[x,y,z,w]_{\m}}{xyz+ t(x^3+y^3+z^3)}\bigg) \leq \lm\bigg(\frac{\mathbb{C}[x,y,z,w]_{\m}}{xyz}\bigg)=1,$$
where $\m=(x,y,z,w)$ and the equality above follows from Theorem~\ref{thm: snc is semistable} in Section~\ref{section: example}. It follows that 
$R\coloneqq\overline{\mathbb{C}(t)}[x,y,z,w]_{\m}/(xyz+t(x^3+y^3+z^3))$ is a normal hypersurface that is Lech-stable. However, it is easy to see that $R$ is not canonical.
\end{remark}

\begin{remark}
The converse to Theorem~\ref{thm: Lech-stable implies canonical} does not hold, even for Gorenstein rings. 
For $d \geq 3$, the degree $d$ Veronese subring $R$ of $\mathbb{C}[x_1, \ldots, x_d]$ is an example of a Gorenstein canonical singularity such that $\eh (R) = d^{d-1} > d!$, such $R$ cannot be Lech-stable by Remark~\ref{rmk: mult bound for Lech-stable and semistable}. On the other hand, we do not know whether there are complete intersection canonical singularities that are not Lech-stable (conjecturally, such singularities have multiplicity at most $2^{d-1}$, see \cite{HMTWFthreshold} and \cite{ShibataMultiplicity} for related results).  
\end{remark}

\subsection{Semi-log canonical singularities} In this subsection, we will partially generalize our results to the non-normal setting. We first recall that an $(\textnormal{S}_2)$ and $(\textnormal{G}_1)$ scheme $X$ is called {\it semi-log canonical} if $X$ is demi-normal, $K_X$ is $\Q$-Cartier and the pair $(X', D')$ is log canonical where $X'$ is the normalization of $X$ and $D'$ is the effective divisor defined by the conductor ideal, see \cite[Chapter 5.2]{KollarSingularitiesMMP} (note that $K_X$ is $\Q$-Cartier implies $K_{X'}+D'$ is $\Q$-Cartier).

The main result will use the following variant of Theorem~\ref{thm: lim-stable main technical result} in the non-normal $\Q$-Gorenstein setting. Its proof relies on a careful analysis using rational powers of ideals and might also lead to an alternative and more algebraic approach of Theorem~\ref{thm: lim-stable main technical result}.

\begin{theorem}
\label{thm: lim-stable main technical result slc}
Let $\pi \colon Y \to X\coloneqq \Spec(R)$ be a projective birational map of $(\textnormal{S}_2)$ and $(\textnormal{G}_1)$ schemes, where $(R, \mf m)$ is an excellent local domain of dimension $d\geq 2$ admitting a dualizing complex. Suppose the following conditions hold: 
\begin{enumerate}
    \item $\pi$ is the blowup of some $\m$-primary ideal $I\subseteq R$;
    \item $R$ is $\Q$-Gorenstein and $K_{Y}-\pi^*K_X=\sum_i a_i E_i$ with $a_i<-1$, where $E_i$ runs over all $\pi$-exceptional prime divisors on $Y$;
    \item the generic point of each $E_i$ is regular.
\end{enumerate} 
Then $\limlm(R)>1$. 
\end{theorem}
\begin{proof}
The hypothesis implies that $\pi_*\sO_Y=\sO_X$ (see the proof of Corollary~\ref{cor: Riemann-Roch for rational powers}). Write $I\sO_Y=\sO_Y(-E)$. By Proposition~\ref{prop: geometric rational powers}, we know that $\Gamma(Y,\sO_Y(-\lceil\frac{n}{\rho}E\rceil))=I^{\frac{n}{\rho}}$ where $\rho$ is the Rees period of $I$. Consider the Noetherian graded family $I_n\coloneqq\{I^\frac{n}{\rho}\}_n$. By Theorem~\ref{thm: criterion for lim unstable via derivative},
$$h(t) \coloneqq \sum_{j\geq 0} \length(I_j/I_{j+1})t^j = \frac{f(t)}{(1-t)^d}$$
where $f(t)$ is a rational function with $f(1)\neq 0$ and all poles of $f(t)$ have order at most $d$. By Theorem~\ref{thm: criterion for lim unstable via derivative}, it is enough to show that $f'(1)>\frac{d}{2}f(1)$.

We next express $f(t)$ in a different way. Observe that 
$$\alpha(t)\coloneqq\sum_{n\geq 1} \length(R/I_n)t^{n-1} = \frac{f(t)}{(1-t)^{d+1}}$$
and that by Corollary~\ref{cor: Riemann-Roch for rational powers}, when $n = k\rho + h$, we have
    \[
        \length (R/I_n)= - \frac{(-E)^d}{d!}k^d + \frac{K_Y \cdot (-E)^{d-1}}{2 (d-1)!}k^{d-1} + \frac{\lceil \frac{h}{\rho}E \rceil (-E)^{d-1}}{(d-1)!} k^{d-1} + O(n^{d-2}).
    \]
    Note that this formula still holds for $h = \rho$. Thus 
    the function $a(n) \coloneqq \length (R/I_{n+1})$ satisfies the assumptions 
    of Lemma~\ref{lemma: generating function 2} below: 
    we may set $A = -(-E)^d$
    and 
    \begin{align*}
        B(n) = B(h) = \frac{d+1}{2} (-E)^d + \frac{K_Y \cdot (-E)^{d-1}}{2}
        + \left \lceil \frac {(h+1)}{\rho} E \right \rceil \cdot (-E)^{d-1}.
    \end{align*}
Now, by Lemma~\ref{lemma: generating function 2} the desired inequality $f’ (1) > \frac{d}{2} f(1)$ is equivalent to showing that 
    \[
    \left(E - K_Y -  \frac {2}{\rho}\sum_{h = 1}^{\rho} \left \lceil \frac{h}{\rho}E \right \rceil \right) \cdot (-E)^{d-1} > 0. 
    \]
    In order to prove this, it suffices to verify that the coefficient of each $E_i$ in 
	\[
		E - K_Y +\pi^*K_X -  \frac {2}{\rho}\sum_{h = 1}^{\rho} \left \lceil \frac{h}{\rho} E \right \rceil
	\]
	is positive. We write $E=\sum c_iE_i$ and note that $c_i = \rho/d_i$ where $d_i\in\N$. 
    Thus 
    \[
    \sum_{h = 1}^{\rho} \left \lceil \frac{hc_i}{\rho} \right \rceil
    = \sum_{h = 1}^{\rho} \left \lceil \frac{h}{d_i} \right \rceil
    = \sum_{j = 1}^{c_i} j d_i
    = d_i \binom{c_i+1}{2}= \frac{\rho (c_i + 1)}{2}.
    \]
    Thus the coefficient of $E_i$ in $E - K_Y +\pi^*K_X -  \frac {2}{\rho}\sum_{h = 1}^{\rho} \left \lceil \frac{hZ}{\rho} \right \rceil$ is 
    \[
    c_i - a_i - c_i - 1 = - a_i - 1 > 0. \qedhere
    \]
\end{proof}

The following lemma is used in the proof of Theorem~\ref{thm: lim-stable main technical result slc} above. 

\begin{lemma}\label{lemma: generating function 2}
Suppose that $a(n) \in \mathbb{C}$ is a quasi-polynomial sequence
such that there is a constant $A$ and a periodic sequence
$B(n)$ with period $\rho$
such that 
\[
a(n) = A\binom{\lfloor n/\rho \rfloor +d}{d} + B(n) \binom{\lfloor n/\rho \rfloor + d-1}{d-1} + O(n^{d-2}).
\]
Then the generating function $\alpha(T) \coloneqq \sum_{n = 0}^\infty a(n)T^n$ can be written as 
$$\alpha(T) = f(T)/(1 - T)^{d+1}$$
where $f(T)$ is a rational function such that $f(1) = A/\rho^d$, $f'(1) \neq 0$,
and 
\[
f'(1) - \frac d 2 f(1) = \frac{-1}{\rho^{d-1}} \left ( \frac {d}{2} A + \frac 1 \rho \sum_{h = 0}^{\rho-1} B(h) \right).
\]
\end{lemma}
\begin{proof}
We may rewrite the generating function for the first term as 
\begin{align*}
\sum_{n \geq 0} \binom{\lfloor n/\rho \rfloor+d}{d} T^{n}
&= \sum_{k = 0}^\infty \sum_{h = 0}^{\rho - 1} 
\binom{k+d}{d} T^{k\rho + h}
= \frac{\left(1+T+ \cdots + T^{\rho-1} \right)}{(1 - T^\rho)^{d+1}}
\\
&= \frac{1}{(1 - T)^{d+1}} \times \frac{ 1}{(1 +T+ \cdots + T^{\rho-1})^{d}}. 
\end{align*}
We then treat the second term similarly, noting first that $B(k \rho + h) = B(h)$ since $B(n)$ is a periodic function with period $\rho$. Thus
\begin{align*}
\sum_{n \geq 0} B(n) \binom{\lfloor n/\rho \rfloor + d-1}{d-1} &= 
\sum_{k = 0}^\infty \sum_{h = 0}^{\rho - 1} 
B(h) \binom{k+d-1}{d-1} T^{k\rho + h}
= \frac{\sum_{h = 0}^{\rho - 1} B(h) T^h}{(1 - T^\rho)^{d}}.
\end{align*}
Since the $O(n^{d-2})$ term is given by a quasi-polynomial, by \cite[Proposition~4.4.1]{StanleyBook}, the pole order of its generating function at $T = 1$ is a most $d - 1$. Combining the three parts, we now may express
\[
\alpha(T) = 
\frac{A}{(1- T)(1 - T^\rho)^{d}} 
+ \frac{\sum_{h = 0}^{\rho-1} B(h) T^{h} }{(1 - T^\rho)^{d}}
+ \frac{r(T)}{(1 - T)^{d-1}}, 
\]
where $r(T)$ has no pole at $T = 1$. 
Therefore
\[
f(T) = 
\frac{A}{(1 + T+ \cdots +  T^{\rho-1})^{d}} 
+ (1 - T)
\frac{\sum_{h = 0}^{\rho-1} B(h) T^{h} }{(1 + T+ \cdots +  T^{\rho-1})^{d}}
+ (1 - T)^2 r(T).
\]
It is now clear that  $f(1) = A/\rho^{d}$ and we may further see that 
\begin{align*}
f'(1) &= \left ( \frac{A}{ (1 + T+ \cdots +  T^{\rho-1})^{d}}  \right)'|_{T = 1} - \frac{1}{\rho^d} \sum_{h = 0}^{\rho-1} B(h)
\\ &= \left ( \frac{-dA \sum_{i = 1}^{\rho-1} iT^{i-1} }{ (1 + T+ \cdots +  T^{\rho-1})^{d+1}}  \right)|_{T = 1} - \frac{1}{\rho^d} \sum_{h = 0}^{\rho-1} B(h)
\\ &= \frac{-dA \binom \rho 2}{\rho^{d+1}} - \frac{1}{\rho^d} \sum_{h = 0}^{\rho-1} B(h)\\
& = - \frac 1{\rho^d} \left (\frac{dA (\rho-1)}{2} + \sum_{h = 0}^{\rho-1} B(h) \right ). 
\end{align*}
Thus it follows by a direct computation that 
\[
f'(1) - \frac{d}{2} f(1) = 
- \frac 1{\rho^d} \left (\frac{dA (\rho-1)}{2} + \sum_{h = 0}^{\rho-1} B(h) \right ) - \frac{dA}{2\rho^{d}}
 = \frac{-1}{\rho^{d-1}} 
 \left ( 
 \frac{dA}{2} + \frac{1}{\rho}\sum_{h = 0}^{\rho-1} B(h)
 \right). \qedhere
\]
\end{proof}

We now prove our main result in this subsection, relating lim-stable and semi-log canonical singularities in the $\Q$-Gorenstein setting. 

\begin{theorem}
\label{thm: lim-stable implies slc}
Let $(R,\m)$ be an excellent local domain that admits a dualizing complex and satisfies $(\textnormal{S}_2)$ and $(\textnormal{G}_1)$. Suppose that $R$ is $\Q$-Gorenstein and lim-stable. In addition, assume one of the following holds:
\begin{enumerate}
    \item $\dim(R) = 2$ and $\charac(R/\m)\neq 2$,
    \item $R$ is essentially of finite type over a field of characteristic zero.
\end{enumerate}
Then $R$ is semi-log canonical.
\end{theorem}
\begin{proof}
In both cases, $R$ is equidimensional, satisfies the condition $(\textnormal{S}_2)$, and lim-stable. Therefore by Corollary~\ref{cor: lim-stable implies demi-normal}, $R$ is demi-normal. We next claim that, in both cases, $R$ admits a semi-log canonical modification, i.e., a projective birational map $f \colon Y\to X\coloneqq\Spec(R)$ such that the pair $(Y,\red{E})$ is semi-log canonical and $K_Y+\red{E}$ is $f$-ample. 

In case (2) this follows directly from \cite[Corollary 1.2]{OdakaXu}. In case (1), essentially the same argument as in \cite[Corollary 1.2]{OdakaXu} applies: the existence of log canonical modifications is guaranteed by the full log MMP for excellent surfaces \cite{TanakaMMPexcellentsurface}, and we use \cite[Theorem 4.1.1]{PosvaGluingSLCsurface} for the gluing argument in dimension two. Here we note that \cite[Theorem 4.1.1]{PosvaGluingSLCsurface} can be generalized to the relative setting as follows: for any excellent local ring $R$ admitting a dualizing complex with residue characteristic $\neq 2$, the normalization map provides a one-to-one correspondence between semi-log canonical surface pair $(S,\Delta)$ proper over $\Spec(R)$ and log canonical surface pairs $(\overline{S}, \overline{D}+\overline{\Delta})$ proper over $\Spec(R)$, with an involution on $(\overline{D}^n, \text{Diff}_{\overline{D}^n}\overline{\Delta})$
that is generically fixed point free on every component. Moreover, $K_S+\Delta$
is ample if and only if  $K_{\overline{S}}+\overline{D}+\overline{\Delta}$ is ample. The proof is essentially the same, except that we replace \cite[Proposition 3.4.1]{PosvaGluingSLCsurface} with \cite[Theorem 1.4]{WitaszekKeelBasePointFree} to argue that the geometric quotient exists (the characteristic zero fiber does not present any issues by \cite[Theorem 5.13]{KollarSingularitiesMMP}).

Now suppose $R$ is not semi-log canonical. We choose a prime ideal $Q\in\Spec(R)$ of minimal height such that $R_Q$ is not semi-log canonical. By Theorem~\ref{thm: localization}, we may replace $R$ by $R_Q$ to assume $R$ is semi-log canonical on the punctured spectrum of $R$ (this step is only necessary in case (2)). 

In both cases, let $\pi \colon Y \to X$ be the semi-log canonical modification. By the uniqueness of semi-log canonical modification $\pi$ is an isomorphism on the punctured spectrum. One easily checks that $Y\to X$ satisfies the assumptions of Theorem~\ref{thm: lim-stable main technical result slc}: 
\begin{itemize}
    \item $K_Y-\pi^*K_X+\red{E}$ is $\pi$-ample $\pi$-exceptional $\Q$-Cartier divisor (so $Y$ is the blowup of some $\m$-primary ideal of $R$); 
    \item the coefficient of each exceptional prime divisor in $K_{Y}-\pi^*K_X$ is $<-1$ by the negativity lemma; 
    \item the generic point of each exceptional prime divisor is regular, since this holds on the corresponding log canonical modification. 
\end{itemize}
Thus, by Theorem~\ref{thm: lim-stable main technical result slc} $R$  cannot be lim-stable, which is a contradiction.
\end{proof}

\newpage
\section{Examples: simple normal crossings and determinantal rings}
\label{section: example}

In this section we present several refinements of Lech's inequality and show how they can be used
to compute the Lech--Mumford constant for special classes of rings.

\subsection{Inclusion-exclusion formulas}
As a technical tool we will need two weaker, diluted, versions of the inclusion-exclusion principle.

\begin{proposition}
\label{prop: combinatorial incl/excl}
Let $S$ be a finite set and let $S_1, \ldots, S_d$ be subsets of $S$ such that $S = \cup_{i = 1}^dS_i$.
Then we have
\begin{align*}
|S| &= \sum_{i = 1}^d |S_i| - \sum_{1 \leq i < j\leq d} |S_i \cap S_j| + \cdots + (-1)^{d+1} |S_1 \cap \cdots \cap S_d|.
\end{align*}
Furthermore, we have inequalities
\[
|S| > \sum_{k = 1}^d \frac{(-1)^{k+1}}{2^k} \sum_{1\leq j_1 < \cdots < j_k \leq d} |S_{j_1} \cap \cdots \cap S_{j_k}|
\geq \sum_{k = 1}^d \frac{(-1)^{k+1}}{k+1} \sum_{1\leq j_1 < \cdots < j_k \leq d} |S_{j_1} \cap \cdots \cap S_{j_k}| \geq \frac{|S|}{2}.
\]
\end{proposition}
\begin{proof}
The first part is well-known: let $a_k$ be the number of elements in $S$ that are contained in exactly $k$ subsets $S_i$, then
$|S| = \sum_{i = 1}^d a_i$, $\sum_{i = 1}^d |S_i| = \sum_{k = 1}^d ka_k$, $\sum_{1\leq i<j\leq d} |S_i \cap S_j| = \sum_{k = 2}^d \binom{k}{2} a_k$, etc., and the conclusion follows from the formula $\sum_{i = 1}^s (-1)^{i+1} \binom{s}{i} = 1$. 

We now analyze the dilutions in a similar manner. First of all, the identity
\[
\sum_{k = 1}^s \left( \frac{-1}{2} \right)^{k+1} \binom{s}{k} = 1 - \left (1 - \frac 12 \right)^{s} = 1 - 1/2^s,
\]
can be applied to obtain the first diluted formula:
\begin{align*}
    A &\coloneqq \sum_{k = 1}^d \frac{(-1)^{k+1}}{2^k} \sum_{1\leq j_1 <\cdots < j_k\leq d} |S_{j_1} \cap \cdots \cap S_{j_k}| =
    \sum_{k = 1}^d \frac{(-1)^{k+1}}{2^k} \sum_{i = k}^d \binom{i}{k} a_i \\ &= \sum_{i = 1}^d a_i \sum_{k = 1}^i \frac{(-1)^{k+1}}{2^k} \binom i k
    = \sum_{i = 1}^d (1 - 1/2^i)a_i.
\end{align*}
Secondly, we also have the identity
\[
\sum_{i = 1}^s \frac{(-1)^{i+1}}{i+1}x^{i+1} \binom{s}{i} = \sum_{i = 1}^s (-1)^{i+1}x^{i+1} \frac{1}{s+1} \binom{s+1}{i+1}
= \frac{1}{s+1} \left((1 - x)^{s+1} - 1\right) + x.
\]
It follows that 
$$\sum_{i = 1}^s \frac{(-1)^{i+1}}{i+1} \binom{s}{i} = 1 - \frac {1}{s+1}.$$
Therefore
\[
B: = \sum_{k = 1}^d \frac{(-1)^{k+1}}{k+1} \sum_{1\leq j_1 <\cdots < j_k\leq d} |S_{j_1} \cap \cdots \cap S_{j_k}| =
\sum_{i = 1}^d \left(1 - \frac{1}{i+1}\right)a_i.
\]
Now clearly, we have $|S| = \sum_{i = 1}^d a_i> A \geq B \geq \sum_{i = 1}^d a_i/2 = |S|/2$.
\end{proof}

\begin{corollary}
\label{cor: monomial inc/exc}
Let $R = K[[x_1, \ldots, x_n]]$, $I$ be an $\m$-primary monomial ideal, and $P_1, \ldots, P_s$ be monomial prime ideals. Then in $S = R/I$ we have 
\small
\begin{align*}
\length\left(\frac{S}{P_1\cap \cdots \cap P_s}\right)
&= \sum_{k=1}^s \sum_{1\leq j_1 < \cdots < j_k \leq s} (-1)^{k+1} \length\left(\frac{S}{P_{j_1}+\cdots +P_{j_k}}\right)
\\&=
\sum_{i = 1}^s \length\left(\frac{S}{P_i}\right)  - \sum_{1\leq i < j\leq s} \length\left(\frac{S}{P_i+P_j}\right) + \cdots + (-1)^{s + 1} \length\left(\frac{S}{P_1+ \cdots + P_s}\right).
\end{align*}
\normalsize
\end{corollary}
\begin{proof}
If we consider the sets $S_j \coloneqq \{\text{monomials in $R$ not contained in } (I, P_j) \}$, then  
\[
\length\left(\frac{S}{P_1\cap \cdots \cap P_s}\right) 
= |\cup_{j=1}^s S_j|
\text{ and }
\length\left(\frac{S}{P_{j_1}+ \cdots + P_{j_c}}\right) = |\cap_{i = 1}^c S_{j_i}|.
\]
The conclusion now follows from the inclusion-exclusion principle in Proposition~\ref{prop: combinatorial incl/excl}.
\end{proof}

We now derive a sharper version of Lech's inequality for a power series ring over a Lech-stable ring.
\begin{theorem}
\label{thm: main Lech incl/excl}
Let $(S, \mf m)$ be a Lech-stable local ring  (i.e., $\lm(S)=1$) of dimension at least one, $R = S[[x_1, \ldots, x_n]]$ for some $n \geq 0$,
and $I$ be an $(\mf m, x_1, \ldots, x_n)$-primary ideal in $R$. Suppose also that $I$ is generated by monomials (i.e., $s \prod x_i^{a_i}$ for some $s \in S$).
Then the inequality 
\begin{align*}
\eh(I) &\leq \dim(R)! \left(\length(R/I) + \sum_{k = 1}^n \frac{(-1)^k}{2^k} \sum_{1\leq j_1 < \cdots < j_k \leq n} \length\left(\frac{R}{I+(x_{j_1}, \ldots, x_{j_k})}\right)\right)
\end{align*}
holds. Moreover, the inequality is strict if $\dim(R)\geq 3$.
\end{theorem}
\begin{proof}
We use induction on $n$. In the base case $n = 0$, $R=S$ is Lech-stable and the conclusion is just a reformulation of $\lm(R) = 1$.

We next consider the inductive step. 
Suppose that $x_n^{c+1} \in I$ but $x_n^{c} \notin I$.
Denote 
$$I_k \coloneqq \{m \in S[[x_1, \ldots, x_{n-1}]]\mid m x_n^k \in I  \}.$$
Note that for $0\leq k\leq c$, each $I_k$ is an $(\mf m, x_1, \ldots, x_{n-1})$-primary monomial ideal in $R'\coloneqq S[[x_1, \ldots, x_{n-1}]]$.
By Proposition~\ref{prop: Mumford 4.3} and the Minkowski inequality for mixed multiplicities (\cite{Teissier, ReesSharp}), for $d \coloneqq \dim(S) + n-1 = \dim(R)-1$ we have
\[
\eh(I) \leq \sum_{k = 0}^c \sum_{j = 0}^d \eh (I_k^{[j]} \mid I_{k+1}^{[d-j]})
\leq \eh(I_c) + \sum_{k = 0}^{c-1} \sum_{j = 0}^d \eh(I_k)^{j/d} \eh(I_{k+1})^{1 - j/d}.
\]
Interpreting the right-hand side as a geometric mean, we further obtain that
\[
\eh(I) \leq \eh(I_c) + \sum_{k = 0}^{c-1} \frac{d+1}{2} \left(\eh(I_{k}) + \eh(I_{k+1})\right )
\leq \frac{d+1}2 \eh(I_0) + (d+1)\sum_{k = 1}^{c} \eh(I_k),
\]
where the second inequality is strict when $d\geq 2$ (i.e., $\dim(R)\geq 3$). 
Since $I$ is generated by monomials, it is easy to see that $\length(R'/I_0) = \length(R/(I, x_n))$.
By the inductive hypothesis
\begin{align}
\label{eqn: eqn in main Lech incl/excl}
\eh(I) \leq (d+1)! &
\sum_{j = 0}^{c}
 \left(\length(R'/I_j) + \sum_{k = 1}^{n-1} \frac{(-1)^k}{2^k} \sum_{1 \leq j_1 < \cdots < j_k \leq n-1} \length\left(\frac{R'}{I_j + (x_{j_1}, \ldots, x_{j_k})}\right) \right) \notag
\\ &- \frac{(d+1)!}{2}\left(\length(R/(I, x_n)) + \sum_{k = 1}^{n-1} \frac{(-1)^k}{2^k} \sum_{1 \leq j_1 < \cdots < j_k \leq n-1} \length\left(\frac{R}{I + (x_{j_1}, \ldots, x_{j_k}, x_n)}\right) \right). 
\end{align}
\normalsize
At this point, it remains to note that, since our ideals are monomial, we have the formula
\[
\length(R/I+(x_{j_1}, \ldots, x_{j_k})) = 
\sum_{j = 0}^{c} \length(R'/I_j + (x_{j_1}, \ldots, x_{j_k})).
\]
The conclusion of the theorem follows directly from (\ref{eqn: eqn in main Lech incl/excl}) by recombining the sums. 
\end{proof}

\begin{corollary}
\label{cor: weak Lech incl/excl}
Let $(R,\m) \coloneqq  K[[x_0, \ldots, x_d]]$ where $K$ is a field and let $I\subseteq R$ be an $\m$-primary monomial ideal.
Then we have the inequality 
\[
\eh(I) \leq (d + 1)! \left(\length(R/I) + \sum_{k = 1}^d \frac{(-1)^k}{k+1} \sum_{1 \leq j_1 <  \cdots < j_k \leq d} \length\left(\frac{R}{I+(x_{j_1}, \ldots, x_{j_k})}\right) \right)
\]
and it is strict if $d\geq 2$. 
\end{corollary}
\begin{proof}
By using the sets corresponding to $P_i = (x_i)$, $1\leq i\leq d$, in the proof of Corollary~\ref{cor: monomial inc/exc},
it follows from Proposition~\ref{prop: combinatorial incl/excl}
that the estimate in Theorem~\ref{thm: main Lech incl/excl} (applied to $S=K[[x_0]]$) is stronger than the conclusion.
\end{proof}

In dimension two, Corollary~\ref{cor: weak Lech incl/excl} recovers a result of Mumford.

\begin{corollary}[{\cite[Lemma 3.15]{Mumford}}]
\label{cor: Mumford Lemma 3.15}
Let $(R,\m)\coloneqq K[[x,y]]$ where $K$ is a field and let $I\subseteq R$ be an $\m$-primary ideal. Then we have
$$\eh(I)\leq 2\length(R/I) - \length(R/I+(x)).$$
\end{corollary}
\begin{proof}
First note that if $I$ is a monomial ideal, then the conclusion is the $d=1$ case of Corollary~\ref{cor: weak Lech incl/excl}. In general, let $\prec$ be the lexicographic order with $x\prec y$. We see that 
$$\length(R/I+(x))= \min\{n \mid y^n + f(x,y) \in I, \text{ where } x|f(x,y)\}= \length(R/\init_{\prec}(I))+(x)).$$
Since $\eh(I)\leq \eh(\init_{\prec}(I)$ and  $\length(R/I)=\length(R/\init_{\prec}(I))$ as discussed in Proposition~\ref{prop: LM multigraded}, the conclusion follows from the monomial case established above.
\end{proof}

\subsection{Simple normal crossings and determinantal rings}
We will now prove main results of this section. We start with the case of simple normal crossing singularities.

\begin{theorem}
\label{thm: snc is semistable}
Let $R \coloneqq K[[x_1, \ldots, x_d]]/(x_1\cdots x_d)$ where $K$ is a field. Then 

\begin{itemize}
\item if $d=1$, then $R$ is Lech-stable; 
\item if $d=2$, then $R$ is semistable but not stable;
\item if $d\geq 3$, then $R$ is stable but not Lech-stable.
\end{itemize}

\end{theorem}
\begin{proof}
If $d=1$, then $R=K$ is a field and the conclusion is clear (see Proposition~\ref{prop: Artinian case}). We next show that when $d\geq 2$ (resp., $d\geq 3$), then $R$ is semistable (resp., stable).
By Proposition~\ref{prop: LM multigraded}, it is enough to prove that
for any $\mf m$-primary monomial ideal $I\subseteq S \coloneqq K[[x_0, \ldots, x_d]]$
\[
\eh(I, S/(x_1 \cdots x_d)) \leq d! \length\left(\frac{S}{I+(x_1 \cdots x_d)}\right),
\]
and the inequality is strict if $d\geq 3$. 

At this point, we note that $\eh(I, S/(x_1\cdots x_d)) = \sum_{i = 1}^{d} \eh(I, S/(x_i))$, so we may bound
\begin{align*}
\sum_{i = 1}^{d} \eh(I, S/(x_i)) 
\leq & \sum_{i = 1}^d
d! \left(\length(S/(I, x_i))  +  \sum_{k = 1}^{d-1} \frac{(-1)^k}{k+1} \sum_{\substack {1 \leq j_1  < \cdots < j_k \leq d\\ j_1, \ldots, j_k \neq i}} \length\left(\frac{S}{I+(x_i, x_{j_1},\dots,x_{j_k})}\right) \right)
\\ 
= & d! \left( \sum_{i = 1}^d \length(S/(I, x_i)) + \sum_{k = 2}^{d} (-1)^{k + 1}\sum_{1 \leq j_1 < \cdots <j_{k} \leq d}
\length\left(\frac{S}{I+(x_{j_1},\dots,x_{j_k})}\right) \right) \\
= & d! \length\left(\frac{S}{I+(x_1 \cdots x_d)}\right),
\end{align*}
where the first inequality follows from Corollary~\ref{cor: weak Lech incl/excl} (and this inequality is strict when $d\geq 3$) and the last equality follows from Corollary~\ref{cor: monomial inc/exc} applied to $P_i=(x_i)$ for $1\leq i\leq d$ (also note that for the equality in the second line above, the coefficient disappeared since each $\length(R/I+(x_{j_1},\dots,x_{j_k}))$ appeared in the original summation $k$ times: one for each $S/(x_{j_i})$). 

It remains to prove that $R$ is not stable when $d=2$ by 
Proposition~\ref{prop: dimension one Lech-stable} and 
and is not Lech-stable when $d\geq 3$ by Corollary~\ref{cor: strict complete intersections}. 
\end{proof}

As an immediate consequence, we see that generic determinantal hypersurfaces are Lech-stable (we will later show a stronger result that generic determinantal rings of maximal minors are Lech-stable).

\begin{corollary}
\label{cor: determinantal hypersurface}
Let $X = \{x_{ij}\}_{1\leq i, j \leq n}$ be an $n\times n$ matrix of indeterminates. Then $R \coloneqq K[[X]]/\det(X)$ is Lech-stable.
\end{corollary}
\begin{proof}
By Corollary~\ref{cor: LM monomial ideal} and Proposition~\ref{prop: LM multigraded} (applied to the lexicographic term order of the variables), we know that 
$$\lm(R)\leq \lm(K[[X]]/(x_{11}x_{22}\cdots x_{nn})) =1$$
where the second equality follows from Theorem~\ref{thm: snc is semistable} and (\ref{eqn: chain of inequalities}). Therefore $\lm(R)=1$ and thus $R$ is Lech-stable.
\end{proof}

Our approach also allows us to compute the Lech--Mumford constant for simple normal crossing in dimension two.

\begin{proposition}
\label{prop: LM of xyz}
Let $R=K[[x,y,z]]/(xyz)$. Then we have $\lm(R) = 3/2$.
\end{proposition}
\begin{proof}
First of all, it is clear that $\eh(\m)/2!\length(R/\m) = 3/2$. Thus all we need to show is that $\eh(I)\leq 3\length (R/I)$ for every $\mf m$-primary ideal $I\subseteq R$.
By Proposition~\ref{prop: LM multigraded}, we may assume that $I$ is a monomial ideal. 
By Corollary~\ref{cor: Mumford Lemma 3.15}, we have that 
\begin{align*}
\eh(I) =& \,\ \eh(I, R/(x)) + \eh(I, R/(y)) + \eh(I, R/(z)) \\
  \leq & \,\ 2\length(R/I+(x))  + 2\length(R/I+(y)) + 2\length(R/I+(z)) \\
   & - \length(R/I+(x, y))- \length(R/I+(y, z)) - \length(R/I+(x, z))\\
= & \,\ 2\length(R/I) + \length(R/I+(x, y))+ \length(R/I+(x, z))+  \length(R/I+(y, z)) -2.
\end{align*}
Thus it suffices to show that 
\begin{equation}\label{eq: 6.2.3 to show}
    \length(R/I) \geq \length(R/I+(x, y))+ \length(R/I+(x, z))+  \length(R/I+(y, z)) -2 .
\end{equation}

Now, by Corollary~\ref{cor: monomial inc/exc}, we have that  
\begin{align*}
\length(R/I)  =& \,\ \length(R/I+(x)) + \length(R/I+(y)) + \length(R/I+(z))  \\
& - \length(R/I+(x, y)) - \length(R/I+(x, z)) - \length(R/I+(y, z)) +1 
\end{align*}
and then we use Corollary~\ref{cor: monomial inc/exc}: $I(R/(x))$ is a monomial ideal in $K[[y,z]]$, so 
$$\length(R/I+(x)) \geq \length(R/I+(x, yz)) = \length(R/I+(x, y)) + \length(R/I+(x, z)) - 1,$$
and similarly,  
$$\length(R/I+(y)) \geq \length(R/I+(y, xz)) = \length(R/I+(y, x)) + \length(R/I+(y, z)) - 1,$$
$$\length(R/I+(z)) \geq \length(R/I+(z, xy)) = \length(R/I+(z, x)) + \length(R/I+(z, y)) - 1.$$
Combining these (\ref{eq: 6.2.3 to show}) follows.
\end{proof}

\begin{remark}
A similar approach can be taken for higher-dimensional simple normal crossings.
Namely, following the proof of Theorem~\ref{thm: snc is semistable}
with the inequality from Theorem~\ref{thm: main Lech incl/excl}, we obtain that for an $\mf m$-primary ideal $I$ in
$R = K[[x_1, \ldots, x_d]]/(x_1\cdots x_d)$,
\[
\eh(I)/(d-1)! \leq \sum_{i = 1}^d \length(R/I+ (x_i)) + \sum_{k = 2}^{d-1} (k-1) \left(\frac{-1}{2} \right)^{k-1} \sum_{1 \leq j_1 < \cdots < j_k \leq d} \length\left(\frac{R}{I+(x_{j_1}, \ldots, x_{j_k})}\right).
\]
It remains to compare the right-hand side with $\length(R/I)$. One may obtain crude estimates using the approach in the proof of Proposition~\ref{prop: combinatorial incl/excl}. We expect that $\lm(R)$ should be attained by a power of the maximal ideal in this case, but we do not see how to show this.
\end{remark}

We end this subsection by proving that generic determinantal rings of maximal minors are Lech-stable.

\begin{theorem}
\label{thm: max minors are Lech-stable}
Let $X = \{x_{ij}\}_{1\leq i\leq n, 1\leq j \leq m}$ be an $n\times m$ matrix of indeterminates with $n\leq m$. Let $I_n(X)$ be the ideal generated by the $n\times n$ minors of $X$. Then $R\coloneqq  K[[X]]/I_n(X)$ is Lech-stable.
\end{theorem}
\begin{proof}
If $n=1$, then $R=K$ is a field and the conclusion is obvious. In what follows, we will assume $n\geq 2$. Under the standard lexicographic monomial order, the maximal minors form a Gr{\"o}bner basis of $I_n(X)$ (see \cite{Sturmfels}), the initial ideal of $I_n(X)$ under this term order is 
\[
J = (x_{1i_1}\cdots x_{ni_n} \mid 1 \leq i_1 < \cdots < i_n \leq m).
\]
By Corollary~\ref{cor: LM monomial ideal}, it is enough to show that $K[[X]]/J$ is Lech-stable. Now since $n\geq 2$, it is easy to see that $x_{n1}$ is a free variable in $K[[X]]/J$ and it is enough to show that the local ring
$K[[x_{j{i_j}} \mid 1 \leq i_1 < \cdots < i_n \leq m]]/J$
is semistable. Furthermore, by Proposition~\ref{prop: LM multigraded}, it is enough to consider only monomial ideals.

We will now use induction on $n$ and $m - n$. For the base case we will use $n = m$ from Corollary~\ref{cor: determinantal hypersurface}.
To set the induction we observe that $J = (x_{11}, J_1) \cap J_2$,
where $J_1$ is the initial ideal of the maximal minors of the $n \times (m-1)$ matrix obtained by removing the first column,
and $J_2$ is the initial ideal of the maximal minors of the $(n-1)\times (m-1)$ matrix obtained by removing the first row and column.
Set $S\coloneqq  K[[x_{j{i_j}} \mid 1 \leq i_1 < \cdots < i_n \leq m]]$ and note that $\hght(x_{11}, J_1)= m-n =\hght J_2$. It follows that 
$\dim(S/(x_{11}, J_1))=\dim(S/J_2) =: d$. From the short exact sequence
\[
0 \to S[[t]]/J \to S[[t]]/(x_{11}, J_1) \oplus S[[t]]/J_2 \to S[[t]]/(x_{11} + J_2) \to 0
\]
we thus obtain that 
\begin{equation}
\label{eqn: eqn multiplicity in determinantal}
\eh(I, S[[t]]/J) = \eh(I, S[[t]]/(x_{11}, J_1)) + \eh(I, S[[t]]/(J_2)) 
\end{equation}
and that
\begin{equation}
\label{eqn: eqn colength in determinantal}
\length\left(\frac{S[[t]]}{I +J}\right) \geq \length\left(\frac{S[[t]]}{I + (x_{11})+ J_1}\right) + \length\left(\frac{S[[t]]}{I + J_2}\right) - \length\left(\frac{S[[t]]}{I + (x_{11})+ J_2}\right).
\end{equation}
We will now analyze $S[[t]]/(x_{11}, J_1)$ and $S[[t]]/(J_2))$ separately.  

We start with $S[[t]]/(x_{11}, J_1)$. 
Observe that $x_{22}, \ldots, x_{nn}$ are free variables in $S/(x_{11}, J_1)$, i.e., we can write $S/(x_{11}, J_1)\cong (S'/J_1)[[x_{22},\dots,x_{nn}]]$ where $S'=K[[x_{j{i_j}} \mid 2 \leq i_1 < \cdots < i_n \leq m]]$.
Thus, by the induction hypothesis, we know that $S'/J_1$ is semistable. Hence by Theorem~\ref{thm: main Lech incl/excl} applied to the Lech-stable ring $(S'/J_1)[[t]]$, we see that for any monomial ideal $I\subseteq S[[t]]$ of finite colength, we have
\begin{align}
\label{eqn: eqn 1 in determinantal}
\frac{\eh(I, S[[t]]/(x_{11}, J_{1}))}{(d+1)!} \leq & \,\ \length\left(\frac{S[[t]]}{I + (x_{11})+ J_1}\right) \notag \\
& + \sum_{k = 1}^{n-1} \frac{(-1)^k}{2^k} \sum_{2 \leq j_1 < \cdots < j_k \leq n} \length\left(\frac{S[[t]]}{I+ (x_{11}) + J_1 + (x_{j_1j_1}, \ldots, x_{j_kj_k})}\right).
\end{align}
By Corollary~\ref{cor: monomial inc/exc} (applied to the prime ideals $(x_{22}), \dots, (x_{nn})$), we have 
\[
\length\left(\frac{S[[t]]}{I+ (x_{11}) +J_1+ (x_{22}\cdots x_{nn})}\right)
= \sum_{k = 1}^{n-1} (-1)^{k+1} \smashoperator{\sum_{\substack{j_1 < \cdots < j_k\\ 2\leq j_i \leq n}}} \length\left(\frac{S[[t]]}{I+ (x_{11}) + J_1 + (x_{j_1j_1}, \ldots, x_{j_kj_k})}\right).
\]
Hence, using the lower bound from Proposition~\ref{prop: combinatorial incl/excl}, we obtain from \ref{eqn: eqn 1 in determinantal} that 
\begin{equation}\label{eqn: final eqn 1 in determinantal}
\frac{\eh(I, S[[t]]/(x_{11}, J_{1}))}{(d+1)!} \leq 
\length\left(\frac{S[[t]]}{I + (x_{11})+ J_1}\right) - \frac 12 
\length\left(\frac{S[[t]]}{I+ (x_{11}) +J_1+ (x_{22}\cdots x_{nn})}\right).   
\end{equation}

Now for $S/J_2$, we note that $x_{11}$ is a free variable. Thus by a similar (and easier) analysis using the inductive hypothesis and Theorem~\ref{thm: main Lech incl/excl}, we have that 
\begin{equation}
\label{eqn: eqn 2 in determinantal}
\frac{\eh(I, S[[t]]/J_2)}{(d+1)!} \leq \length\left(\frac{S[[t]]}{ I+ J_2}\right) - \frac 12 \length\left(\frac{S[[t]]}{I+ J_2+ (x_{11})}\right).
\end{equation}

By combining (\ref{eqn: eqn multiplicity in determinantal}), (\ref{eqn: eqn colength in determinantal}),
(\ref{eqn: final eqn 1 in determinantal}), and (\ref{eqn: eqn 2 in determinantal})
we obtain the bound 
\begin{align*}
\frac{\eh(I, S[[t]]/J)}{(d+1)!} \leq \length\left(\frac{S[[t]]}{I +J}\right)
- \frac 12 
\length\left(\frac{S[[t]]}{I+ (x_{11}) +J_1+ (x_{22}\cdots x_{nn})}\right)
+ \frac 12 \length\left(\frac{S[[t]]}{I+ J_2+ (x_{11})}\right).
\end{align*}
Finally, since $J_1 + (x_{22}\cdots x_{nn}) \subseteq J_2$, we have 
$$\length\left(\frac{S[[t]]}{I+(x_{11}) +J_1+ (x_{22}\cdots x_{nn})}\right) \geq \length\left(\frac{S[[t]]}{I + (x_{11})+ J_2}\right)$$
and the desired inequality follows.
\end{proof}

Given Theorem~\ref{thm: max minors are Lech-stable}, it is natural to ask the following question.

\begin{question}
Let $X = \{x_{ij}\}_{1\leq i\leq n, 1\leq j \leq m}$ be an $n\times m$ matrix of indeterminates  and let $I_t(X)$ be the ideal generated by the $t\times t$ minors of $X$. Then is $R\coloneqq  K[[X]]/I_t(X)$ Lech-stable (or at least semistable)?
\end{question}

\subsection{Rational triple points} 
In \cite{Shah}, Shah established a list of candidates for semistable surface singularities among triple points of embedding dimension four. They all happen to have rational singularities. Shah's list contains several infinite families and several exceptional cases. We record here that our methods can be used to prove semistability for at least one of the infinite families. Following Tjurina's terminology (\cite{Tjurina})
\[A_{l, m, n} = \begin{bmatrix}
x & y & w^{n+1}\\
w^{l+1} & y + w^{m+1} & z
\end{bmatrix},
\]
i.e., the singularity $A_{l, m, n}$ is defined in $k[[x,y,z,w]]$ by the $2\times 2$ minors of the above matrix. Up to change of coordinates, one can check that this is symmetric on $l, m, n$ and thus we may assume that $l\leq m\leq n$. Note that $A_{0,0,0}$ is isomorphic to the (completion of the) cone of the twisted cubic curve. 

The following proposition contains the 
key computation for showing semistability of these singularities. The proposition is a special case of Proposition~\ref{prop: Mumford rational cones} in Section~\ref{section: Lech's inequality revisited}, however we present it here since it follows directly from Theorem~\ref{thm: main Lech incl/excl}.

\begin{proposition}
\label{prop: cone of twisted cubic semistable}
Let $R = k[[a,b,c,d]]/(ab, ac, cd)$. Then $R$ is stable. As a consequence, the (completion of the) cone of the twisted cubic $k[[s^3, s^2t, st^2, t^3]]$ is stable.
\end{proposition}
\begin{proof}
We need to show that $\lm(R[[T]])=1$ and that the supremum is not attained, that is, $\eh(I, R[[T]])< 6\length(R[[T]]/I)$ for every $(a,b,c,d,T)$-primary ideal $I\subseteq R[[T]]$. By Proposition~\ref{prop: LM multigraded}, we may assume that $I$ is generated by monomials. By Corollary~\ref{cor: monomial inc/exc} applied to $P_1=(a, c)$, $P_2=(a, d)$ and $P_3=(b, c)$ in $k[[a, b, c, d, T]]$, we have
\begin{multline*}
 \length(R[[T]]/I) =  \,\ \length\left(\frac{R[[T]]}{I+(a, c)}\right) + \length\left(\frac{R[[T]]}{I+(a, d)}\right) + \length\left(\frac{R[[T]]}{I+(b, c)}\right) \\
 - \length\left(\frac{R[[T]]}{I+(a, c, d)}\right) - \length\left(\frac{R[[T]]}{I+(a, b, c)}\right).   
\end{multline*}
Furthermore, by applying Theorem~\ref{thm: main Lech incl/excl} to each $R[[T]]/(a,c), R[[T]]/(a,d), R[[T]]/(b,c)$, we estimate that 
\begin{align*}
\frac{1}{6} \eh(I, R[[T]]) 
&= \frac{1}{6} \eh(I, R[[T]]/(a, c)) + \frac{1}{6} \eh(I, R[[T]]/(a, d)) + \frac{1}{6} \eh(I, R[[T]]/(b, c)) \\
&< \length\left( \frac{R[[T]]}{I+(a,c)} \right) + \length\left( \frac{R[[T]]}{I+(a,d)} \right) + \length\left( \frac{R[[T]]}{I+(b,c)} \right) \\
&\quad - \length\left( \frac{R[[T]]}{I+(a, c, d)} \right) - \length\left( \frac{R[[T]]}{I+(a, b, c)} \right) \\
&\quad - \frac{1}{2} \length\left( \frac{R[[T]]}{I+(b, c, d)} \right) - \frac{1}{2} \length\left( \frac{R[[T]]}{I+(a, b, d)} \right) \\
&\quad + \frac{3}{4} \length\left( \frac{R[[T]]}{I+(a, b, c, d)} \right) \\
&< \length\left( \frac{R[[T]]}{I+(a,c)} \right) + \length\left( \frac{R[[T]]}{I+(a,d)} \right) + \length\left( \frac{R[[T]]}{I+(b,c)} \right) \\
&\quad - \length\left( \frac{R[[T]]}{I+(a, c, d)} \right) - \length\left( \frac{R[[T]]}{I+(a, b, c)} \right) \\
&= \length(R[[T]]/I).
\end{align*}
\normalsize
This completes the proof of the first result. For the second claim, note that
$$k[[s^3, s^2t, st^2, t^3]]\cong k[[a,b,c,d]]/(ab-d^2, cd-b^2, ac-bd)\coloneqq S.$$ 
Consider the term order $a>c>b>d$, the initial ideal of the defining ideal of $S$ under this term order is $(ab, ac, cd)$ which is the defining ideal of $R$. Thus by Corollary~\ref{cor: LM monomial ideal}, we have $\lm(S[[T]])\leq \lm(R[[T]]) = 1$ and the supremum in $\lm(S[[T]])$ is not attained. Thus $S$ is stable.
\end{proof}

\begin{proposition}\label{prop: triple A}
For all $0 \leq m, n \leq \infty$, the singularities
$A_{0, m, n}$ are stable (thus by symmetry, $A_{m, 0, n}$ and $A_{m, n, 0}$ are stable). 
\end{proposition}
\begin{proof}
If $m, n > 0$, then we may pass to the associated graded ring via  Proposition~\ref{prop: LM associated graded} and reduce to the case where $m = n = \infty$. In the latter case, we have
\[
A_{0, \infty, \infty} = 
\begin{bmatrix}
x & y & 0\\
w & y & z
\end{bmatrix} 
\]    
whose defining ideal is $(xy-yw, xz, yz)$. Under the monomial order $x<y<z<w$, the initial ideal is $(yw, xz, yz)$ and thus the singularity is stable by Corollary~\ref{cor: LM monomial ideal}, Proposition~\ref{prop: LM multigraded}, and Proposition~\ref{prop: cone of twisted cubic semistable}. 

It remains to study the case where more than one of $l, m, n$ is $0$. The defining ideal of $A_{0, 0, \infty}$ is $(x(y + w) - yw, xz, yz)$ and under the  monomial order $x < y < z < w$ its initial ideal is $(yw, xz, yz)$. The defining ideal of $A_{0, 0, 0}$ is $(x (y + w) - yw, yz - (y + w)w, xz - w^2)$ and under the monomial order $x < y < w < z$ its initial ideal is $(yw, yz, xz)$. Thus in both cases, the singularities are stable by Corollary~\ref{cor: LM monomial ideal}, Proposition~\ref{prop: LM multigraded}, and Proposition~\ref{prop: cone of twisted cubic semistable}. 
\end{proof}

\begin{remark}
In \cite{Shah}, Shah proved that $A_{l,m,n}$ is unstable if $\frac{1}{l + 1} + \frac{1}{m+1} + \frac{1}{n+1} < 1$. Thus, to fully understand this family, it remains to determine the (semi)stability of $A_{1,1, n}$ for $1\leq n\leq \infty$ along with the exceptional cases $A_{1,2,2}, A_{1,2,3}, A_{1,3,3}, A_{2,2,2}$. 

We also want to highlight a peculiar connection between the non-normal limit $A_{1,1, \infty}$ and the $A_3$-singularity. The defining ideal of $A_{1,1,\infty}$ is
\[(x(y + w^2) - yw^2, yz, xz)
= (x,y)\cap (z, xy + (x-y)w^2),
\]
which becomes isomorphic to $(x + w^2, y + w^2) \cap (z, xy + w^4)$ after a change of variables.
\end{remark}

\newpage
\section{Lech's inequality revisited}
\label{section: Lech's inequality revisited}

In this section we investigate optimal versions of Lech's inequality (see \cite{HSV} for closely related developments). We establish a sharp version of Lech's inequality for regular local rings of dimension three: equality is attained for powers of the maximal ideal. This result can be also viewed as an extension of the sharp two-dimensional inequality in \cite[Lemma 3.15]{Mumford}. As an application, we show that elliptic and rational polygonal $n$-cones for $n\leq 6$ are semistable -- results which were stated by Mumford in \cite[Propositions~3.18 and 3.19]{Mumford} without proofs. We then propose a general conjecture for the optimal form of Lech’s inequality in higher dimensions, provide partial evidence, and explore its connections with the results and conjectures in \cite{HSV}. 

\subsection{Optimal Lech's inequality in dimension three} 
The main ingredient of the proof of the sharp Lech-type inequality in dimension three is an extension of Mumford's two-dimensional Lech's inequality (see Corollary~\ref{cor: Mumford Lemma 3.15}) to mixed multiplicities. 

\begin{lemma}
\label{lem: mixed multiplicity dim two}
Let $(R,\m)=k[[x,y]]$ and let $I\subseteq J \subseteq R$ be two $\m$-primary monomial ideals. Then we have
\[\eh(I | J) \leq \length(R/I) +  \length(R/J) - \length(R/I+(x)).\]
\end{lemma}
\begin{proof}
Write
$I=(y^{r_0}, y^{r_1}x, \ldots, y^{r_c}x^c, x^{c+1})$, where $r_0 \geq r_1 \geq \cdots \geq r_c > 0$. Since $I \subseteq J$, we may write
$J=(y^{s_0}, y^{s_1}x, \ldots,  y^{s_c}x^c, x^{c+1})$ with $r_j \geq s_j \geq 0$. For convenience, we set $r_j=s_j=0$ if $j\geq c+1$. 
We will use the formula $\eh(I | J)= \length(R/IJ) - \length(R/I)- \length(R/J)$ given in \cite[Exercise 17.8]{SwansonHuneke}.
Thus it suffices to show that
\begin{equation*}
\length(R/IJ)\leq 2\length(R/I)+2\length(R/J)-\length(R/I+(x)).
\end{equation*}

By our way of writing $I$ and $J$, we know that $\length(R/I)=\sum^c_{i = 0} r_i$, $\length(R/I)=\sum^c_{j = 0} s_j$, and $\length(R/I+(x))=r_0$. It is now a straightforward computation that
\begin{align*}
\length(R/IJ) & =\sum_{t=0}^{2c+1} \min \{r_i+s_j \mid i+j=t\} \\
& \leq (r_0+s_0) + (s_0+r_1) + (r_1+s_1) + (s_1+r_2) + \cdots .
\end{align*}
In the sum above, each $s_j$ appears twice and each $r_i$ for $i \geq 1$ appears twice, while $r_0$ appears only once. Therefore we have
\[
\length(R/IJ) \leq 2\sum^c_{i = 0} r_i + 2\sum^c_{i = 0} s_j  - r_0 = 2\length(R/I) +2\length(R/J) - \length(R/I+(x)),
\]
which is precisely what we want to show.
\end{proof}

We now present the Lech-type inequality for power series rings of dimension three which is sharp for powers of the maximal ideal.

\begin{theorem}
\label{thm: optimal Lech dim three}
Let $(S,\m)=k[[x_0, x_1, x_2]]$ and let $I\subseteq S$ be an $\mf m$-primary ideal. Then 
\[\eh(I) \leq 6\length(S/I) -3\length(S/I+(x_1)) -3 \length(S/I+(x_2)) + \length(S/I+(x_1,x_2)),\]
and equality holds when $I=\m^n$.
\end{theorem}
\begin{proof}
It is straightforward to verify the equality for $I=\m^n$. Therefore we only need to prove the inequality. We first handle the case when $I$ is a monomial ideal.
The strategy will be similar to the proof of Theorem~\ref{thm: main Lech incl/excl}. We set $R\coloneqq k[[x_0, x_1]]$ and 
$I_j\coloneqq  \{m \in R \mid mx_2^j\in I \}$, 
a monomial ideal in $R$. By Proposition~\ref{prop: Mumford 4.3} (note that we are in dimension two), we can bound
\begin{align*}
\eh(I) \leq \eh(I_0) + 2\eh(I_1) + \cdots + 2\eh(I_c) + \eh(I_0 | I_1) + \cdots + \eh(I_{c-1} | I_c).  
\end{align*}
We now estimate the right-hand side in the following way. By Corollary~\ref{cor: Mumford Lemma 3.15}, we bound 
\begin{align*}
\eh(I_0) + 2\eh(I_1) + \cdots + 2\eh(I_c) &\leq     
2\length(R/I_0)  - \length(R/(I_0, x_1)) + 2\sum_{j=1}^c\big(2\length(R/I_j) - \length(R/(I_j, x_1))\big)
\\&\leq 
4 \length (S/I) - 2\length(R/I_0) - 2 \length (S/(I, x_1)) + \length (R/(I_0, x_1)), 
\end{align*}
where we are using that $\sum \length(R/I_j)=\length(S/I)$ and $\sum \length(R/(I_j, x_1))=\length(S/(I, x_1))$. 
Similarly, using Lemma~\ref{lem: mixed multiplicity dim two} we obtain the bound 
\begin{align*}
\eh(I_0 | I_1) + \cdots + \eh(I_{c-1} | I_c)
&\leq 2 \length (S/I) - \length (R/I_0) - \length (R/I_c)
- \length (S/(I, x_1)) + \length (R/(I_c, x_1))
\\ &\leq 2 \length (S/I) - \length (R/I_0) - \length (S/(I, x_1)),
\end{align*} 
where we have used that $\length(R/I_c) \geq \length(R/I_c+(x_1))$ (note that the equality may happen, namely when $I_c$ contains $x_1$).
By combining the two multiplicity estimates 
and noting that $\length(R/I_0) =\length(S/(I,x_2))$ and $\length(R/(I_0, x_1))= \length(S/(I, x_1,x_2))$, we obtain that 
\begin{align*}
\eh(I) &\leq 6\length(S/I) - 3\length(R/I_0) -3\length(S/(I,x_1)) + \length(R/(I_0,x_1))
\\
&\leq 6\length(S/I) - 3\length(S/(I, x_1)) - 3\length(S/(I, x_2)) + \length(S/(I, x_1, x_2)).
\end{align*}
This completes the proof when $I\subseteq S$ is an $\m$-primary monomial ideal.

We now prove the general case. We first note that $\eh(I) \leq \eh(\init_{\prec}(I))$, $\length (S/I) = \length (S/\init_{\prec}(I))$, 
and $\length (S/(\init_{\prec}(I), x_i) ) \geq 
\length (S/\init_{\prec}(I, x_i)) = \length (S/(I,x_i))$
for any monomial order $\prec$ since 
we have a containment $\init_{\prec}(I) \subseteq \init_{\prec}(I, x_i)$. 
We next claim that if $\prec$ is the lexicographic order with $x_2 \prec x_1 \prec x_0$ then 
$
\length(S/(I, x_1,x_2))  = \length(S/\init_{\prec}(I)+(x_1,x_2)).
$
To see this, note that 
$$\length(S/(I, x_1,x_2)) = \min \{n \mid x_0^n + f(x_0, x_1,x_2) \in I, \text{ where } f(x_0, x_1, x_2)\in (x_1, x_2)\}.$$ 
But since the $\prec$-leading term of 
$x_0^n + f(x_0, x_1,x_2)$, for $f(x_0, x_1, x_2)\in (x_1, x_2)$, is precisely $x_0^n$, we see that 
$\length(S/(I, x_1,x_2))  = \length(S/\init_{\prec}(I)+(x_1,x_2))$ as wanted. Therefore, we bound
\begin{align*}
\eh(I) & \leq \eh(\init_{\prec}(I)) \\
& \leq 6\length(S/\init_{\prec}(I)) - 3\length(S/\init_{\prec}(I) + (x_1)) - 3\length(S/\init_{\prec}(I) + (x_2)) + \length(S/\init_{\prec}(I) + (x_1, x_2))\\
& \leq 6\length(S/I) - 3\length(S/I + (x_1)) - 3\length(S/I + (x_2)) + \length(S/I + (x_1, x_2)),
\end{align*}
where the second inequality follows from the already established monomial case. This completes the proof of the theorem.
\end{proof}
%

\begin{remark}
If we compare Theorem~\ref{thm: main Lech incl/excl} (for $S=k[[x_0]]$ and $n=2$) with Theorem~\ref{thm: optimal Lech dim three}, we see that there are two improvements: first, Theorem~\ref{thm: main Lech incl/excl} only applies to monomial $\m$-primary ideals while Theorem~\ref{thm: optimal Lech dim three} applies to all $\m$-primary ideals; secondly, the upper bound in Theorem~\ref{thm: main Lech incl/excl} is
$$6\length(S/I) -3\length(S/I+(x_1)) -3 \length(S/I+(x_2)) + \frac{3}{2}\length(S/I+(x_1,x_2))$$
while the bound in Theorem~\ref{thm: optimal Lech dim three} is 
$$6\length(S/I) -3\length(S/I+(x_1)) -3 \length(S/I+(x_2)) + \length(S/I+(x_1,x_2))$$
which is sharper. This latter improvement will be crucial in our applications later. 
\end{remark}

\subsection{Elliptic and Rational polygonal cones}
In this subsection, we provide purely algebraic proofs for \cite[Proposition~3.18 and Proposition~3.19]{Mumford} in arbitrary characteristic. We first introduce definitions following \cite[3.17 and 3.18]{Mumford}. First, in $\mathbb{P}^{n-1}$ take a generic $n$-gon $\cup_{i=0}^n\overline{P_iP_{i+1}}$  where $P_{n+1}=P_0$, its affine cone in $\mathbb{A}^{n}$ is called an \emph{elliptic polygonal $n$-cone}. Similarly, take $n$ generic line segments $\cup_{i=0}^{n-1}\overline{P_iP_{i+1}}$ in $\mathbb{P}^{n}$, its affine cone in $\mathbb{A}^{n+1}$ is called a \emph{rational polygonal $n$-cone}.\footnote{It seems to the authors that there are some notational inaccuracies in \cite[3.18]{Mumford}, when defining rational polygonal $n$-cones.}

\begin{proposition}
\label{prop: Mumford elliptic cones}
Elliptic polygonal $n$-cones are semistable if and only if $1 \leq n \leq 6$. Specifically, these are the singularities:
\[
\begin{cases}
\text{ $R = k[[x,y,z]]/(x^2 - xyz + y^3)$ if $n = 1$,} \\
\text{ $R = k[[x,y,z]]/(x^2-xyz)$ if $n =2$,} \\
\text{ $R =  k[[x_1, \ldots, x_n]]/(P_1 \cap \cdots \cap P_n)$ where }  \\
\text{ \hspace{2em} $P_1= (x_1, \ldots, x_{n-2}), P_2 = (x_2, \ldots, x_{n-1}), \ldots, P_n=(x_n, x_1, \ldots, x_{n-3})$ if $n\geq 3$.} 
\end{cases}
\]
In fact, when $1\leq n\leq 5$ these singularities are stable, while if $n=6$, the singularity is semistable but not stable.
\end{proposition}

\begin{remark}
\label{rmk: Mumford equation of elliptic cones}
In the degenerate cases $n=1$ corresponds to glueing a pair of transversal lines in a plane together and $n=2$ corresponds to glueing two planes to each other along a pair of transversal lines. For $n=1$, the equation is $x^2 - xyz + y^3$, which, when $\charac(k)\neq 2$, is equivalent to the equation $x^2-y^3-y^2z^2$ after a change of variables. Similarly, for $n=2$, the equation is $x^2-xyz$, which, when $\charac(k)\neq 2$, is equivalent to the equation $x^2-y^2z^2$ after a change of variables. These latter equations are Mumford's descriptions of these singularities over $\mathbb{C}$, see \cite[3.17]{Mumford}.
\end{remark}

\begin{proof}[Proof of Proposition~\ref{prop: Mumford elliptic cones}]
First of all, if $n\geq 7$, then $\eh(R)=n\geq 7$ and thus $R$ is unstable by Remark~\ref{rmk: mult bound for Lech-stable and semistable}. 
For $n = 1$, we may assign weights $\deg(x) = 4$, $\deg(y) = 3$, $\deg(z) = 1$, its initial ideal is $(x^2-xyz)$. Thus by Corollary~\ref{cor: grober degeneration}, the $n=1$ case reduces to the $n=2$ case. 

Now we suppose $n = 2$. Let $I\subseteq R[[T]]$ be an $(x,y,z, T)$-primary ideal. By tensoring with $R[[T]]/I$ the short exact sequence 
$$
0\to R[[T]]\to R[[T]]/(x) \oplus R[[T]]/(x-yz) \to R[[T]]/(x, yz)\to 0
$$
we obtain the inequality 
\begin{align*}
\length(R[[T]]/I)  \geq & \,\ \length\left(\frac{R[[T]]}{(I,x)}\right) + \length\left(\frac{R[[T]]}{I + (x-yz)}\right) - \length\left(\frac{R[[T]]}{I + (x,yz)}\right).
\end{align*}

For now on, let $\prec$ be any monomial order  on $R[[T]]=k[[x,y,z,T]]$.
By using the inclusion-exclusion formula of Corollary~\ref{cor: monomial inc/exc}, we have that 
\begin{align*}
\length\left(\frac{R[[T]]}{I + (x,yz)}\right)
= & \,\ \length\left(\frac{R[[T]]}{\init_{\prec} (I + (x,yz))}\right) \\
= & \,\ \length\left(\frac{R[[T]]}{\init_{\prec} (I + (x,yz)) + (y)}\right) +
\length\left(\frac{R[[T]]}{\init_{\prec} (I + (x,yz)) + (z)}\right) \\
& -\length\left(\frac{R[[T]]}{\init_{\prec} (I + (x,yz)) + (y,z)}\right).
\end{align*}
Next, note that $\eh (I, R[[T]]) = \eh(I, R[[T]]/(x- yz)) + \eh(I, R[[T]]/(x))$. Since $R[[T]]/(x- yz)$ and $R[[T]]/(x)$ are both isomorphic to $k[[y,z,T]]$, by applying Theorem~\ref{thm: optimal Lech dim three} to the initial ideal of the image of $I$ inside $R[[T]]/(x- yz)$ and $R[[T]]/(x)$ respectively, we bound 
\begin{align*}
    \eh(I, R[[T]]/(x- yz))
    \leq & \,\ 6\length\left(\frac{R[[T]]}{I + (x-yz)}\right) 
    - 3\length\left(\frac{R[[T]]}{\init_{\prec}(I + (x-yz)) + (y)}\right)  \\
    &  - 3\length\left(\frac{R[[T]]}{\init_{\prec}(I + (x-yz)) + (z)}\right)  + \length\left(\frac{R[[T]]}{\init_{\prec}(I + (x-yz)) + (y,z)}\right)
\end{align*}
\begin{align*}
    \eh(I, R[[T]]/(x))
    \leq & \,\ 6\length\left(\frac{R[[T]]}{I + (x)}\right) 
    - 3\length\left(\frac{R[[T]]}{\init_{\prec}(I + (x)) + (y)}\right)  
    - 3\length\left(\frac{R[[T]]}{\init_{\prec}(I + (x)) + (z)}\right)  \\
    & + \length\left(\frac{R[[T]]}{\init_{\prec}(I + (x)) + (y,z)}\right)
\end{align*}
By collecting all these bounds, we find that 
\begin{align*}
\frac{6\length (R[[T]]/I) - \eh (I, R[[T]])}{3}
\geq &\,\ \length\left(\frac{R[[T]]}{\init_{\prec}(I + (x-yz)) + (y)}\right) - \length\left(\frac{R[[T]]}{\init_{\prec} (I + (x,yz)) + (y)}\right)
\\ &+ 
\length\left(\frac{R[[T]]}{\init_{\prec}(I + (x-yz)) + (z)}\right) - \length\left(\frac{R[[T]]}{\init_{\prec} (I + (x,yz)) + (z)}\right)
\\ &+ 
\length\left(\frac{R[[T]]}{\init_{\prec}(I + (x)) + (y)}\right) - \length\left(\frac{R[[T]]}{\init_{\prec} (I + (x,yz)) + (y)}\right)
\\ 
&+ 
\length\left(\frac{R[[T]]}{\init_{\prec}(I + (x)) + (z)}\right) - \length\left(\frac{R[[T]]}{\init_{\prec} (I + (x,yz)) + (z)}\right)
\\
&+ 2\length\left(\frac{R[[T]]}{\init_{\prec} (I + (x,yz)) + (y,z)}\right) 
\\ &- \frac 13 \length\left(\frac{R[[T]]}{\init_{\prec}(I + (x-yz)) + (y,z)}\right) 
-  \frac 13\length\left(\frac{R[[T]]}{\init_{\prec}(I + (x)) + (y,z)}\right).
\end{align*}
Since $\init_{\prec}(I + (x-yz)) \subseteq \init_{\prec} (I + (x,yz))$ and $\init_{\prec}(I + (x))\subseteq \init_{\prec} (I + (x,yz))$ for any monomial order $\prec$, we get that 
the terms paired in the first four lines of the formula 
are all nonnegative. Therefore we obtain that 
\begin{align*}
6\length (R[[T]]/I) - \eh (I, R[[T]])
\geq & \,\ 6 \length\left(\frac{R[[T]]}{\init_{\prec} (I + (x,yz)) + (y,z)}\right) - \length\left(\frac{R[[T]]}{\init_{\prec}(I + (x-yz)) + (y,z)}\right) 
\\ 
& -  \length\left(\frac{R[[T]]}{\init_{\prec}(I + (x)) + (y,z)}\right).
\end{align*}

At this point, we take $\prec$ to be the lexicographic order with $x \prec y \prec z \prec T$. By the same argument as in the proof of Theorem~\ref{thm: optimal Lech dim three},
this monomial order forces that 
$$\init_{\prec} (I + (x,yz)) + (y,z) = \init_{\prec}(I + (x-yz)) + (y,z) = \init_{\prec}(I + (x)) + (y,z) = I+(x,y,z).$$ 
It follows that
$6\length (R[[T]]/I) - \eh (I, R[[T]]) \geq 4\length\left(R[[T]]/(I + (x,y,z))\right) > 0$, so $R$ is stable.

Finally, we suppose that $3 \leq n \leq 6$. Since the defining ideal is a monomial ideal, by Proposition~\ref{prop: LM multigraded}, in order to show $R$ is stable or semistable, it is enough to consider monomial ideals $I\subseteq R[[T]]$.
Note that $P_i + P_j = (x_1, \ldots, x_n)$ if $j \neq i-1, i, i + 1$, thus by Corollary~\ref{cor: monomial inc/exc}, we obtain the formula
\begin{equation}\label{eqn: 7.2.1 length formula}
\length(R[[T]]/I)= \sum_{i = 1}^n \length(R[[T]]/(I+P_i)) - \sum_{i = 1}^n \length(R[[T]]/(I+P_i+ P_{i+1})) + \length(R[[T]]/I+(x_1,\dots,x_n)),    
\end{equation}
where for convenience we denote $P_{n+1} \coloneqq P_1$.

On the other hand, each $R[[T]]/P_i$ is isomorphic to a power series ring over $k$ in three variables. Thus by Theorem~\ref{thm: optimal Lech dim three}, we bound
\begin{align*}
    \eh(I, R[[T]]/P_i) \leq & \,\ 6\length(R[[T]]/I+P_i) -3\length(R[[T]]/I+P_i+ P_{i+1}) -3 \length(R[[T]]/I+P_i+ P_{i-1}) \\
    & + \length(R[[T]]/I+(x_1,\dots,x_n)).
\end{align*}
Summing these up, we obtain the inequality
\begin{equation*}
\begin{split}
\eh(I, R[[T]]) = & \,\ \sum_{i=1}^n \eh(I, R[[T]]/P_i) \\
\leq & \,\ 6\sum_{i=1}^n\length\left (\frac{R[[T]]}{I+P_i}\right) -6\sum_{i=1}^n\length\left(\frac{R[[T]]}{I+P_i+ P_{i+1}} \right) 
 + n\length\left(\frac{R[[T]]}{I+(x_1,\dots,x_n)} \right).
\end{split}    
\end{equation*}
After comparing with (\ref{eqn: 7.2.1 length formula}) we immediately 
see that $\eh(I, R[[T]]) < 6 \length (R[[T]])$ when $n \leq 5$, so $R$ is stable in this case. On the other hand, when $n = 6$
we only get that $\eh(I, R[[T]]) \leq 6 \length (R[[T]])$, 
showing that $R$ is semistable. 
Since equality holds for the maximal ideal, we see that $R$ is not stable for $n = 6$.
\end{proof}

The result on rational  polygonal $n$-cones is proved similarly. 

\begin{proposition}
\label{prop: Mumford rational cones}
Rational polygonal $n$-cones are semistable if and only if $2 \leq n \leq 6$. Specifically, these are  
$R =  k[[x_1, \ldots, x_{n+1}]]/(P_1 \cap \cdots \cap P_{n}),$
where $P_1= (x_1, \ldots, x_{n-1}), P_2 = (x_2, \ldots, x_{n}), \ldots, P_{n}=(x_{n},x_{n+1}, x_1, \ldots, x_{n-3})$. Moreover,  these singularities are stable if and only if $2\leq n\leq 5$.
\end{proposition}
\begin{proof}
Since $\eh(R)=n$, $R$ is unstable whenever $n \geq 7$ by Remark~\ref{rmk: mult bound for Lech-stable and semistable}. It remains to prove that $R$ is stable for $2\leq n \leq 5$ and is semistable but not stable when $n=6$.

We now assume that $2\leq n\leq 6$. The strategy is similar to the proof of Proposition~\ref{prop: Mumford elliptic cones}. First of all, since the defining ideal is a monomial ideal, by Proposition~\ref{prop: LM multigraded}, in order to show $R$ is stable, it is enough to consider monomial ideals $I\subseteq R[[T]]$. Note that $P_i + P_j = (x_1, \ldots, x_{n+1})$ if $j \neq i-1, i, i + 1$, thus by Corollary~\ref{cor: monomial inc/exc}, we have that
\begin{align*}
\length(R[[T]]/I) 
& = \sum_{i = 1}^{n} \length(R[[T]]/I + P_i) 
- \sum_{i = 1}^{n-1} \length(R[[T]]/I + P_i+ P_{i+1}) \\
& = \sum_{i = 1}^{n} \length(R[[T]]/I + P_i)
- \sum_{i \neq n, n-1} \length (R[[T]]/I + (x_1, \dots, \widehat{x_i},\dots, x_{n+1})), 
\end{align*}
since the terms equal to $\length(R[[T]]/I+ (x_1, \ldots, x_{n+1}))$ cancel each other. 

Next, each $R[[T]]/P_i$ is isomorphic to a power series ring over $k$ in three variables. Thus by Theorem~\ref{thm: optimal Lech dim three}, we have that 
\begin{align*}
\eh(I, R[[T]]) = & \,\ \sum_{i=1}^{n} \eh(I, R[[T]]/P_i) \\
\leq & \,\ 6 \sum_{i = 1}^{n} \length(R[[T]]/I + P_i) - 6 \sum_{i \neq n, n-1} \length\left(\frac{R[[T]]}{I+(x_1, \dots, \widehat{x_i}, \dots,x_{n+1})}\right)\\
& - 3 \length\left(\frac{R[[T]]}{I+(x_1, \ldots, x_{n-1}, x_{n+1})}\right) - 3\length\left(\frac{R[[T]]}{I+(x_1, \ldots, x_{n-2}, x_{n}, x_{n+1})}\right) 
\\ &  + n\length\left(\frac{R[[T]]}{I+(x_1, \ldots, x_{n+1})}\right). 
\end{align*}
Putting these together, we obtain that 
\begin{align*}
\eh(I, R[[T]]) - 6\length(R[[T]]/I) \leq n\length\left(\frac{R[[T]]}{I+(x_1, \ldots, x_{n+1})}\right) - 3 \length\left(\frac{R[[T]]}{I+(x_1, \ldots, x_{n-1}, x_{n+1})}\right) \\ - 3\length\left(\frac{R[[T]]}{I+(x_1, \ldots, x_{n-2}, x_{n}, x_{n+1})}\right).
\end{align*}
Immediately, we obtain that $\eh(I, R[[T]]) - 6\length(R[[T]]/I) < 0$ if $n\leq 5$, showing that $R$ is stable in this case. On the other hand, for $n = 6$ we only get that $\eh(I, R[[T]]) \leq  6\length(R[[T]]/I)$. Indeed, equality holds for the maximal ideal of $R[[T]]$, so $R$ is semistable but not stable.
\end{proof}

\begin{example}
\label{example: Veronese in two variables}
Let $V_n\coloneqq k[x,y]^{(n)}$ be the $n$th Veronese subring of $k[x,y]$ and let $\widehat{V_n}$ denote the completion of $V_n$ at its unique homogeneous maximal ideal. Then
\[
\text{$\widehat{V_n}$ is }
\begin{cases}
\text{Lech-stable if $n\leq 2$,} \\
\text{stable but not Lech-stable if $3\leq n\leq 5$,} \\
\text{semistable but not stable if $n=6$,}\\
\text{not semistable if $n\geq 7,$} \\
\text{not lim-stable if $n\geq 17$.}
\end{cases}
\]
\end{example}
\begin{proof}
First of all, if $n\leq 2$, then $\widehat{V_n}$ is a rational double point and thus Lech-stable by Corollary~\ref{cor: RDP Lech stable}. Secondly, for $3\leq n\leq 6$, note that the defining ideal of $V_n$ is the two by two minors of the matrix
\[
\begin{bmatrix}
z_1 & z_2 & \cdots & z_n\\
z_2 & z_3 & \cdots & z_{n+1}
\end{bmatrix}.
\]
It is easy to see that, under the lexicographic monomial order $z_1>z_2>\cdots > z_{n+1}$, the initial ideal defines the rational polygonal $n$-cone up to a change of variables. Therefore by Corollary~\ref{cor: LM monomial ideal} and Proposition~\ref{prop: Mumford rational cones}, we know that $\widehat{V_n}$ is stable when $3\leq n\leq 5$ and is semistable but not stable when $n=6$. Finally, if $n\geq 7$ then $\eh(\widehat{V_n})=n\geq 7$ so $\widehat{V_n}$ is not semistable by Remark~\ref{rmk: mult bound for Lech-stable and semistable}, while if $n\geq 17$, then $\eh(\widehat{V_n})=n\geq 17$ so $\widehat{V_n}$ is not lim-stable by Corollary~\ref{cor: lim unstable dimension 2}.    
\end{proof}

\begin{remark}
We do not know whether $\widehat{V_n}$ is lim-stable when $7\leq n\leq 16$, nor do we know the values of $\limlm(\widehat{V_n})$ for any $n\geq 7$. We do not know a good method to compute or estimate the (higher) Lech--Mumford constants for quotient singularities in general.
\end{remark}

\subsection{A more general conjecture}\label{subsect: Stirling}
The discussed sharp Lech-type inequalities in dimension two and three 
suggest that in general there should be a sharp improvement,  in terms of the colengths of the projections, of Lech's inequality for monomial ideals in power series rings. The symmetry further implies that the coefficients for projections of the same dimension should be equal, leading us naturally to Conjecture~\ref{conj: best Lech}. 

Let us start with notation. Let $(x)_d\coloneqq x(x-1)\cdots (x-d+1)$. The \emph{unsigned Stirling numbers of the first kind} $\stirlingI{d}{k}$ and the \emph{Stirling numbers of the second kind} $\stirlingII {d}{k}$ arise from the decomposition of  
$(x)_d$ in the standard basis $x^k$ and vice versa: 
\[
(x)_d = \sum_{k = 1}^d (-1)^{d-k}\stirlingI{d}{k} x^k
\text{ and } x^d = \sum_{k = 1}^d \stirlingII {d}{k} (x)_k.
\]

\begin{conjecture}
\label{conj: best Lech}
Let $(R,\m)=k[[x_0, \ldots, x_d]]$ and $I\subseteq R$ be an $\m$-primary monomial ideal. Then 
\begin{align*}
\eh(I) &\leq \sum_{k = 0}^{d} (-1)^k\frac{(d+1-k)!}{\binom{d}{k}}\stirlingII{d+1}{d+1-k} \sum_{1 \leq i_1 < \cdots < i_k \leq d} \length\left(\frac{R}{I+(x_{i_1}, \ldots, x_{i_k})}\right)\\
&= \sum_{k = 0}^{d} (-1)^{d-k}\frac{(k+1)!}{\binom{d}{k}}\stirlingII{d+1}{k+1} \sum_{1 \leq i_1 < \cdots < i_{d-k} \leq d} \length\left(\frac{R}{I+(x_{i_1}, \ldots, x_{i_{d-k}})}\right).
\end{align*}
\end{conjecture}

\begin{remark}
We note that in dimension three (i.e., $d=2$), Conjecture~\ref{conj: best Lech} holds, and this is exactly Theorem~\ref{thm: optimal Lech dim three}. We also point out that, when $d\geq 3$, the coefficients in Conjecture~\ref{conj: best Lech} are not necessarily integers. For example, if $d = 3$, then the Stirling numbers are $1,7,6,1$. Thus the conjecture becomes
\[
\eh(I) \leq 24 \length(R/I) - 12 \sum_{i = 1}^3 \length(R/I+(x_i)) + \frac {14}3 \sum_{1 \leq i < j\leq 3} \length(R/I+(x_i, x_j))  - \length(R/I+(x_1,x_2,x_3)).
\]    
\end{remark}

Our conjecture can be compared to the refinement of Lech's inequality that was conjectured in \cite{HSV}.
\begin{conjecture}[{\cite[Conjecture 2.4]{HSV}}]
\label{conj: HSV}
Let $k$ be an infinite field and $(R, \m)=k[[x_0,\dots,x_d]]$. Set $R_i = R/(\ell_1, \ldots, \ell_i)$ for general linear forms $\ell_1, \ldots, \ell_i \in \m$. Then for all $\m$-primary ideals $I\subseteq R$, we have
\[
\sum_{i = 0}^{d-1} \stirlingI{d+1}{d+1-i} \eh (IR_i) \leq (d+1)!\length (R/I).
\]
\end{conjecture}

\begin{remark}
In \cite{HSV}, the conjecture was stated for all local rings, but it was shown in \cite[Section 3]{HSV} that it reduces to the case of polynomial rings (or power series rings) over an infinite field. 
\end{remark}

We will next show that Theorem~\ref{thm: optimal Lech dim three}
implies Conjecture~\ref{conj: HSV} in dimension three, thus recovering the main result of \cite{HSV}.

\begin{theorem}
\label{thm: conj best implies conj HSV}
Conjecture~\ref{conj: best Lech} implies Conjecture~\ref{conj: HSV}. In particular, Conjecture~\ref{conj: HSV} holds in dimension three.
\end{theorem}
\begin{proof}
Conjecture~\ref{conj: HSV} was reduced to monomial ideals in \cite[Section 3]{HSV}. Henceforth we assume that $I\subseteq R$ is an $\m$-primary
monomial ideal and we work with the left-hand side of Conjecture~\ref{conj: HSV}.
Since $\ell_1,\dots, \ell_i$ are general linear forms, we know that
$$\eh(I, R/(\ell_1, \ldots, \ell_i)) \leq \eh(I, R/(x_{a_1}, \ldots, x_{a_i}))$$ 
for any choice of $1 \leq a_1 < \cdots < a_i \leq d$.
By Conjecture~\ref{conj: best Lech}, we then have
\begin{align*}
\eh(I, &R/(\ell_1, \ldots, \ell_i)) \leq
\sum_{k = 0}^{d-i} (-1)^{d-i-k} \frac{\stirlingII {d+1-i}{k+1}(k+1)!}{\binom{d-i}k } \sum_{b_j} \length\left(\frac{R}{I+ (x_{a_1}, \ldots, x_{a_i}, x_{b_1}, \ldots, x_{b_{d-i-k}})}\right),
\end{align*}
where $b_j$ take values in $\{1, \ldots, d\}\setminus\{a_1, \ldots, a_i\}$.
Averaging these inequalities over various choices of $a_i$ yields that
\begin{align*}
\eh(I, R/(\ell_1, \ldots, \ell_i)) & \leq
\sum_{k = 0}^{d-i} (-1)^{d-i-k} \frac{\binom{d-k}{i}}{\binom{d}{i}}\frac{\stirlingII {d+1-i}{k+1}(k+1)!}{ \binom{d-i}k } \sum_{\substack{i_1 <\cdots <i_{d-k}\\ 1 \leq i_a \leq d}} \length\left(\frac{R}{I+(x_{i_1}, \ldots, x_{i_{d-k}})}\right) \\
& = \sum_{k = 0}^{d-i} (-1)^{d-i-k} \frac{\stirlingII {d+1-i}{k+1}(k+1)!}{ \binom{d}{k} } \sum_{\substack{i_1 <\cdots <i_{d-k}\\ 1 \leq i_a \leq d}} \length\left(\frac{R}{I+(x_{i_1}, \ldots, x_{i_{d-k}})}\right),
\end{align*}
where the equality above used the combinatorial identity $\binom{d}{i}\binom{d-i}{k} = \binom{d-k}{i}\binom{d}{k}$. Therefore
\begin{align*}
\sum_{i = 0}^{d} &\stirlingI{d+1}{d+1-i}  \eh(I, R/(\ell_1, \ldots, \ell_i)) \\
\leq &\sum_{i = 0}^{d} \stirlingI{d+1}{d+1-i}
\sum_{k = 0}^{d-i} (-1)^{d-i-k} \frac{\stirlingII {d+1-i}{k+1}(k+1)!}{ \binom{d}k } \sum_{\substack{i_1 <\cdots <i_{d-k}\\ 1 \leq i_a \leq d}} \length\left(\frac{R}{I+(x_{i_1}, \ldots, x_{i_{d-k}})}\right) \\
= &
\sum_{k = 0}^{d} \frac{(k+1)!}{\binom{d}k}\sum_{\substack{i_1 <\cdots <i_{d-k}\\ 1 \leq i_a \leq d}}  \left (
\sum_{i = 0}^{d-k} (-1)^{d-i-k} \stirlingI{d+1}{d+1-i} \stirlingII {d+1-i}{k+1}
\right )
\length\left(\frac{R}{I+(x_{i_1}, \ldots, x_{i_{d-k}})}\right).
\end{align*}
We need to show that this last expression is equal to $(d+1)!\length(R/I)$,
which amounts to show that the coefficient in the parenthesis is zero for $k < d$. But by the definition of the Stirling numbers we have the identity
\begin{align*}
(x)_{d+1} &= \sum_{i = 0}^{d}  (-1)^{i} \stirlingI{d+1}{d+1-i} x^{d+1-i} = \sum_{i = 0}^{d} (-1)^{i} \stirlingI{d+1}{d+1-i} \sum_{k = 0}^{d-i} \stirlingII{d+1-i}{ k+1} (x)_{k+1} \\
&= \sum_{k = 0}^d \left( \sum_{i = 0}^{d-k} (-1)^{i} \stirlingI{d+1}{d+1-i} \stirlingII{d+1-i}{ k+1} \right) (x)_{k+1}.
\end{align*}
Since $\{(x)_n\}_n$ are linearly independent, it follows that in the last expression above, the coefficient of $(x)_{k+1}$ is zero for $k < d$, which is exactly what we need. The last conclusion then follows from Theorem~\ref{thm: optimal Lech dim three}.
\end{proof}

An application of Theorem~\ref{thm: main Lech incl/excl}
gives us the following weaker version of Conjecture~\ref{conj: HSV}. 


\begin{proposition}
Let $k$ be an infinite field and $(R, \m)=k[[x_0,\dots,x_d]]$. Let $R_1 = R/(\ell_1)$ for a general linear form $\ell_1 \in \m$. Then for all $\m$-primary ideals $I\subseteq R$, we have
\[
\eh(I) + \stirlingI{d+1}{d} \eh(IR_1) \leq (d+1)! \length(R/I).
\]
Moreover, the inequality is strict if $d \geq 2$.
\end{proposition}
\begin{proof}
By the same argument as in \cite[Section 3]{HSV}, we may assume that $I\subseteq R$ is an $\m$-primary monomial ideal.
Since $R_1=R/(\ell_1)$ for a general linear form $\ell_1$, we know that $\eh(IR_1) \leq \eh (I, R/(x_i))$ for every $i$. Thus, after
applying Theorem~\ref{thm: main Lech incl/excl} to $R/(x_i)$ for $i = 1, \ldots, d$, we obtain that
\[
\eh(IR_1) \leq
d! \Bigg(\length(R/I+(x_i)) + \sum_{k = 1}^{d-1}\frac{(-1)^{k}}{2^{k}}
\sum_{\substack{1 \leq j_1 < \cdots < j_k \leq d \\j_n \neq i}} \length\left(\frac{R}{I+(x_i, x_{j_1}, \ldots, x_{j_k})}\right)\Bigg).
\]
Summing these up and noting that each $\length(R/I+(x_{j_1}, \ldots, x_{j_k}))$ will appear exactly $k$ times (one for each $R/(x_{j_n})$) we obtain the inequality 
\begin{equation}\label{eqn: 7.3.6 bound}
d \eh(IR_1) \leq d! \left(\sum_{k = 1}^{d} \frac{(-1)^{k-1}k}{2^{k-1}}
\sum_{1 \leq j_1 < \cdots < j_k \leq d} \length\left(\frac{R}{I+(x_{j_1}, \ldots, x_{j_k})}\right) \right)    
\end{equation}
Since $\stirlingI{d+1}{d} = \binom{d+1}{2}$, we may combine  
$(\ref{eqn: 7.3.6 bound})$ 
with the bound on $\eh(I)$ given by Theorem~\ref{thm: main Lech incl/excl} to derive the inequality  
\begin{align*}
\frac{\eh(I) + \stirlingI{d+1}{d} \eh(IR_1)}{(d+1)!}
\leq 
\length(R/I) + \sum_{k = 1}^{d}\left(\frac{-1}{2}\right)^{k}(1-k)\mspace{-25mu}\sum_{\substack{\quad \\ 1 \leq j_1 < \dots < j_k \leq d}} \mspace{-10mu} \length\left(\frac{R}{I+(x_{j_1}, \ldots, x_{j_k})}\right).
\end{align*}
Since this inequality is strict if $d\geq 2$ by Theorem~\ref{thm: main Lech incl/excl}, it remains to show that the right-hand side of the inequality is no greater than $\length(R/I)$. We will use the method of Proposition~\ref{prop: combinatorial incl/excl}. Since 
\[
\sum_{k = 1}^s \frac {k(-1)^{k-1}}{2^k} \binom sk = \frac s2 \sum_{k = 1}^s \left(\frac{-1}{2} \right)^{k-1} \binom{s-1}{k-1} = \frac {s}{2^s},
\]
it is easy to deduce that 
$$\sum_{k = 1}^s (-1)^{k}\frac {1-k}{2^k}\binom{s}{k} = \frac{s + 1}{2^s} - 1 \leq 0.$$ 
By the corresponding inclusion-exclusion formula (see the proof of Proposition~\ref{prop: combinatorial incl/excl}) we may now show that
$$\sum_{k = 1}^{d}\frac{(-1)^{k}}{2^k}(1-k)\sum_{1\leq j_1 < \cdots < j_k \leq d} \length\left(\frac{R}{I+(x_{j_1}, \ldots, x_{j_k})}\right) \leq 0.$$
This completes the proof.
\end{proof}

\subsection{Three points in $\mathbb{P}^2$} 
We believe that there exists a local ring $R$ such that 
\[
\lm(R[[T_1, \ldots, T_n]]) > \lm(R[[T_1, \ldots, T_{n+1}]]) \text{ for all $n$}.
\]
As an evidence, we will show that the coordinate ring $R$ of three points in $\mathbb{P}^2$ is a one-dimensional local ring $(R,\m)$ such that $\lm(R) > \lm(R[[U]]) > \lm(R[[U, V]])$. It turns out that even in this seemingly simple case, estimating $\lm(R[[U, V]])$ is not easy and requires nontrivial new estimates on the mixed multiplicities in a power series ring in two variables. We have not been able to progress this example further due to the lack of such estimates in higher dimensions. 

We begin with the following variation of Lemma~\ref{lem: mixed multiplicity dim two}.

\begin{lemma}
\label{lem: another mixed multiplicity dim two}
Let $(R,\m)=k[[x,y]]$ and let $I\subseteq J \subseteq R$ be two $\m$-primary monomial ideals. Then we have
    \[
        \eh (I|J) \leq \frac 1 2 \length (R/I) + 3 \length (R/J).
    \]
\end{lemma}
\begin{proof}
By \cite[Exercise 17.8]{SwansonHuneke}, we know that $\eh(I | J)= \length(R/IJ) - \length(R/I)- \length(R/J)$, so 
it suffices to show that
    $\length (R/IJ) \leq \frac{3}2 \length (R/I) + 4 \length (R/J)$.
We now follow the notation of the proof of Lemma~\ref{lem: mixed multiplicity dim two} and express 
    \[
        \length (R/IJ) = \sum_{t = 0}^{2c+1}
        \min \{r_i + s_j \mid i + j = t\}
        \leq  \sum_{t = 0}^{c}
        \min \{r_i + s_j \mid i + j = t\}
        + \sum_{j = 0}^c s_j.
    \]
    Therefore it suffices to show that 
    \[
    L \coloneqq \sum_{t = 0}^{c}
        \min \{r_i + s_j \mid i + j = t\} 
        \leq \frac 3 2 \sum_{j = 0}^c r_j
        + 3 \sum_{j = 0}^c s_j.
    \]
    We demonstrate this inequality using the elementary inequality $\min\{a + b, c+ d \} \leq (a+b + c + d)/2$ and grouping the terms in the following way: 
    \begin{align*}
   L &= \sum_{t = 0}^{c}  \min \{r_i + s_j \mid  i + j = t\} \\
   &\leq  r_0 + s_0 + \frac{r_0 + s_1 + r_1 + s_0}{2}
    + \frac{r_1 + s_1 + r_2 + s_0}{2}
    + \frac{r_1 + s_2 + r_2 + s_1}{2}
    \\ &\phantom{<} + \frac{r_2 + s_2 + r_3 + s_1}{2}
    + \frac{r_3 + s_2 + r_4 + s_1}{2}
    + \frac{r_3 + s_3 + r_4 + s_2}{2}
    + \frac{r_4 + s_3 + r_5 + s_2}{2}
    \\&\phantom{<} + \frac{r_5 + s_3 + r_6 + s_2}{2}
    + \frac{r_5 + s_4 + r_6 + s_3}{2}
    + \frac{r_6 + s_4 + r_7 + s_3}{2}
    + \frac{r_7 + s_4 + r_8 + s_3}{2} + \cdots.
    \end{align*}
    Direct inspection shows that each $r_i$ appears with coefficient at most $3/2$ and each $s_i$ appears with coefficient at most $3$. Thus $L \leq \frac 3 2 \sum_{j = 0}^c r_j
        + 3 \sum_{j = 0}^c s_j$ as desired. 
\end{proof}

We have the following variant of Theorem~\ref{thm: optimal Lech dim three}.

\begin{proposition}
\label{prop: skew Lech}
    Let $(S,\m) = k[[x,y,z]]$ and let $I\subseteq S$ be an $\m$-primary ideal. Then
    \[
        \eh(I) \leq \frac {15} 2\length (S/I) - 5 \length (S/I+(z)).
    \]
\end{proposition}
\begin{proof}
The strategy is completely similar to the proof of Theorem~\ref{thm: optimal Lech dim three}. We first handle the case that $I$ is a monomial ideal. We set $R\coloneqq k[[x, y]]$ and consider monomial ideals 
$I_j\coloneqq  \{m \in R \mid mz^j\in I \}.$
By Proposition~\ref{prop: Mumford 4.3}, we have 
\begin{align}\label{eq for 6.4.2}
\eh(I) \leq \eh(I_0) + 2\eh(I_1) + \cdots + 2\eh(I_c) + \eh(I_0 | I_1) + \cdots + \eh(I_{c-1} | I_c).  
\end{align}
Using Lemma~\ref{lem: another mixed multiplicity dim two} for the mixed multiplicity terms and Lech's inequality (Theorem~\ref{thm Lech}) for the usual multiplicity terms, we bound the right-hand side of (\ref{eq for 6.4.2}) by
    \begin{align*}
    \eh(I) \leq \text{RHS} &\leq 2\length(R/I_0) + 4 \sum_{j=1}^c \length(R/I_j) + \sum_{j=0}^c \big(\frac{1}{2}\length(R/I_j) + 3\length(R/I_{j+1}) \big) \\
    &\leq \frac{5}{2} \length(R/I_0) + \frac{15}{2} \sum_{j=1}^c\length(R/I_j) = \frac{15}{2} \sum_{j=0}^c\length(R/I_j) - 5\length(R/I_0) \\
    &\leq \frac{15}{2} \length(S/I) - 5\length(S/I +(z)) .
    \end{align*}
This completes the proof when $I$ is a monomial ideal. 

To prove the general case, we note that for any monomial order $\prec$ we have the containment $\init_{\prec}(I) \subseteq \init_{\prec}(I, z)$ and thus
$\eh(I) \leq \eh(\init_{\prec}(I))$, $\length (S/I) = \length (S/\init_{\prec}(I))$, 
and $\length (S/(\init_{\prec}(I), z) ) \geq 
\length (S/\init_{\prec}(I, z)) = \length (S/(I,z))$. It follows that
\begin{align*}
\eh(I) \leq \eh(\init_{\prec}(I)) \leq \frac{15}{2} \length(S/\init_{\prec}(I)) - 5 \length(S/\init_{\prec}(I) +(z))\leq  \frac {15} 2\length (S/I) - 5 \length (S/I+(z))
\end{align*}
where the second inequality follows from the already established monomial case. 
\end{proof}

We are now ready to study the (higher) Lech--Mumford constants of the coordinate ring of three points in $\mathbb{P}^2$.

\begin{example}
\label{example: double drop}
Let $(R,\m) = k[[x,y,z]]/(x,y)\cap (y,z) \cap (x,z)$. Then we have 
\[
\begin{cases}
\lm(R)=3; \\
\lm(R[[U]]) = 3/2; \\
\lm(R[[U, V]]) \leq 5/4.
\end{cases}
\]
\end{example}
\begin{proof}
Since $\dim(R)=1$, $\lm(R)=\eh(R)=3$ follows from Proposition~\ref{prop: basic results on cLM}(\ref{part LM dim 1}). Next  we verify that $\lm(R[[U]]) = 3/2$. 
Taking $I = (x,y,z,U)$ shows that $\lm(R[[U]]) \geq 3/2$, 
so it suffices to verify that 
$\eh (I) \leq 3 \length (R[[U]]/I)$ for all monomial $\m+(U)$-primary ideals by Proposition~\ref{prop: LM multigraded}.
Now by Corollary~\ref{cor: monomial inc/exc} and Corollary~\ref{cor: Mumford Lemma 3.15}, we have
\begin{align*}
\eh (I) 
&= \eh (I, R[[U]]/(x,y)) + \eh (I, R[[U]]/(x,z)) + \eh (I, R[[U]]/(y,z)) 
\\ &\leq 
2 \left ( \length \left(\frac{R[[U]]}{I + (x,y)} \right)
+ \length \left (\frac{R[[U]]}{I+(x,z)} \right) + \length \left (\frac{R[[U]]}{I+(y,z)} \right)
\right ) 
- 3 \length \left (\frac{R[[U]]}{I+(x,y,z)}\right )\\
& \leq  3 \left ( \length \left (\frac{R[[U]]}{I + (x,y)}\right )
+ \length \left (\frac{R[[U]]}{I+(x,z)} \right) + \length \left (\frac{R[[U]]}{I+(y,z)} \right)
\right ) 
- 6 \length \left (\frac{R[[U]]}{I+(x,y,z)} \right) \\
& = 3 \length (R[[U]]/I).
\end{align*}
This completes the proof that $\lm(R[[U]])=3/2$.
    
Finally, we estimate $\lm(R[[U, V]])$. We set $S\coloneqq k[[x,y,z,U, V]]$. By Proposition~\ref{prop: LM multigraded} it suffices to consider monomial $(\m+(U,V))$-primary ideals $I\subseteq S$. By Corollary~\ref{cor: monomial inc/exc}, we have
\[
\length (S/I) = \sum_{i = 1}^3 \length (S/I+P_i) - 2 \length (S/I+(x,y,z))
\]
where $P_1=(x,y)S$, $P_2=(y,z)S$, $P_3=(x,z)S$. Since $S/P_i$ is a power series ring in three variables, by Proposition~\ref{prop: skew Lech} we have
\[
\eh (I, S/P_i)
\leq \frac {15}{2} \length (S/I+P_i) - 5 \length (S/I+(x,y,z)). 
\]
Since $\eh (I) = \sum_{i = 1}^3 \eh (I, S/P_i)$, it follows that 
\[
\eh(I) 
\leq \frac {15}{2} \sum_{i = 1}^3 \length (S/(I,P_i)) 
- 15 \length (S/I+(x,y,z))
= \frac{15}{2} \length (S/I).
\]
Therefore 
$\lm(R[[U,V]])\leq 3! \times 15/2 = 5/4$ as wanted.
\end{proof}

\begin{remark}
One sees a drastic difference between the (higher) Lech--Mumford constants of $R=k[[x,y,z]]/(x,y)\cap (y,z) \cap (x,z)$ and $R'=k[[x,y]]/(x^3+y^3)$: while $\lm(R[[t_1,t_2]])\leq 5/4$ (and thus $\limlm(R)\leq 5/4$) by Example~\ref{example: double drop},   by Remark~\ref{rmk: lim LM of x^3+y^3} one has $\limlm(R')\geq 1.26>5/4$ (and thus $\lm(R'[[t_1,\dots,t_n]])>5/4$ for all $n$).
\end{remark}

\subsection{Semi-log canonical hypersurfaces in dimension two}
Two-dimensional semi-log canonical hypersurface singularities over $\mathbb{C}$ are classified, see for example \cite[Table 1]{LiuRollenske}. Using this classification, 
we will show in this subsection that most of these singularities are semistable. In fact, we suspect that all of them are semistable, but our techniques cannot handle the case of simple elliptic singularities,\footnote{In \cite[page 79, Case 3 a)]{Mumford}, it was stated without proof that cones over nonsingular elliptic curves are semistable. We have not been able to verify this claim. A concrete example that we do not know how to show semistability is the singularity $\mathbb{C}[[x,y,z]]/(x^3+y^3+z^3)$.} see Conjecture~\ref{conj: elliptic curve semistable}.

We start by recalling the classification. In what follows, all singularities are hypersurfaces of dimension two over $\mathbb{C}$, i.e., (up to completion) they can be written as $\mathbb{C}[[x,y,z]]/(f)$. 
\begin{enumerate}
    \item[(0)] Canonical singularities are either regular or ADE type singularities, in particular, they are Lech-stable by Theorem~\ref{thm: Lech-stable CM surface}.
\end{enumerate}
Next, log canonical but not canonical singularities consists of the next four families, the first three are simple elliptic singularities, and the fourth family is called cusp singularities:
\begin{enumerate}
\item[(1)] $X_{1,0}\colon$ $f=x^2 + y^4 + z^4 + \lambda xyz$, $\lambda^4 \neq 64$,
\item[(2)] $J_{2,0}\colon$ $f=x^2 + y^3 + z^6 + \lambda xyz$, $\lambda^6 \neq 432$,
\item[(3)] $T_{3,3,3}\colon$ $f=x^3 + y^3 + z^3 + \lambda xyz$, $\lambda^3 \neq -27$,
\item[(4)] $T_{p,q,r}$: $f=xyz + x^p + y^q + z^r$, $\frac 1p + \frac 1q + \frac 1r < 1$.
\end{enumerate}
Lastly, semi-log canonical but not log canonical singularities are classified as follows. First, we have normal crossings $A_\infty$ and pinched point $D_\infty$:
\begin{enumerate}
\item[(5)] $A_\infty$: $f=x^2+y^2$,
\item[(6)] $D_\infty$: $f=x^2+y^2z$.
\end{enumerate}
The rest semi-log canonical singularities are degenerate cusps:
\begin{enumerate}
\item[(7)] $T_{2, \infty, \infty}$: $f = x^2 + y^2z^2$,
\item[(8)] $ T_{2, q, \infty}$: $f = x^2 + y^2(z^2 + y^{q-2})$, $q \geq 3$,
\item[(9)] $T_{\infty, \infty, \infty}$: $f = xyz$,
\item[(10)] $T_{p, \infty, \infty}$: $f = xyz + x^p$, $p \geq 3$,
\item[(11)] $T_{p,q, \infty}$: $f = xyz + x^p + y^q$, $q \geq p \geq 3$.
\end{enumerate}

\begin{theorem}\label{thm: slc semistable}
All the singularities in the families (4)--(11) of the classification exhibited above are semistable.
\end{theorem}
\begin{proof}
We first handle the families (5)--(11). $A_\infty$ and $D_\infty$ are semistable by Theorem~\ref{thm: Lech-stable CM surface} (they are even Lech-stable). $T_{2,\infty, \infty}$ is semistable by Proposition~\ref{prop: Mumford elliptic cones} $n=2$ case (see Remark~\ref{rmk: Mumford equation of elliptic cones}). For $T_{2,q,\infty}$, by assigning $x,y,z$ weights $4, 3, 1$ respectively, the initial term is $x^2+y^2z^2$ which is semistable by the $T_{2,\infty, \infty}$ case. Thus $T_{2,q,\infty}$ is semistable by Corollary~\ref{cor: grober degeneration}. $T_{\infty, \infty, \infty}$ is semistable by Theorem~\ref{thm: snc is semistable}. For $T_{p, \infty, \infty}$ and $T_{p, q, \infty}$, we can assign $x,y,z$ weights $2, 2, 1$ so that the initial term is $xyz$, which is semistable by the $T_{\infty,\infty, \infty}$ case. Thus $T_{p, \infty, \infty}$ and $T_{p, q, \infty}$ are also semistable by Corollary~\ref{cor: grober degeneration}. 

We now handle the family (4), which follows from a degeneration very similar to above. We may assume $p\leq q\leq r$. If $p=2$ and $q=3$, then $r\geq 7$. So by assigning $x,y,z$ weights $3, 2, 1$, the initial term is $xyz+x^2+y^3$, which is semistable by Proposition~\ref{prop: Mumford elliptic cones} $n=1$ case. Thus $T_{2,3,r}$ is semistable by Corollary~\ref{cor: grober degeneration}. If $p=2$ and $q=4$, then $r\geq 5$. By assigning $x,y,z$ weights $2, 1, 1$, the initial term is $xyz+x^2+y^4$, but then assigning $x,y,z$ weights $3,2,1$, the initial term is $xyz+x^2$, which is semistable by Proposition~\ref{prop: Mumford elliptic cones} $n=2$ case. Thus $T_{2,4,r}$ is semistable by Corollary~\ref{cor: grober degeneration} (applied twice). If $p=2$ and $q,r\geq 5$, then by assigning $x,y,z$ weights $2, 1, 1$, the initial term is $xyz+x^2$, which is semistable by Proposition~\ref{prop: Mumford elliptic cones} $n=2$ case. Thus $T_{2,q,r}$ is semistable by Corollary~\ref{cor: grober degeneration}. Finally, if $p\geq 3$, then by assigning $x,y,z$ weights $1, 1, 1$, the initial term becomes either $T_{3, \infty, \infty}$ or $T_{3.3, \infty}$, which are semistable by the already established families (10) and (11). Thus they are semistable by Corollary~\ref{cor: grober degeneration} again. This completes the proof.
\end{proof}


Since our technique is essentially to degenerate the singularities to the monomial case where explicit computations are available, it does not seem to work for the simple elliptic singularities. However, we suspect that all two-dimensional semi-log canonical hypersurface singularities are semistable, see also \cite[page 79]{Mumford}. The following conjecture is all that remains to show.

\begin{conjecture}
\label{conj: elliptic curve semistable}
All the singularities in the families (1)--(3) above are semistable.
\end{conjecture}

More generally, we do not know whether the following could be true.

\begin{question}
Suppose $R$ is essentially of finite type over a field of characteristic zero that is $\Q$-Gorenstein and log canonical (or merely semi-log canonical), and $\eh(R) \leq \dim(R) + 1$. Then is $R$ semistable?
\end{question}

\newpage

\bibliographystyle{alpha}
\bibliography{refs}

\end{document}